\documentclass[a4paper,11pt]{book}
\usepackage{graphicx}
\usepackage{amsmath}
\usepackage{amsthm}
\usepackage{amsfonts}
\usepackage{color}
\usepackage{tikz}
\usepackage{pgfplots}
\usepackage[a4paper]{geometry}
\usepackage{layout}
\usepackage{rotating}
\usepackage{hyperref}
\usepackage{emptypage}
\newtheorem{theorem}{Theorem}
\newtheorem{corollary}{Corollary}
\newtheorem{proposition}{Proposition}
\newtheorem{definition}{Definition}

\newenvironment{customlegend}[1][]{%
    \begingroup
    \csname pgfplots@init@cleared@structures\endcsname
    \pgfplotsset{#1}%
}{%
    \csname pgfplots@createlegend\endcsname
    \endgroup
}%

\def\addlegendimage{\csname pgfplots@addlegendimage\endcsname}

\pgfkeys{/pgfplots/number in legend/.style={%
        /pgfplots/legend image code/.code={%
            \node at (0.125,-0.0225){#1}; 
        },%
    },
}
\pgfplotsset{
every legend to name picture/.style={west}
}

\begin{document}
\frontmatter


\newgeometry{margin=1in}

\begin{titlepage}
\centering

\begin{figure}[!h]
\centering
	$\vcenter{\hbox{\includegraphics[height=1cm]{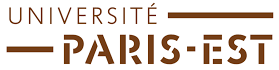}}}$
	\qquad \qquad
	$\vcenter{\hbox{\includegraphics[height=2cm]{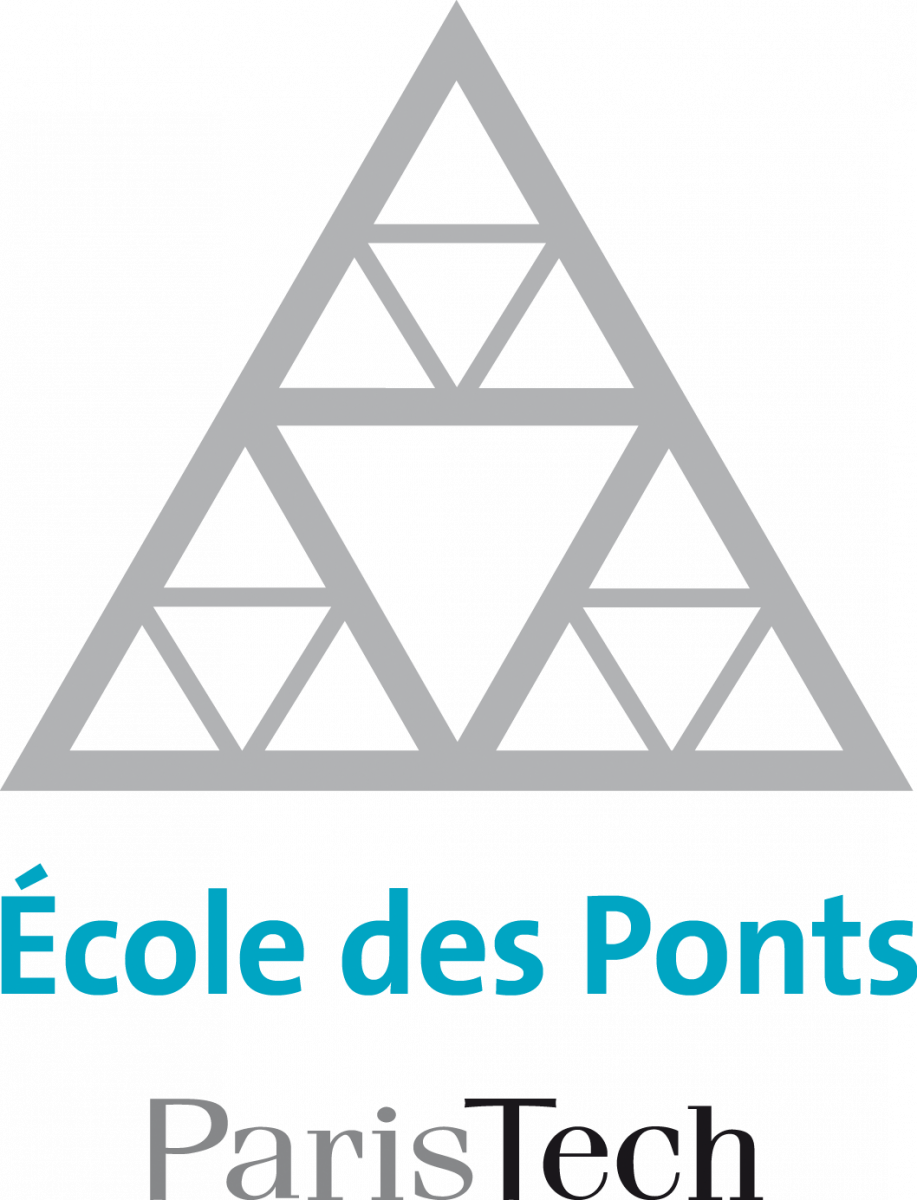}}}$
	\qquad \qquad
	$\vcenter{\hbox{\includegraphics[height=1cm]{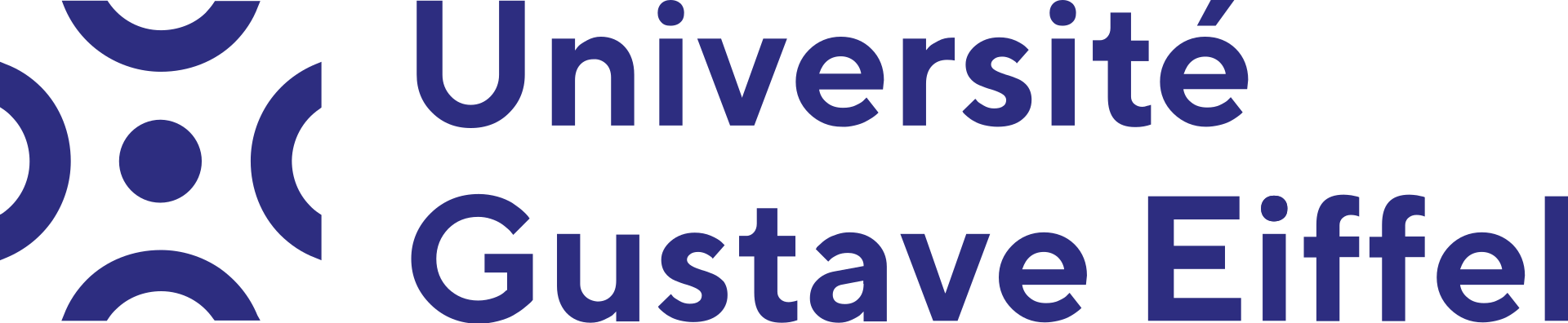}}}$
\end{figure}

\vspace{1cm}
{\fontsize{14pt}{17pt}\selectfont\textsf{École doctorale Ville, Transports et Territoires} \par}
\vspace{1 cm}
{\fontsize{14pt}{17pt}\selectfont \textsf{Th\`ese de doctorat d'Universit\'e Paris-Est\\
 dans le domaine des Sciences de l'ing\'enieur} \par}
\raggedright
\vspace{2.5cm}
{\fontsize{20pt}{20pt}\selectfont \textsf{Florian Sven Schanzenb\"acher}\par}
\vspace{1cm}
{\fontsize{30pt}{30pt}\selectfont \textsf{Modélisation max-plus du trafic ferroviaire des voyageurs sur une ligne à fourche : diagramme fondamental et contrôle dynamique} \par}
\vspace{1 cm}
{\fontsize{14pt}{14pt}\selectfont \textsf{Soutenue le 05 juin 2020}\par}
\vfill
\resizebox{\textwidth}{!} {
\begin{tabular}{ll}
\fontsize{14pt}{16pt}\selectfont \textsf{\textbf{Jury de thèse de docteur}}&\\
&\\
\fontsize{13pt}{16pt}\selectfont \textsf{Président du jury :}  &\fontsize{13pt}{16pt}\selectfont \textsf{Prof. Rob Goverde - TU Delft}\\
&\\
\fontsize{13pt}{16pt}\selectfont \textsf{Directeur de thèse :}  &\fontsize{13pt}{16pt}\selectfont \textsf{Prof. Fabien Leurent, ing\'enieur général des Ponts, des Eaux et For\^ets}\\
&\\
\fontsize{13pt}{16pt}\selectfont \textsf{Encadrent :}  &\fontsize{13pt}{16pt}\selectfont \textsf{Nadir Farhi, chargé de recherche - Université Gustave Eiffel}\\
&\\
\fontsize{13pt}{16pt}\selectfont \textsf{Rapporteur :} &\fontsize{13pt}{16pt}\selectfont \textsf{Anne Bouillard, ingénieur de recherche - Huawei Technologies}\\
&\\
\fontsize{13pt}{16pt}\selectfont \textsf{Rapporteur :} &\fontsize{13pt}{16pt}\selectfont \textsf{Prof. Francesco Corman - ETH Zürich}\\
&\\
\fontsize{13pt}{16pt}\selectfont \textsf{Examinateur :} &\fontsize{13pt}{16pt}\selectfont \textsf{Marianne Akian, directrice de recherche - INRIA}\\
&\\
\fontsize{13pt}{16pt}\selectfont \textsf{Examinateur :} &\fontsize{13pt}{16pt}\selectfont \textsf{Prof. Ahmed Nait-Sidi-Moh - Université de Picardie Jules Verne}\\
&\\
\fontsize{13pt}{16pt}\selectfont \textsf{Examinateur :} &\fontsize{13pt}{16pt}\selectfont \textsf{Paola Pellegrini, directrice de recherche - Université Gustave Eiffel}\\
&\\
\fontsize{13pt}{16pt}\selectfont \textsf{Invité :} &\fontsize{13pt}{16pt}\selectfont \textsf{Gérard Gabriel - RATP}
\end{tabular}
}
\end{titlepage}
\restoregeometry
\cleardoublepage

\subsection*{\centering{Publications list}}
\begin{description}
\item F. Schanzenbächer, N. Farhi, Z. Christoforou, F. Leurent and G. Gabriel. ``A discrete event traffic model explaining the traffic phases of the train dynamics in a metro line system with a junction.'' In \textit{Proceedings of the IEEE Annual Conference on Decision and Control}, pages 6283-6288, Melbourne, Australia, 2017.
\item F. Schanzenbächer, N. Farhi, F. Leurent, and G. Gabriel. ``A discrete event traffic model explaining the traffic phases of the train dynamics on a linear metro line with demand-dependent control.'' In \textit{Proceedings of the Annual American Control Conference}, pages 6335-6340, Milwaukee, USA, 2018.
\item F. Schanzenbächer, N. Farhi, F. Leurent, and G. Gabriel. ``Real-time control of metro train dynamics with minimization of train time-headway variance.'' In \textit{Proceedings of the IEEE Intelligent Transportation Systems Conference}, pages 2747-2752, Maui~(Hawaii), USA, 2018.
\item F. Schanzenbächer, N. Farhi, F. Leurent, and G. Gabriel. ``Comprehensive passenger demand-dependent traffic control on a metro line with a junction and a derivation of the traffic phases.'' In \textit{Annual Meeting online of the Transportation Research Board}. Washington D.C., USA, 2019.
\item F. Schanzenbächer, N. Farhi, F. Leurent, and G. Gabriel. ``A discrete event traffic model for passenger demand-dependent train control in a metro line with a junction.'' In \textit{Proceedings of the ITS World Congress}. Singapore, 2019.
\item F. Schanzenbächer, N. Farhi, Z. Christoforou, F. Leurent, and G. Gabriel. ``Demand-dependent supply control on a linear metro line of the RATP network.'' In \textit{Proceedings of the International Scientific Conference on Mobility and Transport, Transportation Research Procedia}, pages 491-493, Munich, Germany, 2019.
\item F. Schanzenbächer, N. Farhi, F. Leurent, and G. Gabriel. ``Feedback control for metro lines with a junction.'' \textit{IEEE Transactions on Intelligent Transportation Systems}, 2020.
\end{description}

\cleardoublepage

\subsection*{\centering{Abstract en français}}
Cette thèse propose des modèles mathématiques de trafic et des lois de contrôle pour des lignes de métro à fourche. Les modèles sont basés sur ceux 
des lignes linéaires (sans fourche) dans~\cite{FNHL17a,FNHL16}.
La dynamique du trafic des trains est modélisée par un système à événements discrets, avec deux contraintes. 
La première impose une borne inférieure au temps de parcours et au temps de stationnement.
La deuxième impose une borne inférieure au temps de sécurité entre deux trains.
Un modèle de la dynamique du trafic sur la fourche est proposé, ainsi que des lois de contrôle pour les temps de parcours et les temps de stationnement, en fonction de l'affluence voyageur.
La plupart des modèles sont écrits comme systèmes linéaires en algèbre max-plus (algèbre des polynômes matricielles), ce qui permet la caractérisation du régime stationnaire et la dérivation analytique des diagrammes de phase du trafic.

Dans tous les modèles de ce manuscrit, la ligne de métro est discrétisée (dans l'espace) en segments (cantons). Dans le modèle du Chapitre~\ref{1A}, les temps de parcours, de stationnement et de sécurité respectent des bornes inférieures, imposées par segment.
La dynamique du trafic hors fourche est modélisée comme dans le cas des lignes linéaires dans~\cite{FNHL17a,FNHL16}. Une contribution principale du Chapitre~\ref{1A} est le modèle de la dynamique à la fourche.
Il est montré que le modèle de la dynamique du trafic de la ligne entière (avec fourche) est linéaire en algèbre max-plus et que la dynamique atteint un régime stationnaire.
Le taux de croissance moyen asymptotique de la dynamique, interprété ici comme intervalle de temps moyen asymptotique, est dérivé analytiquement.
Il est donné en fonction des temps de parcours, de stationnement et de sécurité, ainsi que du nombre de trains et de la différence entre le nombre de trains sur les branches.
Ce résultat permet d'obtenir des diagrammes de phases de la dynamique du trafic, appelés diagrammes fondamentaux, comme en trafic routier.
Huit phases de trafic sont dérivées analytiquement et interprétées en termes de trafic.
Basé sur cette dérivation analytique, des lois de contrôle macroscopiques sont proposés pour la régulation du trafic sur une ligne à fourche.

Le Chapitre~\ref{Chap-4} propose une extension du modèle du Chapitre~\ref{1A}.
Dans la première section, les temps de parcours et de stationnement du modèle de la dynamique du trafic sur une ligne linéaire~\cite{FNHL17a,FNHL16} sont modélisés en fonction de l'affluence voyageur.
Il est montré que la dynamique reste linéaire en algèbre max-plus. Le régime stationnaire est caractérisé et le taux de croissance moyen asymptotique de la dynamique (intervalle moyen asymptotique) est dérivé analytiquement en fonction des paramètres de la ligne (bornes inférieures aux temps de parcours, de stationnement et de sécurité), ainsi que du nombre de trains et de l'affluence voyageur.
Des marges sur les temps de parcours peuvent être utilisées pour réaliser des temps de stationnement en fonction de l'affluence voyageur. 
Dans la deuxième section du Chapitre~\ref{Chap-4}, cette extension est appliquée au modèle pour une ligne à fourche, combinant ainsi les modèles des Chapitres~\ref{1A} et~\ref{Chap-4} (première section).
Les phases de trafic sont dérivées analytiquement, caractérisées par l'intervalle moyen asymptotique en fonction des paramètres de la ligne ainsi que du nombre de trains, de la différence entre le nombre de trains sur les deux branches, et de l'affluence voyageur.

Finalement, le Chapitre~\ref{sim} propose trois cas de simulation sur la ligne 13 du métro parisien (avec fourche).
Le premier cas montre le contrôle macroscopique du nombre de trains en fonction de l'affluence voyageur.
Le deuxième cas montre le contrôle macroscopique du nombre de trains sur les branches face à une perturbation sur les temps de parcours et/ou les temps de stationnement. 
Le troisième cas considère un état intitial avec des intervalles perturbés et simule la dynamique avec des temps de stationnement et de parcours contrôlés (régulés) en fonction de l'affluence voyageur.
Le contrôle proposé garantit une harmonisation des intervalles sur l'ensemble de la ligne.\\~

\noindent
\textbf{Mots-clés:} Théorie du trafic, systèmes à événements discrets, physique du trafic, modélisation du trafic ferroviaire, contrôle du trafic.

\title{\textsf{Max-plus modeling of traffic on passenger railway lines with a junction: fundamental diagram and dynamic control}}
\author{\textsf{Florian Sven Schanzenb\"acher}}
\date{\textsf{}}
\maketitle

\cleardoublepage
\setcounter{page}{9}

\subsection*{\centering{Abstract in English}}
This thesis proposes mathematical traffic models and control laws for metro lines with one junction. The models are based on the ones developed for linear metro lines (without junction) in~\cite{FNHL17a,FNHL16}.
The train dynamics are described with a discrete event traffic model.
Two time constraints are considered. 
The first one imposes lower bounds on the train run and dwell times. The second one fixes a lower bound on the safe separation time between two trains. A model of the train dynamics on the junction is proposed, as well as
control laws for the train run and dwell times, as a function of the passenger travel demand. 
Most of these models are written as linear systems in the max-plus algebra (polynomial matrix algebra), which permits the characterization of the stationary regime, and the derivation of traffic phase diagrams.

In all the models considered here, the metro line is discretized (in space) in
a number of segments (blocks).
In the model of Chapter~\ref{1A}, the train run, dwell, and safe separation times are lower-bounded on every segment. 
The train dynamics out of the junction are modeled as in the case of linear metro lines~\cite{FNHL17a,FNHL16}. One of the main contributions of Chapter~\ref{1A}
is the model of the train dynamics on the junction.
The train dynamics model for the entire line (with the junction) is shown to be linear in the max-plus algebra 
and is shown to reach a stationary regime. The asymptotic average growth rate of the train dynamics, interpreted as the asymptotic average train time-headway, is derived analytically.
It is given as a function of the train run, dwell, and safe separation times, and of the total number of trains, as well as the difference between the number of trains on the two branches. 
This derivation permits to obtain phase diagrams of the train dynamics, called here fundamental diagrams, as in road traffic. 
Eight traffic phases of the train dynamics are derived analytically and interpreted in terms of traffic. 
Moreover, based on the closed-form solutions of the traffic phases, macroscopic control laws are proposed for the traffic on a line with a junction. 

Chapter~\ref{Chap-4} proposes an extension of the model of Chapter~\ref{1A}.
In the first section, the model of the train dynamics on a linear metro line~\cite{FNHL17a,FNHL16} is extended with the run and dwell times as functions of the passenger travel demand.
It is shown that the train dynamics remain linear in the max-plus algebra.
The stationary regime is characterized, and the asymptotic average growth rate
of the train dynamics (asymptotic average train time-headway) is derived analytically,
as a function of the parameters of the line (lower bounds on the run, dwell, and safe separation times), and of the total number of trains, as well as the passenger travel demand. It is suggested that margins on the train run times can be used to extend train
dwell times at platforms in case of train delays, which improves the robustness
of the train dynamics. In the second section of Chapter~\ref{Chap-4}, this extension is applied to a line with a junction, combining the models of
Chapter~\ref{1A} and Chapter~\ref{Chap-4} (section~1).
Similarly, the traffic phases of the train dynamics are derived analytically, giving the asymptotic average train time-headway as a function of the parameters of the line 
and of the total number of trains, the difference between the number of trains on the two branches, as well as the passenger travel demand.


Finally, Chapter~\ref{sim} proposes three simulation cases illustrated on metro line 13 of Paris (with one junction). The first case illustrates the macroscopic control on the number of trains depending on the passenger travel demand volume. 
The second case shows the macroscopic control of the number of trains on the branches in case of a perturbation on the train travel times. 
The third case gives a simulation of the demand-dependent train dynamics with perturbed initial time-headways and shows how an additional dwell time control harmonizes the train time-headways.\\~ 


\noindent
\textbf{Keywords:} Traffic flow theory, discrete event systems, physics of traffic, railway traffic modeling, traffic control.

\clearpage
\shipout\null
\stepcounter{page}

{\raggedright
\null
\vfill
This PhD thesis has been prepared at:\par
\vspace{.7cm}
{\bfseries ENPC - LVMT}\\
6-8 Avenue Blaise Pascal\\
Cit\'e Descartes, Champs-sur-Marne\\
F-77455 Marne-la-Vall\'ee cedex 2\par
\vspace{1cm}
{\bfseries Université Gustave Eiffel - GRETTIA}\\
14-20 Boulevard Newton\\
Cit\'e Descartes, Champs-sur-Marne\\
F-77447 Marne-la-Vall\'ee cedex 2\par
\vspace{1cm}
In cooperation with:\par
\vspace{.7cm}
{\bfseries RATP}\\
54 Quai de la Rap\'ee\\
F-75599 Paris cedex 12
}

\cleardoublepage
\tableofcontents

\cleardoublepage
\listoffigures


\mainmatter
\cleardoublepage

\chapter{Introduction}
\begin{quote}{This chapter is an introduction to the thesis. It covers in a comprehensive form, a situation of the thesis in its context, followed by a presentation of the environment in which it has been realized. 
Furthermore, the scientific aims and methodology of the thesis, as well as its main scientific and technological contributions are explained, before detailing the outline of the manuscript.}
\end{quote}

\section*{Context}
\subsection*{Role of the Metro in the Context of Cities}
Worldwide, the number of people having been moving to cities exceeds the number of those having been going to rural areas, the well known phenomena is referred to as urbanization. This ongoing process leads, on the one hand, to a drain of the countryside. On the other hand, cities keep on growing further and become dense agglomerations. These urban zones generally have a compact city center. In the context of European cities, this is the historical city center, in those of American cities, this is the Central Business District. With the increasing population, large housing areas are built on the outskirts of the city. 
The suburbs of American cities are characterized by widespread single housing, whereas in European and Asian cities, very dense housing blocks with a small apartment size dominate.

To keep the metropoles functioning and their people moving, efficient means of transportation are important. Without these, the further growth of cities would reach a natural upper limit at some point. This limit is reached when the time their inhabitants spend daily in transportation exceeds a certain value. Transportation incorporates here all different means, from walking, cycling, over individual mobility to public transportation. In the context of a widely spread metropole, people require motorized means of transportation to cover the long distances.
The first way, individual mobility, has recently received a lot of attention in sciences and the public opinion thanks to the technology of driverless and electric vehicles.
However, not only from an environmental point of view, this is only the second best solution. The innovations cited above do not change the fact that the capacity of the car-road-system is significantly lower than the one of the train-railway-system. The capacity, however, is crucial for metropoles where the passenger travel demand is very high, especially in European and Asian cities where the population is extremely dense.
Today, congestion can be observed in many metropoles, because the road network is no longer able to cope with the increasing number of cars.
Finally, particularly European cities, for example Paris, try nowadays to limit their surface occupied by cars, by giving them back to pedestrians, bicycles or by converting them to parks.
To fight congestion, public transportation can play an important role in the future.

The metropolitan (metro) is often underground and therefore independent of road traffic conditions.
Moreover, the signaling system allows the trains to run at high speed independently of the driver's visibility. Therefore metros offer short travel times to passengers, and have consequently known a great success in metropoles over all continents. Metros are urban railway systems characterized by a dense network and a large number of stations where two neighbored ones are often in walking distance. For example, the Paris metro network counts $16 \text{ lines}$ and with a total length of nearly $220 \text{ km}$. Taking the example of Paris metro line~1, the average distance between two stations is of $688 \text{ m}$, with $25$ stations over a total line length of $16.6 \text{ km}$. The lines are operated at a very high frequency, especially in peak hour. For example, on Paris metro network, the train time-headway is as low as $90 \text{ sec}$ on some lines during peak hour.

The easy access, the short waiting and the quick travel times have made in many cities the metro to the inhabitants' preferred transportation mode. First implemented in London in $1863$, similar systems have been opened all over the continent, for example in Paris in $1900$ and in Berlin in $1902$. The construction of new metro systems went on in the second half of the same century, for example in Beijing (1969), in Munich (1971), in Vienna (1978) and still continues today. More recently, the metropolitan has been implemented in many Asian metropoles, such as Hong Kong (1975) and Singapore (1987). From there, it now arrives in Oceania: Auckland, the biggest city of New Zealand plans to open its first metro line in $2024$.
Perfectly adapted to dense metropoles, the metro has known a less important development in the often very dispersed American cities and their suburbs.

A further development of the historic signaling system is the fully automated metro. Since the beginning of the $21^\text{st}$ century, many new metro lines are designed as driverless systems. Since the metro networks in Asian cities are much more recent than those in Europe, most of the automated lines can today be found there. Singapore, for example, has automated all of its metro lines. Fully automated systems present several advantages compared to systems with drivers.

Firstly, the travel time decreases, especially because trains follow the speed profile between the stations more accurately.
Secondly, dwell times are known to be shorter and well respected on automated lines.
Consequently, the regularity increases considerably, making of the metropolitan a metronome of the city~\cite{G18}. In regular time-intervals, metros run on their line, taking passengers on board and pitching them into the pulsing life of their metropoles.
As an ultimate consequence, a higher average speed and an increased regularity maximize the throughput or train frequency of a line, which allows increase the capacity offered to passengers.
Obviously, this is a key issue in metropoles with a very high passenger travel demand.
The first fully automatic metro in the world is based on the \textit{Véhicule Automatique Léger} (automatic light vehicle) technique and was opened in Lille, France, in 1983.
Paris opened the first automated line, metro line 14, in 1998. Since then, the Paris metro operator \textit{R\'egie Autonome des Transports Parisiens} (RATP) has also automated metro line 1 in 2012. The automation of metro line 4 is under way.

\subsection*{Existing Traffic Simulation Tools at RATP}
Development and operations of metros require stakeholders to a priori visualize and evaluate possible modifications and extensions to the system.
Public transportation operators, infrastructure operators or private traffic engineering companies realize these studies by means of simulation.

At RATP, simulations are done by the departments \textit{Etudes g\'en\'erales, Developpement, Territoires} (EDT) and \textit{Ma\^{i}trise d'Ouvrage des Projets} (MOP).
EDT conducts simulations of the passenger travel demand. 
The MATYS entity of the \textit{G\'{e}nie Ferroviaire} (GEF) of MOP is in charge of train and traffic simulations on metro and \textit{R\'{e}seau Express R\'{e}gional} (RER) lines. 

On the train simulation side, MATYS simulates train speed profiles, based on the rolling stock and the infrastructure. 
Together with the dwell times, they are the basis for timetable construction. Both dwell times and the timetable design is done by the metro and RER operations departments, \textit{Métro, Transport et Services} (MTS) and RER. 

The RER is a suburban train network, connecting the outskirts of Paris to the city center at a high frequency. For example, the RER line A runs from the west to the east of the Paris region and has a length of $109 \text{ km}$, with $46 \text { stations}$ and an average distance between two consecutive stops of $2.36 \text{ km}$. The time-headway during peak hour is of $140 \text{ seconds}$.

On the traffic simulation side, MATYS uses the railway simulator OpenTrack. 
Based on a detailed infrastructure and vehicle characteristics model, a given timetable can be simulated. 
Furthermore, MATYS uses an API version of OpenTrack which allows to communicate with the simulator in real-time. Once a simulation has begun, the train positions, among further parameters, can be accessed at any moment. Based on these data, third party tools, which are external to the simulator, allow to modify dwell times and speed profiles of specific trains.

This makes it possible to reproduce precise perturbation scenarios. The recalculated dwell times and speed profiles are then sent back to OpenTrack. They are directly taken into account in the current simulation. For example, the effect of passenger accumulation on the platform due to a perturbation can be simulated. For trains running with a long train time-headway, extended dwell times can be calculated. Moreover, different driver characteristics can be reproduced.

On the one end, traffic simulation allows to visualize different scenarios with a high level of detail.
The thus obtained results, however, only hold for the simulated case.
To obtain results on the global system behavior, many simulations have to be run, with varying input parameters. 
In the context of traffic engineering, traffic simulators have big success because of their flexibility with regard to possible applications. 

On the other end, a system can be studied with a model which is simple enough to be mathematically analyzed.
However, the hypotheses taken to simplify the model, cause that the latter no longer represents the real world system.
Therefore, the choice between simulation on the one side, and analytic modeling on the other side, has to be made carefully, depending on the application.


\section*{Overview of the Thesis}
\subsection*{Objective \& Methodology}
In railway operations, 
in case of a perturbation, the aim of the traffic controller is to re-establish the predefined timetable by recovering delays. This strategy is convenient for linear lines without junction, even though further control actions, such as canceling a train service, or re-assigning drivers and crew can be beneficial and further enhance the service offered to passengers.

For a line with a junction, traffic control in case of perturbations is a more complex problem. The branches which converge at one point influence each other. In case of saturation, the train passing order and the number of trains on each branch have to be controlled. The problem becomes even more complex in case of a high passenger affluence which may further disturb train dwell times.

The principal aim of this thesis is to bring about, in a first step, a general mathematical model of the traffic dynamics of a metro line with a junction. General means that the model can be applied to any line with a junction, satisfying some hypotheses.
The traffic dynamics are modeled as a discrete event system in such a way that they linear in the max-plus algebra. The application of the theory of max-plus linear systems allows to study the properties of the system in the stationary regime, interpreted in this thesis as the physics of traffic of a metro line with a junction.

%

In a second step, the results of the stationary regime characteristics of the traffic on a metro line with a junction, are used for the control of the system.
The application of the mathematical model is twofold, on the one side to strategic planing and scheduling and on the other side to the dynamic control of the system.


To illustrate its possible applications, metro line 13 of Paris is modeled. The metro line has one central part and two branches, connected with a junction. The branches have different lengths. The southern terminus of the line, common to all trains, is \textit{Châtillon -- Montrouge}. The first branch ends at \textit{Saint-Denis -- Université}, the second one at \textit{Les Courtilles (Asnières~--~Gennevilliers)}. The total length is of $24.3 \text{ km}$ with an average distance between two stations of $776 \text{ m}$. In 2017, $131.4 \text{ million}$ passengers used the line which places it on $3^\text{th}$ position of all the lines of the Paris metro network.

The final aim is to provide the operator RATP, on the basis of this thesis, three tools for application.
The first one allows to conduct a theoretic analysis of the capacity of a metro line with a junction.
The second one allows to calculate an optimal train passing order at the convergence, depending on the real-time traffic conditions.
The third one depicts the asymptotic average train frequency as a function of the parameters of the line and of the passenger travel demand. This allows to optimize the train schedule depending on the passenger travel demand.

\subsection*{Contributions}
The thesis is situated between applied mathematics and engineering sciences, in particular modeling, simulation and optimization. It is a mathematical contribution to transportation engineering. 
The modeling consists in describing the principal parameters and variables of the dynamics of a metro line with a junction.

An one-over-two operation of the junction in the stationary regime is assumed.
Taking the example of the convergence, one train entering the central part from branch 1 is then followed by one train from branch 2. This one-over-two rule also applies to the divergence.

The line is discretized in segments or blocks, representing the signaling system. The dependency of the parameters and variables is described by mathematical equations. Even though the focus is on how train move on in time, the system is described as a sequence of discrete events, where an event is a departure from a given segment. More precisely, these are the counted train departure times from a given segment. 

In a first version in Chapter~\ref{1A}, the trains are supposed to respect given lower bounds on the train dwell times at the platforms, the run times on each segment, and the safe separation times between the two consecutive trains. 
The train dynamics are written on two constraints, one on the sum of run and travel time, and one on the safe separation time.

The series of constraints is non-linear in standard algebra but can be represented in a linear matrix form in max-plus algebra. 
The properties of max-plus linear systems in the stationary regime are well known. A main reference is Baccelli et al.~\cite{BCOQ92}. One of the first applications of the theory of discrete event systems to scheduling in transportation is from the Dutch Olsder~\cite{O89}. The research on discrete event systems and the application of the theory to transportation have contributed to the implementation of a synchronized timetable on the Dutch railway network.

The max-plus linear model of the train dynamics and the application of the theory from~\cite{BCOQ92} allows to derive closed-form solutions for the asymptotic average growth rate of the system in the stationary regime. It is interpreted as the asymptotic average train time-headway on the branches and on the central part of the line. It depends on different parameters. The graphical representation is the fundamental diagram of a metro line.
It is already well known in road traffic. It represents the dependency between the macroscopic variables throughput, speed and traffic density.

For the case of a metro line with a junction, the fundamental diagram gives the relation between the macroscopic variables train number, the difference between the number of trains on the branches, asymptotic average train time-headway and the parameters of the line (train run, dwell and safe separation times). It is new and has first been presented by Schanzenbächer et al.~\cite{SFCLG17,SFLG19c}.

In a second version in Chapter~\ref{Chap-4}, the dwell times are modeled as a function of the passenger travel demand. This allows to take into account the effect of an accumulation of passengers on the platform in case of a perturbation, within a margin on the train run times. To guarantee the stability of the system, a control on the train run times is introduced which cancels the effect of dynamic demand-dependent dwell times.

With the combination of demand-dependent dwell times and controlled run times, the system can be written linearly in the max-plus algebra. From the application of the theory of max-plus linear systems, closed-form solutions for the asymptotic average growth rate of the system are derived.
It is interpreted as the asymptotic average train time-headway and represented as the fundamental diagram, see Schanzenbächer et al.~\cite{SFCLG18a,SFLG18b,SFLG19a,SFLG19b}.
The fundamental diagram depicts the relation between the macroscopic variables asymptotic average train frequency, passenger travel demand, and run time margin. The asymptotic average train frequency depends furthermore on the difference between the number of trains on the branches, as well as the line parameters.

Chapter~\ref{sim} demonstrates the possible applications of the traffic model and the fundamental diagram, which are twofold.
Firstly, the fundamental diagram can be used for strategic planing and scheduling. For scheduling, the main control variable is the number of trains and the difference between the number of trains on the branches. The variable of the difference between the number of trains on the branches is new, it has been derived from the model for lines with a junction. The fundamental diagrams also serves for theoretic analysis of the capacity of a metro line.

With regard to dynamic traffic control, several control laws are derived from the traffic model.
First of all, following the optimal difference between the number of trains on the branches in real-time can be realized by instantaneously changing the train passing order at the convergence, depending on the online realized train dwell and run times. 

Moreover, the demand-dependent train dwell and the controlled run times of the model can be applied in real-time, which guarantees demand-dependent stable traffic operations. 
This allows to extend dwell times in case of a long headway to take into account the passenger accumulation. Combined with the run time control, an amplification of an initial delay on a train is excluded and the traffic remains stable. Clearly, dwell times which are optimized depending on the number of passengers willing to alight and board while guaranteeing the stability of the traffic, represents an important amelioration in service quality and reliability offered to travelers.

Finally, an enhanced dynamic control is presented for a linear line, see Schanzenbächer et al.~\cite{SFLG18c}. In case of a perturbation on the train time-headway, an temporary additional control on the dwell times is applied which re-harmonizes the train headways on the line while avoiding excessive train dwell times. This control is calculated depending on the headway of each train and limits over-extension of dwell times due to high passenger affluence. The strength of the control increases with the train time-headway. While applied, train dwell times are no longer optimized only with regard to passenger accumulation, but are a compromise between passenger comfort and train regularity. It is shown, that its application leads to a harmonization of the train headways on the line. Once achieved, the additional control can be released and dwell times are again fully passenger demand-dependent.
This temporary additional control might refrain some passengers from boarding, but harmonizestrain time-headways which is, on the longer term, beneficial for passengers.

\subsection*{Line Diagnostic and Traffic Control Tools Developed}
Three tools have been developed for RATP. The first one 
allows to analyze the capacity of any metro line with a junction. Taking as input the timetable of the line, it is based on the theoretic dwell times and run times. Furthermore, the signaling system is taken into account by the safe separation times for each segment. With these parameters, the tool visualizes the fundamental diagram of the line. The fundamental diagram depicts many important informations for the operation of the line. First, the capacity of the line, and second, the optimal number of trains on the central part and on the branches. Optimal means that the frequency is maximized for a given total number of trains. Taking furthermore into account a passenger demand profile, the required train frequency in order to serve the passenger demand is calculated. The tool automatically checks if the capacity of the line allows to realize this frequency. If this is the case, the fundamental diagram allows to recalculate the corresponding optimal number of trains on the central part and on the branches.

The second tool 
allows to control the train passing order at the convergence in real-time. Its inputs are the same type of parameters as for the preceding tool. Precisely these are the train dwell times, run times and safe separation times. In the contrary to above, the parameters are on-line measured. For example, every time a train stops at a station or passes by a segment, the corresponding dwell and run time is updated which permits to depict the fundamental diagram in real-time. Depending on the on-line dwell and run times and to realize a certain required frequency, the optimal number of trains on the central part and the two branches is recalculated. If the optimal difference between the number of trains on the branches diverts, the train passing order at junction is temporarily modified until the new optimal difference has been retrieved. Modifying the operation of the convergence means, instead of realizing an one-over-two operation (one train from branch 1 followed by a train from branch 2 and so forth), a number of consecutive trains from one branch are entering the central part following each other.

The third tool is for demand-dependent scheduling with demand-dependent dwell times and controlled run times. It is based on up-to-date Origin-Destination matrices which allow to calculate the passenger travel demand. Together with minimum dwell and run times, safe separation times and the run time margin chosen by the operator, the tool calculates demand-dependent dwell and run times and depicts the corresponding asymptotic average train frequency on the line.
\subsection*{Structure of the Manuscript}
The outline of the manuscript is as follows. In Chapter~\ref{lit}, some {\textit{Background on Railway Traffic Control, Max-plus Algebra and Traffic Flow Theory}} is presented.
The chapter covers publications on railway traffic operations towards the derivation of fundamental diagrams,
real-time control of railway systems, including optimal control approaches of the train dynamics on a metro line and operations research approaches for complex rescheduling problems.
Moreover, the bases of max-plus algebra are recalled and some main applications to railway networks are presented.
The chapter finishes with the most recent works studying metro line dynamics.

In Chapter~\ref{1A}, {\textit{The Fundamental Diagram of a Metro Line with a Junction}} is presented. This is the first version of the traffic model for a metro line with a junction and the main results are closed-form solutions of the average train time-headway and frequency depending on many parameters. The analytic solutions are displayed in their graphical form, the fundamental diagrams. Eight traffic phases of the train dynamics can be distinguished, two free flow phases, two congested branches phases, two congestion phases, a capacity and a zero flow phase. Based on the theoretic results, macroscopic control laws for the number of trains depending on the traffic state and the passenger travel demand are derived.

Chapter~\ref{Chap-4}, {\textit{The Effect of the Passenger Travel Demand on the Traffic}} presents the extended version of the traffic model, with demand-dependent dwell times and a run time control which guarantees traffic stability. Once again, closed form solutions for the average train headway and frequency, as well as their graphical form, the fundamental diagrams, are presented. Here, they depend furthermore on the passenger demand and a run time margin chosen by the operator which allows to command optimal dwell times taking into account the passenger affluence.
The fundamental diagrams allow an in-depth study of how all these parameters influence the traffic on a metro line with a junction.

Chapter~\ref{sim}, the {\textit{Simulation of Feedback Traffic Control}}, is presented. The simulation results show how the theoretic results from Chapter~\ref{1A}~and~\ref{Chap-4} can be used for traffic control on the tactical (number of trains) and on the operational level (dwell and run times).

The thesis finishes with a {\textit{Conclusion}} in Chapter~\ref{conclusion}, providing a resume of the manuscript, as well as a brief discussion of its limitations and an outlook with regard to open future research.


\chapter{Background on Railway Traffic Control, Max-plus Algebra and Traffic Flow Theory}\label{lit}
\chaptermark{Literature Review}

\begin{quote}{Based on a bibliographical analysis, this chapter presents an overview of railway traffic control approaches, max-plus algebra and traffic flow theory for railway systems.
Special attention is given to recent publications contributing to the derivation of the fundamental diagram of railway traffic. The fundamental diagram, already widely studied for vehicular traffic, describes the relation between the macroscopic variables traffic flow, density and speed.
In the last years, research has been done on extending this concept to railways.}
\end{quote}


\section{Real-time Railway Traffic Control}
\subsection*{Optimal Control of the Train Dynamics on a Metro Line}
This thesis presents, for the first time, a traffic model and feedback control for metro lines with a junction. In the first part, the train dynamics are supposed to respect given lower bounds on train dwell, run and safe separation times. In the second part, train dwell and run times are controlled depending on the passenger travel demand.
Feedback control for metro lines without junction, including demand-dependent control, has been subject to previous research.

Van Breusegem et al.~\cite{BCB91} have done pioneer work proposing a traffic regulation harmonizing train time-headways. The authors have studied metro loop lines. These lines are a closed system with a constant number of trains. Their model can also be applied to linear metro lines, connecting two terminus stations. By modeling the turnaround at the two terminus stations, the circuit is closed and the system can be seen as a loop line. The authors of~\cite{BCB91} have modeled dwell and travel times as a function of the passenger travel demand. This means, one affects a longer dwell time to a delayed train, to take into account the passenger accumulation on the platform. Consequently, on these lines, train delays are accumulated from circuit to circuit. The main contributions of the work are, first, a discrete event traffic model in state space formulation which proves the natural instability of such lines with demand-dependent dwell times and, second, state feedback control laws in order to guarantee the stability of the system.
The authors of~\cite{BCB91} have written the train dynamics on the departure times of a train from a given station (discrete events). The departure times depend on the constant train run time per inter-station and the demand-dependent dwell time per station. Supposing one train per inter-station, the dynamics are linear.
Based on the linear dynamics, a proposed control minimizes a quadratic performance index leading to a linear quadratic problem. The performance index represents the operations objectives. It minimizes, first, the deviation between the headways over the line and, second, the delay of each train. Simulation results show the interest of the proposed control. It guarantees stability and harmonizes the train positions over the line while keeping the traffic flowing.
However, possible applications are limited. The dynamics are only linear (a necessary condition for the control) in the case where there is one train per inter-station at a time. This is not a realistic condition, since on many metro lines, there can be more than one train per inter-station at a time. Therefore, segments and the signaling system have to be explicitly modeled.

On this basis, Fernandez et al.~\cite{FCVC06} have come up with a linear quadratic predictive control for metro loop line traffic regulation. The control minimizes a cost function over a time horizon. The function takes into account the deviation of each train from the timetable and the headway deviations between two consecutive trains. Furthermore, a main operation constraint is taken into account here. The control actions on the train dwell time are limited.
With a short computing time, the linear quadratic predictive control is efficient for real-time application. The presented method allows to realizes both timetable and headway train regulation. It is particularly interesting for the headway regulation when aiming to guarantee a high commercial speed.
As above, the train dynamics are written on the departure of a train from a platform. Run times and dwell times are controlled. The latter are passenger demand-dependent. The underlying dwell time model has been published in Martinez et al.~\cite{MVFC07}.
New, compared to the dynamics from above~\cite{BCB91}, is that a minimum time-interval is respected between the arrival of a train at a given platform and the departure of the preceding train from this same platform.
Most interesting is the predictive control which offers a lot more possibilities than the simple dwell time control proposed by van Breusegem et al. Here, the control actions can be holding a train at a platform or shortening its dwell time. Similarly, a train can be accelerated or slowed down in the inter-station. The quadratic cost function allows to weight the four types of control to give a preference to any one of the actions. Furthermore, the timetable and the headway deviation are weighted such that a privilege can be given to either timetable control, for example during off peak hour, or headway control, during peak hour.

Schanzenbacher et al.~\cite{SCF16} have applied this model predictive control approach to a part of the RER line A in the greater Paris area. This is Europe's commuter railway line with the highest passenger affluence, 1.3 million travelers per workday. It is operated at a very high train frequency and dwell times are particularly demand-sensitive. In case of a delay, they tend to increase significantly with the accumulation of passengers on the platform. These characteristics make it interesting for the application of the predictive control approach proposed by the authors from above.
Here, the overall aim was to ensure a freely flowing traffic by an optimization of train dwell times. The idea was to ensure a green wave such that a train, once departed from a station, is not perturbed until arriving at the following platform. Based on the green wave, a control action on the dwell time is calculated. This control minimizes a quadratic cost function. The cost function represents two traffic regulation objectives. First, the harmonization of the intervals on the line, in order to avoid interaction between consecutive trains and passenger accumulation on the platform. Second, the minimization of trains delays and control actions. This allows to guarantee a certain throughput on the line and minimizes the risk of very long dwell times.

A working group at the State Key Laboratory of Rail Traffic Control and Safety at the Beijing Jiaotong University continues to develop extensions to the train dynamics model of van Breusegem~\cite{BCB91} cited above, and extends its applications.
Recently, Li et al.~\cite{LSYG16} have designed a robust model predictive controller for train traffic control on metro loop lines. As in~\cite{BCB91}, the model can also be applied to a linear metro line where the turnaround is modeled. The number of trains is considered to be constant. There can be maximum one train per inter-station. Especially this constraint imposed by the model limits its practical relevance.
As a new element, the typically uncertain passenger arrival flow at the platforms are taken into account. A state-feedback controller is presented which optimizes some cost function subject to safety constraints. The paper investigates the possible applications of a robust model predictive control and explores its limits. It is shown, that in case of a large number of stations and trains, the computation time becomes too long, so that a distributed robust model predictive control might perform better.

Another extension by the same laboratory is presented by the authors of~\cite{LDYG17}. The problem of minimizing train time-headway and timetable deviations with overloaded passenger flow is studied. The authors consider a metro-type linear railway with an ordered train set. The train dynamics are developed accordingly the discrete event model in~\cite{BCB91}. Consequently, there can be maximum one train per inter-station at a time. New to the dynamics is their dependence on the number of boarding and alighting passengers. Moreover, the train travel time takes into account a stochastic variable modeling possible perturbations on the train dwell and run time. In order to stabilize the traffic on the line, to minimize the variance on the headway and to increase the commercial speed, a joint optimal dynamic train regulation and a passenger flow control are developed. Therefore, a coupled state-space model for the train departure time and the passenger load dynamics is presented. By applying model predictive control, the optimal train and passenger flow control strategy minimizing headway and timetable deviations while increasing the commercial speed, is found. It is shown that the solution can be computed efficiently using a quadratic programming algorithm. The authors show that the approach is interesting for delays up to a certain limit, whereas for large delays, creating a new timetable is more promising, leading to a rescheduling problem.

The latest work of the group is by Li et al.~\cite{LYG19} and focuses once again on traffic control for metro loop lines facing frequent minor disruptions. The number of trains is considered to be constant and the trains run periodically on the line. Still based on~\cite{BCB91}, the train dynamics are modeled considering the departure of a train from a given station. As above, the number of trains must be smaller than the number of stations, there can be maximum one train per inter-station. With a constant number of trains running in a cycle, the number of stations in the model constantly increases.
The authors aim to guarantee traffic stability, recover delays and ensure headway regularity. The proposed method is a train controller which dynamically adjusts train run and dwell times. It does not necessitate explicit time margins. Therefore, the metro line can at any time be operated under the optimal (minimum) number of trains which serves the passenger travel demand.
The main contribution is the efficient design of the controller. The optimization problem is split into a set of quadratic programming problems. They can be solved in short time which makes the algorithm interesting for real-time applications. 

Recently, Moaveni et al.~\cite{MN18} have presented a discrete event nonlinear traffic model for metro loop lines, based on deviations from the nominal departure times. New to the models from above is that it models buffer times and knock-on delays between two consecutive trains. They can often be observed in high frequency metro lines where trains run close to each other. Consequently, already small perturbations effect several trains and delays propagate through the line by a so called cascade effect. The model also includes uncertainty with regard to the passenger demand evolution over the day. The nonlinear model is approximated by a linear model which allows to design a robust model predictive controller aiming to maximize passenger satisfaction by minimizing timetable and headway deviations.
The model has overcome the shortcoming of the above models, modeling interaction between the trains and reproducing knock-on delays. However, it keeps the deficit that signaling system and block system are not explicitly modeled. In fact, the train dynamics are written on the train departures from the platforms, that means the model can represent maximum one train at a time between two consecutive stations.
The authors present simulation results for a real metro line which suggests the effectiveness of the controller as well as a graphical representation of the delay rate as a function of the number of passenger on the platforms and the headway.

\subsection*{Operations Research for Rescheduling Problems}
For traffic control on more complex railway layouts, such as open lines with several junctions, stations and networks, another approach has been subject to research in the last years. Based on Operations Research techniques, a lot of constraints, related to infrastructure, timetable, rolling stock and the crew can be taken into account. The utilization of solvers allows to find an optimal solution for the problem in generally relatively short time. The solution here is optimal with regard to a criteria chosen which can represent many objectives, for example maximizing the throughput, the number of passengers, or minimizing their waiting time. Obviously, this is particularly interesting for railway systems, which are rather timetable- than headway-operated. Cacchiani et al. give in their review~\cite{C14} a good overview over the state of the art of algorithms for real-time railway management and distinguish three sub-problems which are timetable, rolling stock and crew rescheduling. Approaches exist for rescheduling one of these sub-problems exist, as well as for integrated timetable, rolling stock and crew rescheduling.

First of all, for the \textit{Train Timetable Rescheduling} (TRR) problem, the decision variables are mainly the routing of trains, the departure and arrival times of the trains at the stations and the train order on the tracks. The problem is constraint by the track capacity, where a safe separation time has to be respected between two consecutive trains. The latter is usually modeled by the so called \textit{Blocking Time Stairway}. The TRR is typically close to the NP-complete combinatorial no-wait job shop scheduling problem.
Many approaches focus on deterministic online rescheduling, this means the new timetable is calculated in one optimization, supposing availability of perfect deterministic information. Corman et al.~\cite{CM15} present an overview over recent online rescheduling approaches considering the stochastic nature and uncertainty of the problem. In this case, the timetable is continuously optimized every time updated information is available. 

Secondly, if severe delays or a disruption occurs on a railway line or a network and the traffic has been rescheduled, the original rolling stock allocation might be no longer feasible. In this case, the rolling stock has to be rescheduled, too. Aspects to be considered are allocation of rolling stock to lines, providing sufficient capacity to serve the passenger demand and end-of-day balances at stations. The problem can be solved by combinatorial optimization techniques.

Finally, if the rolling stock has been rescheduled, the crew has to be reallocated, too. It deals with assigning drivers and commercial crews to train services. The original crew schedule is taken as an input to minimize the difference between the new and the original schedule. Typical constraints to be taken into account concern work time restrictions, the aim that the crew should finish their duty at their home depot and respect of break times. 

An integrated train and passenger disruption management for the case of urban railways has been presented by Besinovic et al.~\cite{B19}. In case of a major perturbation, trains are automatically rescheduled and passenger flows controlled, which is important in case of a very high passenger travel demand as it can be observed in many metropoles. The objective of the model is to minimize passenger delays, the number of passengers denied from boarding and to limit the number of rescheduling services and the recover time. The decisions which can be taken include short-turning, canceling and rerouting trains as well as controlling passenger flows including regulating station access.
This paper presents for the first time an integrated train and passenger rescheduling algorithm particularly adapted to high passenger travel demand. The capacity of the platforms is explicitly taken into account.
The traffic management model for train rescheduling is essentially based on works by Caimi et al.~\cite{C12} who has introduced rescheduling based on extended conflict graph models. Based on a model predictive control framework, a rescheduling algorithm for large station areas is proposed. The model assigns deterministic blocking stairways to trains under a number of constraints, including connections and platform constraints.


\section{Introduction to Max-plus Algebra}
\subsection*{Max-plus Linear Discrete Event Systems}

This thesis presents a new application of max-plus algebra to railway traffic modeling and control.

Max-plus algebra~\cite{BCOQ92} is the algebraic structure $(\mathbb R \cup \{-\infty\}, \oplus, \otimes)$ which denotes the set $\mathbb{R}$ with the zero element $\{-\infty\}$ and the two basic operations $\oplus$ and $\otimes$.
The operator $\oplus$ is defined as the maximization
\begin{equation}
a \oplus b = \max(a,b)
\end{equation}
and the operator $\otimes$ is defined as the addition
\begin{equation}
a \otimes b = a + b
\end{equation}
for $a,b \in \mathbb R \cup \{-\infty\}$.
The zero element with respect to the maximization $\oplus$ is $$\varepsilon = -\infty.$$ We have $$a \oplus \varepsilon = a = \varepsilon \oplus a$$ and the absorbing property with regard to the product $$a \otimes \varepsilon = \varepsilon = \varepsilon \otimes a.$$
The identity element with respect to the addition $\otimes$ is $e$. We have $$a \otimes e = a = e \otimes a.$$
Note that the natural order on this structure may be defined using the $\oplus$ operation: $a \leq b \text{ if } a \oplus b = b$.

Extending the basic operations to matrices, consider a set of square matrices where $A$ and $B$ are two max-plus matrices of size
$n \times n$. Then, the addition $\oplus$ and the product $\otimes$ are defined by: $$(A \oplus B)_{ij} := A_{ij} \oplus B_{ij} = \max(a_{ij},b_{ij}),$$
$$(AB)_{i,j} = (A \otimes B)_{ij} := \bigoplus_k^{n}[a_{ik} \otimes b_{kj}] = \max_{k = 1,...,n}(a_{ik} + b_{kj}) \quad \forall i,j.$$
The zero and the unity matrices are also denoted by $\varepsilon$ and $e$ respectively.


Consider the dynamics of a homogeneous $p$-order max-plus system with a family of max-plus matrices $A_l$
\begin{equation}
  x(k) = \bigoplus_{l=0}^p A_l \otimes x(k-l). \label{eq-rev1}
\end{equation}
Define $\gamma$ as the backshift operator applied on the sequences on $\mathbb Z$:
$\gamma^l x(k) = x(k-l), \forall l\in\mathbb N$. Then~(\ref{eq-rev1}) can be written
\begin{equation}
  x(k) = \bigoplus_{l=0}^p \gamma^l A_l x(k) = A(\gamma) x(k),  \label{eq-rev2}
\end{equation}
where $A(\gamma)=\bigoplus_{l=0}^p \gamma^l A_l$ is a polynomial matrix in the backshift operator $\gamma$;
see~\cite{BCOQ92,Gov07} for more details.

$\mu \in \mathbb R_{\max} \setminus \{\varepsilon\}$ is said to be a 
\textit{generalized eigenvalue}~\cite{CCGMQ98} of $A(\gamma)$, with associated
\textit{generalized eigenvector} $v\in \mathbb R_{\max}^n \setminus \{\varepsilon\}$, if
\begin{equation}\label{eq-gev}
  A(\mu^{-1}) \otimes v = v,
\end{equation}
where $A(\mu^{-1})$ is the matrix obtained by evaluating the polynomial matrix $A(\gamma)$ at $\mu^{-1}$.

A directed graph denoted $\mathcal G (A(\gamma))$ can be associated to a dynamic system of type~(\ref{eq-rev2}).
For every $l, 0\leq l\leq p$, an arc $(i,j,l)$ is associated to each non-null ($\neq \varepsilon$) element
$(i,j)$ of Max-plus matrix $A_l$.
A \emph{weight} $W(i,j,l)$ and a \emph{duration} $D(i,j,l)$ are associated to each arc $(i,j,l)$ in the graph,
with $W(i,j,l) = (A_l)_{ij} \neq \varepsilon$ and $D(i,j,l) = l$. 
Similarly, a weight, respectively duration of a cycle (a directed cycle) in the graph is the standard sum of the weights,
resp. durations of all the arcs of the cycle. Finally, the \textit{cycle mean} of a cycle $c$ with a weight $W(c)$ and
a duration $D(c)$ is $W(c)/D(c)$.
A polynomial matrix $A(\gamma)$ is said to be irreducible, if $\mathcal G(A(\gamma))$ is strongly connected.

The findings presented in this thesis are mainly based on the application of the following result of the max-plus algebraic eigenvalue problem.
\begin{theorem} \cite[Theorem 3.28]{BCOQ92} \cite[Theorem 1]{Gov07} \label{max-plus-theorem}
  Let $A(\gamma) = \oplus_{l=0}^p A_l\gamma^l$
  be an irreducible polynomial matrix with acyclic
  sub-graph $\mathcal G(A_0)$. Then $A(\gamma)$ has an unique generalized eigenvalue $\mu > \varepsilon$ and finite eigenvectors $v > \varepsilon$
  such that $A(\mu^{-1}) \otimes v = v$, and $\mu$ is equal to the maximum cycle mean of $\mathcal G(A(\gamma))$, given as follows:
  $\mu = \max_{c\in\mathcal C} W(c) / D(c)$,
  where $\mathcal C$ is the set of all elementary cycles in $\mathcal G(A(\gamma))$.
  Moreover, the dynamic system $x(k) = A(\gamma) x(k)$ admits an asymptotic average growth vector (also called cycle time vector here) $\chi$ 
  whose components are all equal to $\mu$.
\end{theorem}

Note that a matrix is irreducible if its associated graph is strongly connected. Seeing $A$ as the adjacency matrix of a directed graph, $A$ is strongly connected if and only if every vertex can be reached from any other vertex.
Accordingly to Schutter et al.~\cite{SBXF19}, the max-plus algebraic eigenvalue can be interpreted as follows. If $\mu_{max}$ is the maximal average weight over all elementary cycles of the strongly connected graph $\mathcal G(A(\gamma))$, then $\mu_{max}$ is a max-plus algebraic eigenvalue of $A$.
Moreover, its max-plus algebraic eigenvalue is unique.
An elementary cycle is a cycle in which no vertex appears more than once, except for the initial vertex which appears exactly twice.

\subsection*{Railway Applications of Max-plus Algebra}
The application of max-plus algebra to railway systems has been studied by a research group at TU Delft.
Recently (2016), Kersbergen et al.~\cite{KRBS16} have presented works on railway traffic management using switching max-plus linear systems. The authors present a model for railway traffic and a model predictive controller for real-time traffic management on a network scale with a periodic train timetable. The main objective is to maximize delay recovery by modifying train departure times, breaking connections, splitting trains and redistributing trains in case of multiple tracks. The railway system is represented by a switching max-plus linear model. Run and dwell times are updated in real-time, the model is then used to determine the optimal control actions by solving a Mixed Integer Linear Programming problem. The authors have applied the algorithm to the Dutch railway network.

Before (2012), van den Boom et al.~\cite{BKS12} have discussed the rescheduling of trains on a network in case of perturbations using a max-plus linear system description. Based on a study of the system matrices, control variables are designed such that they modify the matrices in a convenient way. Constraints for scheduling trains on separate tracks for slow and fast traffic, per direction, and for joining and splitting trains are considered.
The approach assumes a periodic timetable, this means after one cycle time, the trains pass again at a given station. A station is a point where trains can possibly change order, including junctions. The train dynamics are then written on departure times of a train and within a given cycle of the timetable, from a station. These departure times satisfy a number of constraints, including timetable with nominal run and dwell times, as well as connections from other trains. Minimum headways between two consecutive trains at the departure and arrival to a station are also taken into account. In case of a single track railway, a waiting time with regard to trains running in the opposite direction is applied. A train departures as soon as the maximum of the constraints is satisfied.
This system is transformed to max-plus algebra. It is shown that it becomes linear which allows to analyze the system matrices. It is shown that all control actions, including changing the train order, canceling connections, joining and splitting trains can be described by the entries of a control vector to the system. The model can be used to compute the optimal control with a model predictive control approach.

Earlier (2009), Goverde et al.~\cite{GHM09} have used max-plus linear system theory to analyze the stability of a railway timetable. In the contrary to the above used deterministic theory, their work is based on stochastic max-plus linear systems. Therefore, the authors generalize the spectral theory of deterministic max-plus linear systems to stochastic ones. More precisely, the concept of stability of deterministic max-plus linear systems is developed further towards stability analysis considering stochastic process times, which allows to analyze timetable stability. Primary delays on the railway network are modeled using general probability distributions. Delay propagation because of infrastructure and timetable constraints are calculated from the stochastic recursive equations modeling the train dynamics.

Already in 2004, Heidergott et al.~\cite{HWO04} have shown that the asymptotic growth rate of a max-plus system, well known for the deterministic case, exists under some weak conditions for their stochastic max-plus linear systems.

Among the first applications of discrete event systems to railway transportation is the one from Olsder~\cite{O89} (1989).
The theory developed in this work has played a role with for the implementation of a synchronized timetable on the Dutch railway network.

\section{Traffic Flow Theory} 
\subsection*{Standard Algebra Modeling Approaches}
Approaches exist which aim to understand the behavior of a traffic system in its steady state.
Extensive research on this field has been done for vehicular traffic in the past.

Well known for simple road links, fundamental diagrams depict the relation between vehicle flow, density and speed. Typically, two traffic phases are distinguished. Firstly, a free flow phase, where the vehicle flow or throughput, that is the number of vehicles which pass by a certain point within a time interval, increases with the vehicle density. The denser getting traffic reaches a critical value where the throughput is maximized. From this point on, the vehicle flow decreases continuously with an increasing density and congestion occurs.
Moreover, the vehicle flow is speed-dependent. A high speed at low densities increases the vehicle flow. In the contrary, approaching the critical density, a reduced speed allows to increase the maximum vehicle throughput.
The fundamental diagrams for vehicular flows allow to obtain a good understanding of the physics of traffic on road links. On this basis, traffic control laws for optimal speed or optimal density with regard to the vehicular flow can be derived. These well-known relations have been published in various forms, for example by Daganzo~\cite{D97}.

Research on traffic flow theory for urban networks has gathered interest recently. So called aggregated fundamental diagrams (also macroscopic fundamental diagram, network fundamental diagram) have been developed and allow, in the line with fundamental diagrams for simple links, to have an idea of the physics of traffic of some urban road networks. The relationship between vehicular traffic flow, density and speed on a network scale has been subject to a number of publications. Daganzo~\cite{D07} has presented physical models for an aggregated description of traffic flow on urban networks. Based on real traffic data from Yokohama (Japan), Geroliminis et al.~\cite{GD08} have shown that for this case and the specific traffic data, an aggregated fundamental diagram exists.
More recent works by Geroliminis et al.~\cite{GS11} have observed scatter phenomena for the flow-density relation for high densities. The authors of~\cite{GS11} name some properties an urban network should satisfy so that aggregated fundamental diagrams with low scatter exist. Most importantly, an aggregated fundamental diagram with lower scatter can be expected for a homogeneous spatial distribution of the vehicle density.
Most recent applications of aggregated fundamental diagrams for traffic control are from Keyvan-Ekbatani et al.~\cite{KEKPP12,KEYGP15}.

Studying the physics of railway traffic by means of fundamental diagrams has only been subject to very recent research by Corman et al.~\cite{CHKE19}, Saidi et al.~\cite{SWKZ19} and Seo et al.~\cite{SWF17}.

The first work is from Seo et al.~\cite{SWF17} (2017) have proposed a simplified analytic traffic model for urban rail transit lines taking into account the effect of the passenger travel demand.
The model proposes a dwell time depending linearly on the number of passenger willing to board and assumes infinite capacity of the trains.
The train dynamics are modeled based on Newell's car-following model, assuming trains to travel at maximum speed while respecting a safe separation distance.
Both train congestion and passenger boarding congestion have been modeled.
The main result are fundamental diagrams for the stationary regime of the traffic in an idealized environment. Main simplifying assumptions are constant train cruising speeds, safe separation intervals, passenger boarding flows, as well as an equal distance between two consecutive stations, trains stopping at all stations, a constant headway between two trains following each other and the same passenger arrival flow to all platforms.
The fundamental diagram depicts train flow over train density and passenger arrival flow, see Figure~\ref{seo}. The existence of two traffic phases is shown: a free flow phase, where the flow increases with the density, and a congestion phase. From the optimal density on, where capacity is reached, train flow decreases with an increasing train density.
The effect of the passenger demand is double: first, for a fixed train density, an increasing passenger demand decreases the train flow. Second, the optimal density is passenger demand dependent and decreases with an increasing demand.

\begin{figure}
 \centering
 \begin{minipage}{0.49\textwidth}
 	\centering
 	\frame{
  	\includegraphics[width=\textwidth]{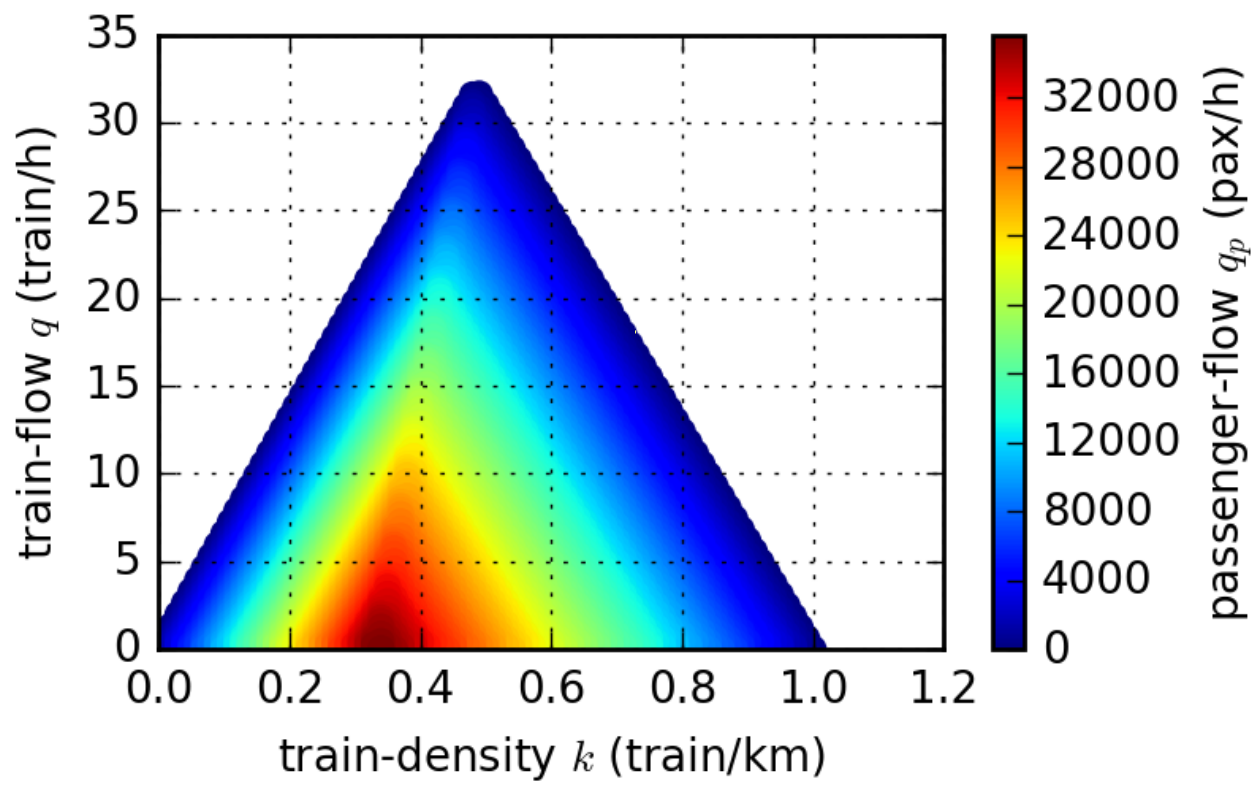}}
	\caption {Train flow, density and passenger flow from a very simplified analytic traffic model~\cite{SWF17}.}
	\label{seo}
	\end{minipage}
 \begin{minipage}{0.49\textwidth}
 	\centering
 	\frame{
  	\includegraphics[width=\textwidth]{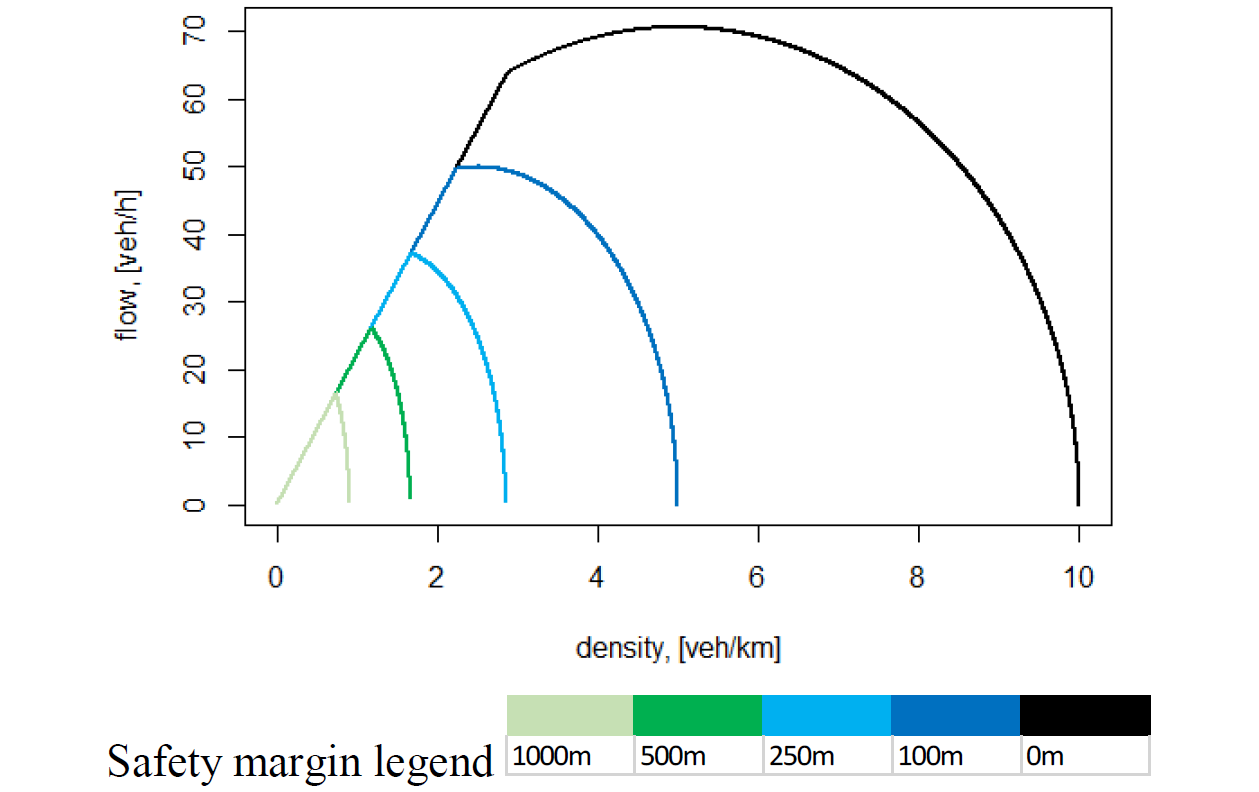}}
	\caption {Analytic fundamental diagram for a railway line, different safe separation intervals and moving block~\cite{CHKE19}.}
	\label{corman}
	\end{minipage}
\end{figure}

More recently, Corman et al.~\cite{CHKE19} (2019) have studied the idea of macroscopic descriptions of railway operations. As for vehicular traffic, fundamental diagrams can help to derive optimal control actions maximizing train speed and flow. Therefore, the authors investigate the possibility to derive equations for macroscopic variables describing the railway traffic, as speed, density and flow. In~\cite{CHKE19} this is done by closed-form solutions whenever possible, or using a railway simulator for more complex cases.
In railway, many microscopic modeling approaches exist allowing to reproduce with an increasing level of detail, the dynamics of one single train, that is acceleration, speed, breaking, position etc. Macroscopic approaches exist, but mainly allow to represent the average conditions of the system, interesting for example for long-term strategy planing purposes. Wendler~\cite{W07} has presented a popular approach using queuing theory. However, they are not able to represent the large scope of traffic states (free flow, capacity, congestion) characterized by some average condition parameter. Here, fundamental diagrams for railways can close the gap. 

The authors of~\cite{CHKE19} aim to understand physics of railway traffic on a line without junction, this means there is neither a merge, nor a divergence. The train dynamics including dwell times are supposed to be independent of the passenger travel demand,  trains respect given travel times from the timetable. 
The authors assume a moving block system, where trains follow each other in absolute braking distance such that trains can always stop in distance between themselves and their predecessor. Moreover, a safe separation time is considered, modeling traffic conditions in case the traffic gets very dense. Consequently, trains cannot exceed a maximum speed such that they can potentially stop, respecting the safe separation distance, before running into the preceding train. The breaking distance depends on the vehicle speed. Therefore, acceleration, deceleration and maximum speed of the trains are given, firstly, by the vehicle characteristics, and secondly, by the traffic conditions. The moving block system considered is one of the most interesting innovations in railway, but has not been implemented on many lines yet. In fact it is easier to model then a block system, since car-following models from road traffic can potentially be applied.

In~\cite{CHKE19}, for the case where all trains run with uniform speed, a diagram giving the dependency of train flow and train density for different safety margins, is presented, see Figure~\ref{corman}. The authors distinguish two traffic phases: a free flow phase and a congestion phase. At their intersection, the train flow is maximized for the  critical train density. Precisely, in the free flow phase, the flow increases linearly with the train density. 
If the traffic gets very dense, congestion occurs and the flow decreases. Whereas in the first phase, flow depends linearly on the density, in the congestion phase the flow is a quadratic function of the density since the braking distance depends quadratically on the train speed.
Finally, the maximum train flow depends on the safety distance. As it can be seen in the figure, a shorter safe separation distance allows to increase the train flow. Note that reducing the train speed allows to reduce the safe separation distance. The authors suggest that the maximum train flow occurs under reduced speed.
The results have been verified by simulation. The simulation results do not show congestion since the dispatching headway (corresponding to vehicle density), can, due to operational constraints, not decrease under a certain limit which would allow to enter in the congestion phase. However, the results allow to identify a region where the train flow increases up to its maximum, under a reduced train speed.

In a parallel work, Saidi et al.~\cite{SWKZ19} (2019) have also recently begun to develop fundamental diagrams for urban railway lines, using a data approach. They have derived a new mesoscopic train-following model from empirical track circuit data from a metro line in Cambridge, Massachusetts. The authors have analyzed train delays between consecutive trains around the bottleneck of the line. 
The resulting model, based on microscopic analysis of the train operations, is used to predict the behavior of the system from any initial state on. The authors claim their model to be more accurate than simplified macroscopic analytical railway traffic models with closed-form solutions for macroscopic traffic variables such as train speed, density and flow. Compared to microscopic train dynamics models which can be exploited by simulation, their results allow a more accurate system behavior analysis and traffic state prediction.

\begin{figure}
 \centering
 \begin{minipage}{0.49\textwidth}
 	\centering
 	\frame{
  	\includegraphics[width=\textwidth]{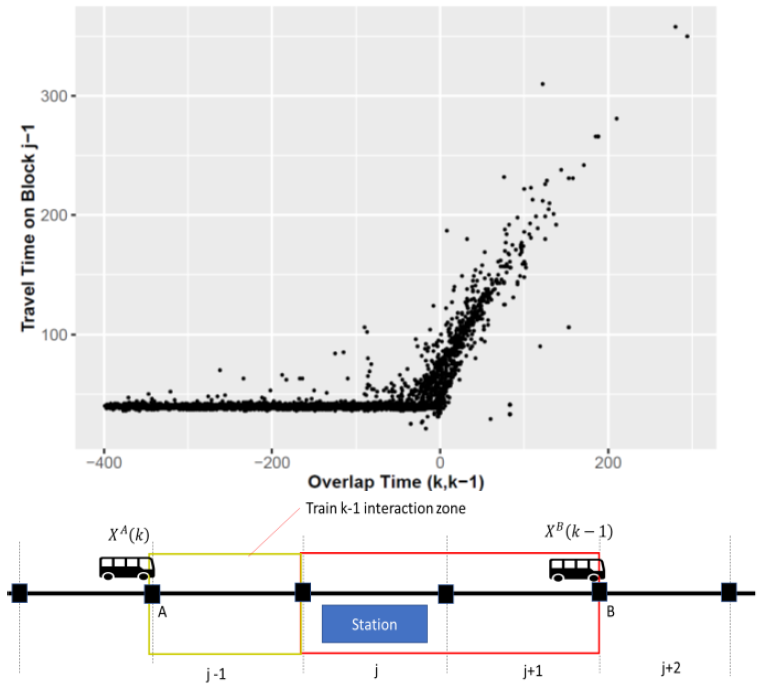}}
	\caption {Train travel depending on the position of the preceding train for block system~\cite{SWKZ19}.}
	\label{saidi-overlap}
	\end{minipage}
\begin{minipage}{0.49\textwidth}
 	\centering
 	\frame{
  	\includegraphics[width=\textwidth]{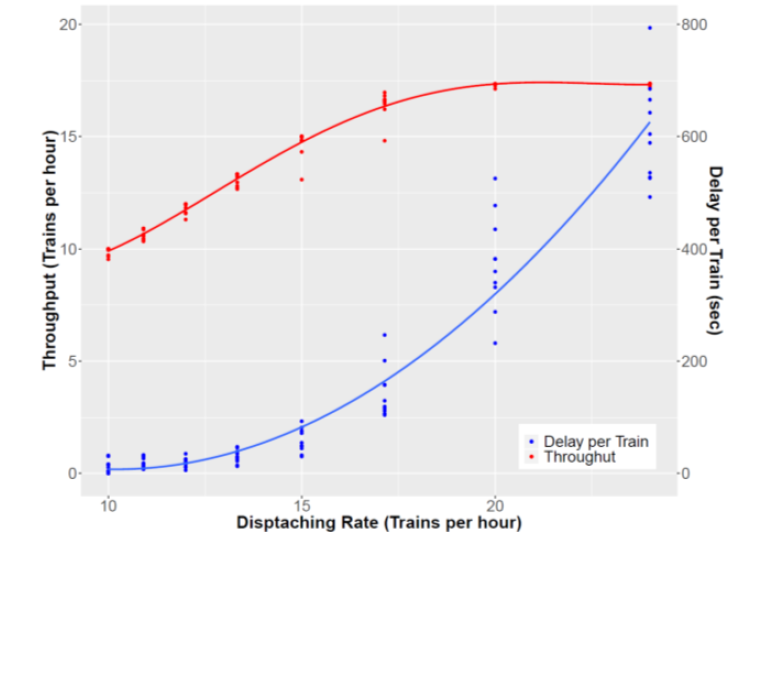} }
	\caption {Train throughput, dispatching rate, speed and delay from traffic data (metro Cambridge), block system~\cite{SWKZ19}.}
	\label{saidi}
	\end{minipage}
\end{figure}

Precisely, based on train operations data including train speed and dwell time variability, the authors of~\cite{SWKZ19} derive a train following model for MBA Red Line, metro Cambridge, with a fixed block signaling system. On this line with the block system, the speed of a train depends on the location of the preceding train, see Figure~\ref{saidi-overlap}. Note that if block $j+1$ is occupied, block $j$ needs to be clear. Trains entering block $j-1$ cannot enter with full speed but receive a reduced speed code. This alerts the driver and makes sure that it can stop before block $j$. 
Still in Figure~\ref{saidi-overlap}, the authors have studied the relation and derived a model for the dependency between travel time on block $j-1$ just before the bottleneck station, and the position of the lead and the following train. Their train interaction is given by the overlap time and accounts for the effect of signals, block boundaries, speed codes and train characteristics. As it is shown in the graph, as soon as trains start to interact (positive overlap time), the travel time on block $j-1$ depends linearly on the overlap.

The authors of~\cite{SWKZ19} have applied this empirically derived model to all block sections of the line to generated macroscopic variables curves, see Figure~\ref{saidi}. The throughput (or train flow) and the delay per train, depending on the dispatching rate (or initial headway), that is the number trains dispatched from the terminus station per time interval, related to the train density. Furthermore, they have analyzed the throughput and have shown that it first depends linearly on the dispatching rate. For high dispatching rates, the maximum throughput is constrained by the bottleneck of the line. Further increasing the initial dispatching rate does therefore not further increase the throughput, but train delays due to congestion around the bottlenecks.
For the relationship between train throughput and dispatching rate (red curve in Figure~\ref{saidi}), a free flow phase can be distinguished, where the flow depends linearly on the initial headway. From a certain point on, the flow can still be increased, but typically by reducing train approaching speeds, which leads to an increasing train delay (blue curve in Figure~\ref{saidi}). This is coherent with the result presented by Corman et al.~\cite{CHKE19} presented above. However, in the contrary to the analysis conducted by~\cite{CHKE19}, the train flow reaches a horizontal plateau phase, where, with an increasing dispatching rate, the train flow remains at its maximum level. This result might be due to the fact that the authors here have modeled the specificity of a railway block signaling system. This is particularly interesting with regard to the results presented in this thesis, where a similar plateau phase can be found.

\subsection*{Max-plus Algebra Modeling Approaches}
The latest works on metro traffic modeling and control are from Farhi et al.~\cite{FNHL16,FNHL17a,FNHL17b,F18}.
The authors have proposed a new approach using max-plus algebra which allows to derive the traffic phases of the train dynamics describing the physics of traffic of a metro line in the stationary regime.

Farhi et al.~\cite{FNHL17a} have proposed a novel model of the train traffic dynamics on a linear metro line. The main idea is to model the dynamics as a discrete events system, where an event is the departure of a train from a segment. Therefore, the metro line is discretized in space into a number of segments. A segment is a part of the line in which there can be at maximum one train at a time. This coincides with the blocks of a standard railway signaling system.
The discrete events are the counted train departures in time from a given segment. The departures are counted separately on every segment.
The new approach presented by Farhi et al. relaxes the shortcoming common to some approaches cited above based on van Breusegem et al.~\cite{BCB91}, who do not explicitly model the block system. The model presented in~\cite{FNHL17a} can represent as many trains per inter-station as there are blocks.
The authors apply the max-plus traffic modeling approach to a linear metro line without junction. The train dynamics are written on two constraints.
One constraint is on the travel time (= sum of run + dwell time) of a train. Trains are supposed respect given lower bounds on the train run and dwell times.
Another constraint is on the safe separation time which is enforced between two trains following each other. 

The traffic model is a max-plus linear analytical model which allows to derive closed-form solutions for the asymptotic average growth rate of the system, interpreted as the average train time-headway on the line, accordingly to Theorem~\ref{max-plus-theorem}.
The main result, illustrated in Figure~\ref{farhi-1} for the case of a linear metro line, consists in a fundamental traffic diagram that links the average train frequency at any point on the line, to the number of trains on the line.
Three traffic phases of the train dynamics can be identified.
Firstly, a free flow traffic phase, where the observed train frequency increases linearly with the number of trains.
This traffic phase is bounded by a horizontal maximum frequency phase, corresponding to the capacity of the line. In this phase, the average frequency is independent of the number of trains on the line.
Finally, for very dense traffic, congestion occurs and the frequency decreases with an increasing number of trains.
The fundamental diagram can be used to identify the minimum number of trains for which the average train frequency is maximized. It lies at the intersection between the first and second traffic phase.

Figure~\ref{farhi-2} shows the variation of train dwell time ($w$), dynamic interval ($g$) and headway ($h$, all in seconds), accordingly to the traffic phases of Figure~\ref{farhi-1}. The dynamic interval is by definition the time interval during which there is no train at a given platform. Consequently, the headway is the sum of dynamic interval and dwell time. It can be seen, that the dwell time is constant in the free flow phase, but increases in the maximum frequency phase when the average frequency still remains constant. This can be interpreted as a stop and go effect around the bottleneck of the line. Not surprisingly, dwell times increase significantly in the congestion phase.
Furthermore, the dynamic interval decreases in the free flow phase with an increasing number of trains. It is minimal at the end of the maximum frequency phase and in the congestion phase. It can be seen that the headway is minimal and constant in the maximum frequency phase. However, with an increasing number of trains, the dynamic interval between two consecutive trains decreases at the cost of an increasing dwell time.

\begin{figure}
 \centering
 \begin{minipage}{0.49\textwidth}
 	\centering
 	\frame{
  	\includegraphics[width=\textwidth]{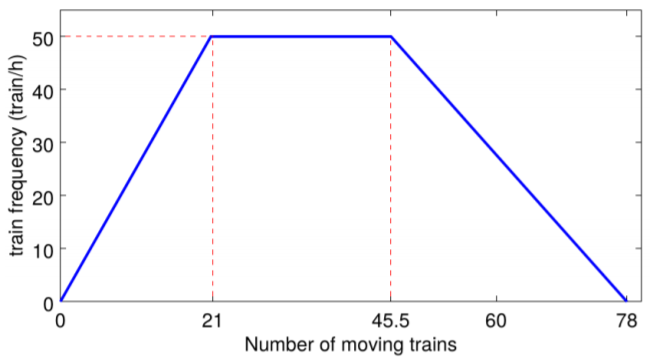}}
	\caption {Asymptotic average train frequency on a linear metro line over number of trains~\cite{FNHL16}.}
	\label{farhi-1}
	\end{minipage}
\begin{minipage}{0.49\textwidth}
 	\centering
 	\frame{
  	\includegraphics[width=\textwidth]{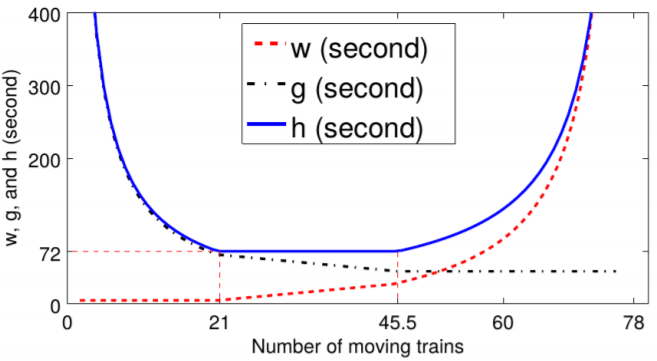} }
	\caption {Dwell time $w$, dynamic interval $g$, headway $h$ on a linear metro line over number of trains~\cite{FNHL16}.}
	\label{farhi-2}
	\end{minipage}
\end{figure}

Farhi et al.~\cite{FNHL17b} have also presented another model, where the train dwell times are a function of the passenger travel demand. Precisely, the train dwell times take into account the number of passenger on the platform, which depends on the average passenger travel demand and on the dynamic interval of the train.
Since the dynamic interval is the time in which there is no train at the platform, passengers accumulate over this interval.
The authors have shown that the train traffic dynamics on a linear line, with demand-dependent dwell times and without any further control of the train dwell and/or run times, are naturally unstable. The instability of these systems has already been pointed out in the initial works of metro traffic modeling and control by van Breusegem et al.~\cite{BCB91}.


\chapter{The Fundamental Diagram of a Metro Line with a Junction}\label{1A}
\chaptermark{The Fundamental Diagram}

\begin{quote}{This chapter presents a discrete event traffic model of the train traffic dynamics on a metro line with a junction where trains respect given lower bounds on train dwell, run and safe separation times.
{It is shown that the model can be written linearly in the max-plus algebra of polynomial matrices. The system reaches a stationary regime, with an unique asymptotic average growth rate, which is derived and interpreted as the average train time-headway on the line.}
From this result, eight traffic phases of the train dynamics are distinguished for a line with a junction and interpreted in terms of traffic. They are represented in a phase diagram, called the fundamental diagram of a metro line with a junction.
Finally, macroscopic feedback control laws for the number of trains on each part of the line are derived from the fundamental diagram.
They allow to control the number of trains with regard to a change in the passenger travel demand volume or to respond to important disturbances on the line, in case included time margins are insufficient to recover perturbations.}
\end{quote}
\section{A Model of the Train Traffic Dynamics on a Metro Line with a Junction}\label{1A-a}
\sectionmark{A Traffic Model of a Line with a Junction}
\subsection{Train Traffic Dynamics Modeling}\label{1A-a-i}

This thesis presents a series of traffic models for linear metro lines or lines with one junction.
All models are mathematical representations of the train traffic dynamics on the metro line.
{The dynamics considered here describe the movement of the trains which are supposed to respect given lower bounds on train run, dwell and safe separation times.}

The entire work is based on a modeling approach developed in~\cite{F18,FNHL17a,FNHL17b,FNHL16} for metro train traffic dynamics on linear lines.
A metro line is called \textit{linear} if it is delimited by exactly two terminus stations, which are connected by two paths, one in each direction. There can be a certain number of platforms on the line between the terminus stations. Their number and localization can be different in the two directions.

A metro line is said to have \textit{one junction} if there are three terminus stations and the line includes a central part and two branches. The branches and the central part are connected via a junction node which is composed of a divergence and convergence. There are two paths between the junction node and the three terminus stations, one in each direction, see Figure~\ref{fig-line}.
%
At the divergence, trains go from the central part onto the branches, while at the convergence trains pull into the central part, coming from the branches.
Trains running on the divergence and the convergence do not run on the same track and do therefore not cross each other.

In the model for a line with a junction proposed here, the metro line is discretized in space in a number of segments. A segment is a part of the line in which can be maximum one train at a time.
Two segments are delimited by a node. { The number of segments on the central part, written $n_0$ is supposed to be even. This can be achieved by an adequate discretization.}
There can, but does not necessarily have to be, a platform on a node.
Two consecutive platforms can be separated by several segments.
Segments and nodes are indexed as in Figure~\ref{fig-line}.

Finally, trains stop at all stations.
The system is closed and is fully observable.

\begin{figure}[]
 \centering
  \includegraphics[width=\textwidth]{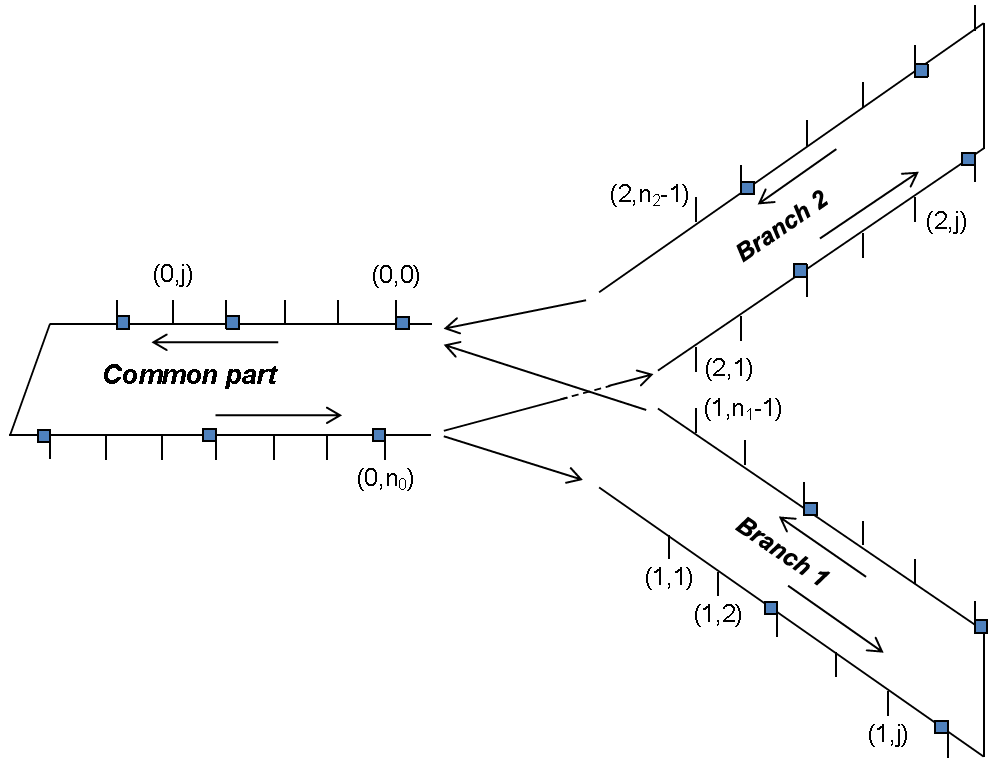} 
\caption [Schema of the metro line with one junction and the corresponding notation.]{Schema of the metro line with one junction and the corresponding notation.
The nodes and segments are indexed $(u,j)$ where a segment with $u=0$ is on the central part, a segment with $u=1$ is on branch 1, and a segment with $u=2$ is on branch 2. For $u=0, j \in \{0,1, \ldots, n_0\}$, for $u=1, j \in \{0,1, \ldots, n_1\}$  and for $u=2, j \in \{0,1, \ldots, n_2\}$. 
The nodes $(1,n_1)$ and $(2,n_2)$ are superposed with node $(0,0)$ as it is the case for the nodes $(1,0)$ and $(2,0)$ with node $(0,n_0)$. }
\label{fig-line}
\end{figure}

\subsubsection{Notations}\label{1A-a-ii} 
In order to model the metro train traffic dynamics on a line with one junction, consider the following notations.

\begin{tabular}{lp{0.8\textwidth}}
  $u$		& $\in \mathcal U = \{0,1,2\}$ indexes the central part if  $u=0$, branch 1 if $u=1$, branch 2 if $u=2$. \\
 $J(u)$	 & the set of indexes on part $u$ of the line, $J(u) = \{1,2, \ldots, n_u\}$.\\        
$(u,j)$ & indexes a node and a segment, in which can be maximum one train at a time. Segment $(u,j)$ connects nodes $(u,j-1)$ and $(u,j)$.\\
 	 $n_u$ & the number of segments on part $u$ of the line.\\
  $m_u$	& the number of trains on the part $u$ of the line.
\end{tabular}
~

\begin{tabular}{lp{0.8\textwidth}}
  $b_{(u,j)}$ & $\in \{0,1\}$. It is $0$ (resp. $1$) if there is no train (resp. one train) on segment $j$ of part $u$, in the intial state (= at time zero).\\
  $\bar{b}_{(u,j)}$ & $= 1 - b_{(u,j)}$.
\end{tabular}
~

\begin{tabular}{lp{0.8\textwidth}}
  $d^k_{(u,j)}$ & the $k^{\text{th}}$ departure time from node $j$, on part $u$ of the line. Notice that $k$ does not index trains, but counts the number of train departures. \\  
  $a^k_{(u,j)}$ & the $k^{\text{th}}$ arrival time to node $j$, on part $u$ of the line.
\end{tabular}
~

\begin{tabular}{lp{0.8\textwidth}}
  $r^k_{(u,j)}$ & the $k^{\text{th}}$ run time (of a train) on segment $(u,j)$ (between nodes $j-1$ and $j$ of part $u$).\\ 
  $w^k_{(u,j)}$ & $=d^k_{(u,j)}-a^k_{(u,j)}$ the $k^{\text{th}}$ dwell time on node $(u,j)$ (separating segments $(u,j)$ and $(u,j+1)$); $w^k_{(u,j)}= 0$ in case there is no platform, $w^k_{(u,j)}> 0$ in case there is a platform.
\end{tabular}
~

\begin{tabular}{lp{0.8\textwidth}}
  $t^k_{(u,j)}$ & $=r^k_{(u,j)}+w^k_{(u,j)}$ the $k^{\text{th}}$ travel time from node $j-1$ to node $j$ on part $u$ of the line.\\      
  $g^k_{(u,j)}$ & $=a^k_{(u,j)}-d^{k-1}_{(u,j)}$ the $k^{\text{th}}$ dynamic interval, between the departure of the preceeding train and the arrival of the following train at the same node $j$ on part $u$.
\end{tabular}
~

\begin{tabular}{lp{0.8\textwidth}}
  $h^k_{(u,j)}$ & $=d^k_{(u,j)}-d^{k-1}_{(u,j)} = g^k_{(u,j)}+w^{k}_{(u,j)}$ the $k^{\text{th}}$ train departure time-headway between the departure of the preceding train and the departure of the following train at the same node $j$ on part $u$ (note that $h^k_{(u,j)} = g^k_{(u,j)}$ if $w^k_{(u,j)} = 0$).\\
  $s^k_{(u,j)}$ & $=g^{k+b_{(u,j)}}_{(u,j)}-r^k_{(u,j)}$ the $k^{\text{th}}$ safe separation time between the departure of the preceeding train at node $j$ and the departure of the following train at the preceding node $j-1$.
\end{tabular}

The (asymptotic) averages on $j$ for each part $u$  and on $k$ of the variables are indexed $u$.
Then $r_u, w_u, t_u, g_u, s_u$ and $h_u$ denote the average run, dwell, travel, dynamic interval and safe separation times and the asymptotic average train time-headway, respectively.

It is easy to check the following relationships:  
  \begin{align}
    & g_u = r_u + s_u, \label{form1} \\
    & t_u = r_u + w_u, \label{form2} \\
    & h_u = g_u + w_u = r_u + w_u + s_u = t_u + s_u.
    \label{form3}
  \end{align}    

Underline notations are used to note minimum 
bounds of the corresponding variables.
Then $\underline{r}_{(u,j)}, \underline{w}_{(u,j)},  \underline{t}_{(u,j)}, \underline{g}_{(u,j)}$ and $\underline{s}_{(u,j)}$
denote respectively minimum run, dwell, travel, dynamic interval and safe separation times.

~\\
\textbf{\textit{One-over-two operational rule of the junction.}}\\
An \textit{one-over-two} operational rule at the junction applies in the steady state.
\begin{itemize}
\item At the divergence, odd departures from the last node of the central part $(0,n_0)$ go to branch~1, even ones go to branch~2.
\item At the convergence, odd departures from the first node of the central part $(0,0)$ come from branch~1, even ones come from branch~2.
\end{itemize}
Consequently, in two $k$ steps of the dynamics, that is an odd and an even departure from the junction, the difference between the number of trains on the branches is constant.

Here in Chapter~\ref{1A}, trains are supposed to respect minimum dwell, run and safe separation times, which holds for example for automated metro lines.
This already allows some conclusions on the train traffic dynamics on a line with one junction, which will be used for macroscopic traffic control.
In Chapter~\ref{Chap-4}, a margin on the run times is introduced, which will allow microscopic traffic regulation.

In the following, based on these hypotheses and the notations,
the train traffic dynamics on a metro line with one junction are modeled.
Three train dynamics are distinguished: the ones out of the junction, the ones on the divergence and the ones on the convergence.

\subsubsection{Train Traffic Dynamics Model}\label{1A-a-iii}
\subsubsection{Train Traffic Dynamics out of the Junction}
Below, the train traffic dynamics on a metro line with one junction are modeled.
Two main constraints are considered to describe the dynamics out of the junction.

~\\
\textbf{\textit{1 - A constraint on the run and dwell times}}
\begin{itemize}
  \item The $k^{\text{th}}$ train departure from node $j$ (of any part of the line) depends on
     the $k^{\text{th}}$ train departure from node $j-1$ in case there was no train 
     on segment $j$ at time zero; and it depends on the $(k-1)^{\text{th}}$ 
     train departure from node $j-1$ in case there was a train on segment $j$ at time zero.
     Between two consecutive departures, a minimum travel time of $\underline{t}_{(u,j)}$ is respected.
     \begin{equation}\label{eq-c1}
        d^k_{(u,j)} \geq d^{k-b_{(u,j)}}_{(u,j-1)} + \underline{t}_{(u,j)}, \; \forall k\geq 0, u\in\mathcal U, {j \in J(u)\setminus \{0\} \text{ and }  j \neq 1 \text{ for } u=1,2.}
     \end{equation}
\end{itemize}

~\\
\textbf{\textit{2 - A constraint on the safe separation time}}
\begin{itemize}   
  \item The $k^{\text{th}}$ train departure from node $j$ is constrained by the $(k-1)^{\text{th}}$ train departure
    from node $j+1$ plus a minimum safe separation time $\underline{s}_{(u,j+1)}$ in case there was no train on segment $j+1$ at time zero; and
    it is constrained by the $k^{\text{th}}$ train departure from node $j+1$ plus a minimum safe separation time $\underline{s}_{(u,j+1)}$
    in case there was a train on segment $j+1$ at time zero.
    \begin{equation}\label{eq-c2}
          d^k_{(u,j)} \geq d^{k-\bar{b}_{(u,j+1)}} _{(u,j+1)} + \underline{s}_{(u,j+1)}, \; \forall k\geq 0, u\in\mathcal U, {j \in J(u) \setminus \{n_0, n_1-1, n_2-1\}.}
    \end{equation}
\end{itemize}

It is assumed that a train departs from node $j$ out of the junction, as soon as the two
constraints~(\ref{eq-c1}) and~(\ref{eq-c2}) are satisfied.
\begin{equation}\label{eq-c3}
  d^k_{(u,j)} = \max \left\{ d^{k-b_{(u,j)}}_{(u,j-1)} + \underline{t}_{(u,j)}, d^{k-\bar{b}_{(u,j+1)}} _{(u,j+1)} + \underline{s}_{(u,j+1)}  \right\}.
\end{equation}
This assumption holds furthermore for all couples of constraints on the divergence and on the convergence which will be given below.

\subsubsection{Train Traffic Dynamics on the Divergence}
As explained at the beginning of this chapter, it is assumed that trains leaving the central part of the line and
pulling into the branches respect the following one-over-two rule.
Odd departures go to branch~1 while even departures go to branch~2.
More precisely, in equation~(\ref{eq-div2}) it is considered that departures from the last node of the central part $(0,n_0)$ are constrained by the preceding departure pulling into the same branch (and not by the departure of the train which has been just in front in the central part).

With this assumption, the constraints on the divergence can be written as follows.

The $k^{\text{th}}$ departures from the central part:
\begin{equation}\label{eq-div1}
  d^k_{(0,n_0)} \geq d^{k-b_{(0,n_0)}}_{(0,n_0-1)} + \underline{t}_{(0,n_0)}, \; \forall k\geq 0, \\
\end{equation}

\begin{equation}\label{eq-div2}
  d^k_{(0,n_0)} \geq \begin{cases}
                      d^{(k+1)/2-\bar{b}_{(1,1)}}_{(1,1)} + \underline{s}_{(1,1)} & \text{if } k \text{ is odd} \\ ~~ \\
                      d^{k/2-\bar{b}_{(2,1)}}_{(2,1)} + \underline{s}_{(2,1)} & \text{if } k \text{ is even} \\
                   \end{cases}
\end{equation}

The $k^{\text{th}}$ departures from the entry of branch 1:
\begin{align}
   d^k_{(1,1)} \geq d^{(2k-1)-2b_{(1,1)}}_{(0,n_0)} + \underline{t}_{(1,1)}, \; \forall k\geq 0, \label{eq-div3}\\
   d^k_{(1,1)} \geq d^{k-\bar{b}_{(1,2)}} _{(1,2)} + \underline{s}_{(1,2)}, \; \forall k\geq 0. \label{eq-div4}
\end{align}

The $k^{\text{th}}$ departures from the entry of branch 2:
\begin{align}
   d^k_{(2,1)} \geq d^{2k-2b_{(2,1)}}_{(0,n_0)} + \underline{t}_{(2,1)}, \; \forall k\geq 0, \label{eq-div5}\\
   d^k_{(2,1)} \geq d^{k-\bar{b}_{(2,2)}}_{(2,2)} + \underline{s}_{(2,2)}, \; \forall k\geq 0. \label{eq-div6}
\end{align}

\subsubsection{Train Traffic Dynamics on the Convergence}
Similarly to the train traffic dynamics on the divergence, it is assumed that trains entering the central part from 
the two branches respect the following one-over-two rule.
Odd departures at node $(0,0)$ towards the central part correspond to trains coming from
branch~1 while even ones correspond to trains coming from branch~2.
At the last nodes on the branches $(1,n_1-1)$, $(2,n_2-1)$, there are two different possibilities to model the dynamics, which are explained in the following.

~\\
\textbf{\textit{Convergence model 1.}}\\
In this convergence model, it is considered that departures from the last nodes on the branches $(1,n_1-1)$, $(2,n_2-1)$ are constrained by a departure from node $(0,0)$, realized by a train from the same branch.
For example, a departure from $(2,n_2-1)$ must respect a departure at $(0,0)$ realized by a train from this branch~2 plus a minimum safe separation time $\underline{s}_{(2,n_2)}$.
Similarly, a departure from $(1,n_1-1)$ must respect a departure at $(0,0)$ realized by a train from this branch~1 plus a minimum safe separation time $\underline{s}_{(1,n_1)}$.
Consequently, there are two positions to stock trains at node $(0,0)$, one for trains from branch~1 and one for trains from branch~2.
The one-over-two junction operations rule is applied directly at node $(0,0)$.
The model is then the following.

The $k^{\text{th}}$ departures from the central part:
\begin{equation}\label{eq-conv1}
  d^k_{(0,0)} \geq \begin{cases}
                      d^{(k+1)/2-b_{(1,n_1)}}_{(1,n_1-1)} + \underline{t}_{(1,n_1)} & \text{if } k \text{ is odd} \\ ~~ \\
                      d^{k/2-b_{(2,n_2)}}_{(2,n_2-1)} + \underline{t}_{(2,n_2)} & \text{if } k \text{ is even} \\
                   \end{cases}
\end{equation}
\begin{equation}\label{eq-conv2}
  d^k_{(0,0)} \geq d^{k-\bar{b}_{(0,1)}}_{(0,1)} + \underline{s}_{(0,1)}, \; \forall k\geq 0, \\
\end{equation}

The $k^{\text{th}}$ departures from the exit of branch 1:
\begin{align}
   d^k_{(1,n_1-1)} \geq d^{k-b_{(1,n_1-1)}}_{(1,n_1-2)} + \underline{t}_{(1,n_1-1)}, \; \forall k\geq 0, \label{eq-conv3}\\
   d^k_{(1,n_1-1)} \geq d^{2k-1-2\bar{b}_{(1,n_1)}}_{(0,0)} + \underline{s}_{(1,n_1)}, \; \forall k\geq 0. \label{eq-conv4}
\end{align}

The $k^{\text{th}}$ departures from the exit of branch 2:
\begin{align}
   d^k_{(2,n_2-1)} \geq d^{k-b_{(2,n_2-1)}}_{(2,n_2-2)} + \underline{t}_{(2,n_2-1)}, \; \forall k\geq 0, \label{eq-conv5} \\
   d^k_{(2,n_2-1)} \geq d^{2k-2\bar{b}_{(2,n_2)}}_{(0,0)} + \underline{s}_{(2,n_2)}, \; \forall k\geq 0. \label{eq-conv6}
\end{align}

This model is used throughout this thesis.

~\\
\textbf{\textit{Convergence model 2.}}\\
Another possibility is to suppose that departures from the last nodes on the branches $(1,n_1-1)$, $(2,n_2-1)$ are constrained by a departure from node $(0,0)$, realized by a train from the other branch.
For example, a departure from $(2,n_2-1)$ must respect a departure at $(0,0)$ realized by a train from branch~1 plus a minimum safe separation time $\underline{s}_{(2,n_2)}$.
Similarly, a departure from $(1,n_1-1)$ must respect a departure at $(0,0)$ realized by a train from branch~2 plus a minimum safe separation time $\underline{s}_{(1,n_1)}$.
This means, a train can only depart from the last node on the branches once both segments $(1,n_1)$ and $(2,n_2)$ are free. Consequently, there can be maximum one train at a time on the segments  $(1,n_1)$, $(2,n_2)$.

The $k^{\text{th}}$ departures from the exit of branch 1, only for the safe separation constraint, to compare to~(\ref{eq-conv4}):
\begin{align}
   d^k_{(1,n_1-1)} \geq d^{2k-2\bar{b}_{(1,n_1)}}_{(0,0)} + \underline{s}_{(1,n_1)}, \; \forall k\geq 0. \label{eq-conv-4bis}
\end{align}

The $k^{\text{th}}$ departures from the exit of branch 2, only for the safe separation constraint, to compare to~(\ref{eq-conv6}):
\begin{align}
   d^k_{(2,n_2-1)} \geq d^{2k+1-2\bar{b}_{(2,n_2)}}_{(0,0)} + \underline{s}_{(2,n_2)}, \; \forall k\geq 0. \label{eq-conv-6bis}
\end{align}

~\\
\textbf{\textit{Example.}}
\begin{table}
\centering
\begin{tabular}{l | l | l | l || l }~\label{ex-conv1}
$k \text{ in}$ & $d_{(1,n_1-1)}$ & $d_{(2,n_2-1)}$ & $d_{(0,0)}$ (model~2) & $d_{(0,0)}$ (model~1, used here) \\
\hline
	& 1	& - 	& $2-2\bar{b}_{(1,n_1)}$ 	& $1-2\bar{b}_{(1,n_1)}$ \\
	& - 	& 1 	& $3-2\bar{b}_{(2,n_2)}$			& $2-2\bar{b}_{(2,n_2)}$\\
	& 2	& - 	& $4-2\bar{b}_{(1,n_1)}$ 	& $3-2\bar{b}_{(1,n_1)}$ \\
	& - 	& 2 	& $5-2\bar{b}_{(2,n_2)}$			& $4-2\bar{b}_{(2,n_2)}$\\
\end{tabular}
\caption{The train departure counter index of the safe separation constraint of the convergence (before the changing of variables), for convergence model~1 and convergence model~2.}
\end{table}

Table~\ref{ex-conv1} depicts the evolution of the train departure counter index $k$ for both convergence models.
Model~1, that is equations~(\ref{eq-conv4}), (\ref{eq-conv6}) are represented in the last column.
Model~2, that is equations~(\ref{eq-conv-4bis}), (\ref{eq-conv-6bis}) are represented in the second last column.
It becomes clear that there is difference of one k-step between the two models (see last two columns).
In model~1, departures from the last nodes on the branches are constraint by a train from the same branch.
In model~2, departures from the last nodes on the branches are constraint by a train from the other branch.

If there is a train, at time 0, on segment $(2,n_2)$, this train has to wait for one train from branch~1 realizing the first departure at node $(0,0)$. Therefore, the train which is on $(2,n_2)$ in the initial state, realizes the second departure at node $(0,0)$.
In both cases, on the convergence (as on the divergence), odd departures always correspond to trains from (to) branch~1 and even departures are trains from (to) branch~2, independently of the initial state.

\subsection{Changing of Variables and New Train Traffic Dynamics}\label{1A-a-iv}

In the preceding part, the train traffic dynamics on a metro line with one junction have been modeled.
The main variable is the $k^\text{th}$ departure time from a node $(u,j)$.
Most importantly, the index \textit{k} counts the number of train departures, but does not index a specific train.
Since the system considered here has an one-over-two operated junction and two branches, the number of train departures indexed \textit{k} on the central part, is in average twice the one on the branches.
To illustrate this, consider the dynamic given by~(\ref{eq-div3}).
The departure from the first node on branch~1 $d^\textbf{k}_{(1,1)}$ is given as a function of the departure from the last node from the central part $d^{\textbf{2k}-1-2b_{(1,1)}}_{(0,n_0)}$.
The two sequences do not have the same growth rate.
Indeed, the growth rate of $d_{(0,n_0)}$ is in average double the one of $d_{(1,1)}$.
This is a consequence of the one-over-two rule at the junction.

In order to have all the sequences of the train dynamics growing with the same speed,
the following change of variables is introduced:
\begin{align}
  & \delta^k_{(0,j)} = d^k_{(0,j)}, \forall k\geq 0, \forall j  \\
  & \delta^{2k}_{(1,j)} = d^k_{(1,j)}, \forall k\geq 0, \forall j  \\
  & \delta^{2k}_{(2,j)} = d^k_{(2,j)}, \forall k\geq 0, \forall j.
\end{align}

The train dynamics after changing of variables are rewritten as follows.

\subsubsection{New Train Traffic Dynamics out of the Junction after Changing of Variables}
The train traffic dynamics out of the junction~(\ref{eq-c1}) and~(\ref{eq-c2}) are rewritten as follows.

On the central part, it is sufficient to replace $d$ with $\delta$:
\begin{align}
   \delta^k_{(0,j)} \geq \delta^{k-b_{(0,j)}}_{(0,j-1)} + \underline{t}_{(0,j)}, \; \forall k\geq 0,{ j \neq 0}, \label{ch-neq-c1}\\
   \delta^k_{(0,j)} \geq \delta^{k-\bar{b}_{(0,j+1)}} _{(0,j+1)} + \underline{s}_{(0,j+1)}, \; \forall k\geq 0, j \neq n_0. \label{ch-neq-c2}
\end{align}

The dynamics on the branches are rewritten as follows:
\begin{equation}
   \delta^{2k}_{(u,j)} \geq \delta^{2k-2b_{(u,j)}}_{(u,j-1)} + \underline{t}_{(u,j)}, \; \forall k \geq 0, u\in\{1,2\}, j \neq 0,1, \label{ch-neq-c3}
\end{equation}
\begin{equation}
   \delta^{2k}_{(u,j)} \geq \delta^{2k-2\bar{b}_{(u,j+1)}} _{(u,j+1)} + \underline{s}_{(u,j+1)}, \; \forall k \geq 0, u\in\{1,2\}, {j \neq n_u-1, n_u}. \label{ch-neq-c4}
\end{equation}

\subsubsection{New Train Traffic Dynamics on the Divergence after Changing of Variables}
The dynamics on the divergence are rewritten as follows.

The $k^{\text{th}}$ departures from the central part as given by equations~(\ref{eq-div1}) and~(\ref{eq-div2}):
\begin{equation}\label{ch-neq-div1}
  \delta^{k}_{(0,n_0)} \geq \delta^{k-b_{(0,n_0)}}_{(0,n_0-1)} + \underline{t}_{(0,n_0)}, \; \forall k\geq 0, \\
\end{equation}

\begin{equation}\label{ch-neq-div2}
  \delta^{k}_{(0,n_0)} \geq \begin{cases}
                      \delta^{k+1-2\bar{b}_{(1,1)}}_{(1,1)} + \underline{s}_{(1,1)} & \forall k = {2p-1, \text{ with } p \in \mathbb{N}}, \\ ~~ \\
                      \delta^{k-2\bar{b}_{(2,1)}}_{(2,1)} + \underline{s}_{(2,1)} & \forall k = {2p, \text{ with } p \in \mathbb{N}}. \\
                   \end{cases}
\end{equation}

The $k^{\text{th}}$ departures from the entry of branch 1 as given by equations~(\ref{eq-div3}) and~(\ref{eq-div4}):
\begin{align}
   \delta^{2k}_{(1,1)} \geq \delta^{2k-1-2b_{(1,1)}}_{(0,n_0)} + \underline{t}_{(1,1)}, \; \forall k \geq 0, \label{ch-neq-div3}\\
   \delta^{2k}_{(1,1)} \geq \delta^{2k-2\bar{b}_{(1,2)}} _{(1,2)} + \underline{s}_{(1,2)}, \; \forall k \geq 0. \label{ch-neq-div4}
\end{align}

The $k^{\text{th}}$ departures from the entry of branch 2 as given by equations~(\ref{eq-div5}) and~(\ref{eq-div6}):
\begin{align}
   \delta^{2k}_{(2,1)} \geq \delta^{2k-2b_{(2,1)}}_{(0,n_0)} + \underline{t}_{(2,1)}, \; \forall k \geq 0, \label{ch-neq-div5} \\
   \delta^{2k}_{(2,1)} \geq \delta^{2k-2\bar{b}_{(2,2)}}_{(2,2)} + \underline{s}_{(2,2)}, \; \forall k \geq 0. \label{ch-neq-div6}
\end{align}

\subsubsection{New Train Traffic Dynamics on the Convergence after Changing of Variables}
The dynamics on the convergence (convergence model 1) are rewritten as follows.

The $k^{\text{th}}$ departures from the central part as given by equations~(\ref{eq-conv1}) and~(\ref{eq-conv2}):
\begin{equation} \label{ch-neq-conv1}
  \delta^k_{(0,0)} \geq \begin{cases}
                      \delta^{(k+1)-2b_{(1,n_1)}}_{(1,n_1-1)} + \underline{t}_{(1,n_1)} & \forall k = {2p-1, \text{ with } p \in \mathbb{N}}, \\ ~~ \\
                      \delta^{k-2b_{(2,n_2)}}_{(2,n_2-1)} + \underline{t}_{(2,n_2)} & \forall k = {2p, \text{ with } p \in \mathbb{N}}. \\
                   \end{cases}
\end{equation}
\begin{equation}\label{ch-neq-conv2}
  \delta^k_{(0,0)} \geq \delta^{k-\bar{b}_{(0,1)}}_{(0,1)} + \underline{s}_{(0,1)}, \; \forall k\geq 0, \\
\end{equation}

The $k^{\text{th}}$ departures from the exit of branch 1 as given by equations~(\ref{eq-conv3}) and~(\ref{eq-conv4}):
\begin{align}
   \delta^{2k}_{(1,n_1-1)} \geq \delta^{2k-2b_{(1,n_1-1)}}_{(1,n_1-2)} + \underline{t}_{(1,n_1-1)}, \; \forall k \geq 0, \label{ch-neq-conv3} \\
   \delta^{2k}_{(1,n_1-1)} \geq \delta^{2k-1-2\bar{b}_{(1,n_1)}}_{(0,0)} + \underline{s}_{(1,n_1)}, \; \forall k \geq 0. \label{ch-neq-conv4}
\end{align}

The $k^{\text{th}}$ departures from the exit of branch 2 as given by equations~(\ref{eq-conv5}) and~(\ref{eq-conv6}):
\begin{align}
   \delta^{2k}_{(2,n_2-1)} \geq \delta^{2k-2b_{(2,n_2-1)}}_{(2,n_2-2)} + \underline{t}_{(2,n_2-1)}, \; \forall k \geq 0, \label{ch-neq-conv5} \\
   \delta^{2k}_{(2,n_2-1)} \geq \delta^{2k-2\bar{b}_{(2,n_2)}}_{(0,0)} + \underline{s}_{(2,n_2)}, \; \forall k \geq 0. \label{ch-neq-conv6}
\end{align}

\subsubsection{Index Simplification and new Train Dynamics}\label{1A-a-v}
Furthermore, the simplification below is applied: 
\begin{align}
& {\delta^{l}_{(0,j)} = \delta^{k}_{(0,j)}, \forall k \in \mathbb{N}, \forall j,} \label{simpl0} \\
& {\delta^{l}_{(1,j)} = \delta^{2k}_{(1,j)}, \forall k \in \mathbb{N}, \forall j,} \label{simpl1}\\
& {\delta^{l}_{(2,j)} = \delta^{2k}_{(2,j)}, \forall k \in \mathbb{N}, \forall j. }\label{simpl2}
\end{align}
In the following, all the dynamics are rewritten with the index simplification.

\subsubsection{New Train Traffic Dynamics out of the Junction after Index Simplification}
The train traffic dynamics out of the junction are rewritten as follows.

On the central part, equations~(\ref{ch-neq-c1}) and~(\ref{ch-neq-c2}) are rewritten accordingly to~(\ref{simpl0}):
\begin{align}
   \delta^{{{l}}}_{(0,j)} \geq \delta^{{{l}-b_{(0,j)}}}_{(0,j-1)} + \underline{t}_{(0,j)}, \; \forall {l\geq 0, j \neq 0}, \label{neq-c1}\\
   \delta^{{{l}}}_{(0,j)} \geq \delta^{{{l}-\bar{b}_{(0,j+1)}}} _{(0,j+1)} + \underline{s}_{(0,j+1)}, \; \forall {l}\geq 0, j \neq n_0. \label{neq-c2}
\end{align}

The dynamics on the branches as given by equations~(\ref{ch-neq-c3}) and~(\ref{ch-neq-c4}) are rewritten with the simplification~(\ref{simpl1}),~(\ref{simpl2}) as follows:
\begin{equation}
   \delta^{{{l}}}_{(u,j)} \geq \delta^{{{l}-2b_{(u,j)}}}_{(u,j-1)} + \underline{t}_{(u,j)}, \; \forall {l = 2p, \text{ with } p \in \mathbb{N}}, u\in\{1,2\}, j \neq {0,1}, \label{neq-c3}
\end{equation}
\begin{equation}
   \delta^{{{l}}}_{(u,j)} \geq \delta^{{{l}-2\bar{b}_{(u,j+1)}}} _{(u,j+1)} + \underline{s}_{(u,j+1)}, \; \forall {l = 2p, \text{ with } p \in \mathbb{N}, u\in\{1,2\}, j \neq n_u-1, n_u}. \label{neq-c4}
\end{equation}

\subsubsection{New Train Traffic Dynamics on the Divergence after Index Simplification}
The dynamics on the divergence are rewritten as follows.

The $l^{\text{th}}$ departures from the central part as given by equations~(\ref{ch-neq-div1}) and~(\ref{ch-neq-div2}):
\begin{equation}\label{neq-div1}
  \delta^{{l}}_{(0,n_0)} \geq \delta^{{l}-b_{(0,n_0)}}_{(0,n_0-1)} + \underline{t}_{(0,n_0)}, \; \forall {l}\geq 0, \\
\end{equation}

\begin{equation}\label{neq-div2}
  \delta^{{l}}_{(0,n_0)} \geq \begin{cases}
                      \delta^{{l}+1-2\bar{b}_{(1,1)}}_{(1,1)} + \underline{s}_{(1,1)} & \forall {l = 2p-1, \text{ with } p \in \mathbb{N}}, \\ ~~ \\
                      \delta^{{l}-2\bar{b}_{(2,1)}}_{(2,1)} + \underline{s}_{(2,1)} & \forall {l \in 2p, \text{ with } p \in \mathbb{N}}. \\
                   \end{cases}
\end{equation}

The $l^{\text{th}}$ departures from the entry of branch 1 as given by equations~(\ref{ch-neq-div3}) and~(\ref{ch-neq-div4}):
\begin{align}
   \delta^{{l}}_{(1,1)} \geq \delta^{{l}-1-2b_{(1,1)}}_{(0,n_0)} + \underline{t}_{(1,1)}, \; \forall {l = 2p, \text{ with } p \in \mathbb{N}}, \label{neq-div3}\\
   \delta^{{l}}_{(1,1)} \geq \delta^{{l}-2\bar{b}_{(1,2)}} _{(1,2)} + \underline{s}_{(1,2)}, \; \forall {l = 2p, \text{ with } p \in \mathbb{N}}. \label{neq-div4}
\end{align}

The $l^{\text{th}}$ departures from the entry of branch 2 as given by equations~(\ref{ch-neq-div5}) and~(\ref{ch-neq-div6}):
\begin{align}
   \delta^{{l}}_{(2,1)} \geq \delta^{{l}-2b_{(2,1)}}_{(0,n_0)} + \underline{t}_{(2,1)}, \; \forall {l = 2p, \text{ with } p \in \mathbb{N}}, \label{neq-div5} \\
   \delta^{{l}}_{(2,1)} \geq \delta^{{l}-2\bar{b}_{(2,2)}}_{(2,2)} + \underline{s}_{(2,2)}, \; \forall {l = 2p, \text{ with } p \in \mathbb{N}}. \label{neq-div6}
\end{align}

\subsubsection{New Train Traffic Dynamics on the Convergence after Index Simplification}
The dynamics on the convergence (convergence model 1) are rewritten as follows.

The $l^{\text{th}}$ departures from the central part as given by equations~(\ref{ch-neq-conv1}) and~(\ref{ch-neq-conv2}):
\begin{equation} \label{neq-conv1}
  \delta^{{{l}}}_{(0,0)} \geq \begin{cases}
                      \delta^{{({l}+1)-2b_{(1,n_1)}}}_{(1,n_1-1)} + \underline{t}_{(1,n_1)} & \forall {l = 2p-1, \text{ with } p \in \mathbb{N}}, \\ ~~ \\
                      \delta^{{{l}-2b_{(2,n_2)}}}_{(2,n_2-1)} + \underline{t}_{(2,n_2)} & \forall {l = 2p, \text{ with } p \in \mathbb{N}}. \\
                   \end{cases}
\end{equation}
\begin{equation}\label{neq-conv2}
  \delta^{{{l}}}_{(0,0)} \geq \delta^{{{l}-\bar{b}_{(0,1)}}}_{(0,1)} + \underline{s}_{(0,1)}, \; \forall {l\geq 0}, \\
\end{equation}

The $l^{\text{th}}$ departures from the exit of branch 1 as given by equations~(\ref{ch-neq-conv3}) and~(\ref{ch-neq-conv4}):
\begin{align}
   \delta^{{{l}}}_{(1,n_1-1)} \geq \delta^{{{l}-2b_{(1,n_1-1)}}}_{(1,n_1-2)} + \underline{t}_{(1,n_1-1)}, \; \forall {l = 2p, \text{ with } p \in \mathbb{N}}, \label{neq-conv3} \\
   \delta^{{{l}}}_{(1,n_1-1)} \geq \delta^{{{l}-1-2\bar{b}_{(1,n_1)}}}_{(0,0)} + \underline{s}_{(1,n_1)}, \; \forall {l = 2p, \text{ with } p \in \mathbb{N}}. \label{neq-conv4}
\end{align}

The $l^{\text{th}}$ departures from the exit of branch 2 as given by equations~(\ref{ch-neq-conv5}) and~(\ref{ch-neq-conv6}):
\begin{align}
   \delta^{{{l}}}_{(2,n_2-1)} \geq \delta^{{{l-2b_{(2,n_2-1)}}}}_{(2,n_2-2)} + \underline{t}_{(2,n_2-1)}, \; \forall {l = 2p, \text{ with } p \in \mathbb{N}}, \label{neq-conv5} \\
   \delta^{{{l}}}_{(2,n_2-1)} \geq \delta^{{{l-2\bar{b}_{(2,n_2)}}}}_{(0,0)} + \underline{s}_{(2,n_2)}, \; \forall {l = 2p, \text{ with } p \in \mathbb{N}}. \label{neq-conv6}
\end{align}

~\\
\textbf{\textit{ Example.}}

\begin{table}
\centering
\begin{tabular}{l | l |l | l | l | l  }~\label{ex-conv2}
${l} \text{ in}$ & $d_{(1,n_1-1)}$ & $d_{(2,n_2-1)}$ & $\delta_{(1,n_1-1)}$ & $\delta_{(2,n_2-1)}$ & $\delta_{(0,0)}$ (model~1)  \\
\hline
	& 1	& -& 2	& - 	 	& $1-2\bar{b}_{(1,n_1)}$ \\
	& -	& 1& - 	& 2 	 & $2-2\bar{b}_{(2,n_2)}$\\
	& 2	& -& 4	& - 		& $3-2\bar{b}_{(1,n_1)}$ \\
	& -	& 2& - 	& 4 	& $4-2\bar{b}_{(2,n_2)}$\\
\end{tabular}
\caption{The train departure counter index of the safe separation constraint of the convergence, after the changing of variables, for convergence model~1.}
\end{table}

Table~\ref{ex-conv2} illustrates the safe separation constraint model~1 of the convergence after the changing of variables.
It depicts the departure counter index {{$l$}} of equations (\ref{neq-conv4}),~(\ref{neq-conv6}).

Trains depart as soon as the two constraints on the travel and on the safe separation time, are satisfied.
Then, all couples of constraints can be rewritten, see equation~(\ref{eq-c3}):
\begin{equation}\label{eq-mp1}
  \delta^{{{l}}}_{(u,j)} = \max \left\{ \delta^{{{l-b_{(u,j)}}}}_{(u,j-1)} + \underline{t}_{(u,j)}, \delta^{{{l-\bar{b}_{(u,j+1)}}}} _{(u,j+1)} + \underline{s}_{(u,j+1)}  \right\}.
\end{equation}

Finally, it can be seen that with the changing of variables, the number of train departures on the branches has been doubled, whereas the one on the central part remains unchanged.
For example, note that the second departure from node $(1,n_1-1)$ is in reality the first departure from this node.

\section{Model in Max-plus Algebra}\label{1A-b}

In this section, the train traffic dynamics on a line with a junction which have been written above separately on the travel time and on the safe separation time for each node (see equations~(\ref{neq-c1})-(\ref{neq-conv6})), and which, when combining travel time and safe separation time constraint, are for all the nodes of the form of constraint~(\ref{eq-mp1}), will be written linearly in the max-plus algebra of polynomial matrices, summarizing the dynamics on the entire line.


Furthermore, remember that the couples of constraints of the dynamics are written on the counted $l^{th}$ departures from a given node (see equation~(\ref{eq-mp1})).
Depending on the initial condition (initial train departures), the $l^{th}$ departure from a node depends either on the $l^{th}$ or on the $(l-1)^{th}$ departure from, first, the node behind (travel time constraint) and, second, the node in front (safe separation time constraint).
Because of the linearity of the dynamics in max-plus algebra and the dependency on the initial condition, a convenient representation of the dynamics of the entire line is the max-plus algebra of polynomial matrices with their backshift operator $\gamma$ between the counted train departures $l$ on two neighbored nodes on the line.



In this section the train dynamics will be written in a max-plus polynomial matric form. For that, a number of matrices will be considered below.
Consider the following matrices:

\begin{equation}
   A_{1}(\gamma) = \begin{pmatrix}
                A_{00}(\gamma) & A_{01}(\gamma) & \varepsilon \\
                A_{10}(\gamma) & A_{11}(\gamma) & \varepsilon \\
                \varepsilon & \varepsilon & A_{22}(\gamma)
             \end{pmatrix},
\end{equation}
\begin{equation}
   A_{2}(\gamma) = \begin{pmatrix}
                A_{00}(\gamma) & \varepsilon & A_{02}(\gamma) \\
                \varepsilon & A_{11}(\gamma) & \varepsilon \\
                A_{20}(\gamma) & \varepsilon & A_{22}(\gamma)
             \end{pmatrix},
\end{equation}
\begin{equation}
   A{'}(\gamma) = \begin{pmatrix}
                A_{00}(\gamma) & A'_{01}(\gamma) & A'_{02}(\gamma) \\
                A'_{10}(\gamma) & A_{11}(\gamma) & \varepsilon \\
                A'_{20}(\gamma) & \varepsilon & A_{22}(\gamma)
             \end{pmatrix},
\end{equation}
\begin{equation}
   A{''}(\gamma) = \begin{pmatrix}
                A_{00}(\gamma) & A''_{01}(\gamma) &  A''_{02}(\gamma) \\
                A''_{10}(\gamma) & A_{11}(\gamma) & \varepsilon \\
                A''_{20}(\gamma) & \varepsilon & A_{22}(\gamma)
             \end{pmatrix}.
\end{equation}

The diagonal blocks of the matrices represent the dynamics out of the junction.
They have the following form, here for $n_0 = 6$:

\begin{equation}
A_{00}(\gamma) =  \begin{pmatrix}
\varepsilon & \gamma^{\bar{b}_{(0,1)}} \underline{s}_{(0,1)} &\varepsilon &\varepsilon &\varepsilon &\varepsilon &\varepsilon\\
\gamma^{b_{(0,1)}} \underline{t}_{(0,1)} & \varepsilon & \gamma^{\bar{b}_{(0,2)}} \underline{s}_{(0,2)} &\varepsilon &\varepsilon &\varepsilon &\varepsilon  \\
\varepsilon & \gamma^{b_{(0,2)}} \underline{t}_{(0,2)} & \varepsilon & \gamma^{\bar{b}_{(0,3)}} \underline{s}_{(0,3)} &\varepsilon &\varepsilon  &\varepsilon   \\
\varepsilon &\varepsilon   &\gamma^{b_{(0,3)}} \underline{t}_{(0,3)} & \varepsilon & \gamma^{\bar{b}_{(0,4)}} \underline{s}_{(0,4)} &\varepsilon &\varepsilon    \\
\varepsilon &\varepsilon &\varepsilon   &\gamma^{b_{(0,4)}} \underline{t}_{(0,4)} & \varepsilon & \gamma^{\bar{b}_{(0,5)}} \underline{s}_{(0,5)} &\varepsilon    \\
\varepsilon   &\varepsilon &\varepsilon &\varepsilon   &\gamma^{b_{(0,5)}} \underline{t}_{(0,5)} & \varepsilon  &\gamma^{\bar{b}_{(0,6)}} \underline{s}_{(0,6)} \\ 
\varepsilon &\varepsilon   &\varepsilon &\varepsilon &\varepsilon   &\gamma^{b_{(0,6)}} \underline{t}_{(0,6)} & \varepsilon   
                 \end{pmatrix}.
\end{equation}

The other blocks represent the dynamics on the junction.
The blocks above the diagonal blocks represent the dynamics on the convergence with regard to $\underline{t}$ and on the divergence with regard to $\underline{s}$. They have the following form, here for $n_0 = n_1 = 6$:

\begin{equation}
A_{01}(\gamma) =  \begin{pmatrix}
\varepsilon & \varepsilon &\varepsilon &\varepsilon &\varepsilon &\varepsilon & \gamma^{2{b}_{(1,n_1)}-1} \underline{t}_{(1,n_1)} \\
\varepsilon & \varepsilon & \varepsilon &\varepsilon &\varepsilon &\varepsilon &\varepsilon  \\
\varepsilon & \varepsilon & \varepsilon &\varepsilon &\varepsilon &\varepsilon &\varepsilon  \\
\varepsilon & \varepsilon & \varepsilon &\varepsilon &\varepsilon &\varepsilon &\varepsilon  \\
\varepsilon & \varepsilon & \varepsilon &\varepsilon &\varepsilon &\varepsilon &\varepsilon  \\
\varepsilon & \varepsilon & \varepsilon &\varepsilon &\varepsilon &\varepsilon &\varepsilon  \\
\gamma^{2\bar{b}_{(1,1)}-1} \underline{s}_{(1,1)}&\varepsilon   &\varepsilon &\varepsilon &\varepsilon  & \varepsilon & \varepsilon   
                 \end{pmatrix}.
\end{equation}

The blocks below the diagonal blocks represent the dynamics on the divergence with regard to $\underline{t}$ and on the convergence with regard to $\underline{s}$. They have the following form, here for $n_0 = n_1 = 6$:

\begin{equation}
A_{10}(\gamma) =  \begin{pmatrix}
\varepsilon & \varepsilon &\varepsilon &\varepsilon &\varepsilon &\varepsilon & \gamma^{2{b}_{(1,1)}+1} \underline{t}_{(1,1)} \\
\varepsilon & \varepsilon & \varepsilon &\varepsilon &\varepsilon &\varepsilon &\varepsilon  \\
\varepsilon & \varepsilon & \varepsilon &\varepsilon &\varepsilon &\varepsilon &\varepsilon  \\
\varepsilon & \varepsilon & \varepsilon &\varepsilon &\varepsilon &\varepsilon &\varepsilon  \\
\varepsilon & \varepsilon & \varepsilon &\varepsilon &\varepsilon &\varepsilon &\varepsilon  \\
\varepsilon & \varepsilon & \varepsilon &\varepsilon &\varepsilon &\varepsilon &\varepsilon  \\
\gamma^{2\bar{b}_{(1,n_1)}+1} \underline{s}_{(1,n_1)}&\varepsilon   &\varepsilon &\varepsilon &\varepsilon  & \varepsilon & \varepsilon   
                 \end{pmatrix}.
\end{equation}

The matrices $A_1(\gamma), A_2(\gamma)$ represent at the junction the dynamics in case where $l$ is odd, respectively even.
The matrices $A'(\gamma), A''(\gamma)$ represent at the junction the dynamics in the case where $l$ is odd on the convergence node but even on the divergence node, respectively where $l$ is even on the convergence node but odd on the divergence node.

The blocks above the diagonal blocks have the following form, here for $n_0 = n_1 = 6$:

\begin{equation}
A'_{01}(\gamma) =  \begin{pmatrix}
\varepsilon & \varepsilon &\varepsilon &\varepsilon &\varepsilon &\varepsilon & \gamma^{2{b}_{(1,n_1)}-1} \underline{t}_{(1,n_1)} \\
\varepsilon & \varepsilon & \varepsilon &\varepsilon &\varepsilon &\varepsilon &\varepsilon  \\
\varepsilon & \varepsilon & \varepsilon &\varepsilon &\varepsilon &\varepsilon &\varepsilon  \\
\varepsilon & \varepsilon & \varepsilon &\varepsilon &\varepsilon &\varepsilon &\varepsilon  \\
\varepsilon & \varepsilon & \varepsilon &\varepsilon &\varepsilon &\varepsilon &\varepsilon  \\
\varepsilon & \varepsilon & \varepsilon &\varepsilon &\varepsilon &\varepsilon &\varepsilon  \\
\varepsilon &\varepsilon   &\varepsilon &\varepsilon &\varepsilon  & \varepsilon & \varepsilon   
                 \end{pmatrix}.
\end{equation}

\begin{equation}
A'_{02}(\gamma) =  \begin{pmatrix}
\varepsilon & \varepsilon &\varepsilon &\varepsilon &\varepsilon &\varepsilon &  \varepsilon\\
\varepsilon & \varepsilon & \varepsilon &\varepsilon &\varepsilon &\varepsilon &\varepsilon  \\
\varepsilon & \varepsilon & \varepsilon &\varepsilon &\varepsilon &\varepsilon &\varepsilon  \\
\varepsilon & \varepsilon & \varepsilon &\varepsilon &\varepsilon &\varepsilon &\varepsilon  \\
\varepsilon & \varepsilon & \varepsilon &\varepsilon &\varepsilon &\varepsilon &\varepsilon  \\
\varepsilon & \varepsilon & \varepsilon &\varepsilon &\varepsilon &\varepsilon &\varepsilon  \\
\gamma^{2\bar{b}_{(2,1)}} \underline{s}_{(2,1)}&\varepsilon   &\varepsilon &\varepsilon &\varepsilon  & \varepsilon & \varepsilon   
                 \end{pmatrix}.
\end{equation}


Consider $n_0 = n_1 = n_2 = 6$. The associated graphs to the matrices $A_1(\gamma)$, $A_2(\gamma)$, $A'(\gamma)$ and $A''(\gamma)$ are depicted below.

\begin{figure}[]
\centering
\tikzset{
    side by side/.style 2 args={
        line width=1.0pt,
        #1,
        postaction={
            clip,postaction={draw,#2}
        }
    }
}
\tikzset{
    side by side 2/.style 2 args={
        line width=1.5pt,
        #1,
        postaction={
            clip,postaction={draw,#2}
        }
    }
}
\definecolor{light-gray}{gray}{0.75}
\begin{tikzpicture}[->,auto, node distance=2.7 cm, state_0/.style={circle,draw},
state_1/.style={circle,draw},
state_2/.style={circle,draw}]
\node [state_0] (1) {0,0};
\node [state_0](2) [left of = 1] {0,1};
\node [state_0](3) [left of = 2] {0,2};

\node [state_0](4) [below left of = 3] {0,3};

\node [state_0](5) [below right of = 4] {0,4};
\node [state_0](6) [right of = 5] {0,5};
\node [state_0](7) [right of = 6] {0,6};

\node [state_1](8) [below right of=7] {1,1};
\node [state_1](9) [below right of=8] {1,2};
\node [state_1](10) [right of=9] {1,3};
\node [state_1](11) [above of=10] {1,4};
\node [state_1](12) [above left of=11] {1,5};

\node [state_2](14) [above of=12] {2,1};
\node [state_2](15) [above right of=14] {2,2};
\node [state_2](16) [above of=15] {2,3};
\node [state_2](17) [left of=16] {2,4};
\node [state_2](18) [below left of=17] {2,5};

\path[]
(1) edge [bend right, line width=1 pt] node [] {} (2)
(2) edge [bend right, line width=1 pt] node [] {} (3)
(3) edge [bend right, line width=1 pt] node [] {} (4)
(4) edge [bend right, line width=1 pt] node [] {} (5)
(5) edge [bend right, line width=1 pt] node [] {} (6)
(6) edge [bend right, line width=1 pt] node [] {} (7)
(7) edge [bend right, line width=1 pt] node [] {} (8)
(8) edge [bend right, line width=1 pt] node [] {} (9)
(9) edge [bend right, line width=1 pt] node [] {} (10)
(10) edge [bend right, line width=1 pt] node [] {} (11)
(11) edge [bend right, line width=1 pt] node [] {} (12)
(12) edge [bend right, line width=1 pt] node [] {} (1)

(14) edge [bend right, line width=1 pt] node [] {} (15)
(15) edge [bend right, line width=1 pt] node [] {} (16)
(16) edge [bend right, line width=1 pt] node [] {} (17)
(17) edge [bend right, line width=1 pt] node [] {} (18)

(2) edge [bend right, dashed, line width=1 pt] node [] {} (1)
(3) edge [bend right, dashed,line width=1 pt] node [] {} (2)
(4) edge [bend right, dashed,line width=1 pt] node [] {} (3)
(5) edge [bend right, dashed,line width=1 pt] node [] {} (4)
(6) edge [bend right, dashed,line width=1 pt] node [] {} (5)
(7) edge [bend right, dashed,line width=1 pt] node [] {} (6)
(8) edge [bend right, dashed,line width=1 pt] node [] {} (7)
(9) edge [bend right, dashed,line width=1 pt] node [] {} (8)
(10) edge [bend right,dashed, line width=1 pt] node [] {} (9)
(11) edge [bend right, dashed,line width=1 pt] node [] {} (10)
(12) edge [bend right, dashed,line width=1 pt] node [] {} (11)
(1) edge [bend right, dashed,line width=1 pt] node [] {} (12)

(15) edge [bend right, dashed,line width=1 pt] node [] {} (14)
(16) edge [bend right, dashed,line width=1 pt] node [] {} (15)
(17) edge [bend right, dashed,line width=1 pt] node [] {} (16)
(18) edge [bend right, dashed,line width=1 pt] node [] {} (17);

\begin{customlegend}[legend cell align=left,
legend entries={ 
Weight $=\underline{t}_{(u,j)}$,
Weight $=\underline{s}_{(u,j)}$,
},
legend style={at={(1,5.5)},font=\footnotesize}]
	\addlegendimage{-stealth,opacity=1}
	\addlegendimage{-stealth,dashed,opacity=1}
\end{customlegend}
\end{tikzpicture}
\caption{The associated graph $\mathcal G(A_{1}(\gamma))$ for $n_0=n_1=n_2=6$.}
\label{fig-G(A1)}
\end{figure}


\begin{figure}[]
\centering
\tikzset{
    side by side/.style 2 args={
        line width=1.0pt,
        #1,
        postaction={
            clip,postaction={draw,#2}
        }
    }
}
\tikzset{
    side by side 2/.style 2 args={
        line width=1.5pt,
        #1,
        postaction={
            clip,postaction={draw,#2}
        }
    }
}
\definecolor{light-gray}{gray}{0.75}
\begin{tikzpicture}[->,auto, node distance=2.7 cm, state_0/.style={circle,draw},
state_1/.style={circle,draw},
state_2/.style={circle,draw}]
\node [state_0] (1) {0,0};
\node [state_0](2) [left of = 1] {0,1};
\node [state_0](3) [left of = 2] {0,2};

\node [state_0](4) [below left of = 3] {0,3};

\node [state_0](5) [below right of = 4] {0,4};
\node [state_0](6) [right of = 5] {0,5};
\node [state_0](7) [right of = 6] {0,6};

\node [state_1](8) [below right of=7] {1,1};
\node [state_1](9) [below right of=8] {1,2};
\node [state_1](10) [right of=9] {1,3};
\node [state_1](11) [above of=10] {1,4};
\node [state_1](12) [above left of=11] {1,5};

\node [state_2](14) [above of=12] {2,1};
\node [state_2](15) [above right of=14] {2,2};
\node [state_2](16) [above of=15] {2,3};
\node [state_2](17) [left of=16] {2,4};
\node [state_2](18) [below left of=17] {2,5};

\path[]
(1) edge [bend right, line width=1 pt] node [] {} (2)
(2) edge [bend right, line width=1 pt] node [] {} (3)
(3) edge [bend right, line width=1 pt] node [] {} (4)
(4) edge [bend right, line width=1 pt] node [] {} (5)
(5) edge [bend right, line width=1 pt] node [] {} (6)
(6) edge [bend right, line width=1 pt] node [] {} (7)
(7) edge [bend right, line width=1 pt] node [] {} (14)
(8) edge [bend right, line width=1 pt] node [] {} (9)
(9) edge [bend right, line width=1 pt] node [] {} (10)
(10) edge [bend right, line width=1 pt] node [] {} (11)
(11) edge [bend right, line width=1 pt] node [] {} (12)
(18) edge [bend right, line width=1 pt] node [] {} (1)

(14) edge [bend right, line width=1 pt] node [] {} (15)
(15) edge [bend right, line width=1 pt] node [] {} (16)
(16) edge [bend right, line width=1 pt] node [] {} (17)
(17) edge [bend right, line width=1 pt] node [] {} (18)

(2) edge [bend right, dashed, line width=1 pt] node [] {} (1)
(3) edge [bend right, dashed,line width=1 pt] node [] {} (2)
(4) edge [bend right, dashed,line width=1 pt] node [] {} (3)
(5) edge [bend right, dashed,line width=1 pt] node [] {} (4)
(6) edge [bend right, dashed,line width=1 pt] node [] {} (5)
(7) edge [bend right, dashed,line width=1 pt] node [] {} (6)
(14) edge [bend right, dashed,line width=1 pt] node [] {} (7)
(9) edge [bend right, dashed,line width=1 pt] node [] {} (8)
(10) edge [bend right,dashed, line width=1 pt] node [] {} (9)
(11) edge [bend right, dashed,line width=1 pt] node [] {} (10)
(12) edge [bend right, dashed,line width=1 pt] node [] {} (11)
(1) edge [bend right, dashed,line width=1 pt] node [] {} (18)

(15) edge [bend right, dashed,line width=1 pt] node [] {} (14)
(16) edge [bend right, dashed,line width=1 pt] node [] {} (15)
(17) edge [bend right, dashed,line width=1 pt] node [] {} (16)
(18) edge [bend right, dashed,line width=1 pt] node [] {} (17);

\begin{customlegend}[legend cell align=left,
legend entries={ 
Weight $=\underline{t}_{(u,j)}$,
Weight $=\underline{s}_{(u,j)}$,
},
legend style={at={(1,5.5)},font=\footnotesize}]
	\addlegendimage{-stealth,opacity=1}
	\addlegendimage{-stealth,dashed,opacity=1}
\end{customlegend}
\end{tikzpicture}
\caption{The associated graph $\mathcal G(A_{2}(\gamma))$ for $n_0=n_1=n_2=6$.}
\label{fig-G(A2)}
\end{figure}


\begin{figure}[]
\centering
\tikzset{
    side by side/.style 2 args={
        line width=1.0pt,
        #1,
        postaction={
            clip,postaction={draw,#2}
        }
    }
}
\tikzset{
    side by side 2/.style 2 args={
        line width=1.5pt,
        #1,
        postaction={
            clip,postaction={draw,#2}
        }
    }
}
\definecolor{light-gray}{gray}{0.75}
\begin{tikzpicture}[->,auto, node distance=2.7 cm, state_0/.style={circle,draw},
state_1/.style={circle,draw},
state_2/.style={circle,draw}]
\node [state_0] (1) {0,0};
\node [state_0](2) [left of = 1] {0,1};
\node [state_0](3) [left of = 2] {0,2};

\node [state_0](4) [below left of = 3] {0,3};

\node [state_0](5) [below right of = 4] {0,4};
\node [state_0](6) [right of = 5] {0,5};
\node [state_0](7) [right of = 6] {0,6};

\node [state_1](8) [below right of=7] {1,1};
\node [state_1](9) [below right of=8] {1,2};
\node [state_1](10) [right of=9] {1,3};
\node [state_1](11) [above of=10] {1,4};
\node [state_1](12) [above left of=11] {1,5};

\node [state_2](14) [above of=12] {2,1};
\node [state_2](15) [above right of=14] {2,2};
\node [state_2](16) [above of=15] {2,3};
\node [state_2](17) [left of=16] {2,4};
\node [state_2](18) [below left of=17] {2,5};

\path[]
(1) edge [bend right, line width=1 pt] node [] {} (2)
(2) edge [bend right, line width=1 pt] node [] {} (3)
(3) edge [bend right, line width=1 pt] node [] {} (4)
(4) edge [bend right, line width=1 pt] node [] {} (5)
(5) edge [bend right, line width=1 pt] node [] {} (6)
(6) edge [bend right, line width=1 pt] node [] {} (7)
(7) edge [bend right, line width=1 pt] node [] {} (14)
(8) edge [bend right, line width=1 pt] node [] {} (9)
(9) edge [bend right, line width=1 pt] node [] {} (10)
(10) edge [bend right, line width=1 pt] node [] {} (11)
(11) edge [bend right, line width=1 pt] node [] {} (12)
(12) edge [bend right, line width=1 pt] node [] {} (1)

(14) edge [bend right, line width=1 pt] node [] {} (15)
(15) edge [bend right, line width=1 pt] node [] {} (16)
(16) edge [bend right, line width=1 pt] node [] {} (17)
(17) edge [bend right, line width=1 pt] node [] {} (18)

(2) edge [bend right, dashed, line width=1 pt] node [] {} (1)
(3) edge [bend right, dashed,line width=1 pt] node [] {} (2)
(4) edge [bend right, dashed,line width=1 pt] node [] {} (3)
(5) edge [bend right, dashed,line width=1 pt] node [] {} (4)
(6) edge [bend right, dashed,line width=1 pt] node [] {} (5)
(7) edge [bend right, dashed,line width=1 pt] node [] {} (6)
(14) edge [bend right, dashed,line width=1 pt] node [] {} (7)
(9) edge [bend right, dashed,line width=1 pt] node [] {} (8)
(10) edge [bend right,dashed, line width=1 pt] node [] {} (9)
(11) edge [bend right, dashed,line width=1 pt] node [] {} (10)
(12) edge [bend right, dashed,line width=1 pt] node [] {} (11)
(1) edge [bend right, dashed,line width=1 pt] node [] {} (12)

(15) edge [bend right, dashed,line width=1 pt] node [] {} (14)
(16) edge [bend right, dashed,line width=1 pt] node [] {} (15)
(17) edge [bend right, dashed,line width=1 pt] node [] {} (16)
(18) edge [bend right, dashed,line width=1 pt] node [] {} (17);

\begin{customlegend}[legend cell align=left,
legend entries={ 
Weight $=\underline{t}_{(u,j)}$,
Weight $=\underline{s}_{(u,j)}$,
},
legend style={at={(1,5.5)},font=\footnotesize}]
	\addlegendimage{-stealth,opacity=1}
	\addlegendimage{-stealth,dashed,opacity=1}
\end{customlegend}
\end{tikzpicture}
\caption{The associated graph $\mathcal G(A{'}(\gamma))$ for $n_0=n_1=n_2=6$.}
\label{fig-G(A1')}
\end{figure}


\begin{figure}[]
\centering
\tikzset{
    side by side/.style 2 args={
        line width=1.0pt,
        #1,
        postaction={
            clip,postaction={draw,#2}
        }
    }
}
\tikzset{
    side by side 2/.style 2 args={
        line width=1.5pt,
        #1,
        postaction={
            clip,postaction={draw,#2}
        }
    }
}
\definecolor{light-gray}{gray}{0.75}
\begin{tikzpicture}[->,auto, node distance=2.7 cm, state_0/.style={circle,draw},
state_1/.style={circle,draw},
state_2/.style={circle,draw}]
\node [state_0] (1) {0,0};
\node [state_0](2) [left of = 1] {0,1};
\node [state_0](3) [left of = 2] {0,2};

\node [state_0](4) [below left of = 3] {0,3};

\node [state_0](5) [below right of = 4] {0,4};
\node [state_0](6) [right of = 5] {0,5};
\node [state_0](7) [right of = 6] {0,6};

\node [state_1](8) [below right of=7] {1,1};
\node [state_1](9) [below right of=8] {1,2};
\node [state_1](10) [right of=9] {1,3};
\node [state_1](11) [above of=10] {1,4};
\node [state_1](12) [above left of=11] {1,5};

\node [state_2](14) [above of=12] {2,1};
\node [state_2](15) [above right of=14] {2,2};
\node [state_2](16) [above of=15] {2,3};
\node [state_2](17) [left of=16] {2,4};
\node [state_2](18) [below left of=17] {2,5};

\path[]
(1) edge [bend right, line width=1 pt] node [] {} (2)
(2) edge [bend right, line width=1 pt] node [] {} (3)
(3) edge [bend right, line width=1 pt] node [] {} (4)
(4) edge [bend right, line width=1 pt] node [] {} (5)
(5) edge [bend right, line width=1 pt] node [] {} (6)
(6) edge [bend right, line width=1 pt] node [] {} (7)
(7) edge [bend right, line width=1 pt] node [] {} (8)
(8) edge [bend right, line width=1 pt] node [] {} (9)
(9) edge [bend right, line width=1 pt] node [] {} (10)
(10) edge [bend right, line width=1 pt] node [] {} (11)
(11) edge [bend right, line width=1 pt] node [] {} (12)
(18) edge [bend right, line width=1 pt] node [] {} (1)

(14) edge [bend right, line width=1 pt] node [] {} (15)
(15) edge [bend right, line width=1 pt] node [] {} (16)
(16) edge [bend right, line width=1 pt] node [] {} (17)
(17) edge [bend right, line width=1 pt] node [] {} (18)

(2) edge [bend right, dashed, line width=1 pt] node [] {} (1)
(3) edge [bend right, dashed,line width=1 pt] node [] {} (2)
(4) edge [bend right, dashed,line width=1 pt] node [] {} (3)
(5) edge [bend right, dashed,line width=1 pt] node [] {} (4)
(6) edge [bend right, dashed,line width=1 pt] node [] {} (5)
(7) edge [bend right, dashed,line width=1 pt] node [] {} (6)
(8) edge [bend right, dashed,line width=1 pt] node [] {} (7)
(9) edge [bend right, dashed,line width=1 pt] node [] {} (8)
(10) edge [bend right,dashed, line width=1 pt] node [] {} (9)
(11) edge [bend right, dashed,line width=1 pt] node [] {} (10)
(12) edge [bend right, dashed,line width=1 pt] node [] {} (11)
(1) edge [bend right, dashed,line width=1 pt] node [] {} (18)

(15) edge [bend right, dashed,line width=1 pt] node [] {} (14)
(16) edge [bend right, dashed,line width=1 pt] node [] {} (15)
(17) edge [bend right, dashed,line width=1 pt] node [] {} (16)
(18) edge [bend right, dashed,line width=1 pt] node [] {} (17);

\begin{customlegend}[legend cell align=left,
legend entries={ 
Weight $=\underline{t}_{(u,j)}$,
Weight $=\underline{s}_{(u,j)}$,
},
legend style={at={(1,5.5)},font=\footnotesize}]
	\addlegendimage{-stealth,opacity=1}
	\addlegendimage{-stealth,dashed,opacity=1}
\end{customlegend}
\end{tikzpicture}
\caption{The associated graph $\mathcal G(A{''}(\gamma))$ for $n_0=n_1=n_2=6$.}
\label{fig-G(A2')}
\end{figure}

%
%
%
%

\newpage
Let the matrix $B(\gamma)$ be defined as follows ($n_0$, $n_1$, $n_2$ are assumed to be
even\footnote{The assumption $n_0$, $n_1$, $n_2$ is even is necessary for the proof of Proposition~\ref{propos1}. It does not affect the applicability of the model.}).
\begin{itemize}
\item If ($n_0 \mod 4 = 0$ and $\sum_j b_{(0,j)}$ is even) or ($n_0 \mod 4 = 2$ and $\sum_j b_{(0,j)}$ is odd), then $B(\gamma) = A_1(\gamma) \otimes A_2(\gamma) \otimes A_2(\gamma) \otimes A_1(\gamma)$.
\item If ($n_0 \mod 4 = 0$ and $\sum_j b_{(0,j)}$ is odd) or ($n_0 \mod 4 = 2$ and $\sum_j b_{(0,j)}$ is even), then $B(\gamma) = A{'}(\gamma) \otimes A{''}(\gamma) \otimes A{''}(\gamma) \otimes A{'}(\gamma)$.
\end{itemize}

%
\begin{proposition}\label{propos1}
The train dynamics~{(\ref{neq-c1})-(\ref{neq-conv6})} are equivalent to the following system of max-plus polynomial matrices:
\begin{equation}
\delta = B(\gamma) \otimes \delta.
\end{equation}
\end{proposition}

\noindent
\textit{Sketch of the proof.} 
Proposition~\ref{propos1} states that the train dynamics, taken on four steps ($B(\gamma)$ is a product of four matrices, see above), are linear
in the max-plus algebra of polynomial matrices.
The train dynamics cannot be linear in one step, because some of the equations of the dynamics depend on whether $l$ is odd or even. This is the case for equations~(\ref{neq-div2}),~(\ref{neq-div3}),~(\ref{neq-div5}) at the divergence and~(\ref{neq-conv1}),~(\ref{neq-conv4}),~(\ref{neq-conv6}) at the convergence.

Moreover, the shifts in $l$ in one step of the dynamics are not the same for all the equations of the dynamics. They can be $0$, $1$ or $2$ (depending on $b_{(u,j)}, \bar{b}_{(u,j)}$ on the central part, respectively $2b_{(u,j)}, 2\bar{b}_{(u,j)}$ on the branches). 
Otherwise the train dynamics
would be linear in two steps.

In the following, a sketch of the proof is given. Consider the train traffic dynamics of a line with a junction, see equations~{(\ref{neq-c1})-(\ref{neq-conv6})}. Of these, only the following constraints distinguish $l$ is odd from $l$ is even (the constraints on the junction):
\begin{itemize}
\item The safe separation constraint of the last node of the central part (divergence), equation~{(\ref{neq-div2}), represented by
two arcs: $(1,1) \to (0,n_0)$ and $(2,1) \to (0,n_0)$:
\begin{equation*}
  \delta^{{l}}_{(0,n_0)} \geq \begin{cases}
                      \delta^{{l}+1-2\bar{b}_{(1,1)}}_{(1,1)} + \underline{s}_{(1,1)} & \forall {l = 2p-1, \text{ with } p \in \mathbb{N}}, \\ ~~ \\
                      \delta^{{l}-2\bar{b}_{(2,1)}}_{(2,1)} + \underline{s}_{(2,1)} & \forall {l \in 2p, \text{ with } p \in \mathbb{N}}, \\
                   \end{cases}
\end{equation*}
\item The travel time constraint of the first node of the central part (convergence), equation~(\ref{neq-conv1})}, represented by two
arcs: $(1,n_1-1) \to (0,0)$ and $(2,n_2-1) \to (0,0)$:
\begin{equation*} 
  \delta^{{{l}}}_{(0,0)} \geq \begin{cases}
                      \delta^{{({l}+1)-2b_{(1,n_1)}}}_{(1,n_1-1)} + \underline{t}_{(1,n_1)} & \forall {l = 2p-1, \text{ with } p \in \mathbb{N}}, \\ ~~ \\
                      \delta^{{{l}-2b_{(2,n_2)}}}_{(2,n_2-1)} + \underline{t}_{(2,n_2)} & \forall {l = 2p, \text{ with } p \in \mathbb{N}}, \\
                   \end{cases}
\end{equation*}
\item The pair of travel time constraints of the first nodes of the branches (divergence), equations~(\ref{neq-div3}) and~(\ref{neq-div5}),
represented by two arcs: $(0,n_0) \to (1,1)$ and $(0,n_0) \to (2,1)$:
\begin{align*}
   \delta^{{l}}_{(1,1)} \geq \delta^{{l}-1-2b_{(1,1)}}_{(0,n_0)} + \underline{t}_{(1,1)}, \; \forall {l = 2p, \text{ with } p \in \mathbb{N}}, \\
   \delta^{{l}}_{(2,1)} \geq \delta^{{l}-2b_{(2,1)}}_{(0,n_0)} + \underline{t}_{(2,1)}, \; \forall {l = 2p, \text{ with } p \in \mathbb{N}}, 
\end{align*}
\item The pair of safe separation constraints of the last nodes of the branches (convergence), equations~(\ref{neq-conv4})
and~(\ref{neq-conv6}), represented by two arcs: $(0,0) \to (1,n_1-1)$ and $(0,0) \to (2,n_2-1)$:
\begin{align*}
   \delta^{{{l}}}_{(1,n_1-1)} \geq \delta^{{{l}-1-2\bar{b}_{(1,n_1)}}}_{(0,0)} + \underline{s}_{(1,n_1)}, \; \forall {l = 2p, \text{ with } p \in \mathbb{N}}, \\
   \delta^{{{l}}}_{(2,n_2-1)} \geq \delta^{{{l-2\bar{b}_{(2,n_2)}}}}_{(0,0)} + \underline{s}_{(2,n_2)}, \; \forall {l = 2p, \text{ with } p \in \mathbb{N}}, 
\end{align*}
\end{itemize}

To correctly represent the dynamics on odd and even $l$, four matrices $A_1(\gamma)$, $A_2(\gamma)$, $A'(\gamma)$, $A''(\gamma)$ have been defined above.
Notice that:
\begin{itemize}
\item The arc corresponding to the first line of equation~(\ref{neq-div2}) ($l$ is odd) exists only in matrix $A_1(\gamma)$ and $A'(\gamma)$ and the arc corresponding to the second line of equation~(\ref{neq-div2}) ($l$ is even) exists only in matrix $A_2(\gamma)$ and $A''(\gamma)$ .
\item The arc corresponding to the first line of equation~(\ref{neq-conv1}) ($l$ is odd) exists only in matrix $A_1(\gamma)$ and $A''(\gamma)$ and the arc corresponding to the second line of equation~(\ref{neq-conv1}) ($l$ is even) exists only in matrix $A_2(\gamma)$ and $A'(\gamma)$.
\item The arc corresponding to equation~(\ref{neq-div3}) ($l$ is odd) exists only in matrix $A_1(\gamma)$ and $A''(\gamma)$ and the arc corresponding to equation~(\ref{neq-div5}) ($l$ is even) exists only in matrix $A_2(\gamma)$ and $A'(\gamma)$ .
\item The arc corresponding to equation~(\ref{neq-conv4}) ($l$ is odd) exists only in matrix $A_1(\gamma)$ and $A'(\gamma)$ and the arc corresponding to equation~(\ref{neq-conv6}) ($l$ is even) exists only in matrix $A_2(\gamma)$ and $A''(\gamma)$ .
\end{itemize}

The nodes for which the application of the correct matrix has to be verified, are the nodes of the divergence and convergence.
More precisely, in order that the train dynamics of the entire line (with the junction) are equivalent to the matrix form of 
Proposition~\ref{propos1}, that is $\delta = B(\gamma) \otimes \delta$, it has to be verified that by iterating the matrix form the following rules are satisfied:
\begin{itemize}
\item every time $A_1(\gamma)$ is applied, $l$ is odd on the convergence node $(0,0)$ and on the divergence node $(0,n_0)$,
\item every time $A_2(\gamma)$ is applied, $l$ is even on the convergence node $(0,0)$ and on the divergence node $(0,n_0)$,
\item every time $A'(\gamma)$ is applied, $l$ is odd on the convergence node $(0,0)$ but even on the divergence node $(0,n_0)$, and
\item every time $A''(\gamma)$ is applied, $l$ is even on the convergence node $(0,0)$ but odd on the divergence node $(0,n_0)$.
\end{itemize}

For that, starting from the node $(0,0)$ or from the node $(0,n_0)$, 
and iterating the matrix $B(\gamma)$ over a number of steps, it has to be checked that every time the nodes $(0,0)$ 
or $(0,n_0)$ are reached, the parity of $l$ (odd or even) matches with the application of the correct matrix ($A_1(\gamma), A_2(\gamma), A'(\gamma), A''(\gamma)$) as defined in the rules above.

Let us start from any node and iterate the dynamics up to another node by applying the matrices step-by-step. A one step matrix multiplication (with $A_1(\gamma)$, $A_2(\gamma)$, $A'(\gamma)$, $A''(\gamma)$) iterates the dynamics to the next node.
Then the change in parity of $l$ between origin and destination nodes depends on the sum over the shifts in $l$ along the path to the destination node.
All the paths between the nodes where the dynamics depend on the parity of $l$ have to be checked:
%
%
\begin{itemize}
\item all the paths linking the node $(0,0)$ with itself,
\item all the paths linking the node $(0,n_0)$ with itself, and
\item all the paths linking the nodes $(0,0)$ and $(0,n_0)$.
\end{itemize}
For theses paths, the shift in $l$ (depending on $b_{(u,j)}, \bar{b}_{(u,j)}$ on the central part, respectively $2b_{(u,j)}, 2\bar{b}_{(u,j)}$ on the branches) between origin and destination nodes has to guarantee the application of the correct matrix at the destination node accordingly to the rules defined above.




\begin{itemize}
\item First of all, let us start with the paths along the branches.
There are two main types of paths on the branches to distinguish:
\begin{enumerate}
\item The paths from $(0,0) \rightarrow (0,0)$ and from $(0,n_0) \rightarrow (0,n_0)$ by passing on the branches: In this case, when iterating the matrix form ($A_1(\gamma) \otimes A_2(\gamma) \otimes A_2(\gamma) \otimes A_1(\gamma)$) or ($A'(\gamma) \otimes A''(\gamma) \otimes A''(\gamma) \otimes A'(\gamma)$), the matrix applied at $(0,0)$ or $(0,n_0)$ will be the same. Note that the shift in $l$ is always even because of the changing of variables ($2b_{(u,j)} + 2\bar{b}_{(u,j)} = 2$), which is concordant with the matrix form of the train traffic dynamics.
\item The paths from $(0,0) \rightarrow (0,n_0)$ and from $(0,n_0) \rightarrow (0,0)$ along the branches: The shift in $l$ between the two junction nodes, by passing on the branches, is always a multiple of $2$ because of the changing of the variables and the fact that divergence and convergence are operated in the same way (odd departures from/ to branch~1, even ones from/ to branch~2)~\footnote{see equations~(\ref{neq-c3}),~(\ref{neq-c4}),~(\ref{neq-div2}),~(\ref{neq-div3}),~(\ref{neq-div5}),~(\ref{neq-conv1}),~(\ref{neq-conv4}),~(\ref{neq-conv6})}. 
The condition $n_1,n_2$ is even guarantees the application of the correct matrix at the convergence or divergence node at the end of the path, accordingly to the rules above. 
For example, consider Graph~\ref{fig-G(A1)}: the shift in $l$ over the path against the travel direction along branch 1 is: $2\bar{b}_{(1,n_1)} -1 + 2\bar{b}_{(1,5)} + 2\bar{b}_{(1,4)} + 2\bar{b}_{(1,3)} + 2\bar{b}_{(1,2)} + 2\bar{b}_{(1,1)} + 1 = \text{even}$~\footnote{see equations~(\ref{neq-conv4}),~(\ref{neq-div3}),~(\ref{neq-c4})}.
On the same graph, for $l = 1$, an iteration along branch~1 with $n_1=6$ is completed by applying $A_1(\gamma) \otimes A_2(\gamma) \otimes A_2(\gamma) \otimes A_1(\gamma) \otimes A_1(\gamma) \otimes A_2(\gamma) \otimes A_2(\gamma) \otimes A_1(\gamma)$ such that at the convergence/ divergence node at the end of the path, with $A_1(\gamma)$, the correct matrix is applied, since $l$ is again odd because of the changing of variables. Other paths accordingly. 
\end{enumerate}
\item For the paths passing by the central part of the line, there are 4 types of paths to distinguish:
\begin{enumerate}
\item The path from $(0,0) \rightarrow (0,0)$ and from $(0,n_0) \rightarrow (0,n_0)$ with a number of arcs $n$ such that $n \mod 4 = 2$: In this case, when iterating the matrix form ($A_1(\gamma) \otimes A_2(\gamma) \otimes A_2(\gamma) \otimes A_1(\gamma)$) or ($A'(\gamma) \otimes A''(\gamma) \otimes A''(\gamma) \otimes A'(\gamma)$) of the dynamics, the matrix applied at $(0,0)$ or $(0,n_0)$ will switch. Note that the shift in $l$ is odd since, $b_{(u,j)} + \bar{b}_{(u,j)} = 1$, which is concordant with the matrix form of the train traffic dynamics.
\item The path from $(0,0) \rightarrow (0,0)$ and from $(0,n_0) \rightarrow (0,n_0)$ with a number of arcs $n$ such that $n \mod 4 = 0$: In this case, when iterating the matrix form ($A_1(\gamma) \otimes A_2(\gamma) \otimes A_2(\gamma) \otimes A_1(\gamma)$) or ($A'(\gamma) \otimes A''(\gamma) \otimes A''(\gamma) \otimes A'(\gamma)$) of the dynamics, the matrix applied at $(0,0)$ or $(0,n_0)$ will be same.  The shift in $l$ is even since, $2b_{(u,j)} + 2\bar{b}_{(u,j)} = 2$, which is concordant with the train traffic dynamics.
\item The path from $(0,0) \rightarrow (0,n_0)$ going in the direction of the traffic (forward):
\begin{itemize}
\item If $n_0 \mod 4 = 2$ and if $\sum_j b_{(0,j)}$ is odd, iterating the matrix form ($A_1(\gamma) \otimes A_2(\gamma) \otimes A_2(\gamma) \otimes A_1(\gamma)$) of the dynamics imposes to switch the matrix $A_1(\gamma)$ or $A_2(\gamma)$, which is concordant with the odd shift in $l$.
\item If $n_0 \mod 4 = 2$ and if $\sum_j b_{(0,j)}$ is even, iterating the matrix form ($A'(\gamma) \otimes A''(\gamma) \otimes A''(\gamma) \otimes A'(\gamma)$) of the dynamics imposes to switch the matrix, which is concordant with the even shift in $l$.
\item If $n_0 \mod 4 = 0$ and if $\sum_j b_{(0,j)}$ is odd, iterating the matrix form ($A'(\gamma) \otimes A''(\gamma) \otimes A''(\gamma) \otimes A'(\gamma)$) of the dynamics imposes to stay on the same matrix, which is concordant with the odd shift in $l$.
\item If $n_0 \mod 4 = 0$ and if $\sum_j b_{(0,j)}$ is even, iterating the matrix form ($A_1(\gamma) \otimes A_2(\gamma) \otimes A_2(\gamma) \otimes A_1(\gamma)$) of the dynamics imposes to stay on the same matrix, which is concordant with the even shift in $l$.
\end{itemize}
\item The path from $(0,n_0) \rightarrow (0,0)$ going against the traffic direction (backward).
\begin{itemize}
\item If $n_0 \mod 4 = 2$ and if $\sum_j \bar{b}_{(0,j)}$ is odd, iterating the matrix form ($A_1(\gamma) \otimes A_2(\gamma) \otimes A_2(\gamma) \otimes A_1(\gamma)$) of the dynamics imposes to switch the matrix, which is concordant with the fact that the shift in $l$ is odd.
\item If $n_0 \mod 4 = 2$ and if $\sum_j \bar{b}_{(0,j)}$ is even, iterating the matrix form ($A'(\gamma) \otimes A''(\gamma) \otimes A''(\gamma) \otimes A'(\gamma)$) of the dynamics imposes to switch the matrix, which is concordant with the fact that the shift in $l$ is even.
\item If $n_0 \mod 4 = 0$ and if $\sum_j \bar{b}_{(0,j)}$ is odd, iterating the matrix form ($A'(\gamma) \otimes A''(\gamma) \otimes A''(\gamma) \otimes A'(\gamma)$) of the dynamics imposes to stay on the same matrix, which is concordant with the fact that the shift in $l$ is odd.
\item If $n_0 \mod 4 = 0$ and if $\sum_j \bar{b}_{(0,j)}$ is even, iterating the matrix form ($A_1(\gamma) \otimes A_2(\gamma) \otimes A_2(\gamma) \otimes A_1(\gamma)$) of the dynamics imposes to stay on the same matrix, which is concordant with the fact that the shift in $l$ is even.
\end{itemize}
\end{enumerate}
\end{itemize}

Consequently when iterating the train traffic dynamics by applying $\delta = B(\gamma) \otimes \delta$ with $B(\gamma)$ as defined above, the matrix form of the train dynamics is equivalent to the train traffic dynamics model, see equations~{(\ref{neq-c1})-(\ref{neq-conv6})}.

\section{The Steady State Train Dynamics}
\subsection{The Asymptotic Average Train Time-headway}
In this part, the properties of the stationary regime of the train dynamics in a metro line with one junction will be studied.
The main result presented below is based on~\cite{BCOQ92} and~\cite{FNHL16}, where the following eigenvalue problem for matrices~$A(\gamma)$ with entries in the max-plus algebra has been studied.

\begin{equation}\label{eq-eig}
A(\gamma) \otimes \upsilon = \mu \otimes \upsilon.
\end{equation}

\begin{definition}
$\mu \in \mathbb{R}_{max} \backslash \{\varepsilon\}$ is said to be a generalized eigenvalue of $A(\gamma)$, with associated generalized eigenvector $\upsilon \in \mathbb{R}^n_{max} \backslash \{\varepsilon\}$, if $A(\mu^{-1}) \otimes \upsilon = \upsilon$, where $A(\mu^{-1})$ is the matrix obtained by evaluating the polynomial matrix $A(\gamma)$ at $\mu^{-1}$ [Theorem 2.3 in~\cite{FNHL16}].
\end{definition}

%

Accordingly to~\cite{BCOQ92} and~\cite{FNHL16}:
\begin{theorem}~\label{th-mpa}
Let $B(\gamma) = \oplus^{p}_{l=0} A_l\gamma^l$ be an irreducible poynomial matrix with acyclic sub-graph $G(B_0)$. Then there exists one and only one generalized eigenvalue $\mu > \varepsilon$ and finite eigenvectors $\upsilon > \varepsilon$. This eigenvalue $\mu$ is equal to the maximum cycle mean of the graph $G(B)$:
$$
\mu = \max_{c \in C} = \frac{W(c)}{D(c)},
$$
where W is the weight,
D is the duration and
$c$ ranges over the set of circuits $C$ of $G(B)$.
Moreover, the dynamic system $\delta = B(\gamma) \otimes \delta$ admits an asymptotic average growth vector $\chi$ whose components are all equal to $\mu$.
\end{theorem}

\subsubsection{Derivation of the Generalized Eigenvalue of $B(\gamma)$ in Max-plus algebra}
The generalized eigenvalue of $B(\gamma)$ corresponds to the asymptotic average growth rate of the polynomial matrix and is interpreted here as the asymptotic average train time-headway.
%
%
Note that if the asymptotic average growth rate $h$ of system $\delta = B(\gamma) \otimes \delta$ exists, it represents the asymptotic average time-headway on the central part, and since the number of \textit{k} steps on the branches has been doubled because of the changing of variables,
the asymptotic average train time-headway on the branches is $2h$.
The asymptotic average growth rate is given by the unique generalized eigenvalue of the homogeneous max-plus system, which can be calculated
from its associated graph, see Theorem~\ref{th-cdc17} below.


The asymptotic average
train frequency of a metro line with a junction, depends on the total number of trains and on the difference between the number
of trains on the branches. Both parameters are invariable in time (in two steps of the train dynamics),
since the one-over-two rule is applied on the divergence and on the convergence.

Consider the following notations.

\begin{tabular}{lp{0.8\textwidth}}
  $m_u$ & the number of trains on part $u$ of the line at time 0.\\  
  $m$ & $= m_0+m_1+m_2$ the total number of trains on the line.
  \end{tabular}
  
  \begin{tabular}{lp{0.8\textwidth}}
  $\Delta m$ & $= m_2 - m_1$ the difference in the number of trains between branches 2 and 1.\\
  $n_u$	& the number of segments on part $u$ of the line. 
  \end{tabular}
  
  \begin{tabular}{lp{0.8\textwidth}}
  $\bar{m}_u$ & $= n_u - m_u, \forall u\in\{0,1,2\}$ the number of free segments on part $u$ of the line.\\ 
  $\bar{m}$ & $= \bar{m}_0 + \bar{m}_1 + \bar{m}_2$ the total number of free segments.
  \end{tabular}
  
  \begin{tabular}{lp{0.8\textwidth}}
  $\Delta \bar{m}$ & $= \bar{m}_2 - \bar{m}_1$ the difference in the number of free segments between branches 2 and 1.
    \end{tabular}
  
  \begin{tabular}{lp{0.8\textwidth}}
  $\underline{T}_u$ & $= \sum_j \underline{t}_{(u,j)}, \forall u\in\{0,1,2\}$ the sum over the travel times on each part of the line. 
    \end{tabular}
  
  \begin{tabular}{lp{0.8\textwidth}}
  $\underline{S}_u$ & $= \sum_j \underline{s}_{(u,j)}, \forall u\in\{0,1,2\}$ the sum over the s times on each part of the line.
\end{tabular}

The two ceses of $m=0$ and $m=n:=n_0+n_1+n_2$ are excluded here. In fact, for these two cases, the train dynamics are fully implicit. 
However, it is known that for these two cases, the train frequency is zero (no train in the case $m=0$, and no movement in the case $m=n$).

\begin{theorem}\label{th-cdc17}
  The dynamic system~$\delta = B(\gamma) \otimes \delta$ admits  a stationary regime, with a common average growth
  rate $h_0$ for all the variables, which represents the average train
  time-headway $h_0$ on the central part and $h_1/2 = h_2/2$ on the branches.
  Moreover 
  $$ h_0 = h_1/2 = h_2/2 = \max \{ h_{fw}, h_{\min} ,h_{bw}, h_{br}\}, $$
  with\footnote{fw: forward, bw: backward, min: minimum, br: branches.}
  $$ h_{fw} = \max\left\{ \frac{\underline{T}_0 + \underline{T}_1}{m - \Delta m},
                          \frac{\underline{T}_0 + \underline{T}_2}{m + \Delta m} \right\}, $$
	$$ h_{\min} = \max \begin{cases}
		  \max_{u,j} (t_{(u,j)} + s_{(u,j)}) & \forall u \in \{0\},\\
		  \max_{u,j} (t_{(u,j)} + s_{(u,j)})/2 & \forall u \in \{1,2\}, \forall j \in J(u) \backslash \{n_u\},
	\end{cases}$$
  $$ h_{bw} = \max\left\{ \frac{\underline{S}_0 + \underline{S}_1}{\bar{m} - \Delta \bar{m}},
                          \frac{\underline{S}_0 + \underline{S}_2}{\bar{m} + \Delta \bar{m}} \right\}, $$
  $$ h_{br} =  \max\left\{\frac{\underline{T}_1 + \underline{S}_2}{2(n_2 - \Delta m)},
                  \frac{\underline{S}_1 + \underline{T}_2}{2(n_1 + \Delta m)}\right\}.$$
\end{theorem}

\begin{proof}
In the following, Theorem~\ref{th-mpa} is applied to the dynamics $\delta = B(\gamma) \otimes \delta$.
\begin{itemize}
\item First, since the dynamics $\delta = B(\gamma) \otimes \delta$ are the application of the original train dynamics in four steps, the 
graph $\mathcal G(B(\gamma))$ associated
to $B(\gamma)$ may (theoretically) have up to four strongly connected components (depending on the number of segments of the line).
Theorem~\ref{th-mpa} is then applied to every sub-system associated to a strongly connected component.
The definition of $B(\gamma)$ implies (shown below) that all the strongly connected components have the same maximum cycle mean, which then 
implies that the train dynamics have the same asymptotic average growth rate (interpreted as the asymptotic average train 
time-headway $h_0$ on the central part of the metro line).
For the case $n_0 = n_1 = n_2 = 6$, there are two strongly connected components of $\mathcal G(B(\gamma))$. One strongly connected component is shown in Figures~\ref{fig-G(A2*A1)} and~\ref{fig-G(B)-2}.
In order to have readable figures, some of the arcs of the strongly connected component studied here (the other has the same maximum cycle mean) are given in Figure~\ref{fig-G(A2*A1)}, and 
some further arcs are given in Figure~\ref{fig-G(B)-2}. 
The isolated nodes in these two figures are those of the second strongly connected component of the graph $\mathcal G(B(\gamma))$.
For example, in Figure~\ref{fig-G(A2*A1)}, the arc $(0,0) \to (0,4)$ results from the sequence of arcs $(0,0) \to (0,1) \to (0,2) \to (0,3) \to (0,4)$
of the original train dynamics.
In Figure~\ref{fig-G(B)-2}, the arc $(0,0) \to (0,2)$ results from the sequence of arcs $(0,0) \to (0,1) \to (0,0) \to (0,1) \to (0,2)$
of the original train dynamics, or from the sequence of arcs $(0,0) \to (0,1) \to (0,2) \to (0,1) \to (0,2)$ of the original train dynamics.
\item Second, since $m\neq 0$ and $m\neq n$, it can be checked that the train dynamics are not fully implicit. They are triangular, that is, 
there exists an order of applying the dynamics on every node, in such a way that the dynamics will be explicit. 
By consequent, the graph $\mathcal G (B_0)$ is acyclic. Therefore, the sub-graphs of $\mathcal G(B_0)$ associated to each strongly 
connected component of $\mathcal G(B(\gamma))$ are also acyclic. 
\end{itemize}
Consequently, Theorem~\ref{th-mpa} can be applied to the strongly connected components of $\mathcal G(B(\gamma))$.
The cycles of the strongly connected components of Figure~\ref{fig-G(A2*A1)} 
and their cycle means are given and calculated below.
\begin{itemize}
  \item One cycle in the travel direction, passing by the central part and by branch~1 of the line.
  This cycle is given in Figure~\ref{fig-G(A2*A1)}: $(0,2) \to (0,6) \to (1,4) \to (0,2)$.
The cycle mean is given by
  $$\frac{\sum_j \underline{t}_{(0,j)} + \sum_j \underline{t}_{(1,j)}}{\sum_j b_{(0,j)} + \sum_j 2b_{(1,j)}}
     = \frac{\underline{T}_0 + \underline{T}_1}{m_0 + 2 m_1} = \frac{\underline{T}_0 + \underline{T}_1}{m - \Delta m}.$$
  \item Another cycle in the travel direction, passing by the central part and by branch~2 of the line.
 This cycle is given in Figure~\ref{fig-G(A2*A1)}: $(0,0) \to (0,4) \to (2,2) \to (0,0)$.
The cycle mean is given by
  $$\frac{\sum_j \underline{t}_{(0,j)} + \sum_j \underline{t}_{(2,j)}}{\sum_j b_{(0,j)} + \sum_j 2b_{(2,j)}}
     = \frac{\underline{T}_0 + \underline{T}_2}{m_0 + 2 m_2} = \frac{\underline{T}_0 + \underline{T}_2}{m + \Delta m}.$$  
  \item[$\Rightarrow$] $h_{fw}$ = maximum of the cycle means of these types of cycles.
  
  \item One cycle against the travel direction, passing by the central part and by branch~1 of the line.
  It is given in Figure~\ref{fig-G(A2*A1)}: $(0,0) \to (1,2) \to (0,4) \to (0,0)$.
The cycle mean is given by
  $$\frac{\sum_j \underline{s}_{(0,j)} + \sum_j \underline{s}_{(1,j)}}{\sum_j \bar{b}_{(0,j)} + \sum_j 2\bar{b}_{(1,j)}}
     = \frac{\underline{S}_0 + \underline{S}_1}{n_0 - m_0 + 2 n_1 - 2 m_1} = \frac{\underline{S}_0 + \underline{S}_1}{\bar{m} - \Delta \bar{m}}.$$
  \item The cycle against the travel direction, passing by the central part and by branch~2 of the line.
  The cycle is given in Figure~\ref{fig-G(A2*A1)}: $(0,2) \to (2,4) \to (0,6) \to (0,2)$.
The cycle mean is given by
  $$\frac{\sum_j \underline{s}_{(0,j)} + \sum_j \underline{s}_{(2,j)}}{\sum_j \bar{b}_{(0,j)} + \sum_j 2\bar{b}_{(2,j)}}
     = \frac{\underline{S}_0 + \underline{S}_2}{n_0 - m_0 + 2 n_2 - 2 m_2} = \frac{\underline{S}_0 + \underline{S}_2}{\bar{m} + \Delta \bar{m}}.$$  
  \item[$\Rightarrow$] $h_{bw}$ = maximum of the cycle means of these types of cycles.
\item The loops over each node, whose realizations in four steps use one segment.
For example, the loop $(0,0) \to (0,0)$ in Figure~\ref{fig-G(A2*A1)} is the realization of the cycle $(0,0) \to (0,1) \to (0,0) \to (0,1) \to (0,0)$
in four steps. 
The cycle mean of such loops, in case $u=0$ (central part of the line), is given by
$$\frac{2\underline{t}_{(u,j)} + 2\underline{s}_{(u,j)}}{2b_{(u,j)} + 2\bar{b}_{(u,j)}} = \frac{2\underline{t}_{(u,j)} + 2\underline{s}_{(u,j)}}{2} = \underline{t}_{(u,j)} + \underline{s}_{(u,j)}.$$
The cycle mean of such loops, in case $u=1$ or $2$ (branches 1 or 2), is given by
$$\frac{2\underline{t}_{(u,j)} + 2\underline{s}_{(u,j)}}{4b_{(u,j)} + 4\bar{b}_{(u,j)}} = \frac{\underline{t}_{(u,j)} + \underline{s}_{(u,j)}}{2}.$$
 \item[$\Rightarrow$] $h_{\min}$ = maximum of the cycle means of these types of cycles.

\item The loops over each node, whose realizations in four steps use two segments.
For example, the loop $(0,0) \to (0,0)$ in Figure~\ref{fig-G(A2*A1)} is the realization of the cycle $(0,0) \to (0,1) \to (0,2) \to (0,1) \to (0,0)$
in four steps. 
The cycle mean of such loops, in case $u=0$ (central part of the line), is given by
$$\begin{array}{l}
\frac{\underline{t}_{(u,j)} + \underline{t}_{(u,j+1)} + \underline{s}_{(u,j)} + \underline{s}_{(u,j+1)} }{b_{(u,j)} + b_{(u,j+1)} + \bar{b}_{(u,j)} + \bar{b}_{(u,j+1)}} \\ \\
  = \frac{(\underline{t}_{(u,j)} + \underline{s}_{(u,j)}) + (\underline{t}_{(u,j+1)} + \underline{s}_{(u,j+1)})}{2} \\ \\
  \leq \max_{u=0,j} \left( \underline{t}_{(u,j)} + \underline{s}_{(u,j)} \right) \leq h_{\min}.
\end{array}$$
The cycle mean of such loops, in case $u=1$ or $2$ (branches 1 or 2 of the line), is given by
$$\begin{array}{l}
\frac{\underline{t}_{(u,j)} + \underline{t}_{(u,j+1)} + \underline{s}_{(u,j)} + \underline{s}_{(u,j+1)} }{2b_{(u,j)} + 2b_{(u,j+1)} + 2\bar{b}_{(u,j)} + 2\bar{b}_{(u,j+1)}} \\ \\
  = \frac{(\underline{t}_{(u,j)} + \underline{s}_{(u,j)}) + (\underline{t}_{(u,j+1)} + \underline{s}_{(u,j+1)})}{4} \\ \\
  \leq \max_{u\in\{1,2\},j} ( \underline{t}_{(u,j)} + \underline{s}_{(u,j)} )/2 \leq h_{\min}.
\end{array}$$
 \item[$\Rightarrow$] The cycle means of all such loops are upper bounded by $h_{\min}$. 
\item All the cycles with two arcs, for example: $(0,0) \to (0,4) \to (0,0)$ in Figure~\ref{fig-G(A2*A1)}.
The cycle mean of such cycles, in case $u=0$ (central part of the line), is given by
     $$\begin{array}{l}
       \frac{\sum_{j=1}^4 \underline{t}_{(u,j)} + \sum_{j=1}^4 \underline{s}_{(u,j)}}{\sum_{j=1}^4b_{(u,j)} + \sum_{j=1}^4\bar{b}_{(u,j)}} \\ \\
        = \frac{\sum_{j=1}^4 \underline{t}_{(u,j)} + \sum_{j=1}^4 \underline{s}_{(u,j)}}{4}  \\ \\
        \leq \max_{u=0,j} \left( \underline{t}_{(u,j)} + \underline{s}_{(u,j)} \right) \leq h_{\min}.
     \end{array}$$  
   The cycle mean of such cycles, in case $u=1$ or $2$ (branches 1 or 2 of the line), is given by
     $$\begin{array}{l}
       \frac{\sum_{j=1}^4 \underline{t}_{(u,j)} + \sum_{j=1}^4 \underline{s}_{(u,j)}}{\sum_{j=1}^4 2b_{(u,j)} + \sum_{j=1}^4 2\bar{b}_{(u,j)}} \\ \\
        = \frac{\sum_{j=1}^4 \underline{t}_{(u,j)} + \sum_{j=1}^4 \underline{s}_{(u,j)}}{8}  \\ \\
        \leq \max_{u\in\{1,2\},j} \left( \underline{t}_{(u,j)} + \underline{s}_{(u,j)} \right)/2 \leq h_{\min}.
     \end{array}$$      
   \item[$\Rightarrow$] The cycle means of these cycles are upper bounded by $h_{\min}$.
\item The cycle in form of a ``8'' passing by the two branches, without passing by the central part. One cycle passes on branch~1 in the travel direction
     and on branch~2 against the travel direction: $(0,6) \to (1,4) \to (2,4) \to (0,6)$.
Its cycle mean is given by
 $$\frac{\sum_j \underline{t}_{(1,j)} + \sum_j \underline{s}_{(2,j)}}{\sum_j 2b_{(1,j)} + \sum_j 2\bar{b}_{(2,j)}}
     = \frac{\underline{T}_1 + \underline{S}_2}{2m_1 + 2 n_2 - 2 m_2}= \frac{\underline{T}_1 + \underline{S}_2}{2(n_ 2 - \Delta m)}.$$
   \item The other cycle passes on branch~2 in the travel direction
     and on branch~1 against the travel direction: $(1,2) \to (2,2) \to (0,0) \to (1,2)$.
Its cycle mean is given by
 $$\frac{\sum_j \underline{s}_{(1,j)} + \sum_j \underline{t}_{(2,j)}}{\sum_j 2\bar{b}_{(1,j)} + \sum_j 2{b}_{(2,j)}}
     = \frac{\underline{S}_1 + \underline{T}_2}{2 n_1 + 2n_2 - 2m_2}= \frac{\underline{S}_1 + \underline{T}_2}{2(n_1 + \Delta m)}.$$
   \item[$\Rightarrow$] $h_{br}$ = maximum of the cycle means of these two cycles.  
\end{itemize}


\noindent
The remaining cycles of the strongly connected component studied here are given in Figure~\ref{fig-G(B)-2},
and their cycle means are calculated below. 
It is shown below that they are all upper bounded by the cycle means of the same component depicted in Figure~\ref{fig-G(A2*A1)}.
Indeed, they are averages of the cycle means of Figure~\ref{fig-G(A2*A1)}.
\begin{itemize}
  \item The cycles in the travel direction, passing by the central part and by branch~1 of the line.
  For example, the cycle given in Figure~\ref{fig-G(B)-2}: $(0,0) \to (0,2) \to (0,4) \to (0,6) \to (1,2) \to (1,4) \to (0,0)$.
Its cycle mean is given by
  $$\begin{array}{l}
    \frac{\sum_j \underline{t}_{(0,j)} + \sum_j \underline{t}_{(1,j)} + \sum_{u,j} (\underline{t}_{(u,2j)} + \underline{s}_{(u,2j)})}{\sum_j {b}_{(0,j)}+\sum_j 2{b}_{(1,j)} + \sum_{j} b_{(0,2j)} + \sum_{j} \bar{b}_{(0,2j)} + \sum_{j} 2{b}_{(1,2j)} + \sum_{j} 2\bar{b}_{(1,2j)} } \\ \\
    = \frac{\underline{T}_0 + \underline{T}_1 + \sum_{u,j} (\underline{t}_{(u,2j)} + \underline{s}_{(u,2j)})}{m_0 + 2 m_1 + n_0/2 + n_1} \\ \\
    = \frac{\underline{T}_0 + \underline{T}_1+ \sum_{u,j} (\underline{t}_{(u,2j)}+ \underline{s}_{(u,2j)})}{m - \Delta m + n_0/2 + n_1} \\ \\
    \leq \max \left\{ \frac{\underline{T}_0 + \underline{T}_1 }{m - \Delta m } , \frac{\sum_{u,j} (\underline{t}_{(u,2j)}+ \underline{s}_{(u,2j)})}{n_0/2 + n_1}  \right\} \\ \\
    \leq \max \left\{ \frac{\underline{T}_0 + \underline{T}_1 }{m - \Delta m } , \max_j (\underline{t}_{(0,2j)}+ \underline{s}_{(0,2j)}),
         \frac{(\underline{t}_{(1,2j)}+ \underline{s}_{(1,2j)})}{2} \right\}. \\ \\
    \leq \max \left\{ h_{fw}, h_{\min} \right\}.
  \end{array}$$  
\item The cycles in the travel direction, passing by the central part and by branch~2 of the line.
  For example, the cycle given in Figure~\ref{fig-G(B)-2}: $(0,0) \to (0,2) \to (0,4) \to (0,6) \to (2,2) \to (2,4) \to (0,0)$.
Its cycle mean is given by
  $$\begin{array}{l}
    \frac{\sum_j \underline{t}_{(0,j)} + \sum_j \underline{t}_{(2,j)} + \sum_{u,j} (\underline{t}_{(u,2j)} + \underline{s}_{(u,2j)})}{\sum_j {b}_{(0,j)}+\sum_j 2{b}_{(2,j)} + \sum_{j} b_{(0,2j)} + \sum_{j} \bar{b}_{(0,2j)} + \sum_{j} 2{b}_{(2,2j)} + \sum_{j} 2\bar{b}_{(2,2j)} } \\ \\
    = \frac{\underline{T}_0 + \underline{T}_2 + \sum_{u,j} (\underline{t}_{(u,2j)} + \underline{s}_{(u,2j)})}{m_0 + 2 m_2 + n_0/2 + n_2} \\ \\
    = \frac{\underline{T}_0 + \underline{T}_2+ \sum_{u,j} (\underline{t}_{(u,2j)}+ \underline{s}_{(u,2j)})}{m + \Delta m + n_0/2 + n_2} \\ \\
    \leq \max \left\{ \frac{\underline{T}_0 + \underline{T}_2 }{m + \Delta m } , \frac{\sum_{u,j} (\underline{t}_{(u,2j)}+ \underline{s}_{(u,2j)})}{n_0/2 + n_2}  \right\} \\ \\
    \leq \max \left\{ \frac{\underline{T}_0 + \underline{T}_2 }{m + \Delta m } , \max_j (\underline{t}_{(0,2j)}+ \underline{s}_{(0,2j)}),
         \frac{(\underline{t}_{(2,2j)}+ \underline{s}_{(2,2j)})}{2} \right\}. \\ \\
    \leq \max \left\{ h_{fw}, h_{\min} \right\}.
  \end{array}$$  
  \item[$\Rightarrow$] Note that the cycle means are upper bounded by the cycle means of the cycles in the same direction passing by the same parts of the same component and those of the loops over one node, depicted in Figure~\ref{fig-G(A2*A1)}.
  
  \item The cycles against the travel direction, passing by the central part and by branch~1 of the line.
  For example, the cycle given in Figure~\ref{fig-G(B)-2}: $(0,0) \to (1,4) \to (1,2) \to (0,6) \to (0,4) \to (0,2) \to (0,0)$.
Its cycle mean is given by
  $$\begin{array}{l}
    \frac{\sum_j \underline{s}_{(0,j)} + \sum_j \underline{s}_{(1,j)} + \sum_{u,j} (\underline{t}_{(u,2j)} + \underline{s}_{(u,2j)})}{\sum_j \bar{b}_{(0,j)}+\sum_j 2\bar{b}_{(1,j)} + \sum_{j} b_{(0,2j)} + \sum_{j} \bar{b}_{(0,2j)} + \sum_{j} 2{b}_{(1,2j)} + \sum_{j} 2\bar{b}_{(1,2j)} } \\ \\
    = \frac{\underline{S}_0 + \underline{S}_1 + \sum_{u,j} (\underline{t}_{(u,2j)} + \underline{s}_{(u,2j)})}{n_0-m_0 + 2 (n_1-m_1) + n_0/2 + n_1} \\ \\
    = \frac{\underline{S}_0 + \underline{S}_1+ \sum_{u,j} (\underline{t}_{(u,2j)}+ \underline{s}_{(u,2j)})}{\bar{m} - \Delta \bar{m} + n_0/2 + n_1} \\ \\
    \leq \max \left\{ \frac{\underline{S}_0 + \underline{S}_1 }{\bar{m} - \Delta \bar{m} } , \frac{\sum_{u,j} (\underline{t}_{(u,2j)}+ \underline{s}_{(u,2j)})}{n_0/2 + n_1}  \right\} \\ \\
    \leq \max \left\{ \frac{\underline{S}_0 + \underline{S}_1 }{\bar{m} - \Delta \bar{m} } , \max_j (\underline{t}_{(0,2j)}+ \underline{s}_{(0,2j)}),
         \frac{(\underline{t}_{(1,2j)}+ \underline{s}_{(1,2j)})}{2} \right\}. \\ \\
    \leq \max \left\{ h_{bw}, h_{\min} \right\}.
  \end{array}$$

\item The cycles against the travel direction, passing by the central part and by branch~2 of the line.
  For example, the cycle given in Figure~\ref{fig-G(B)-2}: $(0,0) \to (2,4) \to (2,2) \to (0,6) \to (0,4) \to (0,2) \to (0,0)$.
  $$\begin{array}{l}
    \frac{\sum_j \underline{s}_{(0,j)} + \sum_j \underline{s}_{(2,j)} + \sum_{u,j} (\underline{t}_{(u,2j)} + \underline{s}_{(u,2j)})}{\sum_j \bar{b}_{(0,j)}+\sum_j 2\bar{b}_{(2,j)} + \sum_{j} b_{(0,2j)} + \sum_{j} \bar{b}_{(0,2j)} + \sum_{j} 2{b}_{(2,2j)} + \sum_{j} 2\bar{b}_{(2,2j)} } \\ \\
    = \frac{\underline{S}_0 + \underline{S}_2 + \sum_{u,j} (\underline{t}_{(u,2j)} + \underline{s}_{(u,2j)})}{n_0-m_0 + 2 (n_2-m_2) + n_0/2 + n_2} \\ \\
    = \frac{\underline{S}_0 + \underline{S}_2+ \sum_{u,j} (\underline{t}_{(u,2j)}+ \underline{s}_{(u,2j)})}{\bar{m} + \Delta \bar{m} + n_0/2 + n_2} \\ \\
    \leq \max \left\{ \frac{\underline{S}_0 + \underline{S}_2 }{\bar{m} + \Delta \bar{m} } , \frac{\sum_{u,j} (\underline{t}_{(u,2j)}+ \underline{s}_{(u,2j)})}{n_0/2 + n_2}  \right\} \\ \\
    \leq \max \left\{ \frac{\underline{S}_0 + \underline{S}_2 }{\bar{m} + \Delta \bar{m} } , \max_j (\underline{t}_{(0,2j)}+ \underline{s}_{(0,2j)}),
         \frac{(\underline{t}_{(2,2j)}+ \underline{s}_{(2,2j)})}{2} \right\}. \\ \\
    \leq \max \left\{ h_{bw}, h_{\min} \right\}.
  \end{array}$$  
  \item[$\Rightarrow$] Note that the cycle means are upper bounded by the cycle means of the cycles in the same direction passing by the same parts of the same component  and those of the loops over one node, depicted in Figure~\ref{fig-G(A2*A1)}.

   \item All the cycles using two segments, and all the cycles going from a node to another node in the graph, and then changing the direction and
     going back to the original node.  Considering those on the central part, for example, the cycle given in Figure~\ref{fig-G(B)-2}: $(0,0) \to (0,2) \to (0,0) $.
Their cycle mean is given by
  $$\begin{array}{l}
  \frac{\sum_{j=1}^3 \underline{t}_{(u,j)} + \sum_{j=1}^3 \underline{s}_{(u,j)}}{\sum_{j=1}^3b_{(u,j)} + \sum_{j=1}^3\bar{b}_{(u,j)}}  \\ \\
    = \frac{\sum_{j=1}^3 \underline{t}_{(u,j)} + \sum_{j=1}^3 \underline{s}_{(u,j)}}{3} \\ \\
    \leq \max \left\{ \underline{t}_{(u,j)} + \underline{s}_{(u,j)}  \right\} \\ \\
    \leq h_{\min}.
  \end{array}$$ 
   \item All the cycles using two segments, and all the cycles going from a node to another node in the graph, and then changing the direction and
     going back to the original node.  Considering those on the branches, for example, the cycle given in Figure~\ref{fig-G(B)-2}: $(2,2) \to (2,4) \to (2,2) $.
Their cycle mean is given by
  $$\begin{array}{l}
  \frac{\sum_{j=1}^3 \underline{t}_{(u,j)} + \sum_{j=1}^3 \underline{s}_{(u,j)}}{\sum_{j=1}^3 2b_{(u,j)} + \sum_{j=1}^3 2\bar{b}_{(u,j)}}  \\ \\
    = \frac{\sum_{j=1}^3 \underline{t}_{(u,j)} + \sum_{j=1}^3 \underline{s}_{(u,j)}}{6} \\ \\
    \leq \max \left\{ \frac{\underline{t}_{(u,j)} + \underline{s}_{(u,j)}}{2}  \right\} \\ \\
    \leq h_{\min}. 
  \end{array}$$ 
   \item[$\Rightarrow$] The cycle means of these cycles are upper bounded by the cycle means of the loops, see Figure~\ref{fig-G(A2*A1)}.

   \item The cycle in form of a ``8'' passing by the two branches, without passing by the central part. One cycle passes on branch~1 in the travel direction
     and on branch~2 against the travel direction: $(0,6) \to (1,2) \to (1,4) \to (0,0) \to (2,4) \to (2,2) \to (0,6)$.
Its cycle mean is given by
  $$\begin{array}{l}
    \frac{\sum_j \underline{t}_{(1,j)} + \sum_j \underline{s}_{(2,j)} + \sum_{u,j} (\underline{t}_{(u,2j)} + \underline{s}_{(u,2j)})}{\sum_j 2{b}_{(1,j)}+\sum_j 2\bar{b}_{(2,j)} + \sum_{u,j} 2{b}_{(u,2j)} + \sum_{u,j} 2\bar{b}_{(u,2j)} } \\ \\
    = \frac{\underline{T}_1 + \underline{S}_2 + \sum_{u,j} (\underline{t}_{(u,2j)} + \underline{s}_{(u,2j)})}{2(m_1 + n_2-m_2 + n_1 + n_2)} \\ \\
    = \frac{\underline{T}_1 + \underline{S}_2+ \sum_{u,j} (\underline{t}_{(u,2j)}+ \underline{s}_{(u,2j)})}{2(m_1 + \bar{m}_2 + n_1 + n_2)} \\ \\
    \leq \max \left\{ \frac{\underline{T}_1 + \underline{S}_2 }{2(m_1 - \bar{m}_2) } , \frac{\sum_{u,j} (\underline{t}_{(u,2j)}+ \underline{s}_{(u,2j)})}{2}  \right\} \\ \\
    \leq \max \left\{ h_{br}, h_{\min} \right\}.
  \end{array}$$ 

   \item The other cycle passes on branch~2 in the travel direction
     and on branch~1 against the travel direction: $(0,6) \to (2,2) \to (2,4) \to (0,0) \to (1,4) \to (1,2) \to (0,6)$.
Its cycle mean is given by
   $$\begin{array}{l}
    \frac{\sum_j \underline{t}_{(2,j)} + \sum_j \underline{s}_{(1,j)} + \sum_{u,j} (\underline{t}_{(u,2j)} + \underline{s}_{(u,2j)})}{\sum_j 2{b}_{(2,j)}+\sum_j 2\bar{b}_{(1,j)} + \sum_{u,j} 2{b}_{(u,2j)} + \sum_{u,j} 2\bar{b}_{(u,2j)} } \\ \\
    = \frac{\underline{T}_2 + \underline{S}_1 + \sum_{u,j} (\underline{t}_{(u,2j)} + \underline{s}_{(u,2j)})}{2(m_2 + n_1-m_1 + n_1 + n_2)} \\ \\
    = \frac{\underline{T}_2 + \underline{S}_1+ \sum_{u,j} (\underline{t}_{(u,2j)}+ \underline{s}_{(u,2j)})}{2(m_2 + \bar{m}_1 + n_1 + n_2)} \\ \\
    \leq \max \left\{ \frac{\underline{T}_2 + \underline{S}_1 }{2(m_2 - \bar{m}_1) } , \frac{\sum_{u,j} (\underline{t}_{(u,2j)}+ \underline{s}_{(u,2j)})}{2}  \right\} \\ \\
    \leq \max \left\{ h_{br}, h_{\min} \right\}.
  \end{array}$$ 
  \item[$\Rightarrow$] Note that the cycle means are upper bounded by the cycle means of the cycles in the same direction passing by the same parts of the same component  and those of the loops over one node, depicted in Figure~\ref{fig-G(A2*A1)}.
\end{itemize}

Consequently, the maximum cycle mean of the strongly connected component (Figures~\ref{fig-G(A2*A1)},~\ref{fig-G(B)-2}) of the system $\delta = B(\gamma) \otimes \delta$ is given by the maximum cycle mean of four types of cycles: $h_{fw}, h_{min}, h_{bw} \text{ and } h_{br}$.

\begin{figure}[]
\centering
\tikzset{
    side by side/.style 2 args={
        line width=1.0pt,
        #1,
        postaction={
            clip,postaction={draw,#2}
        }
    }
}
\tikzset{
    side by side 2/.style 2 args={
        line width=1.5pt,
        #1,
        postaction={
            clip,postaction={draw,#2}
        }
    }
}
\definecolor{light-gray}{gray}{0.75}
\begin{tikzpicture}[->,auto, node distance=2.7 cm, state_0/.style={circle,draw},
state_1/.style={circle,draw},
state_2/.style={circle,draw}]
\node [state_0] (1) {0,0};
\node [state_0](2) [left of = 1] {0,1};
\node [state_0](3) [left of = 2] {0,2};
\node [state_0](4) [below left of = 3] {0,3};
\node [state_0](5) [below right of = 4] {0,4};
\node [state_0](6) [right of = 5] {0,5};
\node [state_0](7) [right of = 6] {0,6};

\node [state_1](8) [below right of=7] {1,1};
\node [state_1](9) [below right of=8] {1,2};
\node [state_1](10) [right of=9] {1,3};
\node [state_1](11) [above of=10] {1,4};
\node [state_1](12) [above left of=11] {1,5};

\node [state_2](13) [above of=12] {2,1};
\node [state_2](14) [above right of=13] {2,2};
\node [state_2](15) [above of=14] {2,3};
\node [state_2](16) [left of=15] {2,4};
\node [state_2](17) [below left of=16] {2,5};

\path[]
(1) edge [bend right, line width=1 pt] node [] {} (5)
(5) edge [ line width=1 pt] node [] {} (14)
(14) edge [bend right, line width=1 pt] node [] {} (1)

(16) edge [ line width=1 pt,dashed] node [] {} (7)
(7) edge [line width=1 pt,dashed] node [] {} (3)
(3) edge [line width=1 pt,dashed] node [] {} (16)
%
%
(3) edge [bend right, line width=1 pt] node [] {} (7)
(7) edge [bend right, line width=1 pt] node [] {} (11)
(11) edge [line width=1 pt] node [] {} (3)
%
(1) edge [line width=1 pt,dashed] node [] {} (9)
(9) edge [line width=1 pt,dashed] node [] {} (5)
(5) edge [line width=1 pt,dashed] node [] {} (1)

(1) edge [loop above,line width=1pt, dotted] node [] {} (1)
(2) edge [loop above,line width=1pt, dotted] node [] {} (2)
(3) edge [loop above,line width=1pt, dotted] node [] {} (3)

(4) edge [loop below,line width=1pt, dotted] node [] {} (4)
(5) edge [loop below,line width=1pt, dotted] node [] {} (5)
(6) edge [loop below,line width=1pt, dotted] node [] {} (6)

(7) edge [loop below,line width=1pt, dotted] node [] {} (7)
(8) edge [loop below,line width=1pt, dotted] node [] {} (8)
(9) edge [loop below,line width=1pt, dotted] node [] {} (9)

(10) edge [loop above,line width=1pt, dotted] node [] {} (10)
(11) edge [loop above,line width=1pt, dotted] node [] {} (11)

(13) edge [loop below,line width=1pt, dotted] node [] {} (13)
(14) edge [loop below,line width=1pt, dotted] node [] {} (14)
(15) edge [loop below,line width=1pt, dotted] node [] {} (15)

(16) edge [loop above,line width=1pt, dotted] node [] {} (16)

(9) edge [line width=1 pt,dotted] node [] {} (14)
(11) edge [line width=1 pt,dotted] node [] {} (16);

\begin{customlegend}[legend cell align=left,
legend entries={ 
Weight $=\underline{t}$, 
Weight $=\underline{s}$, 
Weight $=\underline{t}+\underline{s}$, 
},
legend style={at={(3.3,6.4)},font=\footnotesize}]
	\addlegendimage{-stealth,opacity=1}
	\addlegendimage{-stealth,dashed,opacity=1}
	\addlegendimage{-stealth,dotted, opacity=1}
\end{customlegend}
\end{tikzpicture}
\caption{$\mathcal G(B(\gamma))\text{, with } B(\gamma) = A_2(\gamma) \otimes A_2(\gamma) \otimes  A_1(\gamma) \otimes A_1(\gamma)$  for $n_0=n_1=n_2=6$ (selected arcs are shown, part~1).}
\label{fig-G(A2*A1)}
\end{figure}
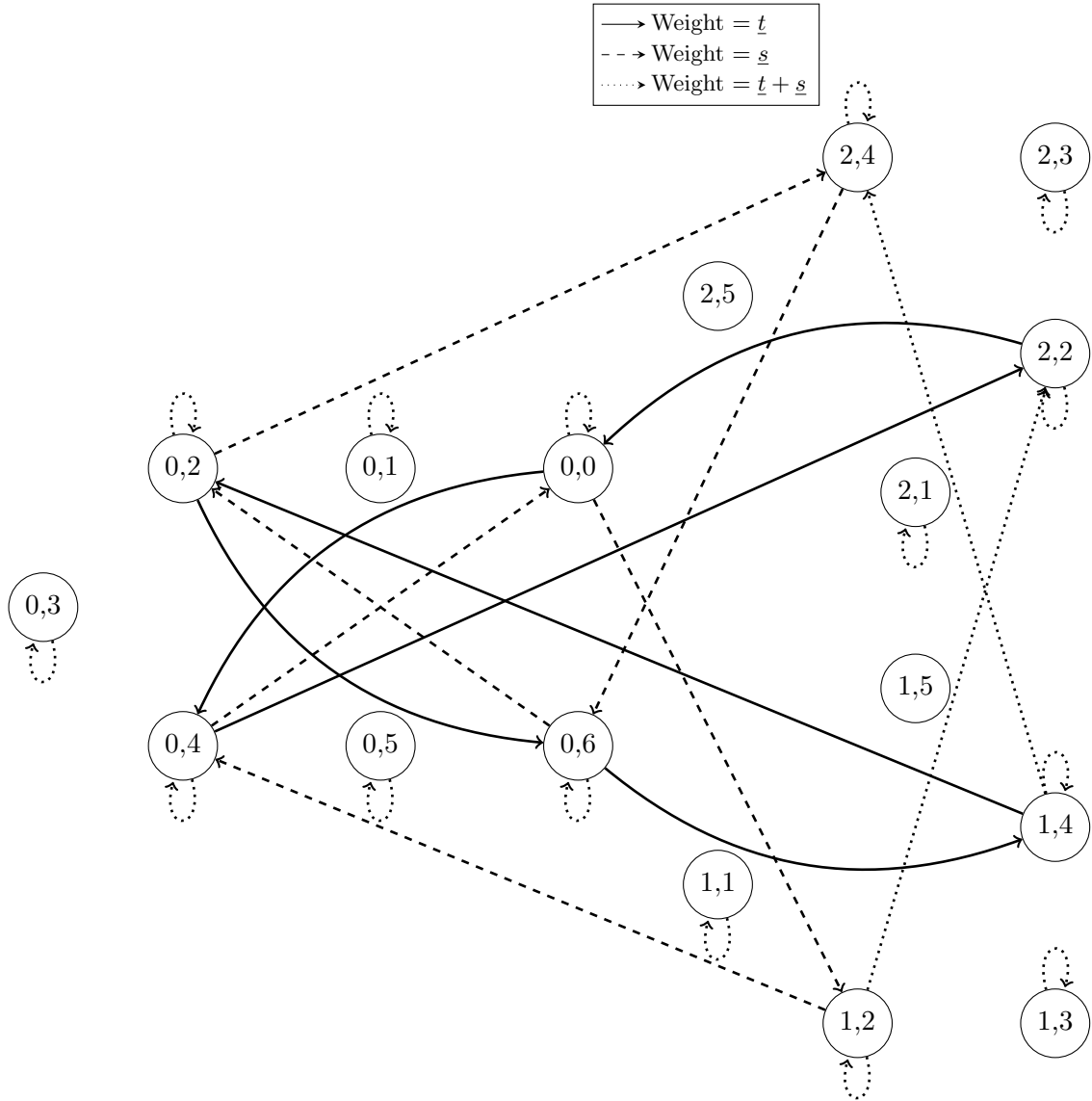

\end{proof}


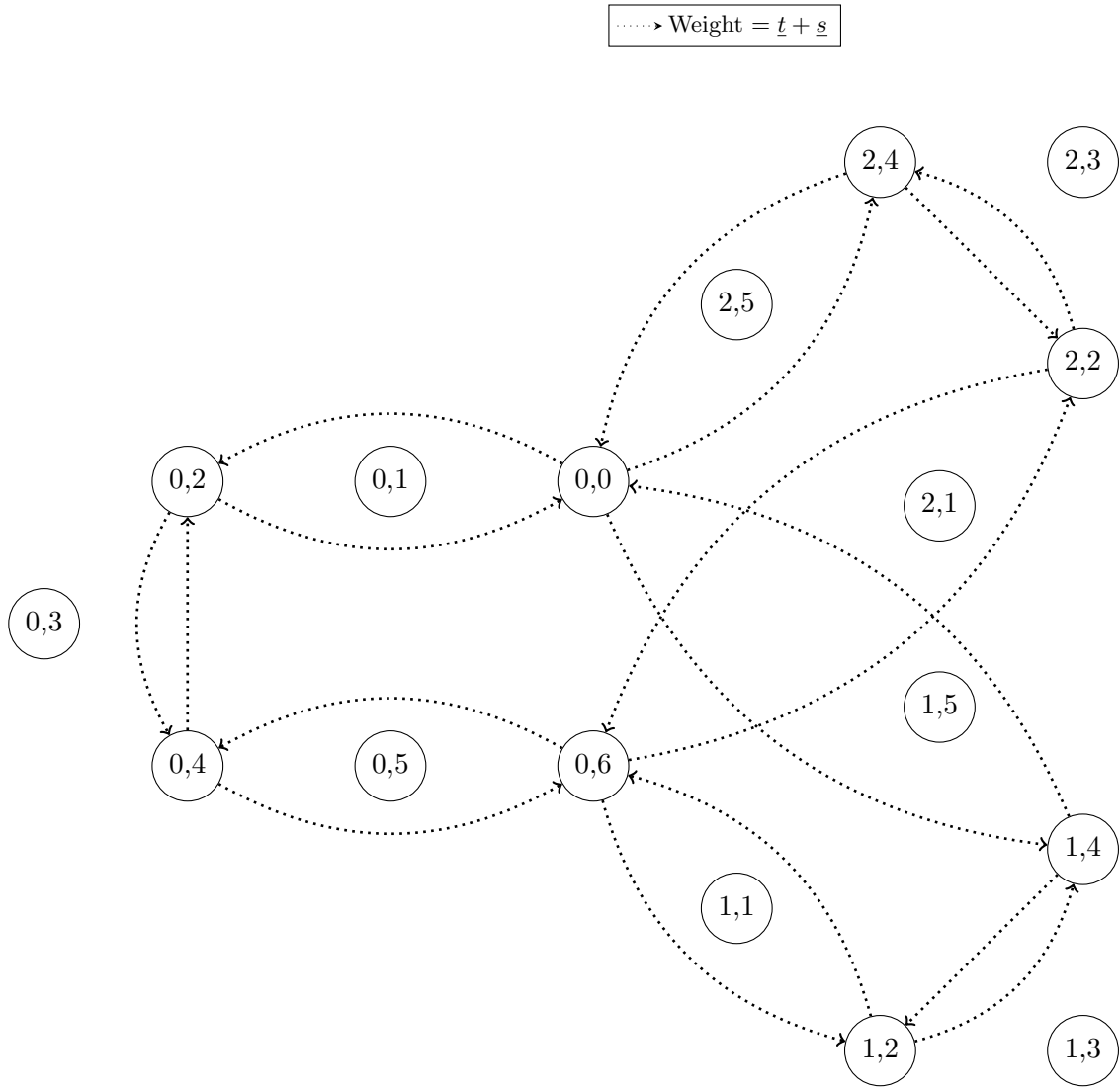
\begin{figure}[]
\centering
\tikzset{
    side by side/.style 2 args={
        line width=1.0pt,
        #1,
        postaction={
            clip,postaction={draw,#2}
        }
    }
}
\tikzset{
    side by side 2/.style 2 args={
        line width=1.5pt,
        #1,
        postaction={
            clip,postaction={draw,#2}
        }
    }
}
\definecolor{light-gray}{gray}{0.75}
\begin{tikzpicture}[->,auto, node distance=2.7 cm, state_0/.style={circle,draw},
state_1/.style={circle,draw},
state_2/.style={circle,draw}]
\node [state_0] (1) {0,0};
\node [state_0](2) [left of = 1] {0,1};
\node [state_0](3) [left of = 2] {0,2};
\node [state_0](4) [below left of = 3] {0,3};
\node [state_0](5) [below right of = 4] {0,4};
\node [state_0](6) [right of = 5] {0,5};
\node [state_0](7) [right of = 6] {0,6};

\node [state_1](8) [below right of=7] {1,1};
\node [state_1](9) [below right of=8] {1,2};
\node [state_1](10) [right of=9] {1,3};
\node [state_1](11) [above of=10] {1,4};
\node [state_1](12) [above left of=11] {1,5};

\node [state_2](13) [above of=12] {2,1};
\node [state_2](14) [above right of=13] {2,2};
\node [state_2](15) [above of=14] {2,3};
\node [state_2](16) [left of=15] {2,4};
\node [state_2](17) [below left of=16] {2,5};

\path[]
(1) edge [bend right, line width=1 pt,dotted] node [] {} (3)
(3) edge [bend right, line width=1 pt,dotted] node [] {} (5)
(5) edge [bend right, line width=1 pt,dotted] node [] {} (7)
(7) edge [bend right, line width=1 pt,dotted] node [] {} (14)
(14) edge [bend right, line width=1 pt,dotted] node [] {} (16)
(16) edge [bend right, line width=1 pt,dotted] node [] {} (1)
(11) edge [bend right, line width=1 pt,dotted] node [] {} (1)
(7) edge [bend right, line width=1 pt,dotted] node [] {} (9)
(9) edge [bend right, line width=1 pt,dotted] node [] {} (11)

(1) edge [bend right, line width=1 pt,dotted] node [] {} (16)
(16) edge [line width=1 pt,dotted] node [] {} (14)
(14) edge [bend right, line width=1 pt,dotted] node [] {} (7)
(1) edge [bend right, line width=1 pt,dotted] node [] {} (11)
(11) edge [line width=1 pt,dotted] node [] {} (9)
(9) edge [bend right, line width=1 pt,dotted] node [] {} (7)
(7) edge [bend right, line width=1 pt,dotted] node [] {} (5)
(5) edge [line width=1 pt,dotted] node [] {} (3)
(3) edge [bend right, line width=1 pt,dotted] node [] {} (1);

\begin{customlegend}[legend cell align=left,
legend entries={ 
Weight $=\underline{t} + \underline{s}$,
},
legend style={at={(3.3,6.4)},font=\footnotesize}]
	\addlegendimage{-stealth,dotted, opacity=1}
\end{customlegend}
\end{tikzpicture}
\caption{$\mathcal G(B(\gamma))\text{, with } B(\gamma) = A_2(\gamma) \otimes A_2(\gamma) \otimes  A_1(\gamma) \otimes A_1(\gamma)$  for $n_0=n_1=n_2=6$ (selected arcs are shown, part~2).}
\label{fig-G(B)-2}
\end{figure}

\newpage
\subsection{The Asymptotic Average Train Frequency}

\begin{corollary}\label{cor-cdc17}
  The asymptotic average train frequency on the central part $f_0$ and on the branches $f_1=f_2$ are given as follows:
  $$f_0 = 2 f_1 = 2 f_2 = \max \left\{ 0, \min \left\{ \frac{1}{h_{fw}}, \frac{1}{h_{\min}} ,\frac{1}{h_{bw}}, \frac{1}{h_{br}} \right\} \right\}.$$
\end{corollary}

\begin{proof}
Directly from Theorem~\ref{th-cdc17}, with $0\leq f=1/h$.
\end{proof}

The asymptotic average train time-headway
(respectively asymptotic average train frequency) is given by Theorem~\ref{th-cdc17}
(respectively Corollary~\ref{cor-cdc17}),
as a function of two parameters: $m$ (the total number of trains running on 
the line) and $\Delta m = m_2-m_1$ (the difference on the number of trains
on branches~2 and~1).
Theorem~\ref{th-cdc17} shows that in a metro line with two branches and with an one-over-two operated junction, the branch with the longest
headway imposes its frequency to the other branch.
Moreover, the frequency on the central part is twice the one on the branches.
\begin{figure}[h]
 \centering
 \frame{
  \includegraphics[width=\textwidth]{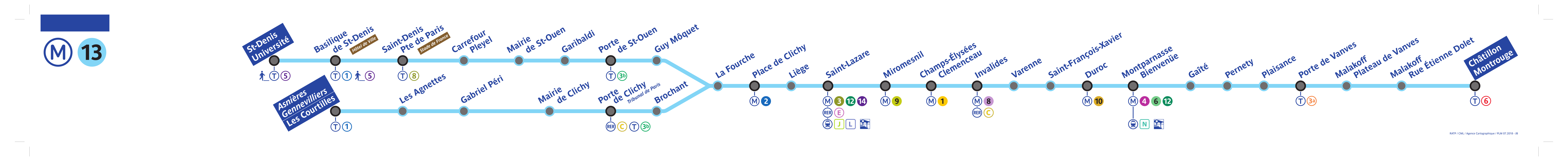} }
\caption {RATP metro line 13, Paris.}
\label{fig-L13}
\end{figure}

\section{The Traffic Phases}\label{transactions}
\subsubsection{Application: RATP Metro Line 13, Paris}
In order to illustrate the theoretic results from above, the model is applied to metro line 13, Paris.
The layout of the line is shown in Fig.~\ref{fig-L13}.
The line consists of one central part and two branches, connected with a junction.
The passenger demand on the two branches is balanced. Therefore, an one-over-two operation of the convergence and of the divergence is optimal from a passengers' point of view.

The operator RATP has provided the minimum (theoretic) run times $\underline{r}_{(u,j)}$, 
dwell times $\underline{w}_{(u,j)}$ and safe separation times $\underline{s}_{(u,j)}$ of line 13.
Note that the dwell times are calculated for each station depending on the passenger travel demand.
Since the line is equipped with a \textit{Grade of Automation 2} (GOA~2) system, where train starting, stopping and door opening (but not door closing) are fully automated,
minimum run times are respected.
It is furthermore supposed that train drivers respect minimum dwell times.
Margins can be included in the minimum travel times to recover small perturbations, for example via a speed and a dwell time control.
Table 3.3 depicts the parameters for the first segments on the central part (from station \textit{La Fourche} towards the terminus \textit{Ch\^atillon -- Montrouge}). 
Travel time $\underline{t}_{(u,j)}$ is the sum of run and dwell times.
With the sum over travel times and safe separation times, on all segments of each part ($\underline{T}_{0,1,2},\underline{S}_{0,1,2}$), all parameters are defined
to depict the phase diagram of the train dynamics in metro line 13, see Fig.~\ref{fig-2D}.

\begin{table}\label{table-par}
\centering
\caption{Extract of model parameters for central part of RATP metro line 13, Paris. There can be several segments $j$ per inter-station.}
\begin{tabular}{|l|cccc|}
\hline
\vspace{2pt}
    		&$\underline{r}_{(u=0,j)}$ & $\underline{w}_{(0,j)}$ & $\underline{t}_{(0,j)}$  & $\underline{s}_{(0,j)}$ \\ \hline
\vspace{2pt}
   $j=1$ \textit{Fourche - Pl. d C.}	 &51.5 s  & 0 s    & 51.5 s    & 14.9 s        \\ \hline
\vspace{2pt}
  $j=2$ \textit{Place de Clichy} 		&17.4 s   & 19 s	   	& 36.4 s    & 42.6 s          \\ \hline
\vspace{2pt}
   $j=3$ \textit{Place d C. - Li\`ege} 	&53.6 s  & 0 s    	& 53.6 s    & 9.3 s       \\ \hline
\vspace{2pt}
   $j=4$ \textit{Place d C. - Li\`ege} 	&75.7 s   & 0 s    	& 75.7 s     & 9.8 s       \\ \hline
\vspace{2pt}
   \vdots &  \vdots &  \vdots    &  \vdots                        & \vdots        \\ \hline
\vspace{2pt}
    	&   	&  	    & $\underline{T}_0$                         & $\underline{S}_0$        \\ \hline
\vspace{2pt}
    $\sum$ &   &    & 54.0 min                         & 26.0 min        \\ \hline
\end{tabular}
\end{table}

\subsubsection{The Phase Diagram of the Train Dynamics}
The steady state dynamics in a metro line with a junction are given by Theorem~\ref{th-cdc17} and Corollary~\ref{cor-cdc17}.
The closed-form solutions prove the existence of four pairs of traffic phases, I-a/I-b, II-a/II-b, III-a/III-b and IV-a/Iv-b, are shown in Fig.~\ref{fig-2D} and are explained in the following.
Every point on the phase diagram 
represents an average frequency $f_0$ on the central part of the line.
It is a function of the number of trains $m$, and
the difference in the number of trains running on branch~2~and~1 $\Delta m = m_2 - m_1$.
In fact, since dynamics depend furthermore on
the temporal parameters, the travel time $\underline{T}_{(u=0,1,2)}$ and minimum
safe separation time $\underline{S}_{(u=0,1,2)}$,
there is a family of phase diagrams of the train dynamics in a metro line.

Note furthermore
$$\underline{T} = (2\underline{T}_0 + \underline{T}_1 + \underline{T}_2)/2,$$
$$\underline{S} = (2\underline{S}_0 + \underline{S}_1 + \underline{S}_2)/2.$$

In the following, the traffic phases derived from Theorem~\ref{th-cdc17} and Corollary~\ref{cor-cdc17} and depicted in Fig.~\ref{fig-2D} are interpreted,
based on the case of metro line 13 of Paris.

\begin{figure}[h]
  \centering
     \frame{\includegraphics[width=\textwidth]{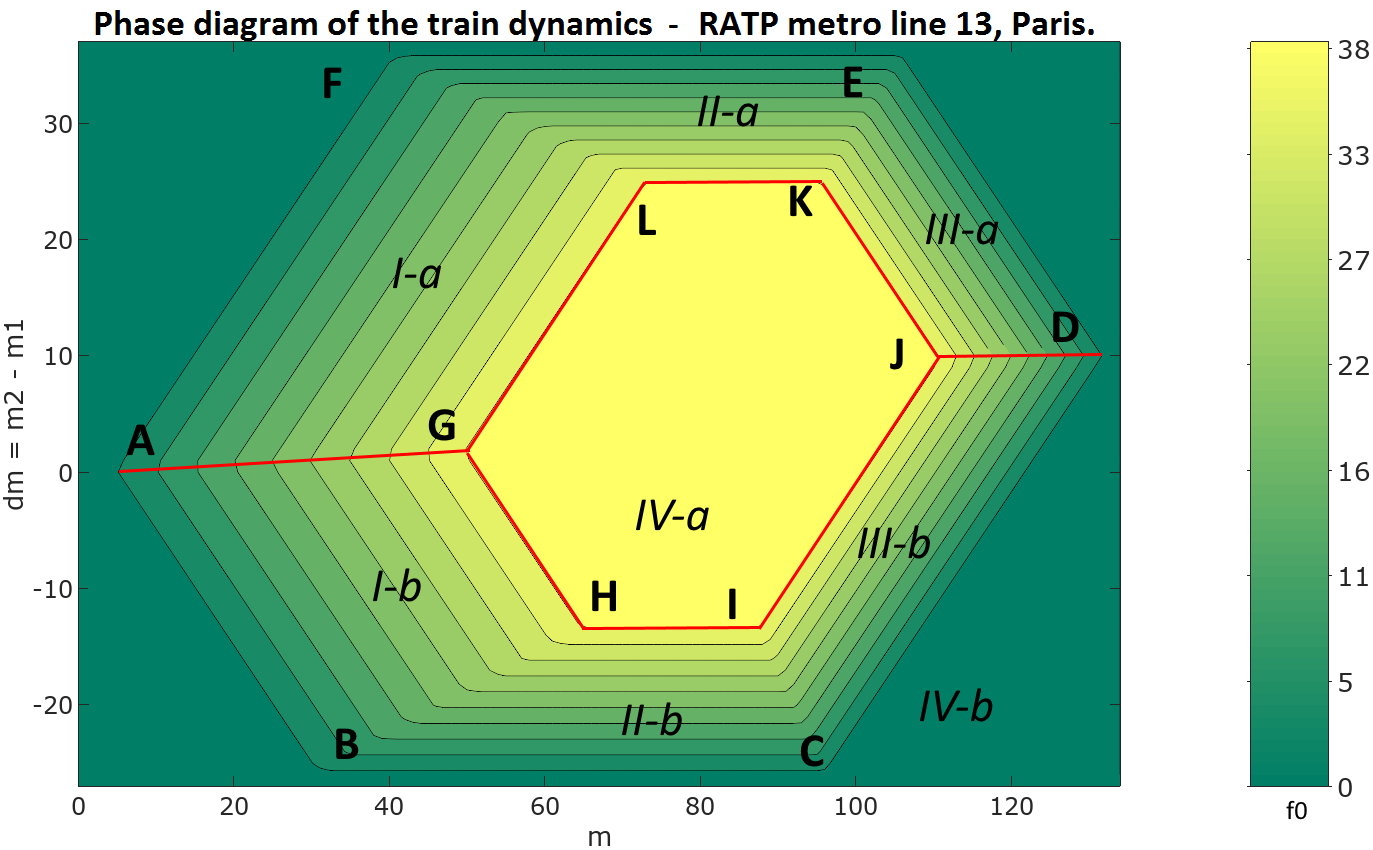}}
  \caption[2D representation of the asymptotic average train frequency $f_0$ (color-scale) on the central part of metro line 13, Paris.]{2D representation of the asymptotic average train frequency $f_0$ (color-scale) on the central part of metro line 13, Paris. The layout of the line is given in Figure~\ref{fig-L13}. The asymptotic average frequency is a function of
  the number of trains \textit{m} and of the difference between the number of trains on the branches~2 and~1 \textit{$\Delta m = m_2 - m_1$}.
  Model parameters (minimum run, dwell and safe separation times) have been provided by the operator RATP.}
    \label{fig-2D}
\end{figure}

\subsubsection{Eight Traffic Phases of the Train Dynamics}
As it can be seen in Fig.~\ref{fig-2D}, the asymptotic average frequency $f_0$ is maximized at the points of the hexagon surface (GHIJKL) (phase IV-a explained below).
Among from the points of this surface, the one requiring minimum number $m$ of trains is the point $G$ whose coordinates $(m^*, \Delta m^*, f_{\max})$ are given as follows.
\begin{itemize}
   \item $f_{\max} = 1 / h_{\min}$.
   \item $m^* := \underline{T} \; f_{\max}$.
   \item $\Delta m^* := \Delta \underline{T} \; f_{\max} / 2 = (\underline{T_2} - \underline{T_1}) f_{max}/2$.   
\end{itemize}

For every $m$, there exists a difference $\Delta m^*$ between the number of trains on the two branches 
that maximizes the asymptotic average train frequency.

\begin{theorem}\label{th-itsc}
$$\forall m, \exists \Delta m^*(m), \forall \Delta m, f_0(m, \Delta m^*(m)) \geq f_0(m, \Delta m).$$
\end{theorem}
\proof It is easy to see from Theorem~\ref{th-cdc17} and Corollary~\ref{cor-cdc17} that for every fixed $m$, the one-variable function
$f_0(m, \Delta m)$, function of $\Delta m$, is concave. \endproof

Optimal $\Delta m^*(m)$ corresponding to every $m$ whose existence is stated by Theorem~\ref{th-itsc} are not necessarily unique.
They are derived analytically from Theorem~\ref{th-cdc17} and Corollary~\ref{cor-cdc17} and are given below in the description of the traffic phases.

The traffic phases are explained on the formula giving the frequencies since it is piece-wise linear.
Eight traffic phases can be distinguished from Theorem~\ref{th-cdc17}, Corollary~\ref{cor-cdc17} and Fig.~\ref{fig-2D}.
They are described by groups.

\subsection{Two Free Flow Phases}
The two phases are given on the 3D space $(m,\Delta m,f_0)$ by the two planes
\begin{align}
   & f_0 = \left(m - \Delta m\right) / \left(\underline{T}_0 + \underline{T}_1\right) \label{I-a} \tag{I-a} \\
   & f_0 = \left(m + \Delta m\right) / \left(\underline{T}_0 + \underline{T}_2\right). \label{I-b} \tag{I-b}
\end{align}
They are bordered by six straight lines, as it can be seen in Fig.~\ref{fig-2D}.
The equations of the six lines are available in Appendix~\ref{appendix}.
In these two phases, trains move freely on the line.
For a given number $m$ of trains on the line, the optimal difference $\Delta m^*(m)$ realizing the maximum average train frequency
in the steady state is such that the point $(m,\Delta m^{*}(m))$  is on the straight line separating the two plan surfaces
corresponding to the two free flow phases (line (AG), Fig.~\ref{fig-2D}):
\begin{equation}\label{opt-dm}
    \Delta m^*(m) = \left( \Delta \underline{T} / (2 \underline{T})  \right) m.
\end{equation}
\begin{proposition}
  $\forall m$ such that $0\leq m \leq m^*$,  $\Delta m^* (m)$ given in~(\ref{opt-dm}) (line (AG), Fig.~\ref{fig-2D}) is unique.
\end{proposition}
\begin{proof}
It is easy to check that the point $(m^*, \Delta m^*(m^*), f_{\max})$ is the intersection point of the three planes~(\ref{I-a}),~(\ref{I-b}) and
the plane $f_0 = f_{\max}$. From Theorem~\ref{th-cdc17} and Corollary~\ref{cor-cdc17}, it can be seen that for every fixed $m, 0\leq m\leq m^*$, the one-variable function
$f_0(\Delta m)$, is strictly concave.
\end{proof}

~\\
\textbf{\textit{I - a - Free flow on central part and branch 1.}}

This case corresponds to points $(m, \Delta m)$ located on the plane~I-a of Fig.~\ref{fig-2D},
whose equation is~(\ref{I-a}), and delimited by four straight lines, whose equations
in the plane $(m, \Delta m)$ are given in Appendix~\ref{appendix}, see (\ref{AG}), (\ref{LG}), (\ref{FL}) 
and (\ref{FA}) respectively.
During this phase, the traffic only moves freely on branch~1 and on the central part.
Branch~2 is overloaded because the difference $\Delta m$ of the number of trains between branches~2 and~1 exceeds the optimum $\Delta m^*$ given by~(\ref{opt-dm}).
Trains on branch~2 have to wait before entering the central part in order to respect the one-by-one rule at the junction.

~\\
\textbf{\textit{I - b - Free flow on central part and branch 2.}}

This case corresponds to points $(m, \Delta m)$ belonging to the plan~I-b of Fig.~\ref{fig-2D},
whose equation is~(\ref{I-b}), and delimited by four straight lines, whose equations
in the plane $(m, \Delta m)$ are given in Appendix~\ref{appendix}, see (\ref{AG}), (\ref{GH}), (\ref{BH}) 
and (\ref{AB}) respectively.
During this phase, the traffic only moves freely on branch~2 and on the central part.
Here, branch~1 is overloaded because the difference $\Delta m$ of the number of trains between branches~2 and~1 is under the optimum $\Delta m^*$ given by~(\ref{opt-dm}). 
Trains on branch 1 have to wait before being able to enter the central part in order to respect the one-by-one rule at the junction.

\subsection{Two Unbalanced Branches Phases}

The two phases are given by the two planes
\begin{align}
   & f_0 = 2 \left( n_2 - \Delta m \right) / \left( \underline{T}_1 + \underline{S}_2\right) \label{II-a} \tag{II-a} \\
   & f_0 = 2 \left( n_1 + \Delta m \right) / \left( \underline{S}_1 + \underline{T}_2\right). \label{II-b} \tag{II-b}
\end{align}
The two phases are located opposite of each other, each one bordered by four straight lines in the plane $(m, \Delta m)$,
see Fig.~\ref{fig-2D}.
These phases are characterized by a large number of running trains and a non-optimal difference between the number of trains on the two branches.
A consequence of that is, that trains on one branch have to wait at the convergence before entering the central part.
The average train frequency corresponds both to free flow of trains on the branch with a small number of trains, and to
congested flow on the branch with big number of trains, where the trains bother each other, as in road traffic.

~\\
\textbf{\textit{II - a - Free flow on branch 1, congestion on branch 2.}}

This case corresponds to points $(m,\Delta m)$ located above the straight line (KL):
\begin{equation}\label{opt-dm_7}
    \Delta m \geq n_2 - f_{\max} \left( \underline{T}_1 + \underline{S}_2 \right) / 2.
\end{equation}
The traffic phase II-a is delimited by four straight lines, whose equations are given in Appendix~\ref{appendix} by~(\ref{EF}), (\ref{FL}), (\ref{KL}) and (\ref{EK}).

~\\
\textbf{\textit{II - b - Free flow on branch 2, congestion on branch 1.}}

This case corresponds to points $(m,\Delta m)$ below the straight line (HI):
\begin{equation}\label{opt-dm_8}
    \Delta m \leq f_{\max} \left( \underline{S}_1 + \underline{T}_2 \right) / 2 - n_1. 
\end{equation}
The traffic phase II-b is delimited by the four straight lines given in Appendix~\ref{appendix} by equations~(\ref{BC}), (\ref{CI}), (\ref{HI}) and (\ref{BH}).

\subsection{Two Congested Traffic Phases}

The two phases are given by the two planes with the following equations:
\begin{align}
   & f_0 = \left( \bar{m} + \Delta \bar{m}\right) / \left( \underline{S}_0 + \underline{S}_2 \right) \label{III-a} \tag{III-a} \\
   & f_0 = \left( \bar{m} - \Delta \bar{m}\right) / \left( \underline{S}_0 + \underline{S}_1 \right). \label{III-b} \tag{III-b}   
\end{align}
They are bordered by six straight lines, as depicted in Fig.~\ref{fig-2D}.
The equations of the six lines are given in Appendix~\ref{appendix}.
These phases are characterized by the fact that the total number $m$ of trains is big enough 
such that at least the central part and one of the two branches are congested.
The asymptotic average train frequency is then given by this congestion. 
The congested traffic phases require a large number of trains and are rather of theoretic interest.

~\\
\textbf{\textit{III - a - Congestion on branch 2 and the central part.}}

This case corresponds to 
\begin{equation}\label{opt-dm_10}
    \Delta \bar{m} \leq \left( \Delta \underline{S} / ( 2 \underline{S} ) \right) \bar{m}.
\end{equation}
A set point on the~\ref{III-a} plane, signifies congestion on branch 2 and the central part.
The number of trains on these two parts of the line is such high with regard to the number of segments,
that trains cannot move freely.

~\\
\textbf{\textit{III - b - Congestion on branch 1 and the central part.}}

This case corresponds to 
\begin{equation}\label{opt-dm_11}
    \Delta \bar{m} \geq \left( \Delta \underline{S} / (2 \underline{S} ) \right) \bar{m}.
\end{equation}
In the contrary to above, if the set point is on the~\ref{III-b} plane, there is congestion on branch 1 and the central part.
Trains on these parts bother each other and have to wait until their predecessor has cleared the downstream segment.

\subsection{Maximum Frequency and Zero Flow Phase} 

The two phases are given by the two planes
\begin{align}
   & f_0 = f_{\max} = 1 / h_{\min}, \label{IV-a} \tag{IV-a} \\
   & f_{0,1,2} = 0. \label{IV-b} \tag{IV-b}
\end{align}
Here, the train frequency is independent of the total number of trains $m$ and of the difference $\Delta m$.

~\\
\textbf{\textit{IV - a - Maximum frequency phase.}}

This phase corresponds to the surface delimited by the polygon (GHIJKL) in Fig.~\ref{fig-2D}.
The total number $m$ of trains is sufficiently high and well allocated between branches and central part to realize the maximum train-frequency and is 
sufficiently low to avoid congestion.
The optimal set point $(m^*, \Delta m^*)$ is located at point $G$ in Fig.~\ref{fig-2D}.
An important detail is that two border lines lying opposite of each other are not necessarily parallel.
Only the lines (HI) and (KL) which are border to phases~\ref{II-a},~\ref{II-b} are parallel to each other and to the m-axis.

~\\
\textbf{\textit{IV - b - Zero flow phase.}}

This phase includes all the points outside of the polygon (ABCDEF) or on its border, in Fig.~\ref{fig-2D}.
Only the points $(m, \Delta m)$ at the border of the polygon are feasible, where the train frequency is zero.
The points outside of the polygon are non-realizable.
 The border lines, starting from point A $(m,\Delta m = 0)$, against the clock are:
 $\Delta m \equiv m$ (all trains on branch 1),
 $\Delta m \equiv -n_1$ (branch 1 full, additional trains on the central part),
 $\Delta m \equiv m - n_0 - 2n_1$ (branch 1 and central part full, additional trains on branch 2).
 Here, point D $(m,\Delta m = n,\Delta n)$ is reached, where there is a train on every segment.
 Now with the clock, again starting from point A, to point D:
 $\Delta m \equiv - m$ (all trains on branch 2),
 $\Delta m \equiv n_2$ (branch 2 full, additional trains on the central part),
 $\Delta m \equiv - m + n_0 + 2n_2$ (branch 2 and central part full, additional trains on branch 1).

\subsubsection{Lines' Equations and Points' Coordinates}

\label{appendix}

\textit{\textbf{Coordinates of the points $A,B,\ldots, L$ in the $m / \Delta m$ plane.}}\\
$$\begin{array}{lr}
    A = (0,0) &  B = (n_1,-n_1) \\
    C = (n_0 + n_1, - n_1) & D = (n, \Delta n) \\
    E = (n_0 + n_2, n_2) & F = (n_2,n_2)
\end{array}$$

$$\begin{array}{l}
    G = \left( \underline{T} f_{\max}, \Delta \underline{T} f_{\max} / 2 \right) \\
    H = \left( \frac{2\underline{T}_0 + \underline{T}_2 - \underline{S}_1}{2h_{min}} + n_1, \frac{\underline{S}_1 + \underline{T}_2}{2h_{min}} - n_1 \right) \\
    I = \left( n_0 + n_1 + \frac{\underline{T}_2 - 2\underline{S}_0 - \underline{S}_1}{2h_{min}}, \frac{\underline{S}_1 + \underline{T}_2}{2h_{min}} - n_1 \right) \\
    J = \left( n - \frac{2 \underline{S}_0 + \underline{S}_1 + \underline{S}_2}{2h_{min}}, \Delta n - \frac{\Delta \underline{S}}{2h_{min}} \right) \\
    K = \left( n_0 +n_2 + \frac{\underline{T}_1 - 2\underline{S}_0 - \underline{S}_2}{2h_{min}}, n_2 - \frac{\underline{T}_1 + \underline{S}_2}{2h_{min}} \right) \\
    L = \left( \frac{2\underline{T}_0 + \underline{T}_1 - \underline{S}_2}{2h_{min}} + n_2, n_2 - \frac{\underline{T}_1 + \underline{S}_2}{2h_{min}} \right)
\end{array}$$


\textit{\textbf{Straight lines and semi-planes delimiting phase IV-b.}}\\
\begin{align}
    m        & \equiv  - \Delta m, \label{AB} \tag{AB}\\
    \Delta m & \equiv  - n_1, \label{BC} \tag{BC}\\
    m        & \equiv  \Delta m + (n_0 + 2n_1), \label{CD} \tag{CD}\\
    m        & \equiv  - \Delta m + (n_0 + 2n_2), \label{DE} \tag{DE}\\
    \Delta m & \equiv  n_2, \label{EF} \tag{EF}\\
    m        & \equiv  \Delta m. \label{FA} \tag{FA}
\end{align}

\textit{\textbf{Straight lines and semi-planes delimiting phase IV-a.}}\\
\begin{align}
    m        & \equiv  \frac{\underline{T}_0+\underline{T}_2}{h_{min}} - \Delta m, \label{GH} \tag{GH} \\
    \Delta m & \equiv  \frac{\underline{S}_1 + \underline{T}_2}{2h_{min}} - n_1, \label{HI} \tag{HI} \\
    m        & \equiv  n_0 + 2n_1 - \frac{\underline{S}_0 +\underline{S}_1}{h_{min}} + \Delta m, \label{IJ} \tag{IJ} \\
    m        & \equiv  n_0 + 2n_2 - \frac{\underline{S}_0+\underline{S}_2}{h_{min}} - \Delta m, \label{JK} \tag{JK} \\
    \Delta m & \equiv  n_2 - \frac{\underline{T}_1 + \underline{S}_2}{2h_{min}}, \label{KL} \tag{KL} \\
    m        & \equiv  \Delta m + \frac{\underline{T}_0 + \underline{T}_1}{h_{min}}. \label{LG} \tag{LG}
\end{align}

\textit{\textbf{Straight lines and semi-planes delimiting phases I-a/I-b, I-b/II-b, II-b/III-b, III-b/III-a, III-a/II-a, II-a/I-a:}}\\
\begin{align}
    \Delta m & \equiv  \left( \Delta \underline{T} / \underline{T} / 2 \right) \; m, \label{AG} \tag{AG} \\
    \Delta m & \equiv  \frac{-m (\underline{S}_1 + \underline{T}_2) + 2n_1 (\underline{T}_0 + \underline{T}_2)}{\underline{S}_1 - 2\underline{T}_0 - \underline{T}_2}, \label{BH} \tag{BH}
    \end{align}
    \begin{align}
    \Delta m & \equiv  \frac{- 2n_1 (\underline{S}_0 - \underline{T}_2) + (\underline{S}_1 + \underline{T}_2) (n_0 - m)}{2\underline{S}_0 + \underline{S}_1 - \underline{T}_2}, \label{CI} \tag{CI} \\
    \Delta m & \equiv  \Delta n + \frac{\Delta \underline{S}(m - n)}{2\underline{S}_0 + \underline{S}_1 + \underline{S}_2}, \label{DJ} \tag{DJ}
        \end{align}
    \begin{align}
    \Delta m & \equiv  \frac{2n_2(\underline{S}_0-\underline{T}_1)-(\underline{S}_2+\underline{T}_1)(m-n_0)}{2\underline{S}_0 + \underline{S}_2 - \underline{T}_1}, \label{EK} \tag{EK} \\
    \Delta m & \equiv  \frac{- 2n_2 (\underline{T}_0 + \underline{T}_1) + m(\underline{S}_2 + \underline{T}_1)}{\underline{S}_2 - 2\underline{T}_0 - \underline{T}_1}. \label{FL} \tag{FL}
\end{align}

\section{Feedback Control Laws for the Number of Trains}\label{mac-control}
Above, the traffic phases of the train dynamics on a metro line with a junction where trains respect minimum travel times, have been derived. An ensemble of stable set points in the stationary regime has been characterized where trains respect minimum dwell and run times. These minimum values can include time margins to handle small perturbations.
In case these margins are insufficient to recover perturbations, a macroscopic control shifting the set point of the system can optimize metro operations.
Some macroscopic control laws for a feedback control of the number of trains to respond to a change in passenger travel demand or a perturbation are presented below.

The derivation of the traffic phases of a metro line with a junction gives an ensemble of optimal operating points
for different train frequencies.
More precisely, these optimal operating points are all the ones of the straight line (AG) in Fig.~\ref{fig-2D}
separating the two surfaces I-a and I-b of the free flow phases (see equations~(\ref{I-a}) and~(\ref{I-b})).
This section concerns the control of the number of running trains ($m$ and $\Delta m$), to track the optimal operating point, changing with different parameters: run and dwell times and the required train frequency (responding to the level of passenger demand). As mentioned above, the macroscopic control we propose here is applied in the case where travel time margins are insufficient to recover perturbations.

Consider the following notations.

\begin{tabular}{lp{0.8\textwidth}}
$\underline{T}$ & $:= \left( 2\underline{T}_0 + \underline{T}_1 + \underline{T}_2 \right) / 2$.\\
  $\Delta \underline{T}$ & $:= \underline{T}_2 - \underline{T}_1$.\\
 $\underline{S}$ & $:= \left( 2\underline{S}_0 + \underline{S}_1 + \underline{S}_2 \right) / 2$. \\
  $\Delta \underline{S}$ & $:= \underline{S}_2 - \underline{S}_1$.
\end{tabular}

For the regulation, the following is considered.
\begin{itemize}
  \item Observed state variables: $(\underline{T}, \Delta \underline{T})$ or $f_0$.
  \item Set point: $f_0$.
  \item Control variables: $m, \Delta m$.
\end{itemize}
The procedure of traffic regulation is then the following.
Assuming that the metro line is operated at an optimal operating point on the straight line (AG) of Fig.~\ref{fig-2D},
the run and dwell times are on-line measured. Every time a train runs (stops) on a given segment, the corresponding run (dwell) time
and the total travel times ($\underline{T}, \Delta \underline{T}$) are updated.
If the coordinates of the set point in the plan of Fig.~\ref{fig-2D} change, the control retrieves the new $m$ and $\Delta m$ corresponding
to the new coordinates of the set point.
Acting on $m$ consists in inserting or removing a train.
Acting on $\Delta m$ can be done by instantaneously changing the train order at the junction.
Indeed it is sufficient to once change the order at the divergence or at the convergence,
to realize a one step shift of $\Delta m$.
However, it is preferable to act only on the train passing order at the convergence, because at the divergence, changing the train destination means that passengers have to change trains which leads to passenger congestion at the divergence and degrades the quality of service.
This control is macroscopic, since it regulates the number of trains on the three parts of the line, the junction,
and consequently the average headways on the line.

Below are derived the main equations needed for the feedback control, distinguished in three cases.
The system of equations of the straight line (AG) in the 3D space is given directly by the
two equations~(\ref{I-a}) and~(\ref{I-b}).
Equivalently, as shown in the following, it can be written as a combination of its 2D equation (see Appendix~(\ref{AG}) in the $(m, \Delta m)$ plane
with equation~(\ref{I-a}) (or equivalently~(\ref{I-b})).
\begin{align}
    & 2 \Delta m / m = \Delta \underline{T} / \underline{T}, \label{ceq1} \\
    & f_0 = (m - \Delta m) / (\underline{T}_0 + \underline{T}_1). \label{ceq2}
\end{align}
Solving for $\Delta m$ in~(\ref{ceq1}) and replacing it in~(\ref{ceq2}) gives
\begin{equation}\label{rel1}
  f_0 = m / \underline{T}.
\end{equation}
Solving for $m$ in~(\ref{ceq1}) and replacing it in~(\ref{ceq2}) gives
\begin{equation}\label{rel2}
  f_0 = 2 \Delta m / \Delta \underline{T}.
\end{equation}

Formula~(\ref{rel1}) is the slope of the 3D straight line (AG) in the $(m,f)$ plane.
It retrieves the one for a linear metro line without junction, derived in~\cite{FNHL16}.
Indeed, (\ref{rel1}) can also be written $f_0 = \rho v$, where $\rho := m/L$ is the average train density on the metro line
and $v := L / \underline{T}$ is the free flow (or maximum) train speed on the metro line.

Formula~(\ref{rel2}) is new. It gives the slope of the 3D straight line (AG) in the 2D $(\Delta m,f)$ plane.
Formula~(\ref{rel2}) can not be written with train densities and speeds.
These equations serve for traffic control. 
From~(\ref{rel1}) and~(\ref{rel2}) the following feedback laws are derived.
\begin{align}
   & m = \underline{T} \; f_0. \label{law1} \\
   & \Delta m = \Delta \underline{T} \; f_0 / 2. \label{law2}
\end{align}

Feedback laws~(\ref{law1}) and~(\ref{law2}) allow to control the number of running trains.
The controls $m$ and $\Delta m$ need to be rounded to integers before applying them.
This does not affect stability, since all the set points are stable.

 In the following, three cases of changing traffic conditions are distinguished.
 \begin{itemize}
   \item Due to an enduring affluence of arrival passengers, there is the need to increase the train frequency of the
     metro line in order to respond to the travel demand. As a consequence, the set point $f_0$ changes.
   \item Due to an incident, train travel times increase. Consequently, a changing in
	$\underline{T}$ and/or in $\Delta \underline{T}$ is observed. However, there is no possibility to insert or to cancel trains
	in order to respond to the incident.
   \item As above, due to an incident on the metro line, train travel times increase. Consequently, a changing in
	$\underline{T}$ and/or in $\Delta \underline{T}$ is observed. This time, trains can be inserted or canceled to respond to the incident.
 \end{itemize}

\subsubsection{Control Law for the Number of Trains in Feedback of the Passenger Travel Demand}
\label{control-1}
For a transition from one time period to another one with different travel demand levels (for example when passing from off-peak to peak hour), 
where the difference between the travel demands on the two branches does not change (and where no further disturbances affect operations)
the train frequency needs to be updated to respond to the new travel demand. The controls $m$ and $\Delta m$ are then calculated as functions
of the new train frequency $f_0$ with the following feedback law.
\begin{equation}\label{lin2}
  \begin{pmatrix}
     m \\ \Delta m
  \end{pmatrix}
  =
  \begin{pmatrix}
     \underline{T} \\ \Delta \underline{T} /2
  \end{pmatrix}
  \;
  f_0.
\end{equation}

In practice, the new train frequency $f_0$ is evaluated with regard to the new travel demand.
Then the new controls $m$ and $\Delta m$ are derived from the feedback law~(\ref{lin2}),
where $\underline{T}$ and $\Delta \underline{T}$ are supposed to remain unchanged.

With regard to Fig.~\ref{fig-2D}, for a new train frequency level $f_0$, the feedback~(\ref{lin2}) 
will move the optimal operating point on the 2D graphic of Fig.~\ref{fig-2D} by remaining on the straight line (AG), 
to reach a new optimal operating point corresponding to the required new train frequency set point $f_0$.

\subsubsection{Control Law for the Number of Trains in Case of a Perturbation on the Branches}
\label{control-2}
In the following, a time period with constant travel demand level and
disturbances on $\underline{T}$ and $\Delta \underline{T}$ is considered.
However, there is no possibility to act on the total number $m$ of trains, that means trains cannot be inserted or canceled.
This is a common constraint in transportation, for example due to the non-availability of additional trains and/or drivers.
A consequence of this assumption is that the train frequency will decrease due to the disturbances on $\underline{T}$ and $\Delta \underline{T}$.
The optimal control $\Delta m$ that permits to remain on the straight line (AG) of Fig.~\ref{fig-2D} is the following.
\begin{equation}\label{lin4}  
     \Delta m =  m \Delta \underline{T} / (2\underline{T}).
\end{equation}
The new train frequency changes with respect to $\underline{T}$, independent of $\Delta \underline{T}$, as follows.
\begin{equation} \nonumber
  f_0 = m / \underline{T}.  
\end{equation}
With respect to Fig.~\ref{fig-2D}, a disturbance on $\Delta \underline{T}$ and $\underline{T}$ modifies the phase diagram such that the
resulting $\Delta m$ re-places the frequency set point $f_0$ back on the straight line (AG), but on a degraded level, due to the prolongation of the travel time $\underline{T}$.

\subsubsection{Control Law for the Number of Trains on the Branches in Case of a Perturbation on the Branches}
\label{control-3}
During a time period where the travel demand level is constant,
the train frequency needed to absorb the travel demand does not change.
Therefore, during the considered time period, the set point $f_0$ remains the same over all time steps. 
In this case, the feedback law system is linear on $\underline{T}$ and $\Delta \underline{T}$,
where the train frequency $f_0$ is assumed to be unchanged.
Moreover, the two feedback laws are independent of each other since the gain matrix is diagonal.
\begin{equation}\label{lin1}
  \begin{pmatrix}
     m \\ \Delta m
  \end{pmatrix}
   = 
   \begin{pmatrix}
      f_0 & 0 \\ 0 & f_0/2
   \end{pmatrix}
   \;
   \begin{pmatrix}
      \underline{T} \\ \Delta \underline{T}
   \end{pmatrix}.
\end{equation}

In this case, the two feedback laws can be applied separately. 
That is to say that feedback law~(\ref{law1}) can be applied every time step a change in $\underline{T}$ is observed,
and feedback law~(\ref{law2}) can applied every time step a change in $\Delta \underline{T}$ is observed.

With respect to Fig.~\ref{fig-2D}, a disturbance on $\underline{T}$ will modify the phase diagram
such that the resulting $m$ from feedback law~(\ref{law1}) combined with the unchanged $\Delta m$ will 
guarantee the same train frequency $f_0$.
Similarly, a disturbance on $\Delta \underline{T}$ will modify the 2D graphic of Fig.~\ref{fig-2D}
such that the resulting $\Delta m$ from feedback law~(\ref{law2}) combined with the unchanged $m$ will 
maintain the train frequency set point $f_0$.


Table~\ref{table-setpointcontrol} summarizes for all types of perturbations studied above,
how the phase diagram and the optimal set point change, and
how control variables have to be adjusted to respond to the perturbations.

\begin{table}[h]
\centering
\caption{Summary over control actions as a function of: 1.~Disturbances on observed variables; 2. The set point.}
\begin{tabular}{|l|l|c|c|cc|}
\hline
\vspace{2pt}
\textbf{Feedback} &\textbf{{Observation}}		&\textbf{{Phase}}	&\textbf{{Set}} &\multicolumn{2}{c}{\textbf{{Control}}}\vline\\
\textbf{control} &		&\textbf{{diagram}}	& \textbf{{point}}  &\multicolumn{2}{c}{\textbf{{variables}}}\vline \\ 
\vspace{2pt}
&      														&				&$f_0$		  				& $m$ 		&$\Delta m$ \\ \hline
\vspace{2pt}
Part \ref{control-1}&  $\nearrow f_0$	  										&same			& $\nearrow$						&	$\nearrow$ & $|\Delta \underline{m}| \nearrow$	\\
& & & & & \\ \hline
\vspace{2pt}
Part \ref{control-2}&    $(\underline{T}\nearrow, \Delta \underline{T} \updownarrow)$		& new  			& $\searrow$ 					& same   & follows \\
&    	&   			&  					&    & $ \Delta \underline{T}$ \\ \hline
\vspace{2pt}
Part \ref{control-3}&    $(\underline{T}\nearrow, \Delta \underline{T} \updownarrow)$		& new  			& same 					& $\nearrow$   & follows \\ 
&   	&   			&  					&    & $\Delta \underline{T}$ \\ \hline
\end{tabular}
\label{table-setpointcontrol}
\end{table}

\chapter{The Effect of the Passenger Travel Demand on the Traffic}\label{Chap-4}
\chaptermark{The Effect of the Passenger Travel Demand}
\begin{quote}{In Chapter~\ref{1A}, a discrete event traffic model of the train dynamics on a metro line with a junction where trains respect given lower bounds on train dwell, run and safe separation times has been presented.
In this chapter, the model is extended. The train dwell times are modeled as a function of the passenger travel demand, such that they are extended for trains with a long time-headway within a margin on the run times. The train run times are controlled to cancel a possible extension of the train dwell times.
Two applications are presented in the following. A first one to a linear line and a second one to a line with a junction.
In both cases, it is shown that the dynamics of the entire line, with the combination of demand-dependent dwell times and controlled run times, are max-plus linear.
The main result is the derivation of the unique asymptotic average growth rate, interpreted as the asymptotic average train time-headway.
The traffic phases of the train dynamics are derived. Their illustration is the fundamental diagram for a linear line and for a line with a junction, relating the variables train frequency, number of trains, passenger travel demand, as well as the run time margin.
}
\end{quote}
\section{A Model of the Demand-dependent Train Dynamics on a Linear Line}\label{acc}
\sectionmark{A Traffic Model of a Linear Line}
\subsection{Passenger Arrivals Modeling}
In the preceding Chapter~\ref{1A}, a model of the train dynamics in a line with a junction has been presented. It has been supposed that trains respect given minimum values on dwell and run times. Indeed, these minimum values can include time margins to respond to smaller perturbations.
In this chapter, the model is extended taking into account the effect of the passenger demand on the train dynamics. This modeling of the passenger arrivals is first applied to the model of a linear line, presented by the authors of~\cite{FNHL17a}. In the second part of this chapter, the model of the train dynamics on a line with a junction of Chapter~\ref{1A} is extended, modeling the passenger arrivals to the platforms and their effect on the train dynamics.

The passenger travel demand is modeled by taking into account average passenger arrival and departure rates to, respectively from the platforms. These rates (or flows) can be calculated from Origin-Destination matrices.
Moreover, an average boarding, respectively average alighting capacity is considered. These average flow-capacities depend on the layout of the trains, especially the length and the number of doors, and the number of passengers on the platform and in the train. To simplify the model, one average boarding and one average alighting capacity is considered per platform. This means that the passenger exchange process is supposed to be independent of the number of passengers on the train and on the platform. Therefore, congestion phenomena cannot be modeled here.

Finally, it is supposed that the train dwell time is calculated such that the passenger alighting and boarding demand on a certain platform is completely satisfied. This means that all the passengers who want to alight and all the ones who are on the platform waiting to board a train, can do so. It is considered that the capacity of train platforms and trains is sufficient to satisfy the demand and are therefore not explicitly modeled here.
Accordingly, consider the following notations.

~\\
\begin{tabular}{lp{0.8\textwidth}}
  $\lambda^{\text{in}}_i$ & $=\sum_j \lambda_{ij}$ the average passenger arrival rate on origin platform $i$ to any destination platform. \\      
  $\lambda^{\text{out}}_j$ & $=\sum_i \lambda_{ij}$ the average passenger departure rate from destination platform $j$, from any origin platform. \\
  $\alpha^{\text{in}}_j$ & average passenger boarding rate on platform $j$. \\      
  $\alpha^{\text{out}}_j$ & average passenger alighting rate on platform $j$. \\  
\end{tabular}
~\\

Using the parameters defined above, the number of boarding passengers is modeled by multiplying the average passenger arrival flow to a platform (number of passengers per time interval) with the train time-headway. The train time-headway is the time interval between two consecutive departures from a platform. Dividing the number of boarding passengers by the boarding capacity gives the time for passenger boarding.
The number of alighting passengers, depends on the number of passengers having boarded a train on the preceding stations. Here, it is directly modeled using the average passenger departure rate from the alighting platform, the corresponding alighting capacity and the time-headway of the train at the alighting platform, such that the time for passenger alighting is extended for a train with a long time-headway. Then, consider the following notations.

~\\
\begin{tabular}{lp{0.8\textwidth}}
  $\frac{\lambda^{\text{in}}_i}{\alpha^{\text{in}}_i} \; h_i$ & time for passenger boarding at platform $j$. \\
  $\frac{\lambda^{\text{out}}_j}{\alpha^{\text{out}}_j}\; h_j$ & $\approx \sum_{i} \lambda_{ij} h_{i} / \alpha_{j}^{out}$ time for passenger alighting at platform $j$.
  Note that the number of alighting passengers at a platform $j$ is estimated with the headway $h_j$ of the train at the arrival platform $j$. This is an approximation since the number of alighting passengers depends on the number of passengers having boarded the train on the previous platforms, which depend on $h_i$.\\
\end{tabular}
~\\

Finally, $x_j$ is defined as the passenger demand parameter incorporating average passenger arrival and departures rates per platform, and average passenger boarding and alighting capacities per platform
\begin{equation} \label{eq-x}
   x_j = \left( \frac{\lambda^{\text{out}}_j}{\alpha^{\text{out}}_j} + \frac{\lambda^{\text{in}}_j}{\alpha^{\text{in}}_j} \right),
\end{equation}
such that $x_j h_j$ represents the time needed for passenger alighting and boarding at platform $j$.

\subsection{Demand-dependent Train Dynamics}
\begin{figure}[h]
  \centering
  \includegraphics{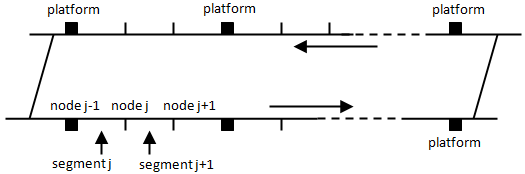}
  \caption{A linear metro line representation.}
  \label{fig-loop-line}
\end{figure}
The notation for the modeling of the train dynamics on a linear line, used throughout this chapter, are similar to those in Chapter~\ref{1A} above.
Similarly to the model for a line with a junction presented in Section~\ref{1A-a} above, the line is discretized into $n$ segments,
as shown in Fig.~\ref{fig-loop-line}.
Considering a linear metro line, it is sufficient to consider the equations of the train dynamics out of the junction~(\ref{neq-c1}) and~(\ref{neq-c2}) and to set, for all nodes and all segments, $u = 0$ in order to retrieve an adequate notation.

Equation~(\ref{neq-c2}) models the constraint on the safe separation time and is for a linear line written
\begin{equation}\label{const2}
       d^k_j \geq d^{k-\bar{b}_{j+1}}_{j+1} + \underline{s}_{j+1}, \quad \forall k \geq 0.
\end{equation}
Other than that, it remains unchanged.

Equation~\ref{neq-c1} models the travel time constraint and is for a linear line written as follows
\begin{equation}\label{const1}
d^k_j \geq d^{k-b_{j}}_{j-1} + \underline{t}_j, \quad \forall k \geq 0.
\end{equation}
More precisely the minimum travel time $\underline{t}_j$ is the sum of minimum run time $\underline{r}_j$ and minimum dwell time $\underline{w}_j$.
\begin{equation}
d^k_j \geq d^{k-b_{j}}_{j-1} + \underline{r}_j + \underline{w}_j, \quad \forall k \geq 0.
\end{equation}
This constraint is rewritten in the following by replacing the minimum run and dwell times by a run time function and a dwell time function.
\begin{itemize}
\item The dwell time function calculates the dwell time depending on the passenger alighting and boarding times defined above.
\item The run time function models a control that cancels a possible extension of the dwell times in case of a high passenger affluence and long passenger exchange times.
\end{itemize}

Replace constraint~(\ref{const1}) with the following.
\begin{equation}\label{const3}
   d^k_j \geq d^{k-b_{j}}_{j-1} + t^k_j(h^k_j,x_j) = d^{k-b_{j}}_{j-1} + r^k_j(h^k_j,x_j) + w^k_j(h^k_j,x_j).
\end{equation}
In constraint~(\ref{const3}), $r^k_j, w^k_j$ and thus $t^k_j$ are functions of the train time-headway $h^k_j$,
and of the passenger demand.
The function $t^k_j(h^k_j,x_j)$ is the control law which makes the train dynamics adaptive with respect to 
train delays and to passenger demand.

In this model, minimum travel times are replaced by a travel time which is a function
of the passenger demand, summarized in parameter $x_j$.
Moreover, the dwell time, the first component of the travel time, is a linear function of the train time-headway and takes into account the passenger demand.
The run time, the second component of the travel time, is as well a function of the train time-headway and allows to recover eventual extensions of the train dwell time.
The two functions $r^k_j(h^k_j)$ and $w^k_j(h^k_j)$ are detailed below.

\subsubsection{Dynamic Programming Traffic Model}
\label{dyn_prog}
Two models of the train dwell time accounting for the
passenger travel demand have been proposed by the authors of~\cite{FNHL17b}.
The first model is of the form
\begin{equation}\label{mod1}
  w^k_j = \frac{\lambda^{\text{in}}_j}{\alpha_j} h^k_j.
\end{equation}
It can be seen that the dwell time takes into account the needed boarding time without taking into account the
needed alighting time. This model permits to adjust the dwell times on platforms in function of
the arrival demand.
However, it has been shown in~\cite{FNHL17b} that the train dynamics is not stable in this case.
This means train delays are amplified over time and propagate backwards through the metro line.

To deal with the latter problem, the authors of~\cite{FNHL17b} have proposed a second model for the train dwell times at
platforms. The model is of the form
\begin{equation}\label{mod2}
  w^k_j = \max\left\{\underline{w}_j , \bar{w}_j - \theta_j \frac{\lambda^{\text{in}}_j}{\alpha_j} h^k_j \right\},
\end{equation}
with $\theta_j$ being a control parameter to be fixed.
The authors have shown that the model~(\ref{mod2}) guarantees the stability of the train dynamics,
and that the dynamics admit an asymptotic regime with an asymptotic average train time-headway.
The latter is derived by simulation in function of the number of trains and of the level of
the passenger demand.
However, with the control law of  model~(\ref{mod2}), a delayed train at platform $j$, which induces an accumulation of passengers at platform $j$, will reduce the train dwell time at that platform,
what is absurd for passengers.

The authors of~\cite{FNHL17b} have shown that with dwell time control~(\ref{mod2}), and under the condition
$0\leq \theta_j\leq 1, \forall j$, the train dynamics
admits an asymptotic regime with a unique average growth rate, which can be interpreted as the average train time-headway. 
Therefore the train dynamics are stable. Moreover, they can be interpreted as the dynamic programming system of 
a stochastic optimal control problem of a Markov chain.

\subsubsection{Max-plus Linear Traffic Model}
\label{max-plus}
The model presented below resolves the problem that the dynamics on a linear line with demand-dependent dwell times accordingly to~(\ref{mod1}) are unstable. Moreover, as it has just been discussed, the dwell time control~(\ref{mod2}) guaranteeing stability is absurd from a passengers point of view, since the dwell times of trains running with a long time-headway are shortened.

~\\
\textbf{\textit{Dwell time model}}\\
The following dwell time model~(\ref{eq-dwell}) is similar to the model~(\ref{mod1}), but it also takes into account
the attraction term of the travel demand.
\begin{align}
  & w_j^k(h_j^k,x_j) = \min (x_j h_j^k, \bar{w}_j ), \label{eq-dwell}
\end{align}

where the $k^{th}$ dwell time on platform $j$, $w^k_j$ is the minimum 
between
the time needed for passenger alighting and boarding ${x}_jh^k_j$,
and an upper bound on the train dwell time $\bar{w}_j$.

Furthermore, the minimum and maximum dwell time are related to the train time-headway $h_j$ and the dynamic interval $g_j$ as follows.
\begin{align}
  & \underline{w}_j = x_j \underline{h}_j,\\
   & \bar{w}_j = x_j \bar{h_j},
\end{align}
where $\underline{h}_j$ and $\bar{h}_j$ are derived from given $\underline{g}_j$ and $\bar{g}_j$, from the formula 
$h = g + w$ of~(\ref{form3}), and from the dwell time law~(\ref{eq-dwell}), as follows.
\begin{align}
   & \underline{h}_j = \underline{g}_j + \underline{w}_j = \underline{g}_j + x_j \underline{h}_j.\\
    & \bar{h}_j = \bar{g}_j + \bar{w}_j = \bar{g}_j + x_j \bar{h}_j.
\end{align}
Then
\begin{align}
   & \underline{h}_j = 1/(1-x_j) \; \underline{g}_j,\\  
    & \bar{h}_j = 1/(1-x_j) \; \bar{g}_j, 
\end{align}
and then
\begin{align}
   & \underline{w}_j = X_j \underline{g}_j, \label{eq-min-dwell}\\
    & \bar{w}_j = X_j \bar{g}_j, \label{eq-max-dwell}
\end{align}
where $X_j := x_j / (1 - x_j)$.

Equation~(\ref{eq-min-dwell}) gives the dwell time for passenger alighting and boarding at platform $j$ for a train which arrives at this platform with the minimum dynamic interval $\underline{g}_j~=~g^k_j~=~a^k_j-d^{k-1}_j$. This minimum dynamic interval is an infrastructure constraint, imposed by the signaling system. The corresponding dwell time $\underline{w}^k_j$ is the minimum dwell time.

Equation~(\ref{eq-max-dwell}) defines the maximum dwell time such that the passenger alighting and boarding demand at platform $j$ is completely satisfied for a train which arrives at this platform with the maximum dynamic interval $\bar{g}_j = g^k_j = a^k_j - d^{k-1}_j$.
In order to guarantee the stability of the trains dynamics, a run time margin is defined below. It will allow to extend dwell times to the maximum $\bar{w}_j$ as defined by equation~(\ref{eq-max-dwell}). The stability conditions are then explained in \textit{Demand-dependent train dynamics}, after Theorem~\ref{th1-acc18}.

%


~\\
\textbf{\textit{Run time model}}\\
In order to deal with the instability issue, the dwell time model is completed with a run time
model which cancels the term that causes instability in the dwell time model.
More precisely, for a train delay at a platform $j$, the dwell time model extends the dwell time
at that platform in order to satisfy the minimum alighting and boarding times given by equation~(\ref{eq-min-dwell}).
The run time model will reduce the run time from platform $j$ to platform $j+1$
in order to compensate in such a way that the whole travel time ($t=w+r$) remains stable.

Consider the following run time law.
\begin{align}
  & r_j^k(h_j^k,x_j) = \max \left\{ \underline{r}_j, \underline{r}_j + \Delta r_j - x_j \left( h_j^k - \underline{h}_j \right) \right\}, \label{eq-run}
\end{align}
where the sum of minimum run time $\underline{r}_j$ and run time margin $\Delta r_j$ is the nominal run time of trains on segment $j$, $\tilde{r}_j$.
The model~(\ref{eq-run}) gives the run time as the maximum between a given minimum run time $\underline{r}_j$
and a term that subtracts $x_j \left( h_j^k - \underline{h}_j \right)$ from the nominal run time.
The term $x_j \left( h_j^k - \underline{h}_j \right)$ expresses a deviation of the boarding and alighting time, 
due to a deviation of the train time-headway.
Notice here that the term $x_j h^k_j$, appearing in the dwell time law~(\ref{eq-dwell}) with a sign ``$+$'',
appears in the run time law~(\ref{eq-run}) with a sign ``$-$''.

Combining the dwell time law~(\ref{eq-dwell}) with the run time law~(\ref{eq-run}), gives the following new travel time law.
\begin{equation}\label{eq-travel}
   t_j^k(x_j) = r_j^k(h_j^k,x_j) + w_j^k(h_j^k,x_j).
\end{equation}

~\\
\textbf{\textit{Demand-dependent train dynamics}}\\
Consider the following notations.
\begin{align}
    & \Delta h_j := \bar{h}_j - \underline{h}_j, \quad \Delta g_j := \bar{g}_j - \underline{g}_j, \nonumber \\
    & \Delta w_j := \bar{w}_j - \underline{w}_j, \quad \Delta r_j := \tilde{r}_j - \underline{r}_j. \nonumber
\end{align}

It is then easy to check the following.
$$\Delta w_j = x_j \Delta h_j = X_j \Delta g_j, \forall j.$$

The following result shows that with the dwell time and run time models chosen above, the train dynamics on a linear line with demand-dependent dwell times are stable in the stationary regime.
\begin{theorem}\label{th1-acc18}
  If $h^1_j \leq \bar{h}_j=1/(1 - x_j) \; \bar{g}_j, \forall j$ and if $\Delta r_j \geq \Delta w_j = X_j \Delta g_j, \forall j$, then the constraints~(\ref{const2})-(\ref{const3}) is
  a max-plus linear system, and is equivalent to
  \begin{align}
    & d^k_j \geq d^{k-b_{j}}_{j-1} + \tilde{r}_j + X_j \underline{g}_j. \label{eq-mp7} \\
     & d^k_j \geq d^{k-\bar{b}_{j+1}}_{j+1} + \underline{s}_{j+1}. \label{eq-mp8}
  \end{align}

The dynamics is then written as follows.

\begin{equation}\label{acc-dynamics}
d^k_j = \max \left\{
	\begin{array}{l}
	    d^{k-b_{j}}_{j-1} + \tilde{r}_j + X_j \underline{g}_j, \\~~\\
	   d^{k-\bar{b}_{j+1}}_{j+1} + \underline{s}_{j+1}.
	\end{array} \right.
\end{equation}
\end{theorem}

\begin{proof}
By induction, it can be shown that~(\ref{const2})-(\ref{const3}) is equivalent to~(\ref{eq-mp7})-(\ref{eq-mp8}) for $k=1$.
\begin{itemize}
  \item On the one side, $x_j h^1_j \leq x_j \bar{h}_j = \bar{w}_j$ holds. Therefore, the first term realizes the minimum in~(\ref{eq-dwell}).
    That is,~(\ref{eq-dwell}) is equivalent to
    \begin{equation}\label{eq-dwell1}
       w_j^1(h_j^1,x_j) = x_j h_j^1.
    \end{equation}
  \item On the other side,
    $$\begin{array}{ll}
            \tilde{r}_j - x_j(h^1_j - \underline{h}_j)  & \geq \tilde{r}_j - x_j\left( \bar{h}_j - \underline{h}_j \right)\\
                                                        & = \tilde{r}_j - \Delta w_j \geq \tilde{r}_j - \Delta r_j = \underline{r}_j
      \end{array}$$
	holds.
    Therefore, the second term realizes the maximum in~(\ref{eq-run}).
    That is,~(\ref{eq-run}) is equivalent to
    \begin{equation}\label{eq-run1}
       r_j^1(h_j^1,x_j) = \tilde{r}_j - x_j \left( h_j^1 - \underline{h}_j \right).
    \end{equation}
\end{itemize}
Consequently,~(\ref{eq-travel}) gives
\begin{equation}\label{eq-travel1}
   t_j^1(x_j) = \tilde{r}_j + x_j \underline{h}_j = \tilde{r}_j + X_j \underline{g}_j.
\end{equation}
Then~(\ref{const3}) can be written
\begin{equation}\nonumber
   d^1_j \geq d^{1-b_{j}}_{j-1} + \tilde{r}_j + X_j \underline{g}_j.
\end{equation}

In the following it is shown, that if~(\ref{const2})-(\ref{const3}) is equivalent to~(\ref{eq-mp7})-(\ref{eq-mp8}) for a given $k$, then
it holds also for $k+1$.
Since it holds for $k$, then a Max-plus linear dynamics~(\ref{eq-mp7})-(\ref{eq-mp8}) will be applied for $k$.
Note by $\mathbf f$ the max-plus map of the max-plus dynamics.
Max-plus linear maps are $1$-Lipschitz for the sup. norm.
The assertion holds for $k$ means that $h^k_j \leq \bar{h}_j, \forall j$.
Then $||d^k - d^{k-1}||_{\infty} \leq \bar{h}_j$.
Hence
$$||d^{k+1} - d^{k}||_{\infty} = ||\mathbf f(d^k) - \mathbf f(d^{k-1})||_{\infty} \leq ||d^k - d^{k-1}||_{\infty} \leq \bar{h}_j.$$
Therefore $h^{k+1}_j \leq \bar{h}_j, \forall j$.
Then, it can easily be shown (as done for $k=1$) that~(\ref{const2})-(\ref{const3}) is equivalent to~(\ref{eq-mp7})-(\ref{eq-mp8}) for $k+1$.
\end{proof}

The two conditions of Theorem~\ref{th1-acc18} can be interpreted as follows.
\begin{itemize}
 \item Condition $\Delta r_j \geq \Delta w_j = X_j \Delta g_j , \forall j$ limits the margin on the train dwell times
   to the margin on the train run times. 
   With this this condition, the upper bound on the train dwell time in equation~(\ref{eq-dwell}) can be fixed to:
\begin{align}\label{max-dwell}
\bar{w}_j = \underline{w}_j + \Delta w_j =  X_j \underline{g}_j + \Delta r_j.
\end{align}
   Indeed, as explained above, the model consists in responding to disturbances and train delays by first extending
   the train dwell times so that the passengers accumulated in the train and on the platforms have time to alight (respectively board) the train, and second,
   by recovering the dwell time extension by reducing the train run times on the inter-stations ahead.
   The following condition ensures this recovering to be possible while fully serving the passenger travel demand, and then the dynamic system to be stable.
 \item Condition $h^1_j \leq \bar{h}_j=1/(1 - x_j) \; \bar{g}_j, \forall j,k$ limits the initial headway $h^1_j$
   (that is the initial condition) to 
   its upper bound $\bar{h}_j$, which is given by the level of the passenger travel demand $x_j$ at platform $j$, the maximum dynamic interval $\bar{g}_j$ and the run time margin $\Delta r_j$ as follows.~\label{cond2}
\begin{align}
   & \bar{h}_j = 1/(1-x_j) \bar{g}_j.
\end{align}
By replacing $\bar{g}_j = \bar{h}_j-\bar{w}_j$:
\begin{align}
& \bar{h}_j = 1/(1-x_j)(\bar{h}_j-\bar{w}_j)\\
& \bar{h}_j = (1/x_j) \bar{w}_j.
\end{align}
Then, replace $\bar{w}_j = X_j \underline{g}_j + \Delta r_j$ accordingly to equation~(\ref{max-dwell}) of the condition above:
\begin{align}
&\bar{h}_j = (1/x_j)(X_j \underline{g}_j + \Delta r_j)\\
&\bar{h}_j = 1/(1-x_j)\underline{g}_j + (1/x_j) \Delta r_j.\label{max-h}
\end{align}
The maximum headway difference $\Delta h_j$ under which the passenger demand is fully served and the train dynamics are stable are related to the run time margin as follows. From equation~(\ref{max-h}), limiting the intial condition to the maximum headway:
\begin{align}
&\Delta h_j = \bar{h}_j - \underline{h}_j= 1/(1-x_j)\underline{g}_j + (1/x_j) \Delta r_j - \underline{w}_j - \underline{g}_j.\\
&\Delta h_j = X_j\underline{g}_j - \underline{w}_j + (1/x_j) \Delta r_j.
\end{align}
And with $\underline{w}_j = X_j\underline{g}_j$:
\begin{align}
&\Delta h_j = X_j\underline{g}_j - X_j \underline{g}_j + (1/x_j) \Delta r_j\\
&\Delta h_j = (1/x_j) \Delta r_j.
\end{align}
\end{itemize}

Since a big value of $h^k_j$ corresponds to a delay of the $k^{th}$ train passing by platform $j$,
this condition tells that if all the delays expressed by $h^k_j, \forall j$ are limited to $\bar{h}_j, \forall j$,
then the dynamic system is max-plus linear, and is then stable, and admits a stationary regime.
In other words, the train dynamics is stable under small disturbances.~\label{cond1}

Note that the nominal run time on segment~$j$ is written $\tilde{r}_j$.
$\tilde{r}_j$ is given by stability condition
\begin{align}
&\Delta r_j \geq \Delta w_j = x_j \Delta h_j, \forall j.
\end{align}
and depends on $\Delta{r}_j$.
Indeed $\Delta{r}_j$ influences the robustness and the frequency
of the metro system, that means it can either be fixed equal to $\Delta{w}_j$ to enhance the average frequency, or it can be chosen greater than $\Delta{w}_j$ to reinforce robustness towards perturbations.



The system can be optimized with regard to train frequency or stability.
In the following section, the traffic phases for a linear line with demand-dependent dwell times and controlled run times are derived and the analytic formulas for the asymptotic average train time-headway $h$ and the asymptotic average
frequency $f$ are presented.

\section{The Steady State Train Dynamics} 

The authors of~\cite{FNHL17a} have presented a max-plus linear traffic model for a linear line, where trains respect lower bounds on dwell and run times.
They have derived the traffic phases of the train dynamics on a linear line, the main result is recalled below.
\begin{equation}\label{th-seattle}
h (m) = \max \left\{ \frac{\sum_j \underline{t}_j}{m}, \max_j (\underline{t}_j+\underline{s}_j), \frac{\sum_j \underline{s}_j}{n-m} \right\},
\end{equation}
with

\begin{tabular}{lp{0.8\textwidth}}
  $m$ & the number of trains,\\
  $n$ & the number of segments on the line.\\  
\end{tabular}

The max-plus theorem derived in~\cite{FNHL17a} gives the asymptotic average train time-headway on a linear line as a function of the number of trains, the number of segments, minimum dwell, run and safe separation times.
Since the traffic model presented above in Section~\ref{acc} with demand-dependent dwell times and controlled run times is max-plus linear, too, $\underline{t}_j$ in~(\ref{th-seattle}) can directly be replaced by $t_j(x_j)$.
\begin{equation}\label{th2-acc18}
  h (m) = \max \left\{ \frac{\sum_j t_j(x_j)}{m}, \max_j (t_j(x_j)+\underline{s}_j), \frac{\sum_j \underline{s}_j}{n-m} \right\}.
\end{equation}

By replacing $t^k_j(x_j)$ using~(\ref{eq-travel}), one finds:

\begin{theorem}\label{th2-acc18}
The asymptotic average train time-headway of the max-plus linear system with dynamic demand-dependent dwell times and controlled run times is given by the asymptotic average growth rate of the system.
The average headway depends on the number of trains $m$ on the line and the passenger travel demand parameter for every platform $X_j$.
   $$  h(m,X) = \max \left\{ \begin{array}{l}
                               \frac{\sum_j (\underline{g}_j X_j + \tilde{r}_j)}{m}, \\~\\
                               \max_j ((\underline{g}_j X_j + \tilde{r}_j)+\underline{s}_j),\\~\\
                               \frac{\sum_j \underline{s}_j}{n-m}.
                            \end{array} \right.$$                                            
\end{theorem}

\begin{proof}
Under the conditions of Theorem~\ref{th1-acc18}, the system can be written in max-plus algebra.
It has been shown that in this case, the asymptotic average growth rate of the system can be analytically derived and
corresponds to the average train time-headway.
Replacing $\underline{t}_j$ in equation~(\ref{th-seattle}) (the max-plus theorem of~\cite{FNHL17a}) allows to obtain the demand-dependent asymptotic average train time-headway on a linear line.
\end{proof}
Notice that the asymptotic average train time-headway $h$ in the free flow phase depends not only on the number of trains $m$, but furthermore on a weighted mean of the passenger travel demand and the nominal run times $\sum_j (\underline{g}_j X_j + \tilde{r}_j)$.
The maximum frequency of the system is given by the maximum over passenger demand, nominal run time and safe separation time, over all segments $\max_j ((\underline{g}_j X_j + \tilde{r}_j)+\underline{s}_j)$.

The following analytic formulas for the asymptotic average train frequency on a linear line are directly obtained from Theorem~\ref{th2-acc18}.
\begin{corollary}\label{cor-acc18}
The asymptotic average frequency of the max-plus linear system with demand-dependent dwell times and controlled run times is a function of the number of trains and the passenger travel demand.
	$$f (m,X) = \min \left\{ \begin{array}{l}
					\frac{m}{\sum_j (\underline{g}_j X_j + \tilde{r}_j)}, \\~\\
					\frac{1}{\max_j ((\underline{g}_j X_j + \tilde{r}_j)+\underline{s}_j)}, \\~\\
					\frac{n-m}{\sum_j \underline{s}_j}.
					 \end{array} \right.$$
\end{corollary}

\begin{proof}
Directly from Theorem~\ref{th2-acc18} with $f = 1/h$.
\end{proof}

Note that in Theorem~\ref{th2-acc18} and Corollary~\ref{cor-acc18}, the asymptotic average train time-headway $h$, respectively asymptotic average frequency $f$ depend linearly on the passenger travel demand parameter $X_j = x_j/(1-x_j)$ with 
$$ x_j = \left( \frac{\lambda^{\text{out}}_j}{\alpha^{\text{out}}_j} + \frac{\lambda^{\text{in}}_j}{\alpha^{\text{in}}_j} \right).$$
Consequently, the dependency between average passenger arrival and departure rates to and from the platforms, and the train dynamics is exponential, not linear.
However, since the average boarding and alighting capacities are typically a lot bigger than the average passenger arrival and departure rates to and from the platforms, the following holds:
$$\lambda^{\text{out}}_j, \lambda^{\text{in}}_j << \alpha^{\text{out}}_j, \alpha^{\text{in}}_j,$$
and $x_j << 1$ can be assumed.
By consequent, for small values of $x_j$, the train dynamics can be said to depend approximately linearly on $x_j$ and on the passenger travel demand.

\section{{The Traffic Phases}}\label{phases-acc18}



The Figures~\ref{acc-1},~\ref{acc-2},~\ref{acc-3},~\ref{acc-6}~and~\ref{acc-7} below depict the asymptotic average train frequency as given by Corollary~\ref{cor-acc18}.
The parameters of Corollary~\ref{cor-acc18} are estimated with the values representing the central part of metro line 13 of Paris.

Precisely, the minimum dynamic intervals $\underline{g}_j$ and the safe separation times $\underline{s}_j$ represent constraints imposed by the block system and have been provided by the operator RATP.
The passenger travel demand is taken into account by the average passenger arrival rates to and the passenger departure rates from the platforms. These flows are estimated based on data from the operator on the number of boarding and alighting passengers per time interval of 30 minutes for each platform. 
Passenger arrival rates to and departure rates from the platforms can be updated continuously. 
Together with the boarding and alighting capacity, the passenger demand parameter $$X_j = x_j/(1-x_j)$$ with $$x_j = \lambda^{in}_j/\alpha^{in}_j + \lambda^{out}_j/\alpha^{out}_j,$$ which is non-zero for every platform, can be calculated.
The train boarding and alighting capacity has been estimated based on average values provided by the operator.

The run times include a margin here, that is $$\tilde{r}_j = \underline{r}_j + \Delta r_j.$$ The minimum run times $\underline{r}_j$ are taken into account as calculated by the MATYS entity of the operator RATP. They are based on train speed profiles, account for train characteristics such as traction, braking curves and maximum speed. They depend furthermore on infrastructure characteristics including curve radius and gradient.
In addition, a margin $\Delta r_j$  is applied on the minimum train run times. The margin is a parameter that can be chosen freely.
It allows to realize dynamic dwell times depending on the passenger volume on the platform while guaranteeing the stability of the train dynamics within the margin. More precisely, dwell times are extended up to a maximum value depending on the run time margin, in case a perturbation causes a long headway on a train accordingly to formula~(\ref{eq-dwell}): $$w^k_j (h^k_j,x_j) = \min(x_jh^k_j,\bar{w}_j).$$ A possible extension of the dwell time is canceled by controlling the train run times by accelerating the train in the inter-station within the run time margin, respecting the minimum run time given by the maximum speed, see formula~(\ref{eq-run}): $$ r^k_j(h^k_j,x_j) = \max\{\underline{r}_j, \tilde{r}_j - x_j(h^k_j-\underline{h}_j)\}. $$
Finally, as it has been derived by the authors of~\cite{FNHL17a} for the traffic phases of the train dynamics on a linear line with constant run and dwell times, the asymptotic average train frequency (and train time-headway) depend on the number of trains.
Note that $n$ is the number of segments on the line.

In Figures~\ref{acc-1},~\ref{acc-2},~\ref{acc-3},~\ref{acc-6}~and~\ref{acc-7} three traffic phases of the train dynamics can be distinguished accordingly to Theorem~\ref{th2-acc18} and Corollary~\ref{cor-acc18}. The asymptotic average frequency depends on:
\begin{itemize}
\item the number of trains $m$,
\item the passenger travel demand and
\item the run time margin.
\end{itemize}


\begin{figure}[h]
    \centering
    \fbox{\includegraphics[width=\textwidth]{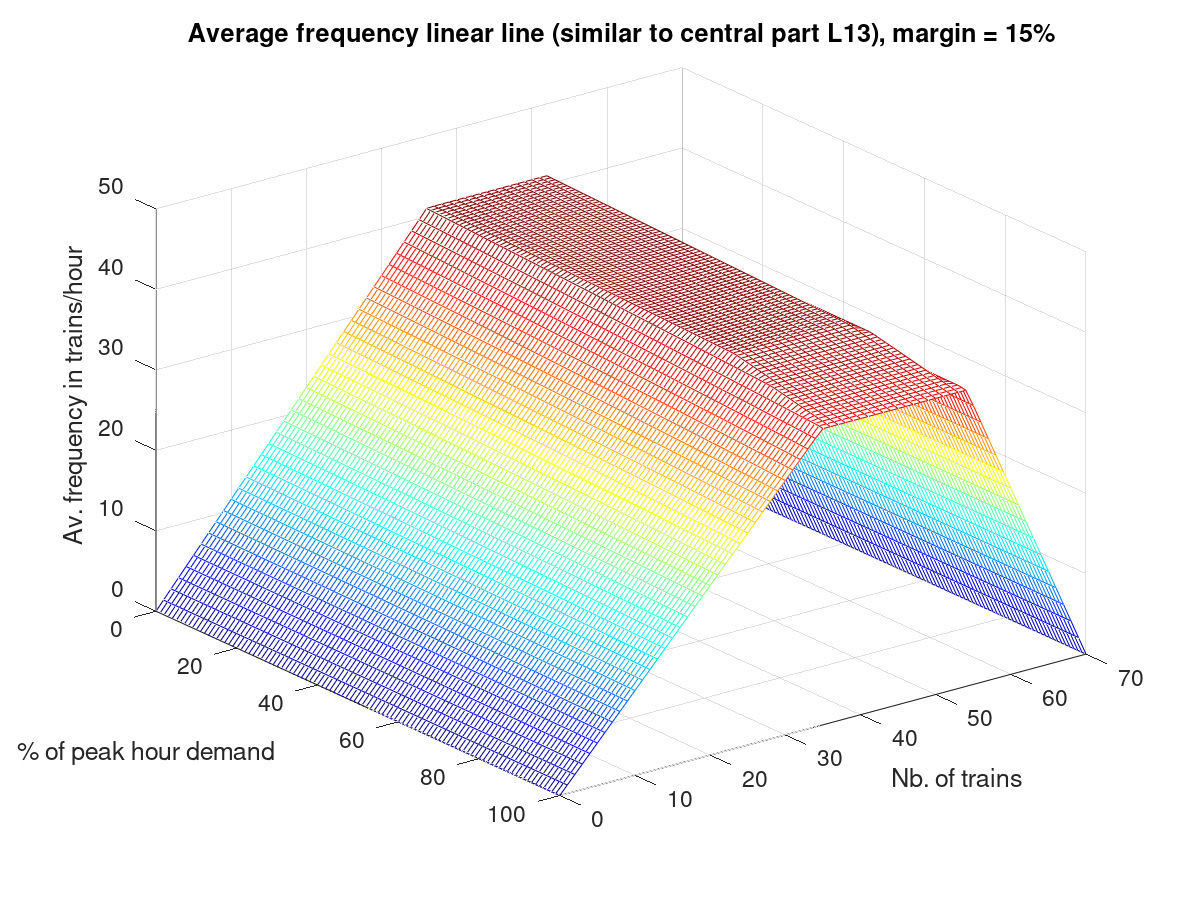}}
    \caption{Average frequency over total number of trains and passenger travel demand, for a run time margin of 15\%, on a linear line similar to the central part of Paris metro line 13 (view 1).}
    \label{acc-1}
\end{figure}

\subsection{The Effect of the Passenger Travel Demand}
\subsubsection{Free Flow Phase}

In the beginning, from the first train running up to a certain limit, the system is in the free flow phase. In this phase, trains run freely on the line, without bothering each other.
As it can be seen in Figure~\ref{acc-1} for a fixed run time margin of 15\%, the frequency increases linearly with the number of trains, for a fixed passenger demand.
For a fixed number of trains, the average frequency decreases non-linearly with increasing passenger arrival and departure flows (here represented in \% of peak hour demand), as it can be seen in Figure~\ref{acc-3}, that is with an increasing passenger travel demand parameter $x_j$.
As it has been shown above, the asymptotic average frequency (and train time-headway) depend linearly on $X_j$ but non-linearly on $x_j$ since $X_j = x_j/(1-x_j)$. However, for small values of $x_j$, the dependency on $x_j$ is approximately linear.
Note in Figure~\ref{acc-3} that contour lines are non-parallel to the passenger demand axis. 
This is a direct consequence of the demand-dependent dwell times, which increase with the passenger affluence. 

\begin{figure}[h]
    \centering
    \fbox{\includegraphics[width=\textwidth]{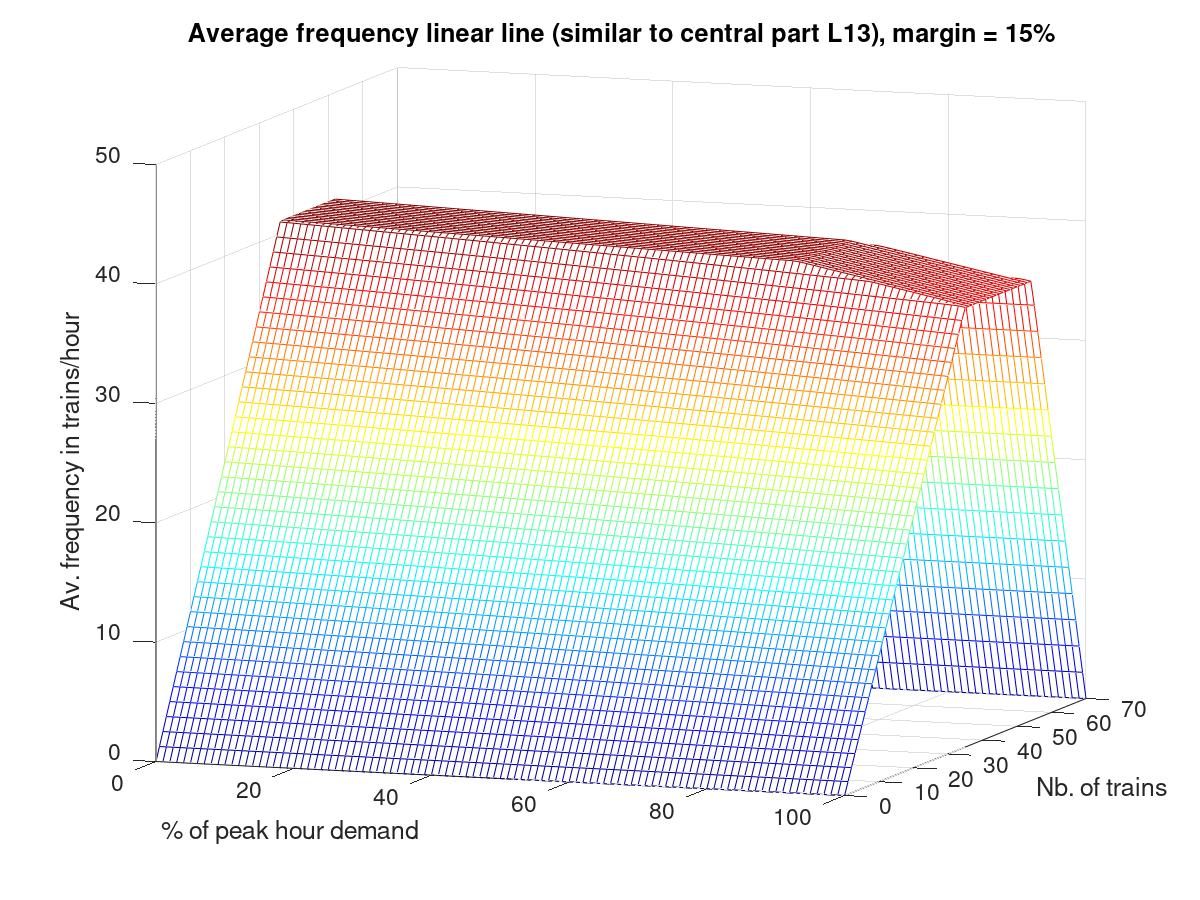}}
    \caption{Average frequency over total number of trains and passenger travel demand, for a run time margin of 15\%, on a linear line similar to the central part of Paris metro line 13 (view 2).}
    \label{acc-2}
\end{figure}

\subsubsection{Maximum Frequency Phase}

The second phase is the maximum capacity phase. Here, the minimum train time-headway, respectively the maximum frequency is realized. The minimum train time-headway is determined by the segment which realizes the maximum for the sum of minimum run time, run time margin, demand-dependent dwell time and safe separation time: $$\max_j{((\underline{g}_jX_j + \tilde{r}_j) + \underline{s}_j)}.$$
As it can be seen in Figure~\ref{acc-2}, in this phase, the maximum frequency is reached and is independent of the number of trains. The system is operated at capacity. As for the standard model for linear lines with fixed run and dwell times, developed by the authors of~\cite{FNHL17a}, the optimal number of trains is reached at the intersection of free flow phase and maximum frequency phase. This number maximizes the frequency for a minimal number of trains. Therefore, there is no interest in running more trains than the one represented by the straight line at the intersection of the two traffic phases.
In the contrary to the standard model in~\cite{FNHL17a}, the maximum train frequency depends on the passenger travel demand.
Figure~\ref{acc-2} allows to analyze this dependency for the case of the central part of Paris metro line 13. Note that with an increasing passenger demand, the capacity decreases. This is a consequence of the demand-dependent dwell times. In Figure~\ref{acc-2}, at around 80\% of peak hour demand a discontinuity in the gradient with respect to passenger travel demand parameter $x_j$ can be identified in the maximum frequency phase. At this point, the bottleneck of the line switches from one segment to another. 
For example, for the central part of Paris metro line 13, for a low demand level, the bottleneck is the segment of platform \textit{Gaîté}, northbound, because of its high run and safe separation times. In the contrary, for a very high demand level, the dwell time becomes important at the platform \textit{Montparnasse -- Bienvenüe}, northbound, which is an important interchange station. Although the segment of that platform has optimized run and safe separation times, the high dwell times make this segment the new bottleneck of the line. 

\begin{figure}[h]
    \centering
    \fbox{\includegraphics[width=\textwidth]{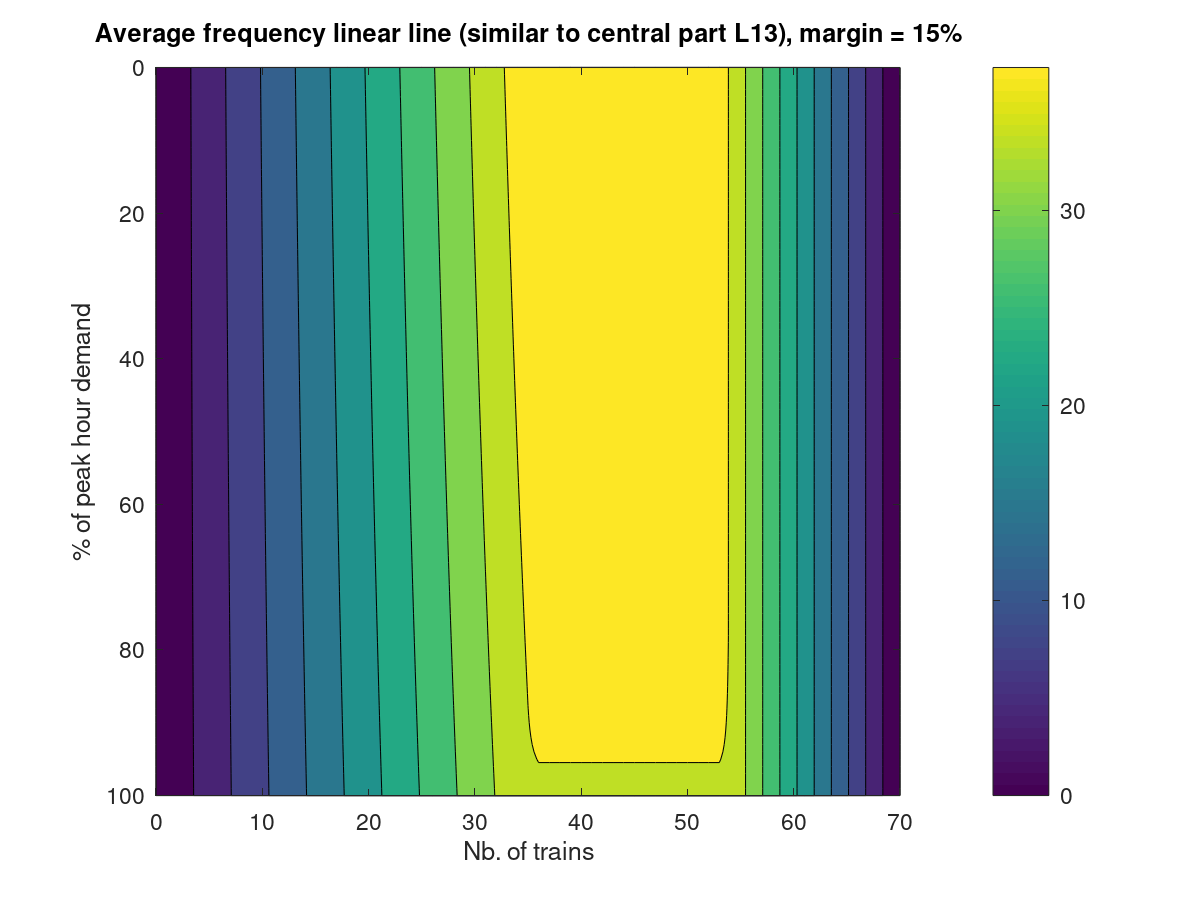}}
    \caption{Average frequency over total number of trains and passenger travel demand, for a run time margin of 15\%, on a linear line similar to the central part of Paris metro line 13}.
    \label{acc-3}
\end{figure}

Finally, Figure~\ref{acc-3} shows that the optimal number of trains depends on the passenger travel demand.
This is a direct consequence of the demand-dependent capacity of the line.
Note that the intersection line between the free flow phase and the maximum frequency phase is not parallel to the passenger travel demand axis.
For example, for some run time margin, here $15 \text{ \%}$, with an increasing passenger demand, the capacity of the system decreases. Moreover, to reach the decreased capacity, the number of trains has to be increased.
In the contrary, the intersection between the maximum frequency phase and the congestion phase does not depend on the passenger travel demand. Congestion occurs from a certain number of trains on, independently of the passenger demand.

\subsubsection{Congestion Phase}

The third phase is the congestion phase. In this phase, the line is overloaded with trains, so that they start to interact with each other. For example, trains will have to wait frequently in inter-station until the preceding train has cleared the downstream platform. Here, the average frequency decreases for an increasing number of trains. Apart from the number of trains, it depends only on the total number of segments and their safe separation time. Consequently, the contour lines of same frequency in Figure~\ref{acc-3} are parallel to the passenger travel demand axis. Note that the frequency decreases linearly in $m$ in this phase. Any operator will avoid that its system reaches this phase.

\subsection{The Effect of the Run Time Margin}
\begin{figure}[h]
    \centering
    \fbox{\includegraphics[width=\textwidth]{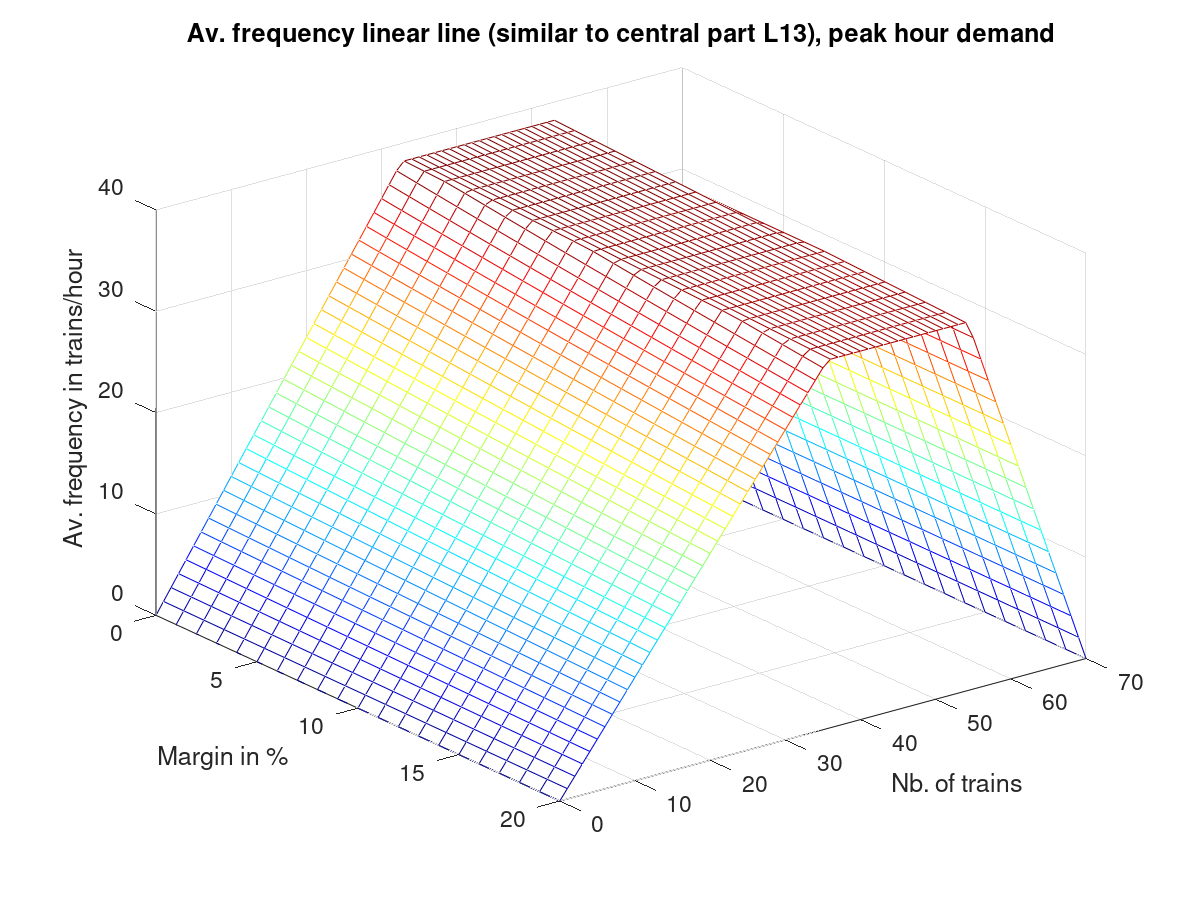}}
    \caption{Average frequency over total number of trains and run time margin, for peak hour demand, on a linear line similar to the central part of Paris metro line 13.}
    \label{acc-6}
\end{figure}
\begin{figure}[h]
    \centering
    \fbox{\includegraphics[width=\textwidth]{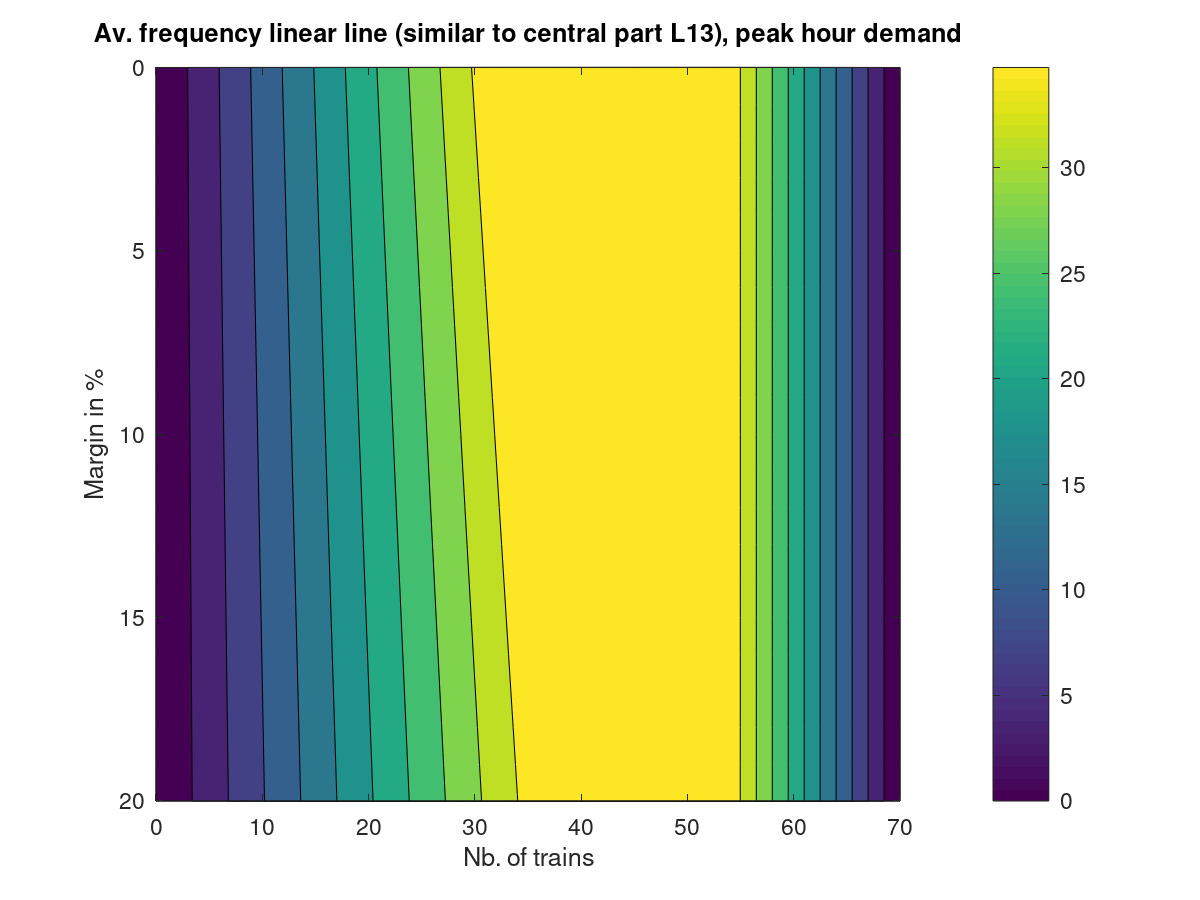}}
    \caption{Average frequency over total number of trains and run time margin, for peak hour demand, on a linear line similar to the central part of Paris metro line 13.}
    \label{acc-7}
\end{figure}

Refer to Figures~\ref{acc-6}~and~\ref{acc-7} to study the effect of the run time margin on the asymptotic average frequency. Again, a free flow phase, a maximum frequency phase and a congestion phase can be distinguished. The two figures depict the asymptotic average frequency on the linear line over the number of trains and the run time margin, for peak hour demand.
 
\subsubsection{Free Flow Phase}
In the free flow traffic phase, for a fixed number of trains, the frequency decreases linearly with an increasing margin. For a fixed margin, the frequency increases linearly with the number of trains.
Moreover, the contour lines of Figure~\ref{acc-7} are a good way to study the cost of an increased time margin. If an operator wants to know how many additional trains have to be inserted in order to guarantee the same frequency with increased run time margins, it is sufficient to follow the corresponding contour line in Figure~\ref{acc-7}.
Taking the example of $f = 30$ and a margin of $5\%$, it can be seen that approximately $30$ trains are necessary to operate the line at the desired frequency. If an operator chooses to increase the robustness of its timetable by increasing the margin to $15\%$, around $3$ additional trains are required to ensure the same average frequency.

\subsubsection{Maximum Frequency Phase}
Refer to Figure~\ref{acc-6} to study the effect of the run time margin on the maximum frequency phase. With an increasing margin, the capacity of the system decreases. This means, increasing the robustness of the timetable by increasing the run time margin affects negatively the capacity of the system. This has to be considered, especially for mass transit lines with a high passenger travel demand.
Moreover, note in Figure~\ref{acc-7} that the optimal number of trains, represented by the intersection between the free flow phase and the maximum capacity phase, increases linearly with an increasing margin at the cost of a lower capacity. Consequently, for an increasing margin not only capacity decreases, but, to reach capacity, the number of trains has to be increased.

\subsubsection{Congestion Phase}
As it can be seen in Figure~\ref{acc-7}, this phase is independent of the run time margin and all contour lines of one frequency are parallel to the margin-axis.

%
%
%


\section{A Model of the Demand-dependent Train Dynamics on a Line with a Junction}\label{cdc18}
\sectionmark{A Traffic Model of a Line with a Junction}
\subsection{Passenger Arrivals Modeling}
This section is based on the model of the train dynamics on a line with a junction, developed in Chapter~\ref{1A} which will be extended here with a passenger demand model, accordingly to the one for a linear line in Section~\ref{acc}.
In the model of Chapter~\ref{1A}, trains are supposed to respect minimum run and dwell times. The train dynamics are modeled as discrete events and the $k^{th}$ departure time from each node $(u,j)$ is written under two constraints, one on the train travel time (run + dwell time), and one on the safe separation time.
\begin{align}
  & d^k_{(u,j)} \geq d^{k-b_{(u,j)}}_{(u,j-1)} + \underline{r}_{(u,j)} + \underline{w}_{(u,j)}. \label{eq-2} \\
  & d^k_{(u,j)} \geq d^{k-\bar{b}_{(u,j+1)}}_{(u,j+1)} + \underline{s}_{(u,j+1)}. \label{eq-3}
\end{align}

Similar to the model for a linear line in Section~\ref{acc}, the minimum run times $\underline{r}_{(u,j)}$ as well as the minimum dwell times $\underline{w}_{(u,j)}$ will be replaced by a demand-dependent control while guaranteeing stability of the dynamics in the stationary regime.
This section follows the modeling approach of~\ref{acc}.
Consider the following additional notations for the passenger
travel demand.

~\\
\begin{tabular}{lp{0.8\textwidth}}
$\lambda_{(u,i),(v,j)}$		& passenger travel demand from platform~$(u,i)$ on part $u$ to platform~$(v,j)$ on part $v$, when $i$ and $j$ denote platforms, and 
					$\lambda_{(u,i),(u,j)} = 0$ if $i$ or $j$ is not a platform node.\\ 
$\lambda^{\text{in}}_{(u,i)}$	& $=\sum_{(v,j)} \lambda_{(u,i),(v,j)}$ the average passenger arrival rate on origin platform $(u,i)$ to any destination platform.\\  
$\lambda^{\text{out}}_{(v,j)}$& $=\sum_{(u,i)} \lambda_{(u,i),(v,j)}$ the average passenger departure rate from destination platform $(v,j)$ from any origin platform.\\
  $\alpha^{\text{in}}_{(u,i)}$	& average passenger boarding capacity on platform ${(u,i)}$. \\      
  $\alpha^{\text{out}}_{(v,j)}$ & average passenger alighting capacity on platform $(v,j)$.
\end{tabular}
~\\

Accordingly to the model in Section~\ref{acc}, the time for passenger alighting and boarding can be estimated as follows. The number of passenger waiting on the platform to board the next train can be calculated by multiplying the passenger arrival flows to the platforms (passengers per time interval) with the train time-headway. Dividing this number by the boarding capacity gives the required passenger boarding time. The boarding capacity can be approximated by an average value which depends on the layout of the platforms, the trains and especially the door width.
The time for passenger alighting is approximated by multiplying the average passenger departure flows from a platform with the train time-headway at the same platform, divided by the passenger alighting capacity. Note, that the number of passengers willing to alight depends on the number of passenger having boarded the train at the preceding platforms. However, the dependency of the alighting time on the train time-headway will still extend the alighting time in case of a long headway, which is a good approximation.

~\\
\begin{tabular}{lp{0.8\textwidth}}
  $\frac{\lambda^{\text{in}}_{(u,i)}}{\alpha^{\text{in}}_{(u,i)}} \; h_{(u,i)}$ & time for passenger boarding at platform $(u,i)$. \\
%
  $\frac{\lambda^{\text{out}}_{(v,j)}}{\alpha^{\text{out}}_{(v,j)}}\; h_{(v,j)}$ & $\approx \sum_{(u,i)} \lambda_{(u,i)} h_{(u,i)} / \alpha_{(v,j)}^{out}$ time for passenger alighting at platform $(v,j)$.
  Note that the number of alighting passengers at a platform $(v,j)$ is estimated with the headway $h_{(v,j)}$ of the train at the arrival platform $(v,j)$. This is an approximation since the number of alighting passengers depends on the number of passengers having boarded the train on the previous platforms, which depend on $h_{(u,i)}$.\\
\end{tabular}
~\\

A passenger demand parameters $x_{(u,j)}$ is defined accordingly to the model for a linear line in Section~\ref{acc}.
\begin{equation}
x_{(u,i)} = \left(\frac{\lambda_{(u,i)}^{out}}{\alpha_{(u,i)}^{out}} + \frac{\lambda_{(u,i)}^{in}}{\alpha_{(u,i)}^{in}}\right),
\label{eq-6}
\end{equation}
such that $x_{(u,i)} h_{(u,i)}$ gives the time needed for passenger alighting and boarding at platform $(u,i)$.
Furthermore, define 
\begin{equation}
  X_{(u,i)} = \frac{x_{(u,i}}{1 - x_{(u,i)}},
\end{equation}
such that, accordingly to the dwell time model in Section~\ref{acc}
\begin{align}
x_{(u,i)} h_{(u,i)} = X_{(u,i)} g_{(u,i)},
\forall \underline{h}_{(u,i)} \leq h_{(u,i)} \leq \bar{h}_{(u,i)} \text{ and }
\underline{g}_{(u,i)} \leq g_{(u,i)} \leq \bar{g}_{(u,i)}, \forall u,i.
\end{align}

\subsubsection{Demand-dependent Dwell Times \& Controlled Run Times}
The minimum run times $\underline{r}_{(u,j)}$ and minimum dwell times $\underline{w}_{(u,j)}$ in~(\ref{eq-2}) and in the model of Chapter~\ref{1A} are here functions of
the passenger travel demand parameter $x_{(u,j)}$ and of the train time-headway 
$h_{(u,j)}$ as in the model for a linear line, see Section~\ref{acc}.
\begin{align}
  & w_{(u,j)}^k = \min\left\{x_{(u,j)}h_{(u,j)}^k, \bar{w}_{(u,j)}\right\}, \label{eq-7}
\end{align}

The first term in equation~(\ref{eq-7}) corresponds to the minimum dwell time satisfying the demand.
The second term in equation~(\ref{eq-7}) relates the upper bound on the dwell time to the passenger travel demand and the maximum dynamic interval. 
The maximum dynamic interval, this means the time interval during which there is no train at a platform, has to be fixed accordingly to the passenger travel demand.

In order to deal with the instability issue, the dwell time model is completed with a run time
model which cancels the term that causes instability in the dwell time model.
More precisely, for a train delay at a platform $(u,j)$, the dwell time model extends the dwell time
at that platform in order to satisfy the minimum alighting and boarding times given by equation~(\ref{eq-7}).
The run time model will reduce the run time from platform $(u,j)$ to platform $(u,j+1)$
in order to compensate in such a way that the whole travel time ($t=w+r$) remains stable.

Consider the following run time law.
\begin{align}
  & r_{(u,j)}^k = \max\left\{\tilde{r}_{(u,j)} - x_{(u,j)}\Delta h_{(u,j)}^k, \underline{r}_{(u,j)}\right\}. \label{eq-8}
\end{align}
where $\tilde{r}_{(u,j)}$ denotes the nominal run times on segment $j$ of part $u$ including a possible run time margin,
and $\Delta h^k_{(u,j)} := h^k_{(u,j)} - \underline{h}_{(u,j)}$. 
The model~(\ref{eq-8}) gives the run time as the maximum between a given minimum run time $\underline{r}_{(u,j)}$
and a term that subtracts $x_{(u,j)} \left( h_{(u,j)}^k - \underline{h}_{(u,j)} \right)$ from the nominal run time.
The term $x_{(u,j)} \left( h_{(u,j)}^k - \underline{h}_{(u,j)} \right)$ expresses a deviation of the boarding and alighting time, 
due to a deviation of the train time-headway.
Notice here that the term $x_{(u,j)} h^k_{(u,j)}$, appearing in the dwell time law~(\ref{eq-7}) with a sign ``$+$'',
appears in the run time law~(\ref{eq-8}) with a sign ``$-$''.

Combining the dwell time law~(\ref{eq-7}) with the run time law~(\ref{eq-8}), gives the following new travel time law for lines with a junction.
\begin{equation}\label{eq-travel-new}
   t_{(u,j)}^k(x_{u,j)}) = r_{(u,j)}^k(h_{(u,j)}^k,x_{(u,j)}) + w_{(u,j)}^k(h_{(u,j)}^k,x_{(u,j)}).
\end{equation}

\subsection{Demand-dependent Train Dynamics}
\subsubsection{Dynamics out of the Junction}
Consider the following notations.

\begin{tabular}{ll}
  $\Delta r_{(u,j)}$ & $:= \tilde{r}_{(u,j)} - \underline{r}_{(u,j)}$. \\
  $\Delta w_{(u,j)}$ & $:= \bar{w}_{(u,j)} - \underline{w}_{(u,j)}$. \\
  $\Delta g_{(u,j)}$ & $:= \bar{g}_{(u,j)} - \underline{g}_{(u,j)}$.\\
    $\Delta h_{(u,j)}$ & $:= \bar{h}_{(u,j)} - \underline{h}_{(u,j)}$.
\end{tabular}

It is then easy to check the following.
$$\Delta w_{(u,j)} = x_{(u,j)} \Delta h_{(u,j)} = X_{(u,j)} \Delta g_{(u,j)}, \forall u,j.$$
Notice here that $\Delta r_{(u,j)}$ is a margin on the run time in order
to be able to recover disturbances. Minimum train run times $\underline{r}_{(u,j)}$ are determined 
by taking into account all the characteristics of the infrastructure, in particular, the maximum train speed
on every segment.

From Theorem~\ref{th1-acc18} it can directly be derived that if
$$h^1_{(u,j)} \leq \bar{h}_{(u,j)} = 1/(1-x_{(u,j)}) \bar{g}_{(u,j)}, \forall u,j$$
and if $$\Delta r_{(u,j)} \geq \Delta w_{(u,j)} = X_{(u,j)} \Delta g_{(u,j)}, \forall u,j,$$
then~(\ref{eq-7}) and~(\ref{eq-8}) sum to
\begin{equation}\label{eq_t}
  t_{(u,j)}^k = r_{(u,j)}^k + w_{(u,j)}^k = \tilde{r}_{(u,j)} + X_{(u,j)} \underline{g}_{(u,j)}.
\end{equation}
With travel time equation~(\ref{eq_t}), the minimum run and dwell times in~(\ref{eq-2}) can replaced and the new constraint on the travel time is written
\begin{equation}\label{eqq-2}
    d^k_{(u,j)} \geq d^{k-b_{(u,j)}}_{(u,j-1)} + \tilde{r}_{(u,j)} + X_{(u,j)} \underline{g}_{(u,j)}.
\end{equation}
Then assuming that the $k^{th}$ departure from segment $(u,j)$ is realized as soon as the two constraints~(\ref{eqq-2}) 
and~(\ref{eq-3}) are satisfied, the train dynamics is written as follows.
\begin{equation}\label{tr_dyn}
   d^k_{(u,j)} \geq \max \left\{ \begin{array}{l}
                                d^{k-b_{(u,j)}}_{(u,j-1)} + \tilde{r}_{(u,j)} + X_{(u,j)} \underline{g}_{(u,j)}, \\
                                d^{k-\bar{b}_{(u,j+1)}}_{(u,j+1)} + \underline{s}_{(u,j+1)}.
                            \end{array} \right.
\end{equation}
Therefore, the train dynamics~(\ref{tr_dyn}) outside of the junction is max-plus linear, see~\cite{FNHL17a, SFLG18b}.

\subsubsection{Dynamics at the Divergence}
In the following, the train dynamics at the divergence of the model for a line with a junction of Chapter~\ref{1A} is adapted to include the passenger demand.
To do so, replace $\underline{t}_{(u,j)}$ 
on every segment $(u,j)$ by $t^k_{(u,j)}$ given in~(\ref{eq_t}).
Assuming that odd departures from the central part go to branch~1,
while even ones go to branch~2,
the following constraints for the train dynamics on the divergence apply.

The $k^{\text{th}}$ departures from the central part.
\begin{equation}\label{eq-d1}
  d^k_{(0,n_0)} \geq d^{k-b_{(0,n_0)}}_{(0,n_0-1)} + \tilde{r}_{(0,n_0)} + X_{(0,n_0)} \underline{g}_{(0,n_0)}, \; \forall k\geq 0, \\
\end{equation}

\begin{equation}\label{eq-d2}
  d^k_{(0,n_0)} \geq \begin{cases}
                      d^{(k+1)/2-\bar{b}_{(1,1)}}_{(1,1)} + \underline{s}_{(1,1)} & \text{for } k \text{ is odd} \\ ~~ \\
                      d^{k/2-\bar{b}_{(2,1)}}_{(2,1)} + \underline{s}_{(2,1)} & \text{for } k \text{ is even} \\
                   \end{cases}
\end{equation}

The $k^{\text{th}}$ departures from the entry of branch 1.
\begin{eqnarray}
   d^k_{(1,1)} \geq d^{(2k-1)-2b_{(1,1)}}_{(0,n_0)} + \tilde{r}_{(1,1)} + X_{(1,1)} \underline{g}_{(1,1)}, \; \forall k\geq 0, \label{eq-d3}\\
   d^k_{(1,1)} \geq d^{k-\bar{b}_{(1,2)}} _{(1,2)} + \underline{s}_{(1,2)}, \; \forall k\geq 0. \label{eq-d4}
\end{eqnarray}

The $k^{\text{th}}$ departures from the entry of branch 2.
\begin{eqnarray}
   d^k_{(2,1)} \geq d^{2k-2b_{(2,1)}}_{(0,n_0)} + \tilde{r}_{(2,1)} + X_{(2,1)} \underline{g}_{(2,1)}, \; \forall k\geq 0, \label{eq-d5}\\
   d^k_{(2,1)} \geq d^{k-\bar{b}_{(2,2)}}_{(2,2)} + \underline{s}_{(2,2)}, \; \forall k\geq 0. \label{eq-d6}
\end{eqnarray}

\subsubsection{Dynamics at the Convergence}
For the train dynamics at the convergence, assuming that odd departures at node $(0,0)$ to the central part 
correspond to trains coming from branch~1 while even ones correspond to trains coming from branch~2, the dynamics are written as follows.

The $k^{\text{th}}$ departures from the central part.
\begin{equation}\label{eq-m1}
  d^k_{(0,0)} \geq \begin{cases}
                      d^{(k+1)/2-b_{(1,n_1)}}_{(1,n_1-1)} + \tilde{r}_{(1,n_1)} + X_{(1,n_1)} \underline{g}_{(1,n_1)} & \text{odd } k \\ ~~ \\
                      d^{k/2-b_{(2,n_2)}}_{(2,n_2-1)} + \tilde{r}_{(2,n_2)} + X_{(2,n_2)} \underline{g}_{(2,n_2)} & \text{even } k \\
                   \end{cases}
\end{equation}
\begin{equation}\label{eq-m2}
  d^k_{(0,0)} \geq d^{k-\bar{b}_{(0,1)}}_{(0,1)} + \underline{s}_{(0,1)}, \; \forall k\geq 0, \\
\end{equation}

The $k^{\text{th}}$ departures from the entry of branch 1.
\begin{eqnarray}
   d^k_{(1,n_1-1)} \geq d^{k-b_{(1,n_1-1)}}_{(1,n_1-2)} + \tilde{r}_{(1,n_1-1)} + X_{(1,n_1-1)} \underline{g}_{(1,n_1-1)}, \label{eq-m3}\\
   d^k_{(1,n_1-1)} \geq d^{2k-1-2\bar{b}_{(1,n_1)}}_{(0,0)} + \underline{s}_{(1,n_1)}. \label{eq-m4}
\end{eqnarray}

The $k^{\text{th}}$ departures from the entry of branch 2.
\begin{eqnarray}
   d^k_{(2,n_2-1)} \geq d^{k-b_{(2,n_2-1)}}_{(2,n_2-2)} + \tilde{r}_{(2,n_2-1)} + X_{(2,n_2-1)} \underline{g}_{(2,n_2-1)}, 
   \label{eq-m5} \\
   d^k_{(2,n_2-1)} \geq d^{2k-2\bar{b}_{(2,n_2)}}_{(0,0)} + \underline{s}_{(2,n_2)}, \; \forall k\geq 0. \label{eq-m6}
\end{eqnarray}

Notice here that a changing of variable is necessary to transform the train dynamics
at the junction to a max-plus linear dynamics. This can be done as in Chapter~\ref{1A} without any modification.

In the following, the closed-form solutions for the asymptotic average train time-headway on the central part $h_0$ and on the branches $h_1,h_2$, as well as for the asymptotic average frequency on the central part $f_0$ and on the branches $f_1,f_2$ are derived from this traffic model.
The traffic phases of the train dynamics depending on the number of trains, the passenger travel demand and the run time margin are derived.

\section{The Steady State Train Dynamics}
The following result giving the asymptotic average train time-headway and the asymptotic average frequency on a line with a junction, with demand-dependent dwell times and controlled run times, are based on the main result of Chapter~\ref{1A}, where the asymptotic average train time-headway (Theorem~\ref{th-cdc17}) and frequency (Corollary~\ref{cor-cdc17}) on a line with a junction where trains respect minimum dwell and run times, have been presented.
The effect of the passenger demand on the train travel time (dwell + run) is resumed in~(\ref{eq_t}).
The following result is derived by replacing $\underline{t}_{(u,j)}$ in Theorem~\ref{th-cdc17} with $t^k_{(u,j)}$ given by~(\ref{eq_t}).

%

\begin{theorem}\label{thm-4}
The train dynamics admit an asymptotic regime with a unique asymptotic average growth rate 
which represents here the asymptotic average train time-headway $h_0$ on the central part of the metro line.
$$\begin{array}{ll}
        h_0 (m,\Delta m, X) & = h_1(m,\Delta m, X)/2 = h_2(m,\Delta m, X)/2 \\~~\\
                            & = \max \left\{h_{fw}, h_{min}, h_{bw}, h_{br}\right\},
    \end{array}$$
      with\footnote{fw: forward, bw: backward, min: minimum, br: branches.}

$$h_{fw} = \max \left\{ \begin{array}{l}
                           \frac{\sum_{0,1,j} (\underline{g}_{u,j} X_{(u,j)} + \tilde{r}_{(u,j)})}{m - \Delta m}, \\~~\\
                           \frac{\sum_{0,2,j} (\underline{g}_{u,j} X_{(u,j)} + \tilde{r}_{(u,j)})}{m + \Delta m}, 
                        \end{array} \right.$$
                        
$$h_{min} = \max 
  \begin{cases}
		  \max_{0,j} ((\underline{g}_{(0,j)} X_{(0,j)} + \tilde{r}_{(0,j)}) + \underline{s}_{(0,j)}), & \\~~\\
		  \max_{u,j} ((\underline{g}_{(u,j)} X_{(u,j)} + \tilde{r}_{(u,j)}) + \underline{s}_{(u,j)})/2, & \quad u \in \{1,2\}, j \neq \{n_u\},
  \end{cases} $$
  
$$h_{bw} = \max\left\{ \frac{\underline{S}_0 + \underline{S}_1}{\bar{m} - \Delta \bar{m}},
                  \frac{\underline{S}_0 + \underline{S}_2}{\bar{m} + \Delta \bar{m}} \right\},$$
  
$$h_{br} = \max \left\{ \begin{array}{l}
                \frac{\sum_{1,j} (\underline{g}_{(1,j)} X_{(1,j)} + \tilde{r}_{(1,j)}) + \underline{S}_2}{2(n_2 - \Delta m)}, \\~~\\ \frac{\underline{S}_1 + \sum_{2,j} (\underline{g}_{(2,j)} X_{(2,j)} + \tilde{r}_{(2,j)})}{2(n_1 + \Delta m)}.
            \end{array} \right.$$
\end{theorem}~~\\~~

\begin{proof}
Accordingly to proof of Theorem~\ref{th-cdc17} in Chapter~\ref{1A}.
\end{proof}

\begin{corollary}\label{coro-1}
The asymptotic average train frequency $f_0$ at the central part and $f_1$ and $f_2$ at the two branches of the metro line
are given as follows. 
$$\begin{array}{ll}
    f_0(m,\Delta m, X) & = 2f_1(m,\Delta m, X) = 2f_2(m,\Delta m, X) \\~~\\
                    & = \max\{0, \min \left\{ f_{fw}, f_{min}, f_{bw}, f_{br} \right\}\},
        \end{array} $$
        
$$f_{fw} = \min \begin{cases}
		    \frac{m - \Delta m}{\sum_{0,1,j} (\underline{g}_{u,j} X_{(u,j)} + \tilde{r}_{(u,j)})}, & \\~~\\
		    \frac{m + \Delta m}{\sum_{0,2,j} (\underline{g}_{u,j} X_{(u,j)} + \tilde{r}_{(u,j)})},
		\end{cases} $$

$$f_{max} = \min  \begin{cases}
		     \min_{0,j} 1/(\underline{g}_{(0,j)} X_{(0,j)} + \tilde{r}_{(0,j)} + \underline{s}_{(0,j)}), & \\~~\\
		     \min_{u,j} 1/(\underline{g}_{(u,j)} X_{(u,j)} + \tilde{r}_{(u,j)} + \underline{s}_{(u,j)})/2, & \quad \forall u \in \{1,2\}, j \neq \{n_u\},
		  \end{cases} $$

$$f_{bw} = \min \left\{\frac{\bar{m} - \Delta \bar{m}}{\underline{S}_0 + \underline{S}_1}, \frac{\bar{m} + \Delta \bar{m}}{\underline{S}_0 + \underline{S}_2} \right\}, $$
  
$$f_{br} = \min \begin{cases}                 
                    \frac{2(n_2 - \Delta m)}{\sum_{1,j} (\underline{g}_{(1,j)} X_{(1,j)} + \tilde{r}_{(1,j)}) + \underline{S}_2}, & \\~~\\
                    \frac{2(n_1 + \Delta m)}{\underline{S}_1 + \sum_{2,j} (\underline{g}_{(2,j)} X_{(2,j)} + \tilde{r}_{(2,j)})}. &
                \end{cases} $$
\end{corollary}

\begin{proof}
Directly from Theorem~\ref{thm-4}, with $f = 1/h$.
\end{proof}

Figure~\ref{cdc18-1} is a schematic graph, depicting the average frequency on the central part, over the total number of trains $m$ and the difference between the number of trains on the branches $\Delta m = m_2-m_1$, for three passenger travel demand levels, accordingly to Corollary~\ref{coro-1}.
It becomes clear, that for a given passenger travel demand, the resulting form of the graph is the one having been derived from the model with fixed train run and dwell times, see Corollary~\ref{cor-cdc17} in Chapter~\ref{1A} above.
Precisely, the eight traffic phases of the train dynamics can be identified. There are two free flow phases, two congested branches phases, two congestion phases, one maximum frequency and one zero-flow phase.
Three levels of passenger travel demand are depicted.
\begin{itemize}
\item Firstly, demand level $a1$, which serves as reference.
\item Secondly, demand level $a2$ which represents a symmetric increase of the passenger arrival rates to and passenger departure rates from all platforms.
\item And last, the demand level $a3$ where the passenger demand increases only on the central part and on branch 2.
\end{itemize}

It can be seen that with regard to the reference scenario $a1$, if the demand level increases similarly over all platforms (scenario $a2$), the capacity of the system (GLKJIH plane in Figure~\ref{cdc18-1}) decreases. By consequent, the optimal number of trains decreases, too. However, congestion only occurs from a higher number of trains on (point~J), which is a direct consequence of the reduced maximum frequency.

Furthermore, if considering a scenario where the passenger demand increases asymmetrically only on the central part and on branch 2 ($a3$), resulting travel times on branch~2 are longer than travel times on branch~1, which causes that the optimal difference between the number of trains on branch~2 and branch~1 $m_2-m_1$ changes. By consequent, the optimal operating point (point G) maximizing the average frequency for a minimum number of trains, diverts towards larger values in $m$ and a more important difference between the number of trains on the branches $m_2-m_1$.

To sum up, as for the traffic phases of the train dynamics on a metro line with a junction and with constant run and dwell times presented in Corollary~\ref{cor-cdc17}, the average frequency depends on the number of trains $m$ and on the difference between the number of trains on branch~2 and branch~1, $\Delta m$.
Furthermore, there exists an optimal difference $\Delta m^{*}$ for every $m$ accordingly to Theorem~\ref{th-itsc}.
In the following, the traffic phases of the train dynamics on a line with a junction with demand-dependent dwell times and controlled run times will be described.

\begin{figure}[h]
    \centering
    \fbox{\includegraphics[width=\textwidth]{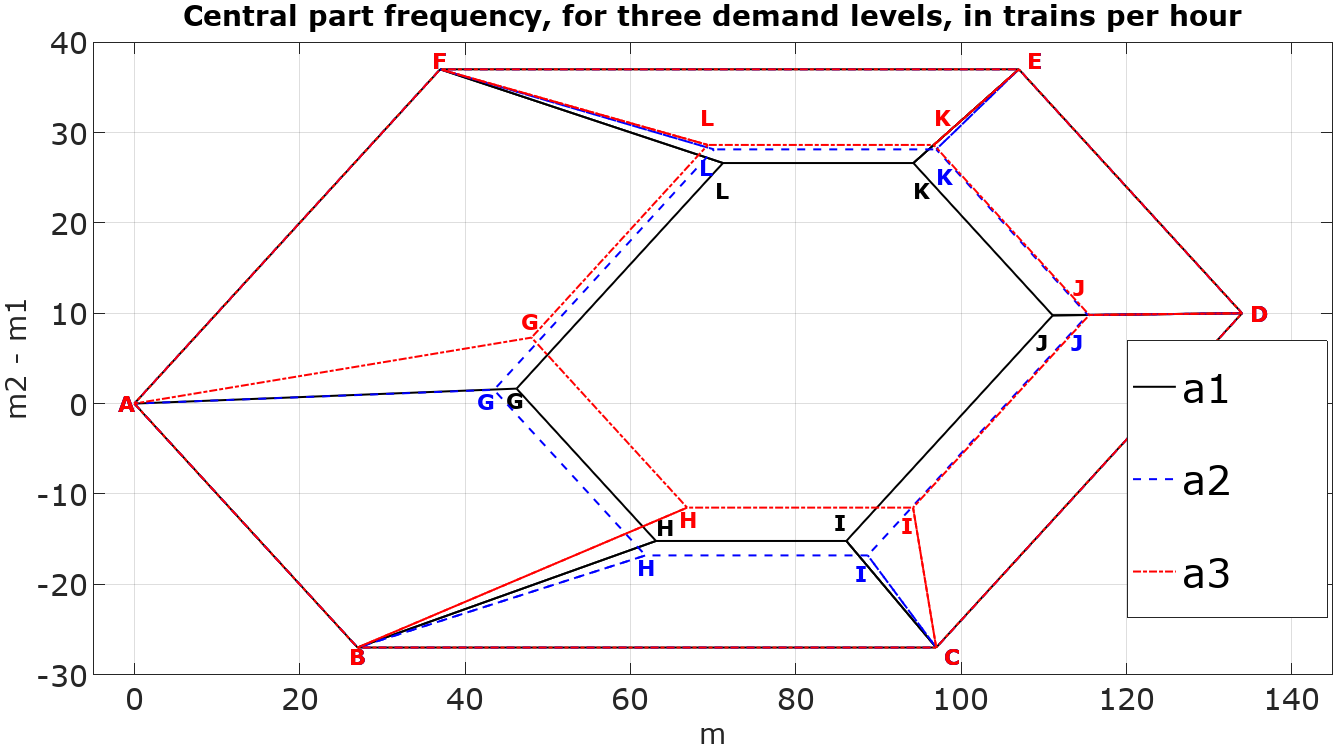}}
    \caption[Schematic illustration of Corollary~\ref{coro-1}. Average central part frequency over number of trains $m$, the difference on the branches $m_2-m_1$, for three demand levels.]{Schematic illustration of Corollary~\ref{coro-1}. Average central part frequency over number of trains $m$, the difference on the branches $m_2-m_1$, for three demand levels: reference scenario $a1$, a symmetric demand increase $a2$, an asymmetric demand increase $a3$.}
    \label{cdc18-1}
\end{figure}

\section{{The Traffic Phases}}
The Figures~\ref{cdc-dm-demand},~\ref{cdc-dm-margin}~\ref{cdc-f-demand},~\ref{cdc-f-demand-contourf},~\ref{cdc-f-margin}~and~\ref{cdc-f-margin-contourf} in the following illustrate Theorem~\ref{thm-4} and Corolllary~\ref{coro-1} from above. They depict the traffic phases of the train dynamics on a metro line with a junction with demand-dependent control. The figures in the following represent the traffic phases of metro line 13, Paris. This line is composed of a central part and two branches, connected with a junction. Furthermore, the traffic phases of metro line 13 are depicted with the asymptotic average train frequency on the central part $f_0$. Note that accordingly to Corollary~\ref{coro-1}, the frequencies on the branches can be obtained with $f_1 = f_2 = f_0/2$.

The infrastructure parameters $\underline{g}_{(u,i)}$ and $\underline{s}_{(u,j)}$ in Corollary~\ref{coro-1} are per definition the minimum dynamic interval, this is the minimal time during which there is no train on a platform, and the safe separation time, the minimum time to be respected between the moment a first train clears a segment and a second train enters the same segment. Both are infrastructure constraints, the parameters have been provided by the operator RATP.
Moreover, the passenger travel demand is modeled by the passenger arrival flows to and the passenger departure flows from the platforms $\lambda^{in}_{(u,j)}, \lambda^{out}_{(u,j)}$. Average passenger flows over time intervals of 30 min have been used here, accordingly to data provided by RATP. The passenger flows can be updated at any time interval. Finally, an average value of the boarding and alighting capacity $\alpha^{in}_{(u,j)}, \alpha^{out}_{(u,j)}$, that is the number of passengers who can board, respectively alight from a train per time interval, has been provided by the operator RATP. From these inputs, the passenger demand parameter is calculated
$$X_{(u,j)} = x_{(u,j)}/(1 - x_{(u,j)}),$$
with
$$x_{(u,j)} = \lambda^{in}_{(u,j)}/\alpha^{in}_{(u,j)} + \lambda^{out}_{(u,j)} / \alpha^{out}_{(u,j)}.$$

The nominal train run times are the sum of a minimum run time and a run time margin which allows to realize the demand-dependent dwell times:
$$\tilde{r}_{(u,j)} = \underline{r}_{(u,j)} + \Delta r_{(u,j)}.$$
The minimum run times for metro line 13, Paris, have been simulated at the MATYS entity of RATP. The simulation takes into account the rolling stock, for example traction, maximum speed and braking characteristics. Moreover, infrastructure constraints are considered, such as the gradient and curve radius.
In addition to this minimum run time, a run time margin can be applied which increases the nominal run times and reduces the nominal train speed. In case of a perturbation on the train time-headways, dwell times can then be extended in the limit of the run time margin to take into account the accumulation of passengers on the platform:
$$w^{k}_{(u,j)}(h^{k}_{(u,j)},x_{(u,j)}) = \min(x_{(u,j)}h^{k}_{(u,j)},\bar{w}_{(u,j)}).$$

The train which has realized a longer dwell time, can then be accelerated within the run time margin in the following inter-station, respecting the maximum speed, that is the minimum run time:
$$r^{k}_{(u,j)}(h^k_{(u,j)},x_{(u,j)}) = \max\{\underline{r}_{(u,j)}, \tilde{r}_{(u,j)} - x_{(u,j)}(h^k_{(u,j)} - \underline{h}_{(u,j)})\}.$$
This combined traffic control allows to realize demand-dependent dwell times while guaranteeing stability of the traffic within the run time margin.
Note that $n_{u}$ is the number of segments on each part of the line.

\subsection{The Number of Trains on the Branches}
\begin{figure}[h]
 	\centering
 	\frame{
  	\includegraphics[width=\textwidth]{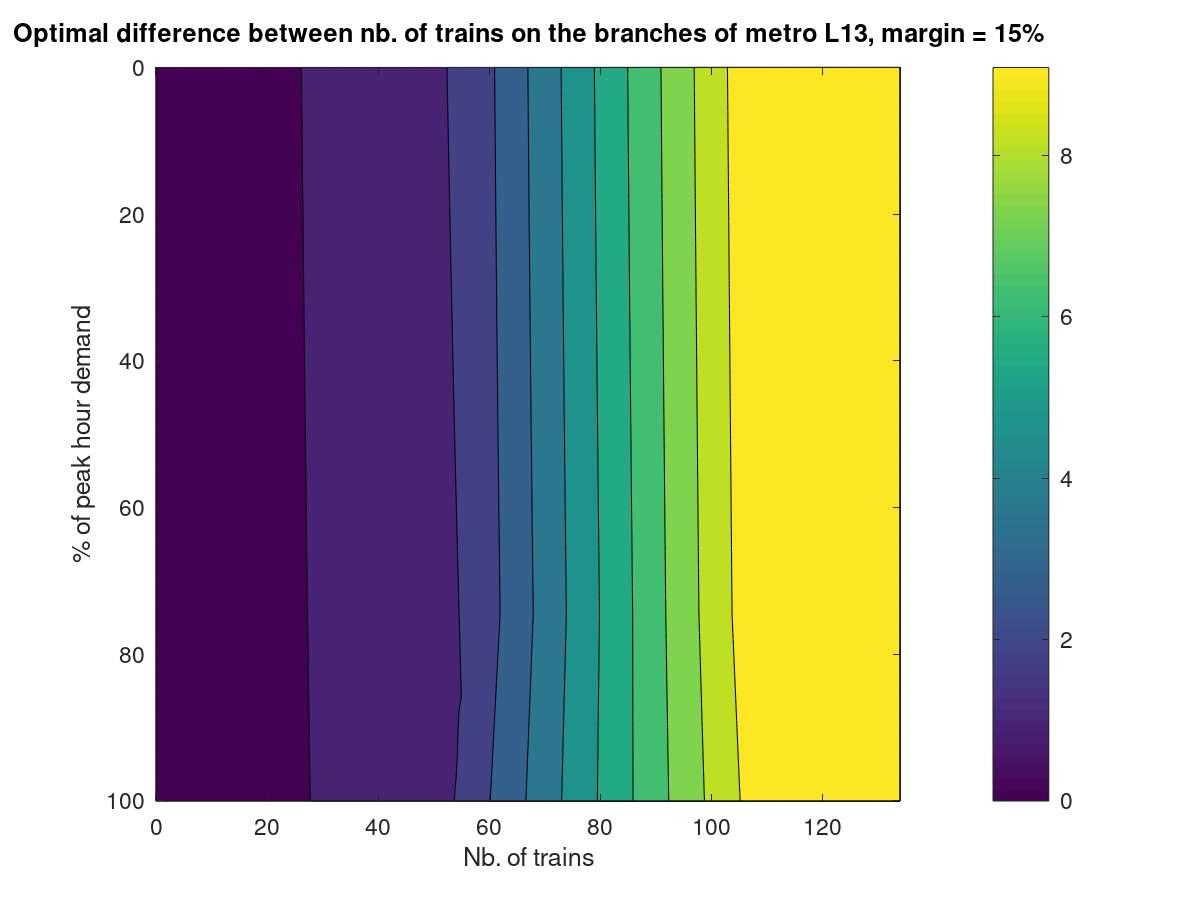}}
	\caption {Optimal $\Delta m^{*}$, over $m$ and passenger demand, margin of 15\%.}
	\label{cdc-dm-demand}
\end{figure}

\begin{figure}[h]
 \centering
 	\frame{
  	\includegraphics[width=\textwidth]{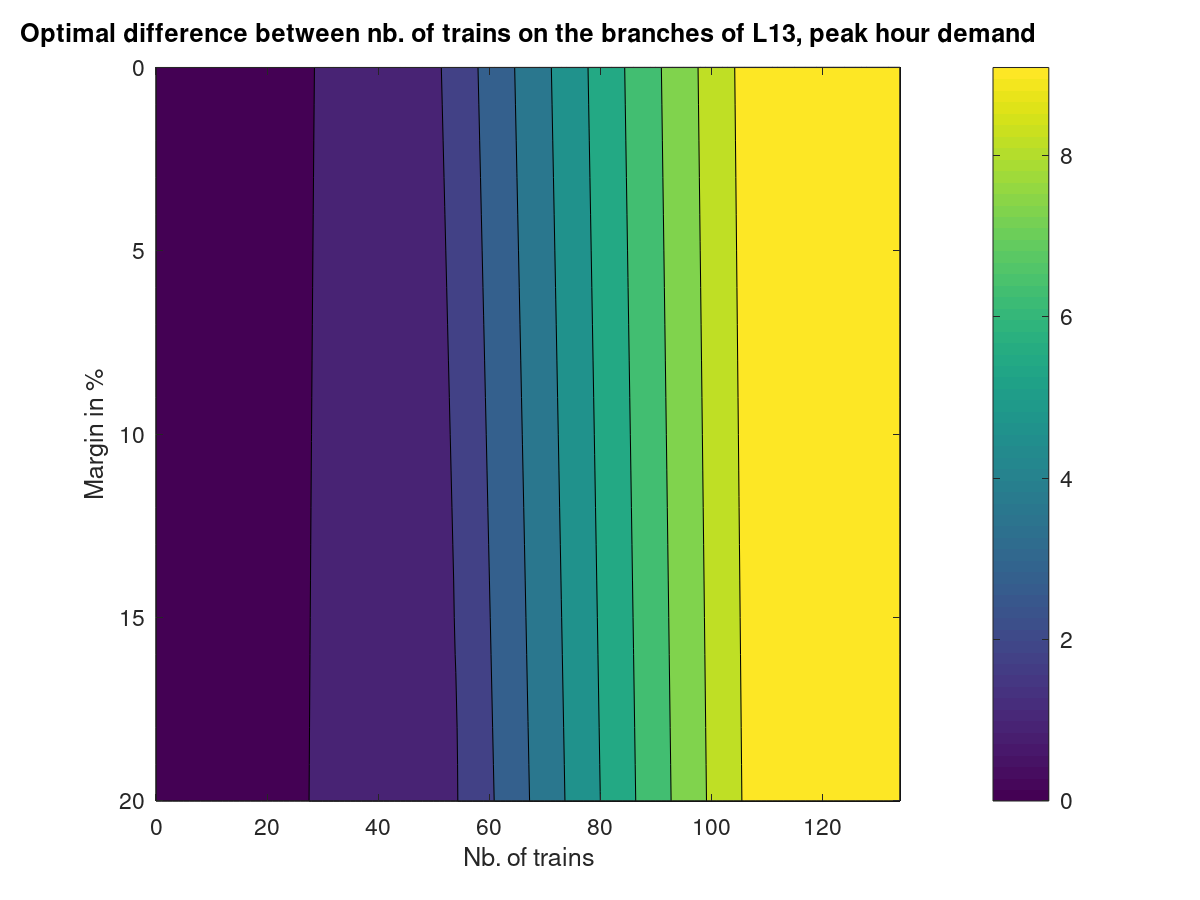}}
	\caption {Optimal $\Delta m^{*}$, over $m$ and run time margin, peak hour demand.}
	\label{cdc-dm-margin}
\end{figure}

In Figures~\ref{cdc-f-demand},~\ref{cdc-f-demand-contourf},~\ref{cdc-f-margin}~and~\ref{cdc-f-margin-contourf} three traffic phases of the train dynamics can be distinguished, for an optimal difference between the number of trains on the branches $\Delta m{*}$, shown in Figures~\ref{cdc-dm-demand}~and~\ref{cdc-dm-margin}. The optimal difference $\Delta m^{*}$ can be derived accordingly to Chapter~\ref{1A}, Section~\ref{transactions} for the traffic phases on a line with a junction with minimum run and dwell times.
With regard to Figure~\ref{cdc18-1} the optimal difference are all the values for $\Delta m$ along a line AG and JD.
\begin{equation}\label{opt-deltam}
\Delta m^{*}(m) =
\begin{cases}
\frac{\Delta T}{2 T}m \text{ if } m \leq T f_{max}\\ ~~ \\
\Delta n + \frac{\Delta \underline{S}}{2 \underline{S}}(m-N) \text{ if } m \geq N - \underline{S} f_{max},
\end{cases}
\end{equation}

where

$$T = (2T_0 + T_1 + T_2)/2$$

and

$$\underline{S} = (2\underline{S}_0 + \underline{S}_1 + \underline{S}_2)/2.$$

Note that for an optimal $\Delta m^{*}$, the diagram depicting the traffic phases of a metro line with a junction with demand-dependent dwell times and controlled run times, see Figures~\ref{cdc-f-demand},~\ref{cdc-f-demand-contourf},~\ref{cdc-f-margin}~and~\ref{cdc-f-margin-contourf}, has the same shape as the one of the traffic phases of a linear line with demand-dependent dwell times and controlled run times, derived in Section~\ref{phases-acc18}.
More precisely, the asymptotic average train frequency on a line with a junction and demand-dependent dwell times and controlled run times depends on the following variables:
\begin{itemize}
\item the number of trains $m$,
\item the passenger travel demand and
\item the run time margin.
\end{itemize}

Figure~\ref{cdc-dm-demand} depicts the optimal difference between the number of trains on the branches $\Delta m^{*}$ accordingly to formula~(\ref{opt-deltam}) for metro line 13, Paris, for a fixed run time margin of 15\% over the number of trains $m$ and the passenger travel demand.
With an increasing number of trains the optimal difference between the number of trains on the branches increases, too. Note that at the intersection between free flow phase and maximum frequency phase at around $m=50$ trains, the optimal difference is $\Delta m = 2$. From here to the intersection between maximum frequency phase and congestion phase at around $m=110$ trains, the optimal difference is $\Delta m = 10$. The effect of the passenger travel demand on the optimal difference is less important. With a symmetrically over all platforms increasing passenger demand, the difference between the travel time on the branches - and consequently the optimal difference - does not change considerably.

 Figure~\ref{cdc-dm-margin} depicts the optimal difference accordingly to formula~(\ref{opt-deltam}) for metro line 13, Paris, over the number of trains and the run time for margin, for a peak hour passenger demand.
 Again, it can be seen that with an increasing number of trains, $\Delta m^{*}$ increases. As before, at the intersection between free flow phase and maximum frequency phase at around $m=50$ trains, $\Delta m^{*} = 2$, and at the intersection between maximum frequency phase and congestion phase at around $m=110$ trains, $\Delta m^{*}=10$. The effect of the run time margin is minor.
Especially with regard to real metro operations and the fact the optimal difference has to be rounded to integer values to be applied, it can be concluded that a symmetric prolongation of travel times does not have an impact on $\Delta m^{*}$.
In the following the effect of the passenger demand and the run time margin on the traffic phases for an optimal $\Delta m$ accordingly to equation~(\ref{opt-deltam}), corresponding to the line AGJD in Figure~\ref{cdc18-1} is derived.

Note, that in the maximum frequency phase, the asymptotic average frequency derived in Theorem~\ref{thm-4} is constant and consequently independent of $m, \Delta m$ within the GLKJIH plane, see Figure~\ref{cdc18-1}.

\subsection{The Effect of the Passenger Travel Demand}
\subsubsection{Free Flow Phase}
\begin{figure}[h]
 	\centering
 	\frame{
  	\includegraphics[width=\textwidth]{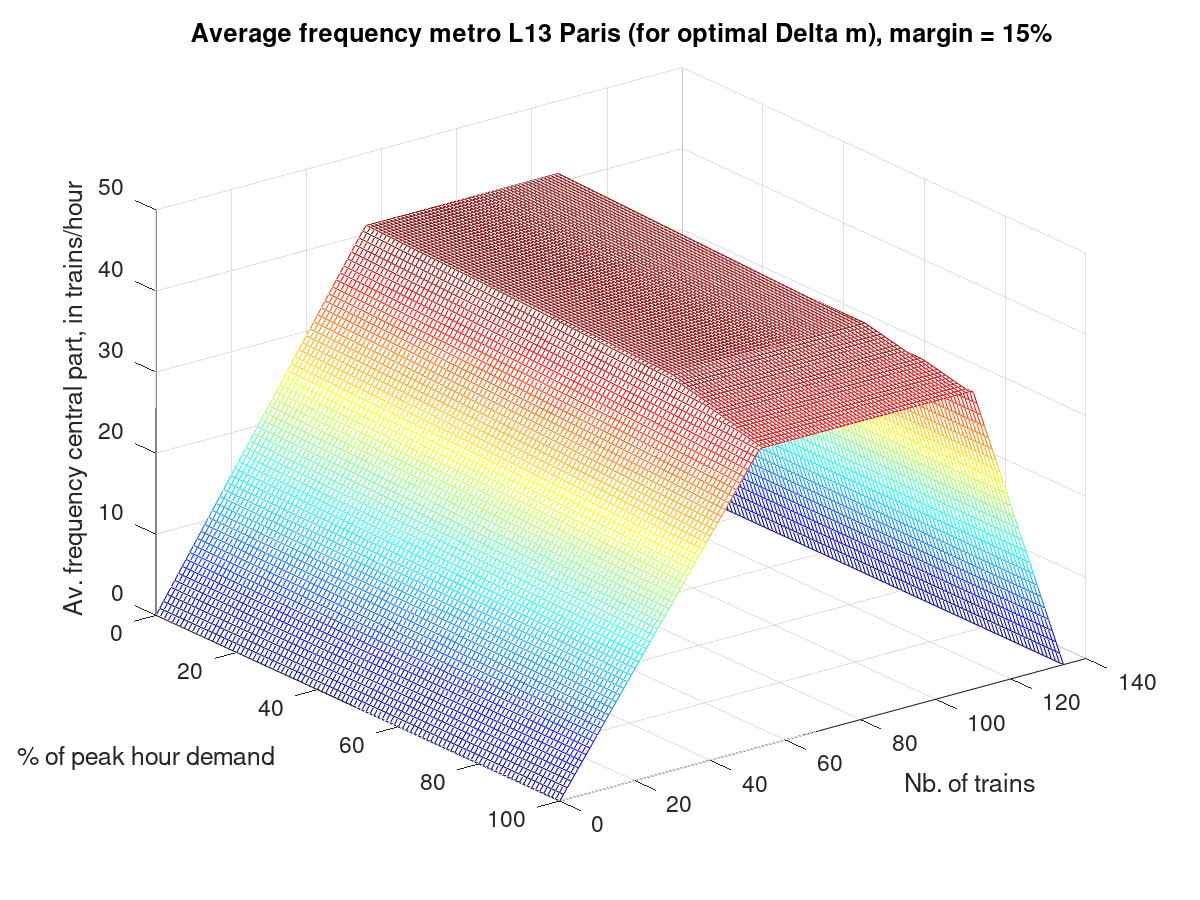} }
	\caption {Average frequency for optimal $\Delta m^{*}$, over $m$ and demand, fixed margin (view 1).}
	\label{cdc-f-demand}
\end{figure}

\begin{figure}[h]
    \centering
    \fbox{\includegraphics[width=\textwidth]{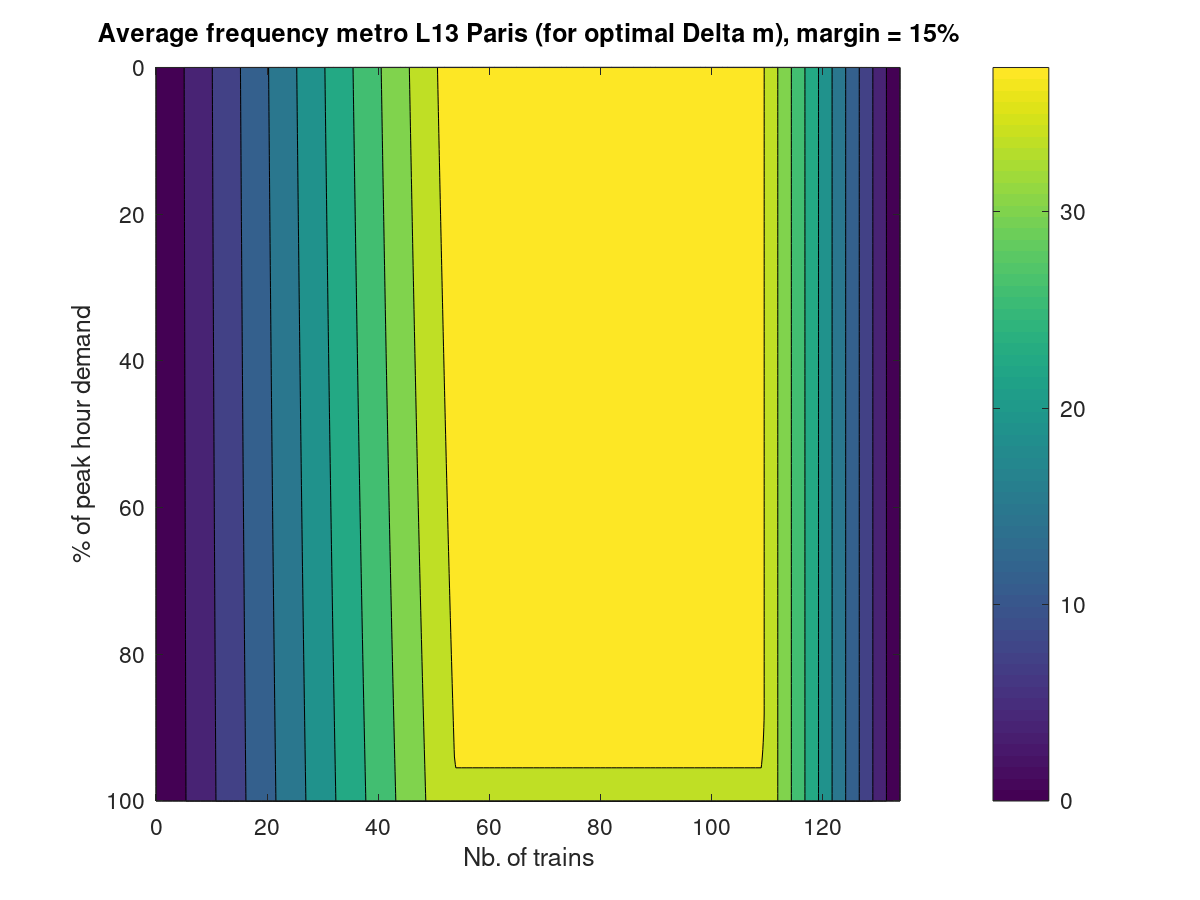}}
    \caption{Average frequency for optimal $\Delta m^{*}$, over number of trains and passenger demand, for a fixed run time margin (view 2).}
    \label{cdc-f-demand-contourf}
\end{figure}

For an optimal $\Delta m^{*}$, Figure~\ref{cdc-f-demand} depicts the three traffic phases. First of all, there is a free flow phase, where the asymptotic average frequency increases with the number of trains for a fixed passenger demand level.
Refer to Figure~\ref{cdc-f-demand-contourf} to study the influence of the passenger demand. For a fixed number of trains, the asymptotic average frequency decreases with an increasing passenger travel demand. This is due to a prolongation of train dwell and travel times. If one aims to realize a certain frequency level and dwell times have to extended due to an increasing passenger demand, this might effect the number of trains required. For example, in Figure~\ref{cdc-f-demand-contourf}, if the passenger demand is at $40 \%$ of peak hour, approximately $m=40$ trains are needed to realize a frequency $f_0=30$ trains per hour. If the demand increases to peak hour level, around $m=48$ trains are necessary to ensure the same average frequency. Note that a change in the total number of trains might also require a change in $\Delta m^{*}$. In this example, if the number of trains change to $m= 40 + 8 = 48$, the optimal difference could be changed by one to $\Delta m^{*} = 1+1 = 2$, see figure~\ref{cdc-dm-demand}.
Finally, the optimal number operating point, where the frequency is maximized for a minimal number of trains can be found. For the case of metro line 13, Figure~\ref{cdc-f-demand-contourf} shows that the optimal number of trains is at around $m=50$ for minimum passenger demand and increases to around $m=55$ for passenger peak hour demand level.

\subsubsection{Maximum Frequency Phase}

In the maximum frequency phase, the capacity of the system is reached. The maximum frequency is given by the segment on which the sum of dwell time, run time and safe separation time realizes the maximum over the entire line. Note, that due to the one-over-two operations rule at the junction, the corresponding sums on the branches have to be divided by factor $2$.
Figure~\ref{cdc-f-demand} shows that the maximum frequency depends on the passenger travel demand (for a fixed run time margin). This is a direct consequence of the demand-dependent dwell times, which increase with the passenger volume. Furthermore, at around 80\% of peak hour demand, a discontinuity in the gradient can be identified. Here, the bottleneck of the system switches from the segment of the station \textit{Gaîté}, direction northbound, to the segment of the station \textit{Montparnasse -- Bienvenüe}, direction northbound. This can be interpreted as follows: for a low passenger demand, dwell times are relatively short. Then, the most constraining segment of the line is one on which the run time and the safe separation time are relatively long, this is the case for the station \textit{Gaîté}. However, with increasing passenger demand, dwell times get considerably longer on platforms with a high passenger volume. This is the case for the interchange station \textit{Montparnasse -- Bienvenüe}. Even though around this station block distances are optimized such that run times and safe separation are minimal, this platform is the bottleneck of the line due to important dwell times during peak hour.

From Figure~\ref{cdc-f-demand-contourf} it can be seen that the capacity phase is bordered by some straight lines at the intersection with the free flow phase and the congestion phase. The intersection with the free flow phase represents the optimal number of trains, since it is the one maximizing the frequency at minimum cost.
This number depends on the passenger travel demand. For an increasing demand, the capacity decreases, whereas the travel times on the line increase due to longer dwell times. As a consequent, not only capacity does decrease, but a higher number of trains is necessary to attain the reduced capacity.
From the intersection between maximum frequency phase and congestion phase on, congestion occurs. 
Congestion always occurs from the same number of trains on, dependent on the infrastructure of the line, but independent of the passenger travel demand. Note, that the intersection line is parallel to the demand-axis. For the case of metro line 13, this is at around $m=110$ trains.


\subsubsection{Congestion Phase}

Figure~\ref{cdc-f-demand-contourf} shows that in the congestion phase, the average frequency decreases with an increasing number of trains, but it is independent of the passenger travel. Note that lines of same frequency are parallel to the demand-axis.

\subsection{The Effect of the Run Time Margin}
\subsubsection{Free Flow Phase}
\begin{figure}[h]
 	\centering
 	\frame{
  	\includegraphics[width=\textwidth]{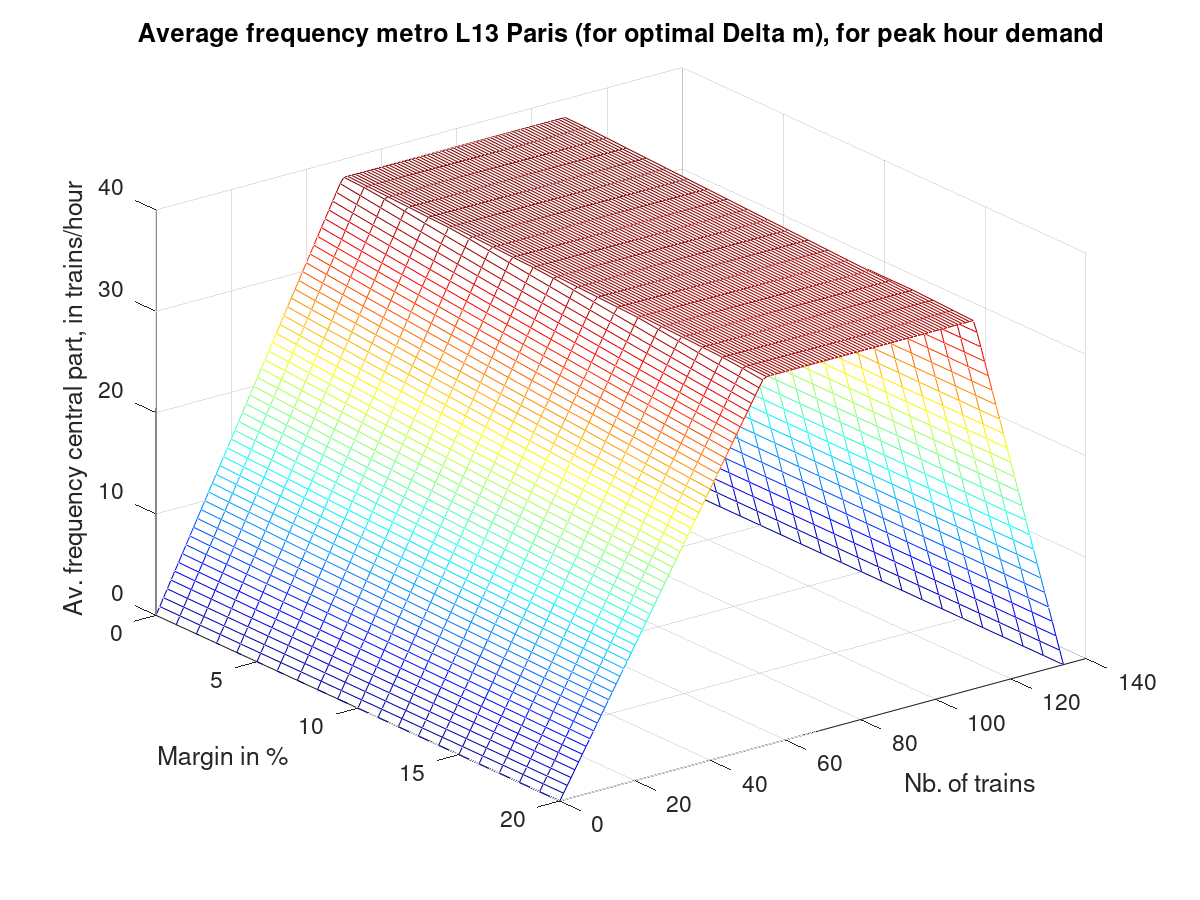} }
	\caption {Traffic phases for optimal $\Delta m^{*}$ over $m$ and run time margin for peak hour demand (view 1).}
	\label{cdc-f-margin}
\end{figure}

Again for an optimal $\Delta m^{*}$, Figure~\ref{cdc-f-margin} depicts the three traffic phases for peak hour demand level. This allows to study the influence of the run time margin on the traffic phases.
In the free flow phase, the asymptotic average frequency increases with the number of trains. As it can be be seen in Figure~\ref{cdc-f-margin-contourf}, for a fixed number of trains, an increasing run time margin has a negative impact on frequency. For example, to realize an average frequency of $f_0 = 30$ trains per hour with a margin of $5\%$, approximately $m=42$ trains are necessary. If it is chosen to increase the run time margin to $20\%$ for a better robustness of the timetable, around $m= 42+5=47$ trains are required to realize the same average frequency. Increasing the number of trains might also lead to change in the optimal difference $\Delta m^{*}$.

\subsubsection{Maximum Frequency Phase}
\begin{figure}[h]
 	\centering
 	\frame{
  	\includegraphics[width=\textwidth]{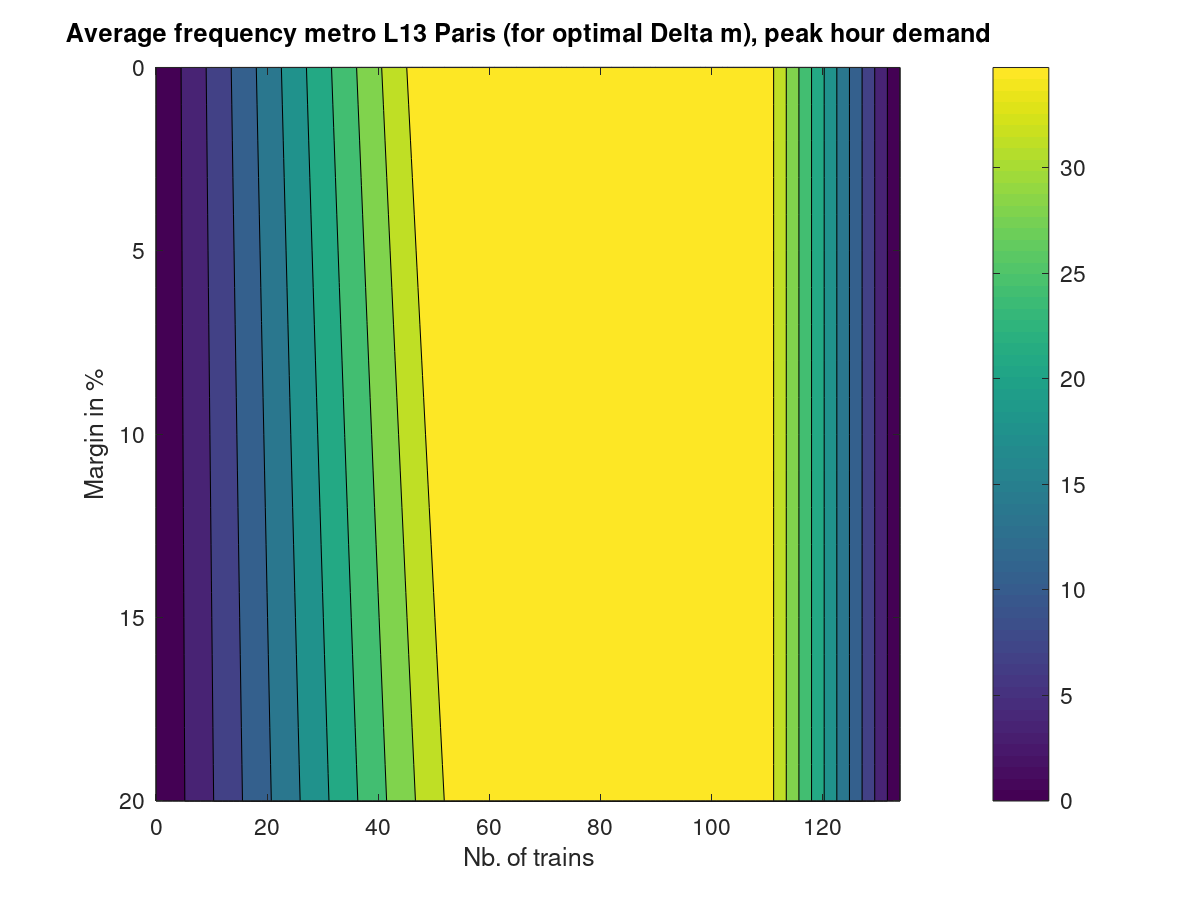} }
	\caption {Traffic phases for optimal $\Delta m^{*}$ over $m$ and run time margin for peak hour demand (view 2).}
	\label{cdc-f-margin-contourf}
\end{figure}

With regard to Figure~\ref{cdc-f-margin}, it can be seen that an increasing run time margin has a restraining effect on the maximum frequency. This is due to the longer travel times an all segments which will lead to longer minimum train time-headways on the line. Most importantly, if one aims to include higher margins to increase the robustness of a timetable, the limiting effect on the capacity has to be considered. This is relevant for mass transit systems with a high passenger demand which are driven on the capacity limit. Moreover, as it can be seen in Figure~\ref{cdc-f-margin-contourf}, the intersection line between the free flow phase and maximum frequency phase it not parallel to the margin-axis. This means, that the total number of trains required to reach the optimal operating point which maximizes the frequency for a minimal number of trains, increases with the run time margin.

\subsubsection{Congestion Phase}
Figure~\ref{cdc-f-margin-contourf} shows that from around $m=110$ trains on, congestion occurs on metro line~13. Here, trains start to interact with each other an the average frequency decreases with the number of trains. As it is the case for the passenger travel demand, the average frequency in the congestion phase is independent of the run time margin. Refer also to Theorem~\ref{thm-4}, which shows that in this phase, the frequency depends only on the safe separation times.

Throughout this thesis, discrete event traffic models for metro lines with a junction have been been developed.
In Chapter~\ref{1A}, the traffic phases of the train dynamics on a line with a junction considering minimum train dwell and run times, have been derived. It has been shown, that the asymptotic average train time-headway on the line depends on train dwell, run and safe separation times and on the number of trains and on the difference of the number of trains on the branches.
In this model, the effect of the passenger demand on the train dynamics is not modeled. The minimum dwell and run times can be seen as the nominal timetable values and can include a margin to recover small perturbations.
In Section~\ref{mac-control}, laws for macroscopic traffic control have been presented.
This macroscopic control is interesting for lines with a junction, where it consists in controlling the number of trains and the difference between the number of trains on the branches.
For example, in case of a changing passenger travel demand or a perturbation which exceeds the included margins, macroscopic control on the number of trains on the central part and on the branches can be applied. In feedback of the demand, respectively of the observed train travel times, the optimal set point of the system can calculated.

In Chapter~\ref{Chap-4} the model is extended by a microscopic control of the train dwell and run times. The minimum dwell times are replaced by a function of the average passenger travel arrival rate to and departure rate from the platforms and the train time-headway, bounded by a maximum train dwell time. To guarantee the stability of the dynamics, minimum train run times are replaced by a function canceling an extension of the dwell times in case of a long headway, bounded by a minimum run time. This model has first been developed for linear metro lines, see Section~\ref{acc}, and then extended for lines with a junction, see~\ref{cdc18}. The traffic phases of the train dynamics have been derived and the effect of new parameters, the margin on the run time and the passenger travel demand on the asymptotic average frequency, have been studied.

\chapter{Simulation of Feedback Traffic Control}\label{sim} 

\begin{quote}{
In this chapter, three simulation cases, applied to metro line 13, Paris, are presented.
They illustrate the traffic control proposed in the preceding chapters.
In the first section, the macroscopic control law presented in~\ref{control-1} for the number of trains on a line with a junction in feedback of a changing passenger travel demand is illustrated.
In the second section, the macroscopic control law presented in~\ref{control-3} for the number of trains on the branches in case of a perturbation on the train dwell and/or run times on the branches is illustrated. Therefore, the closed-loop train dynamics on a line with a junction are simulated for the case of a perturbation of the train travel times on one of the branches. A scenario without control is compared to a scenario with control. 
In the third section, the closed-loop train dynamics with demand-dependent dwell times and controlled run times as presented in~\ref{acc} are simulated. It is shown that an additional control on the train dwell times leads to an harmonization of the train time-headways on the line.}
\end{quote}

\section{Macroscopic Control of the Number of Trains in Feedback of the Passenger Travel Demand Volume}\label{mac-control1}
\sectionmark{Responding to a Changing Travel Demand Volume}
\begin{figure}
 \centering
    \frame{\includegraphics[width=\textwidth]{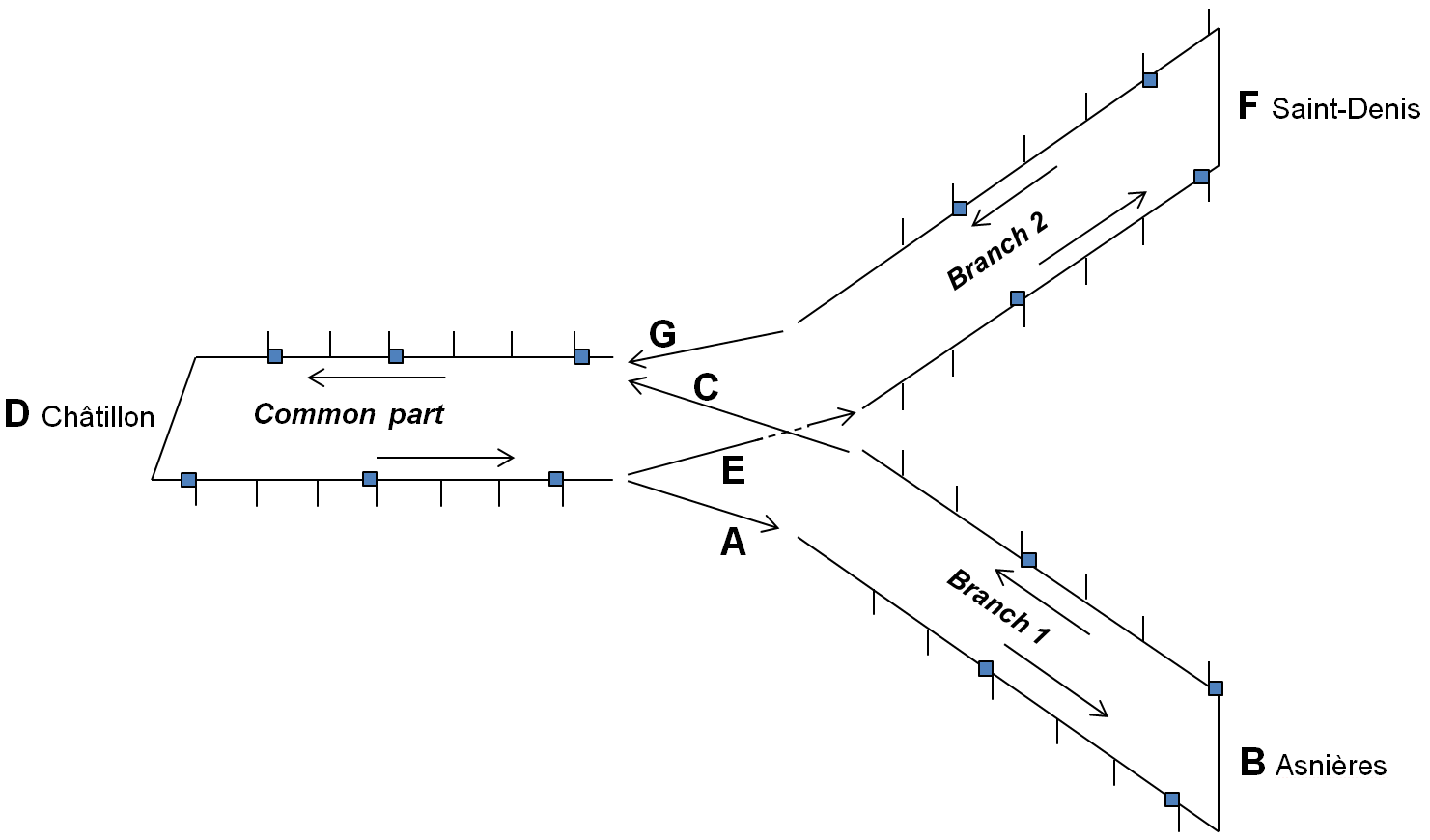}}
    \caption[Schema of metro line 13, Paris.]{Schema of metro line 13, Paris, illustrating the positions of the points A,B,C,D,E,F and G of the $y$-axis of Figure~\ref{fig_perturb_without} and Figure~\ref{fig_perturb_with} on the metro line.}
    \label{fig_schema}
\end{figure}
%
%

Consider Paris metro line 13 with a junction as presented in Figure~\ref{fig_schema}.
Figure~\ref{fig_demand} depicts the maximum passenger charge on the central part and on the branches,
for time intervals of 30 minutes during morning peak hour.
Passenger demand data has been provided by RATP.
More precisely, the operator disposes of the average passenger arrival rates to and the departure rates from the platforms per time interval of 30 minutes. By consequent, for this time interval, the passenger charge, that is the number of passenger to be transported per inter-station, can be calculated. For the central part and each branch, evolution of the number of passengers in the inter-station with the most important passenger demand is depicted in Figure~\ref{fig_demand}. It can be seen, that the passenger charge on the branches is approximately half the one on the central part and balanced over both branches. The demand peak is reached between 8h30 and 9h00 in the morning.
As the passenger travel demand level on the two branches is nearly the same, see Figure~\ref{fig_demand}, RATP has chosen an one-over-two operation of the junction.

The aim here is to control $m, \Delta m$ in feedback of a new set point $f_0$ due to a changing passenger travel demand volume.
Let $c_{(u,j)}$ be the passenger charge on board on a segment $(u,j)$.
RATP calculates the train capacity with a passenger load factor of 4 passengers per $m^2$.
The capacity of one train of metro line 13 is thus $\kappa = 584$ passengers.
The required maximum train time-headway $h^{\text{req}}_0$ (central part), with respect to the passenger demand is calculated as follows.
\begin{align}\label{f/rho}
  h^{\text{req}}_0 = \min \left\{ \min_j \frac{\kappa}{c_{(0,j)}}, \min_j \frac{\kappa}{2 c_{(1,j)}}, \min_j\frac{\kappa}{2 c_{(2,j)}} \right\}.
\end{align}
The feasible maximum train time-headway $h^{\text{fea}}_0$ (central part), with respect to the passenger demand is calculated as follows.
\begin{equation} \label{h_fea}
  h^{\text{fea}}_0 = \max (h^{\text{req}}_0, h_{\min}).
\end{equation}
The feasible headway $h^{\text{fea}}_0$ on metro line 13 (central part) with respect to the passenger demand, is depicted in Figure~\ref{fig_headway}.
Furthermore, Figure~\ref{fig_headway} shows the optimal $m$ and $\Delta m$ calculated with~(\ref{lin2}), where $f^{\text{fea}}_0 = 1/h^{\text{fea}}_0$,
$\underline{T} = (2\underline{T}_0 + \underline{T}_1 + \underline{T}_2)/2 = 84.2$ min, and
$\Delta \underline{T} = 5.8$ min.

\begin{figure}
 \centering
 \frame{
  \includegraphics[width=\textwidth]{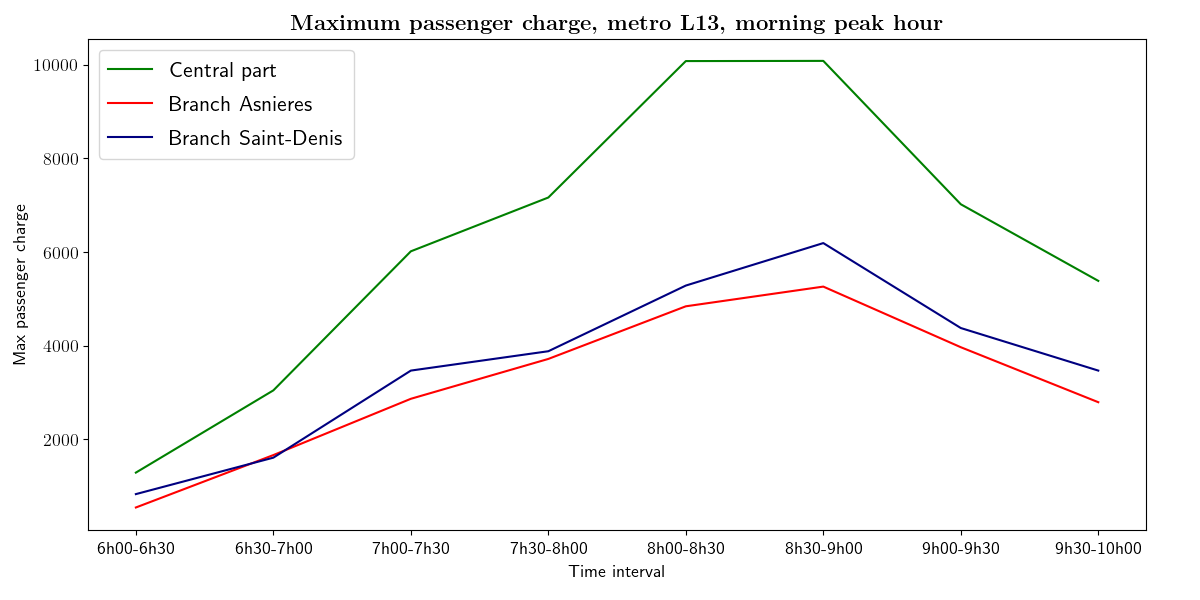} }
\caption {Maximum passenger charge $c_{(u,j)}$ during morning peak hour on metro line 13. Data from RATP.}
\label{fig_demand}
\end{figure}

\begin{figure}
 \centering
 \frame{
  \includegraphics[width=\textwidth]{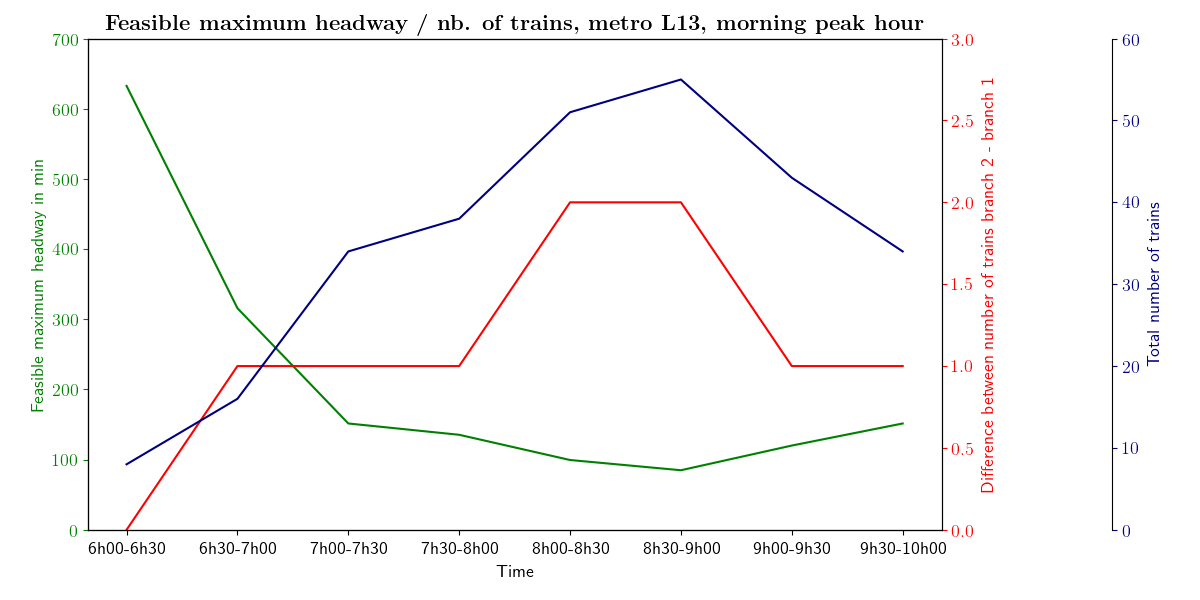} }
\caption {Feasible maximum train time-headway $h^{\text{fea}}_0$ calculated with~(\ref{h_fea}) on RATP metro line 13 Paris (green), total number $m$ of trains (blue),
  and difference $\Delta m$ in the number of trains on the two branches (red).}
\label{fig_headway}
\end{figure}

As traffic remains in the free flow phases, the travel times are the sum of theoretic minimum dwell and run times.
A \textit{Grade of Automation 2} (GOA~2) system implemented on metro line 13, ensures that trains respect the
theoretic minimum run times. Moreover, it is assumed that theoretic dwell times are respected.
With the asymptotic average train time-headway $h_0$ chosen to the feasible time-headway $h^{\text{fea}}_0$, with respect to the passenger demand,
calculated accordingly to~(\ref{h_fea}), and $f^{\text{fea}}_0 = 1 / h^{\text{fea}}_0$, control law~(\ref{lin2}) is applied. The results are given in Table~\ref{table-1} below
with the following control variables: 1) number $m$ of trains on the line (rounded to the nearest integer), and
2) optimal difference $\Delta m$ between the number of trains on the two branches (rounded to the nearest integer).

\begin{table}[h]\label{table-1}
\centering
\caption{Control variables $m, \Delta m$ for optimal set points. Metro line 13, Paris, morning peak hour.}
\begin{tabular}{|l|cccccc|}
\hline
\vspace{2pt}
    		&$h^{\text{req}}_0$ & $f^{\text{req}}_0$ & $h^{\text{fea}}_0$  & $f^{\text{fea}}_0$ & $m$ & $\Delta m$ \\ \hline
\vspace{2pt}
   6.00-6.30 &632.5   & 5.7     &632.5   & 5.7             &  8          &  0\\ \hline
\vspace{2pt}
   6.30-7.00 &315.8   & 11.4     &315.8   & 11.4          &  16          &  1\\ \hline
\vspace{2pt}
   7.00-7.30 &151.5   & 23.8     &151.5   & 23.8           &  34          &  1\\ \hline
\vspace{2pt}
   7.30-8.00 &135.4   & 26.6     &135.4   & 26.6             &  38          &  1\\ \hline
\vspace{2pt}
    8.00-8.30 &99.5   & 36.2   &99.5   & 36.2             &  51          & 2\\ \hline
\vspace{2pt}
   \textcolor{red}{8.30-9.00} &\textcolor{red}{84.4}   & \textcolor{red}{42.4}     & \textcolor{red}{92.6}
          & \textcolor{red}{38.9}            &  \textcolor{red}{55}          &  \textcolor{red}{2}\\ \hline
\vspace{2pt}
    9.00-9.30 &120.1   & 30.0  & 120.1          & 30.0           &  43         & 1\\ \hline
\vspace{2pt}
   9.30-10.00 &151.5   & 23.8     & 151.5           & 23.8        &  34          &  1\\ \hline
\end{tabular}
\label{table-1}
\end{table}

The feasible train time-headway on the central part $h^{\text{fea}}_0$ follows inversely the evolution of the passenger travel demand, see~(\ref{f/rho}),~(\ref{h_fea}).
Moreover, the total number of trains $m$ follows directly the evolution of the passenger travel demand, see~(\ref{lin2}).
Furthermore, it can be seen that the total number $m$ of trains is highly sensitive to a changing set point $f_0$ (with a factor of $\underline{T} = 84.2 \text{ min}$, accordingly to feedback
control law~(\ref{lin2})), whereas $\Delta m$ is less sensitive to $f_0$ (with a factor of $\Delta\underline{T} = 5.8 \text{ min}$, see feedback
control law~(\ref{lin2})).
All the required train frequencies for corresponding passenger demand levels are feasible, except for the one corresponding to
the 8.30-9.00 time period. Indeed, Table~\ref{table-1} shows that for the 8.30-9.00 time period (the highest level of passenger demand,
see Figure~\ref{fig_demand} and Figure~\ref{fig_headway}), the required train frequency responding to the passenger demand level is $f^{\text{req}}_0 = 42.4 \text{ trains/hour}$, while
the feasible one is $f^{\text{fea}}_0 = 38.9 \text{ trains/hour}$.


\section[Macroscopic Control of the Number of Trains on the Branches in Case of a Perturbation]{Macroscopic Control of the Number of Trains on the Branches in Case of a Perturbation on the Branches \sectionmark{Responding to a Perturbation on the Branches}}
\sectionmark{Responding to a Perturbation on the Branches}
\label{mac-control2}


Still consider Paris metro line 13 as shown in Figure~\ref{fig_schema}. The aim here is to control the train passing order at junction. This means a new order which is different to the one-over-two rule is applied temporarily. This can be interesting in case of a perturbation to avoid long train dwell times at junction when forcing to respect the one-over-two rule. Based on the observed train dwell and run times on the line, the train order at the convergence can be optimized. It is shown in the following that firstly, the time to recover the perturbation is reduced, and secondly, undesired dwell times at the convergence before entering the central part are avoided.

This part focuses on the 7.00-7.30 time period during morning peak hour.
The traffic is in the stable steady state, as defined in the corresponding line of Table~\ref{table-1} above.
Consider a minor incident on branch~1 which leads to an extension of train dwell times, and consequently a prolongation of the travel times, on this branch by twice the train time-headway on the central part:
$2h_0 = 303\text{ sec}$.
Figure~\ref{fig_perturb_without} depicts the train trajectories on the line for this case.
Note that points A, B, C, D, E, F, G represent a position on metro line 13 accordingly to Figure~\ref{fig_schema}. Part A-B-C represents branch~1, part C-D-E the central part and part E-F-G branch~2. Trains coming from and going to branch~1 are represented in blue, whereas trains coming from and going to branch~2 have red trajectories. Firstly, it can be seen the average frequency on the central part is twice the one on the branches, this is due to the one-over-two rule. At around $t=10$ minutes, a perturbation occurs near B on branch~1. A number of trains following each other realize significantly longer travel times, which can for example represent a passenger incident. Without control and forcing the nominal one-over-two rule at the junction, trains coming from branch~2 which have not been perturbed, have to wait near point G before entering the central part. The perturbation from branch~1 fully propagates on the central part.

Minor incident means here that the total travel time $\underline{T}$ over the line is assumed not to be seriously affected (it
remains unchanged).
In this case, train regulation on the affected branch and at the junction becomes necessary, in order to allow trains from the unaffected branch to enter
the central part without waiting for the delayed trains from the affected branch.
With $\Delta \underline{T}$ being observed, a decrease of $\Delta \underline{T}$ to $\Delta \underline{T} - 2h_0$ is stated.
Accordingly to~(\ref{lin4}), the new $\Delta m'$ is calculated as follows.
\begin{equation}
   \Delta m' = m(\Delta \underline{T} - 2h_0)/(2\underline{T}) = \Delta m - 1.
\label{sim-1}
\end{equation}

The new $\Delta m' =\Delta m - 1 = 1 - 1 = 0$,
$\Delta m$ is sensitive even to minor perturbations on the 
travel time on the branches.
In fact, $\Delta m$ has decreased by one train.
Therefore, the new optimal set point is located one unit below the old one, with regard to the $y$-axis of Figure~\ref{fig-2D}.
In this case, by
letting pass two subsequent trains, instead of only one train, from the unaffected branch entering the central
part, the new optimal set point will be reached, and the traffic will be well regulated at the convergence.
Note that $m$ does not change here.

Once the incident has terminated,
a diminution of the travel times on branch~1 by $-2h_0$ is observed, corresponding to a change in $\Delta \underline{T}$ of $+2h_0$, with respect to the perturbed situation.
Therefore, a new $\Delta m''$ is calculated as follows.
\begin{equation}
   \Delta m'' = m(\Delta \underline{T})/(2\underline{T}) = \Delta m' + 1 = \Delta m.
   \label{sim-2}
\end{equation}

\begin{sidewaysfigure}[]
 \centering
  \frame{ \includegraphics[width=\textwidth]{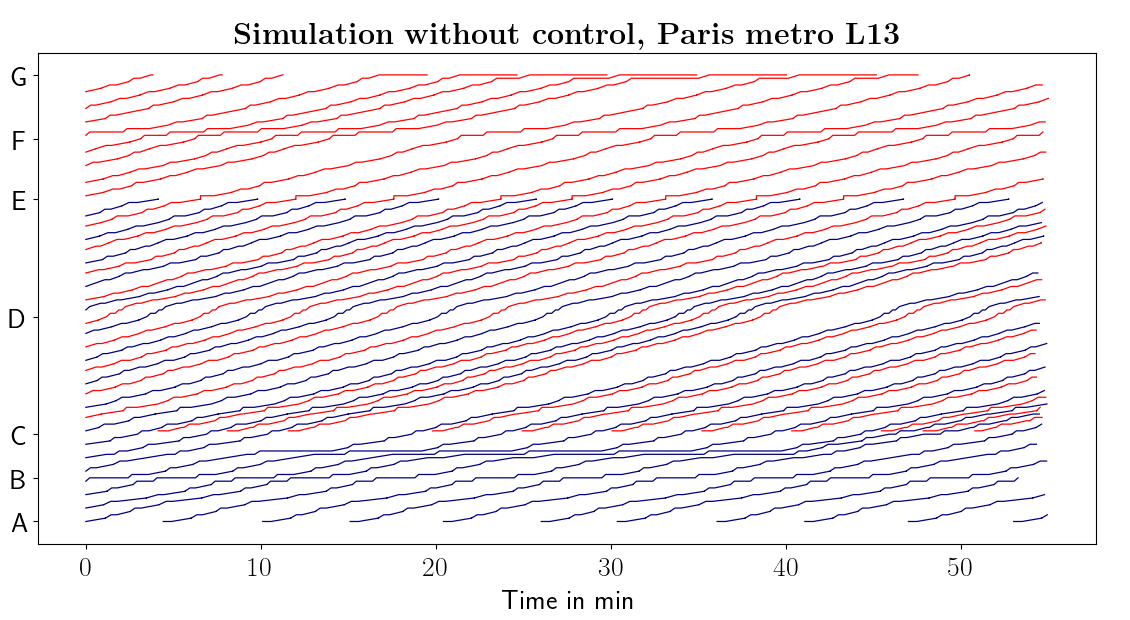}}
  \caption[Simulation results for metro line 13, Paris. A perturbation is injected on branch 1. Case without control.]{Simulation results for metro line 13, Paris. The localization of the points A,B,C,D,E,F and G of the $y$-axis is given in Figure~\ref{fig_schema}.
    A perturbation is injected on branch 1 (between B and C). Here, without
    control, trains from branch~2 have to wait at the convergence to respect the one-over-two rule (point G).}
\label{fig_perturb_without}
\end{sidewaysfigure}

\begin{sidewaysfigure}[]
 \centering
  \frame{ \includegraphics[width=\textwidth]{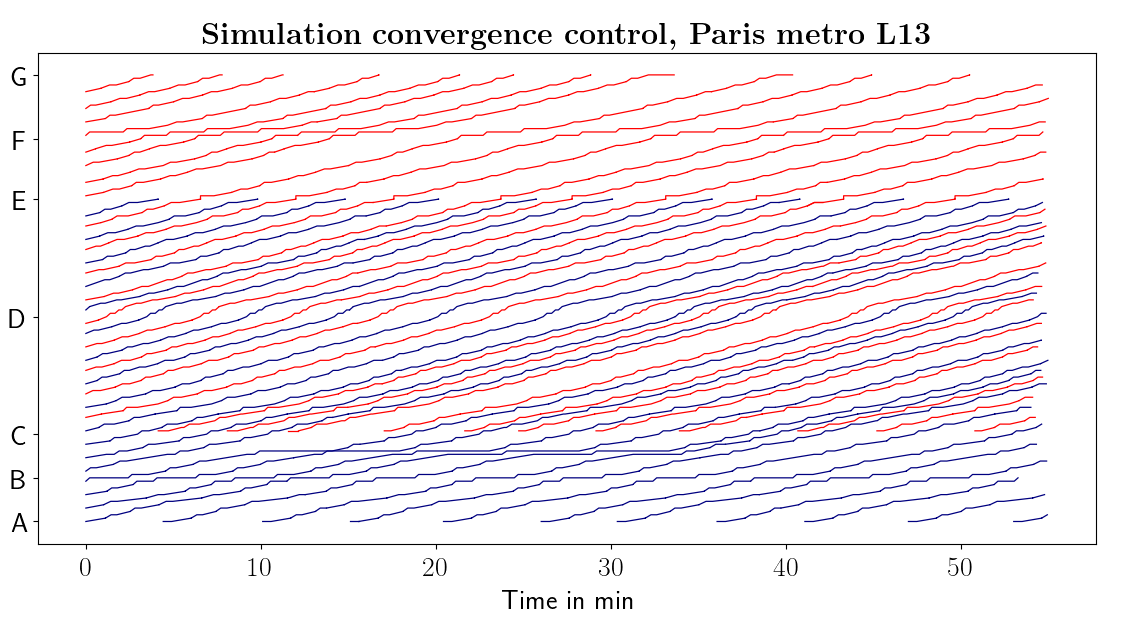}}
  \caption[Simulation results for metro line 13, Paris. A perturbation is injected on branch 1. Case with control.]{Simulation results for metro line 13, Paris. The localization of the points A,B,C,D,E,F and G of the $y$-axis is given in Figure~\ref{fig_schema}.
    A perturbation is injected on branch 1 (between B and C). Here, the train passing order at the convergence is controlled accordingly to~(\ref{sim-1}),~(\ref{sim-2}). The perturbation (at the convergence) disappears.} 
\label{fig_perturb_with}
\end{sidewaysfigure}

Here, the nominal value
$\Delta m'' = \Delta m' + 1 = 0 + 1 = 1$ is re-obtained.
By letting pass two consecutive trains, instead of only one, from the previously affected branch, 
$\Delta m'$ will be increased by one train, such that the initial optimal set point is retrieved.


Figure~\ref{fig_perturb_with} shows the simulation results with traffic control at the convergence.
Note that the initial train positions are the same as in Figure~\ref{fig_perturb_without} and at around $t=10$ minutes, a perturbation occurs near point B on branch~1. A number of consecutive trains realize significantly longer travel times, this can be due to a passenger incident. In the contrary to the simulation without control, the train passing order at convergence is modified. Right after the perturbation has occurred, two consecutive trains from branch~2 enter the central part at point~G. Here, trains from
branch~2 arriving at the convergence (point~G) can directly pass onto the central part, since traffic has
been well regulated at the convergence: At $t=15\text{ min}$ (respectively $t=35\text{ min}$), the train
order has been changed following~(\ref{sim-1}),~(\ref{sim-2}).
This allows, firstly, to reduce the time-headway on the central part compared to the situation in Figure~\ref{fig_perturb_without} and avoids long train dwell times near point G, which means that travel times for passenger from branch~2 are close the nominal value and are significantly shorter than in the situation without control.
When the perturbation on branch~1 disappears, when applying the macroscopic control, two consecutive trains coming from branch~1 enter the central part (point~G). This allows to avoid congestion at the convergence as it can be observed in Figure~\ref{fig_perturb_without}. With the macroscopic control, the traffic is re-stabilized after $35$ minutes, compared to $40$ minutes in the case without control.

\section[Real-time Control of Train Dynamics with Harmonization of the Train Time-headways]{Real-time Control of Train Dynamics with Harmonization of the Train Time-headways \sectionmark{Harmonization of Train Time-headways}}\label{dynamics-control}
\sectionmark{Harmonization of Train Time-headways}
%
\subsection{Controlled Train Dynamics}
In Chapter~\ref{Chap-4}, Section~\ref{acc} a traffic model for linear lines has been presented, where train dwell times are a function of the average passenger arrival rates to and the average passenger departures rates from the platforms and the train time-headway. To ensure that the train dynamics remain stable in the stationary regime, a control is applied on the run times which cancels an extension of the train dwell time in case of a long time-headway.
This model has been adopted to lines with a junction in Section~\ref{cdc18}.
It has been shown that the train dynamics remain stable under conditions.
First the initial train time-headways have to satisfy
\begin{equation}
h^1_j \leq \bar{h}_j = 1/(1-x_j) \bar{g}_j, \forall j,
\end{equation}
which can be written
\begin{equation}\label{stability_cond1}
h^1_j \leq \bar{h}_j = 1/(1-x_j)\underline{g}_j + (1/x_j)\Delta r_j, \forall j.
\end{equation}
Equation~(\ref{stability_cond1}) limits the initial condition to an expression of the passenger travel demand $x_j$, the minimum dynamic interval which is an infrastructure constraint $\underline{g}_j$ and the run time margin chosen $\Delta r_j$.
Note that in case this condition is satisfied, the following condition is automatically satisfied train dynamics on a linear line with demand-dependent dwell times and controlled run times are stable.
\begin{equation}
\Delta w_j \leq \Delta r_j, \forall j,
\end{equation}
which is equivalent to
\begin{equation}\label{stability_cond2}
X_j \Delta g_j \leq r_j, \forall j,
\end{equation}
which means that applying the dwell time equation~(\ref{eq-dwell}), train dwell times $\Delta w_j$ are at maximum extended up to the margin on the train run times $\Delta r_j$:
\begin{equation}\label{eq-dwell-remind}
w^k_j = \min\left\{ x_j h^k_j, \bar{w}_j \right\}, \forall k,j.
\end{equation}
In case~(\ref{stability_cond1}) is not satisfied, the application of~(\ref{stability_cond2}) ensures that dwell times are not extended over the run time margin and train dynamics remain stable. However, this implies that train dwell times cannot be extended to fully serve the passenger travel demand accordingly to dwell time equation~(\ref{eq-dwell-remind}).
The corresponding train dynamics have been proven to be stable, this means the asymptotic average train time-headway in the stationary regime exists.
Note the variance on the train time-headways
\begin{equation}
\Delta h_j := \bar{h}_j - \underline{h}_j.
\end{equation}

Then, the max-plus linear traffic model of Section~\ref{max-plus} guarantees 
stability of the train dynamics and controls both dwell and run times
to take into account the passenger travel demand. The stability is guaranteed under a variance on the train time-headways $\Delta h_j$ within the initial conditions~(\ref{stability_cond1}), and~(\ref{stability_cond2}) however, it does not harmonize the train time-headways on the line.

In Section~\ref{dyn_prog}, the dynamic programming model
from~\cite{FNHL16} has been presented which guarantees stability but controls 
only the dwell times at platforms without controlling the run times. This model is better in term of harmonization
of train departure time intervals, comparing to the max-plus one.
The model in this section extends the max-plus model~\ref{max-plus}
to a dynamic programming model,
as the model of Section~\ref{dyn_prog}.
Consequently, the new model benefits from the advantages of both models.

\begin{itemize}
\item Accounting for the passenger travel demand with (partially) demand-dependent dwell times and a run time control,
\item Harmonization of the train time-headways via a dwell time control without degrading the asymptotic average frequency.
\end{itemize}

In the following, the unchanged train run time control law from~(\ref{eq-run}) is considered, while a new train dwell time control law is proposed, which modifies~(\ref{eq-dwell-remind}) as follows.

\begin{align}
    w^k_j = \min \left\{(1-\gamma_j) x_j h^k_j , \bar{w}_j \right\}, \text{ with } 0 \leq \gamma_j \leq 1.
    \label{dwell-eq}
\end{align}

Notice that for $\gamma_j = 0$, equation~(\ref{dwell-eq}) is equivalent to~(\ref{eq-dwell-remind}) and the train dynamics are max-plus linear.
By activating the control ($1 \geq (1-\gamma_j) > 0$) in case of an extension of the time-headway due to a train delay, excessively long dwell times are limited, which would have been a direct consequence of a long headway, see equation~(\ref{eq-dwell-remind}).

Consider the following notations.
\begin{itemize}
  \item $\Delta w_j := \bar{w}_j - \underline{w}_j$.
  \item $\Delta g_j := \bar{g}_j - \underline{g}_j$.
  \item $\Delta r_j := \tilde{r}_j - \underline{r}_j$.
\end{itemize}
It is then easy to check that $\Delta w_j = 1/(1-x_j) \Delta g_j$.
\begin{proposition}~\label{proposition_1}
   If for all $j$, $h^1_j \leq \bar{h}_j = 1/(1-x_j) \underline{g} + (1/x_j) \Delta r_j$, then $\Delta r_j \geq \Delta w_j$, 
   and then $t^k_j = \tilde{r}_j + X_j \underline{g}_j - \gamma_j X_j g^k_j, \forall j$.
\end{proposition}

\begin{proof}
The proof is by induction. It is similar to the one of Theorem~\ref{th1-acc18} in Section~\ref{max-plus}.
\end{proof}

The train dynamics is then written as follows.
\begin{equation}\label{itsc-dynamics}
    d^k_j = \max \left\{
	\begin{array}{l}
	    (1 - \delta_j) d^{k-b_j}_{j-1} + \delta_j d^{k-1}_j + (1-\delta_j) (\tilde{r} + X_j \underline{g}_j), \\~~\\
	    d^{k-\bar{b}_{j+1}}_{j+1} + \underline{s}_{j+1},\\
	\end{array} \right.
\text{with } \delta_j := (\gamma_j x_j) / (1 + \gamma_j x_j).
\end{equation}

The dynamics presented here~(\ref{itsc-dynamics}) extends the one of Theorem~\ref{th1-acc18},
since in the case where $\delta_j = 0$ ($\gamma_j = 0$),
(\ref{itsc-dynamics}) coincides with the one presented in
Theorem~\ref{th1-acc18}.
This is a direct consequence of the fact that the train dwell time control law~(\ref{dwell-eq})
extends~(\ref{eq-dwell-remind}).

If $m=0$ (zero trains) or if $m=n$ (the metro line is full of trains), then the dynamic system~(\ref{itsc-dynamics}) is fully implicit. 
Indeed, if $m=0$, then $b_j = 0, \forall j$, that is the first term of the maximum operator in~(\ref{itsc-dynamics}) is implicit for every $j$.
Similarly, if $m=n$, then $\bar{b}_j = 0, \forall j$, that is the second term of the maximum operator in~(\ref{itsc-dynamics}) is implicit for every $j$.
In both cases, $m=0$ and $m=n$, no train movement is possible.
On the other side it is not difficult to check that if $0 < m <n$, then the dynamic system is implicit but triangular, that means
there exists an order on $j$ of updating the variables $d^k_j$ in such a way that the system will be explicit.
In fact, this order corresponds to the one of the train movements on the metro line.
In the following, only the case $0<m<n$ is considered. Therefore, the dynamic system~(\ref{itsc-dynamics}) admits an equivalent triangular system.

The train dynamics~(\ref{itsc-dynamics}) can be written as follows.
\begin{align}\label{markov_chain}
    d^k_j = \max_{u \in U} \left\{ \left(M^u d^{k-1}\right)_j + \left(N^u d^k\right)_j + c^u_j \right\}, \forall j,k,
\end{align}
where $M^u, u\in\mathcal U$ and $N^u, u\in \mathcal U$ are two families of square matrices, and $c^u, u\in\mathcal U$ is a family of column vectors.
Moreover, since $0\leq \delta_j < 1, \forall j$ by definition, then $M^u_{ij} \geq 0, \forall u,i,j$ and $N^u_{ij} \geq 0, \forall u,i,j$.
Finally $\sum_j \left(M^u_{ij} + N^u_{ij}\right) = 1, \forall u,i$ holds.

The equivalent triangular system of system~(\ref{itsc-dynamics}) can be written as follows.
\begin{align}\label{markov_chain2}
    d^k_j = \max_{u \in U} \left\{ \left(\tilde{M}^u d^{k-1}\right)_j + \tilde{c}^u_j \right\}, \forall j,k,
\end{align}
where $\tilde{M}^u, u\in\mathcal U$ is a family of square matrices, and $\tilde{c}^u, u\in\mathcal U$ is a family of column vectors,
which can be derived from $M^u, N^u$ and $c^u, u\in\mathcal U$.
Moreover, $\tilde{M}^u_{ij} \geq 0, \forall u,i,j$ and $\sum_j \tilde{M}^u_{ij} = 1, \forall u,i$ is still guaranteed.
By consequent, the system~(\ref{markov_chain2}) can be seen as a dynamic programming system of an optimal control problem of a Markov chain,
whose transition matrices are $\tilde{M}^u, u\in\mathcal U$ associated to every control action $u\in\mathcal U$, and whose 
associated rewards are $\tilde{c}^u, u\in\mathcal U$.

\begin{theorem}~\label{thm-eigenvalue}
    If $0 < m < n$, then the dynamics~(\ref{itsc-dynamics}) admits a stationary regime, with a unique asymptotic average
    growth vector (independent of the initial state vector $d^0$) whose components are all equal to $h$, which is
    interpreted here as the asymptotic average train time-headway.    
\end{theorem}

\begin{proof}
It follows a sketch of the proof here. 
Since the dynamics~(\ref{itsc-dynamics}) is interpreted as the dynamic programming system of a stochastic optimal control problem
of a Markov chain,
it is sufficient to prove that the Markov chain in question, which is acyclic since $0 < m < n$,
is irreducible for every control strategy.
This is not obvious from the dynamics~(\ref{itsc-dynamics}). It has to be shown on the equivalent triangular system~(\ref{markov_chain2}).
An alternative proof is the one of Theorem~5.1 in~\cite{FNHL16}.
\end{proof}

Theorem~\ref{thm-eigenvalue} does not give an analytic formula for the asymptotic average train time-headway $h$.
However, it guarantees its existence and its uniqueness. Therefore, $h$ can be approximated by numerical simulation
based on the value iteration, as follows.
\begin{align}
    h \approx d^K_j/K, \forall j, \text{ for a large } K.
\end{align}

The objective of the model proposed here~(\ref{itsc-dynamics}) which extends the max-plus linear dynamics~(\ref{acc-dynamics}), is to harmonize the train time-headways on the line. 

As mentioned above, if $\gamma_j = 0, \forall j$, the dynamics~(\ref{itsc-dynamics}) are max-plus linear, and are equivalent to
~(\ref{acc-dynamics}).
In this case, if the dynamics~(\ref{itsc-dynamics}) or equivalently~(\ref{acc-dynamics}) are written under the triangular form~(\ref{markov_chain2}), then
the associated matrices $\tilde{M}^u$ are boolean circulant matrices. The latter may then have eigenvalue~$1$ with multiplicity bigger than~$1$.
Therefore, although the average growth rate of the dynamics is unique and independent of the initial state $d^0$, the asymptotic
state $d^k$ (up to an additive constant) may depend on the initial state $d^0$.

However, in case of the dynamics~(\ref{itsc-dynamics}) with $0 < \gamma_j \leq 1$ for some~$j$, the application of $\gamma_j$  will force the dynamics
to converge a stationary regime where the activated matrices $\tilde{M}^u$ are the ones
having eigenvalue~$1$ as a simple one, that is with multiplicity~$1$. In this case, $d^k$ will converge, up to an additive constant,
independent of the initial state $d^0$. Moreover, the asymptotic state $d$ will correspond to the case where the train time-headway
is the same at all the platforms. 

With regard to dwell time equation~(\ref{dwell-eq}), notice that a control $1 > (1-\gamma_j) > 0$ is applied temporarily in case a variance on the train time-headways is observed. During the time interval when the control is applied, the train time-headways at the platforms converge under the asymptotic average train time-headway. Indeed, with~(\ref{dwell-eq}), dwell times are slightly shortened to allow train time-headway harmonization. 
By consequent, during the time period when the control is activated, the passenger travel demand cannot be fully served, this means the train dwell times do not completely account to the accumulation of passenger on the platform (and in the train).
However, the asymptotic average train time-headway is guaranteed, since dwell times are shortened.
Note that the control is deactivated as soon as the time-headways are harmonized and the dwell times correspond again to the passenger travel demand.

On lines with a very high passenger travel demand, as on the RATP network in Paris, the here proposed control can be more interesting compared to a classic holding control. It allows a train time-headway harmonization under a required time-headway at the cost of temporarily shortened dwell times, whereas the holding control is at the cost of an increased asymptotic average train time-headway which might lead to serious passenger accumulation on the platforms and overcrowded trains.

In the next section, by means of numerical simulation, it will be shown that in this case,
the train dynamics converge to a traffic state where the train time-headways are harmonized.
In other words, while the train time-headway converges to its asymptotic value $h$, the variance of the train time-headways,
that is its deviation with respect to its asymptotic average value $h$, converges to zero.

\subsection{{Simulation Results}}

\begin{sidewaysfigure}[h]
	\centering
	\fbox{\includegraphics[width=\textwidth]{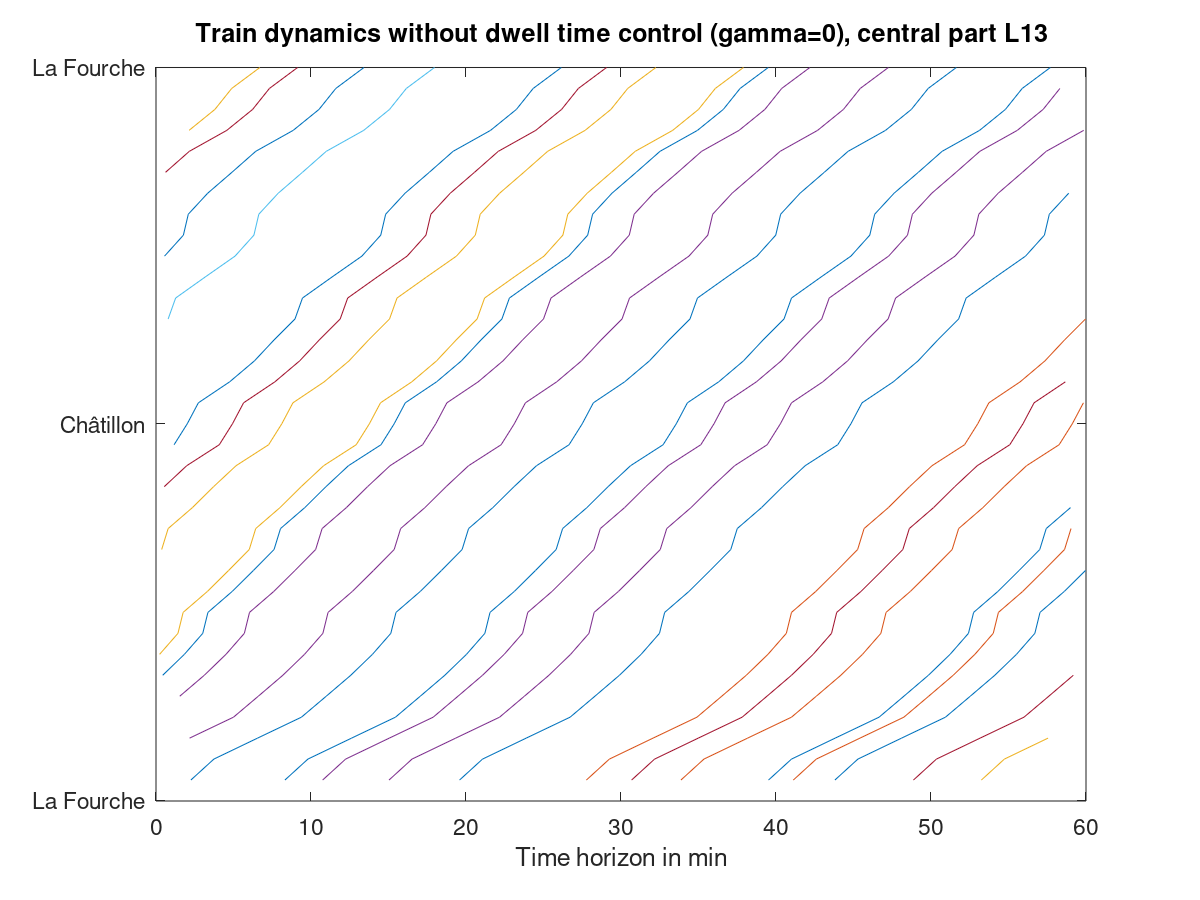}}
	\caption[Simulation of train trajectories with perturbed initial train time-headways, central part of Paris metro line 13. Max-plus linear train dynamics.]{Simulation of train trajectories with perturbed initial train time-headways, central part of Paris metro line 13. Max-plus linear train dynamics ($\gamma_j = 0$)  accordingly to~(\ref{itsc-dynamics}), with demand-dependent dwell times, run time margin $15\%$.}
	\label{its_gamma_0}
\end{sidewaysfigure}

\begin{sidewaysfigure}[h]
	\fbox{\includegraphics[width=\textwidth]{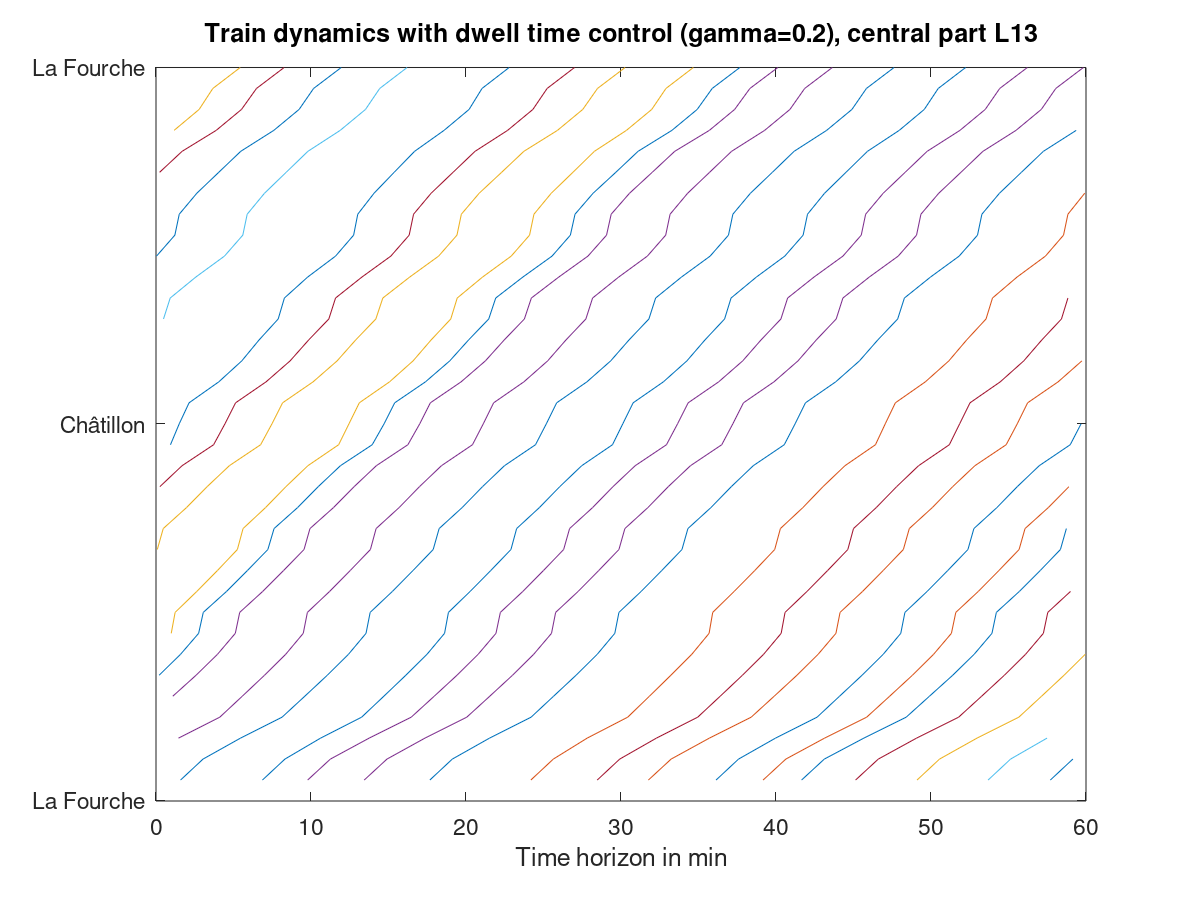}}
	\caption[Simulation of train trajectories with perturbed initial train time-headways, central part of Paris metro line 13. Train dynamics with dwell time control ($\gamma_j = 0.2$).]{Simulation of train trajectories with perturbed initial train time-headways, central part of Paris metro line 13. Train dynamics with dwell time control ($\gamma_j = 0.2$) accordingly to~(\ref{itsc-dynamics}), run time margin $15\%$.}
	\label{its_gamma_02}
\end{sidewaysfigure}

\begin{sidewaysfigure}[h]
	\fbox{\includegraphics[width=\textwidth]{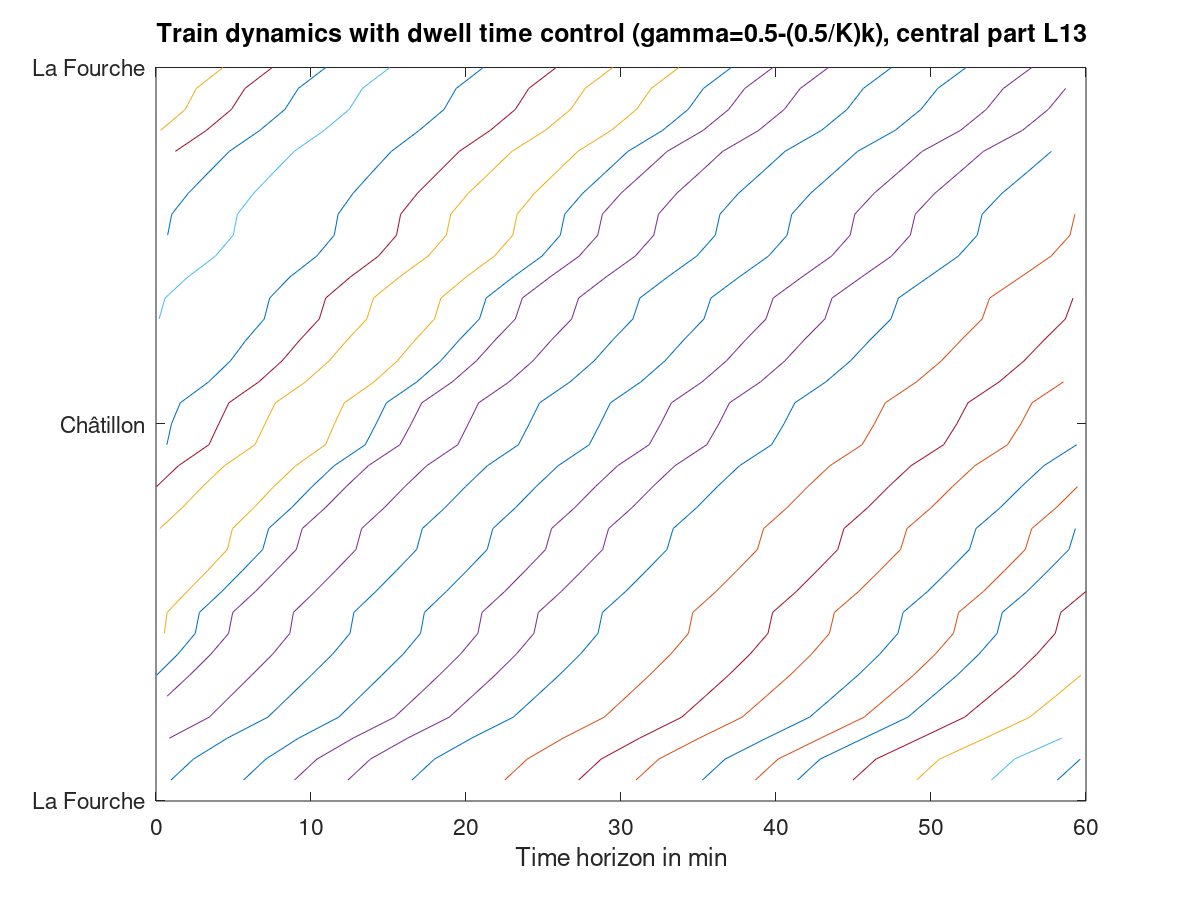}}
	\caption[Simulation of train trajectories with perturbed initial train time-headways, central part of Paris metro line 13. Train dynamics with dwell time control ($\gamma_j = 0.5-(0.5/K)k$).]{Simulation of train trajectories with perturbed initial train time-headways, central part of Paris metro line 13. Train dynamics with dwell time control ($\gamma_j = 0.5-(0.5/K)k$) accordingly to~(\ref{itsc-dynamics}), run time margin $15\%$.}
	\label{its_gamma_05}
\end{sidewaysfigure}



The following figures~\ref{its_gamma_0},~\ref{its_gamma_02},~\ref{its_gamma_05} present simulation results of the train trajectories for the train dynamics of equation~(\ref{itsc-dynamics}) for different values of parameter $\gamma_j$.
First, the max-plus linear dynamics are depicted, this means $\gamma_j = 0$.
Second, the control is fixed to $\gamma_j = 0.2$, which is close to the max-plus linear dynamics.
Third, with $\gamma_j = 0.5 - (0.5/K)k$, the control degrades linearly from $0.5$ to $0$, over the simulation horizon $K$.
The simulation horizon $K$ is fixed accordingly to the time span, after which the train positions should be harmonized. 
It has been chosen to $K = 15$ departures here.

Figure~\ref{its_gamma_0} depicts the train trajectories on the central part of Paris metro line 13 over a time horizon of one hour. The trajectories are given by the max-plus linear train dynamics with $\gamma_j = 0$. In the initial state ($t = 0$) the initial train time-headways are perturbed. For example, at the station \textit{Châtillon -- Montrouge}, the first time-headways which can be observed (starting with the blue trajectory) are $h^1 = 175 \text{ sec}, h^2 = 191 \text{ sec}, h^3 = 339 \text{ sec}, h^4 = 96 \text{ sec}$.
In this case train dwell times are a function of the passenger arrival rates to and the passenger departure rates from the platforms, and the train time-headway such that they are extended for trains with a long headway. In the example, this is the case for the second yellow train with an initial headway of $h^3 = 339 \text{ sec}$. Note that this train is also accelerated in the inter-station to cancel the effect of the extended dwell time. Since the trajectories in the figures below show only the train departures $d^k_j$ from the platforms, this detail is not visible in the graph. The corresponding train dynamics have been proven to be stable in Section~\ref{max-plus} and to reach a stationary regime with an asymptotic average train time-headway. Figure~\ref{its_gamma_0} strongly suggests that this asymptotic average exists since train travel times and the number of trains are constant. Finally it can be seen that without dwell time control, the variance between the train time-headways does not disappear.

Figure~\ref{its_gamma_02} gives the train trajectories on the central part of line 13 for the dynamics accordingly to~(\ref{itsc-dynamics}) with $\gamma_j = 0.2$. The initial condition, that is the train positions and the initial headways are the same as before. In this case, the control on the dwell times is activated. Comparing the figure to the case of the max-plus linear dynamics, it can be seen that the time-headways converge over the horizon. The long initial headways disappear over time.
At the example of the second yellow train with the initial headway of $h^3 = 339 \text{ sec}$, the effect of the reduced dwell times can be studied. At the station La Fourche, the headways of this train with respect to its predecessor and it successor are nearly harmonized. Furthermore, it can be seen that the headway harmonization control does not lead to a decreasing frequency. In the contrary, at the end of the horizon at $t = 60 \text{ min}$ at first station, \textit{La Fourche}, the number of departures realized within the time horizon is increased by $+1$. In case of perturbation an accumulation of passengers on the platforms can typically be observed. With the here presented control, the asymptotic average frequency is increased while harmonizing the train time-headways.

Figure~\ref{its_gamma_05} depicts the train trajectories on the same line and still accordingly to~(\ref{itsc-dynamics}) for $\gamma_j = 0.5 - (0.5/K)k$ which means the control decreases linearly over the time horizon. In the beginning, the control applied is stronger which allows to harmonize the train time-headways more rapidly and therefore, to decrease the control linearly so that it reaches $\gamma_j = 0$ at the end of the horizon. It can be seen that the time-headways are harmonized faster. For example the first train (blue trajectory) at the station \textit{Châtillon -- Montrouge} has, at the end of the line at \textit{La Fourche}, nearly equalized time-headways with regard to its predecessor and its successor. Comparing this to Figure~\ref{its_gamma_02} with a less strong control $\gamma_j = 0.2$, it can be seen that the harmonization is much quicker.
Once a perturbation on the train time-headways is observed, it can be interesting to temporarily apply a strong control so that the perturbation rapidly disappears.

\chapter{Conclusion}\label{conclusion}
\section*{Summary}
This thesis is on metro traffic modeling and control. 
The train dynamics of a metro line with a junction are modeled with a discrete event traffic model. 
The discrete event model is based on a modeling approach first developed by Farhi et al.~\cite{FNHL16} for a linear line and it represents a novel model of railway traffic. The discrete events considered here are the train departures from the nodes of the segments on the line, corresponding to the signaling blocks. The train dynamics are stated by two constraints. The first one is on train run and dwell times. The second one is on the safe separation time at each segment. 

Original with respect to existing research is, first of all, 
the modeling of the train dynamics of a metro line with an one-over-two operated junction as a discrete event system.
In Chapter~\ref{1A}, a first model is proposed where lower bounds are considered on the train run, dwell and safe separation times on the segments. 
The traffic model for a line with a junction is written linearly in the max-plus algebra. By application of the max-plus algebra theory, it is shown that the traffic on a line with a junction reaches a stationary regime with an asymptotic average growth rate, interpreted as the asymptotic average train time-headway. The main result is the closed-form solution of the asymptotic average train time-headway as a function of the train run, dwell, and safe separation times, of the number of trains and of the difference between the number of trains on the branches. It is shown that in the free flow (and in the congestion) phase, and for a fixed number of trains, there exists an optimal difference between the number of trains on the branches, maximizing the asymptotic average frequency.
From this result, eight traffic phases of the train dynamics are derived analytically. These are two phases for free flow, two phases representing congested branches, two
phases representing congestion on the central part of the line, a phase for the train line capacity (maximum frequency), and a zero frequency phase.
Based on the closed-form solutions of the traffic phases, three macroscopic control laws are derived. They allow to control the total number of trains as well as the difference between the number of trains on the branches, in cases of: 1) a changing passenger travel demand volume, 2) a perturbation on the travel times for a fixed number of trains, and 3) a perturbation on the travel times where trains can be inserted or canceled. This macroscopic control is interesting for lines with a junction since it allows to control the train passing order at the convergence by following the optimal difference between the number of trains on the branches and in the case where time margins included for microscopic control are insufficient to re-stabilize the traffic.

In Chapter~\ref{Chap-4}, an extended model is proposed, where the lower bounds on the  train run and dwell times are replaced by functions of the passenger travel demand and
the train time-headway, allowing a microscopic control of the train dynamics. Train dwell times are extended for trains with long time-headways in order to take into account the accumulation of passengers on the platforms. On the other side, and for those trains with long time-headways, the train run times are shortened by accelerating the trains in the upcoming inter-stations, in order to compensate the extension of the train dwell times. 
The model remains max-plus linear and the train dynamics reach a stable stationary regime with an asymptotic average train time-headway. Closed-form solutions of the asymptotic average train time-headway are derived depending, in addition to the parameters of the first model, furthermore on the passenger travel demand and on the run time margin. The traffic phases of the train dynamics are then derived and allow the analysis of the effect of the passenger demand and of the train run time margins on the asymptotic average train frequency.
For the case of a linear line, three traffic phases are distinguished: a free flow phase, a maximum frequency phase and train congestion phase. For the case of a line with a junction, eight traffic phases are distinguished, as in the standard model. 

Finally, three simulation cases at the example of the metro line 13 of Paris (with one junction) are presented in Chapter~\ref{sim}. The first case illustrates the macroscopic control on the number of trains depending on the passenger travel demand volume. In this case, the train dynamics do not change and therefore there is no need for simulation of the train dynamics. The computed control actions can directly be applied, and the stability of the system is proved. 
The second case shows the macroscopic control of the number of trains on the branches in case of a perturbation of the train travel times. Since the train dynamics are perturbed here, they are simulated, and the results with and without control are compared.
The third case gives a simulation of the demand-dependent train dynamics with perturbed initial time-headways.

\section*{Outreach and Limitations}
The thesis brings about several contributions to railway traffic modeling and control.
Their applications in planning and in operations, their limits and perspectives are discussed below.

First of all, thanks to the novel modeling approach, fundamental diagrams for metro lines with a junction have been derived. They can serve for planing issues, for example the timetable construction or the capacity analysis of existing or planed metro lines. The diagrams which depict the relationships between the macroscopic variables: Total number of trains and the difference between the number of trains on the two branches, and the asymptotic average train frequency, depending on the train run, dwell and safe separation times, allow to derive an optimal train configuration with respect to a requested train frequency.

% ...
Moreover, from the underlying max-plus linear traffic model, control laws have been derived, which allow real-time macroscopic control of the number of trains and of the difference between the number of trains on the branches, depending on the train 
run, dwell and safe separation times.
The macroscopic control of the difference between the number of trains on the branches can be realized with a control of the train passing order at the junction. The control is optimal with regard to the asymptotic average train frequency in the stationary regime. 
More attention needs to be given to the control at the junction in the transient traffic regime.

Moreover, in an extended version of the traffic model, a dynamic control of the train dwell times, accounting for a possible passenger accumulation on the platform in case of a long time-headway, and a dynamic control of the train run times guaranteeing the system stability has been proposed. Both have direct applications in metro operations within a margin on the train run times.
The dynamic control guarantees system stability combined with demand-dependent dwell times in case of a perturbation on the train time-headways. However, it is limited by the margin on the train run times, such that in case a perturbation exceeds the margin, the maximum dwell time is applied and optimality with regard to the passenger travel demand is no longer guaranteed. The dynamic dwell time does not take into account the train capacity.
A model of the train capacity will enhance the field of applications, for example in case of important perturbations.

A final possible application in metro operations is the dynamic dwell time control, reducing the dwell time for trains with a long time-headway. It leads to the harmonization of the train time-headways on the line, see Chapter~\ref{sim}.
Based on the dynamic dwell time model Chapter~\ref{Chap-4}, this control guarantees the harmonization of the time-headways within a limit on the longest time-headway, depending on the run time margin chosen. Moreover, reducing train dwell times in case of a long time-headway is negative from a passenger point of view and needs technical installations which guarantee its applicability. A possible solution is an automatic closing of platform doors and trains doors.
Its combination with a holding control could ensure headway harmonization even in case of important perturbations.

The railway system considered here is a metro line. This allows to simplify the description of the system at some points, compared to a full railway, for example with regard to a simplified block and signaling system. 
For example, constant train run times do not represent acceleration and braking in case a train is affected by the signaling system.
To allow the application of the model to full railway systems, for example RER lines (Paris), suburban lines or high speed lines, a more specific model of the train dynamics is necessary.

The metro line modeled here has one junction. In the steady state, an one-over-two operations rule is applied at the divergence and at the convergence.
Many metro lines with a junction are operated accordingly to the one-over-two rule because it theoretically ensures constant time-headways between trains on both branches. However, different rules exist and can be modeled with the approach of this thesis.

With regard to full railway systems, more complex configurations with several branches and junctions exit. These configurations can also be modeled using the here presented approach. However, the representation of the traffic phases of the train dynamics is difficult for more complex lines. The fundamental diagram of a line with a central part and two branches derived here has three dimensions. Consequently, the fundamental diagram of a line with more than two branches has more than three dimensions which makes the graphical interpretation of the phase diagrams more complex.

Finally, trains are supposed to stop at all the platforms and the number of trains in the system is supposed to be constant over a certain period in time.
Using the same modeling approach, an extension representing the operation of express trains which skip stops is possible.

\section*{Perspectives}
As highlighted above, this thesis has contributed to the theoretical understanding of the physics of metro traffic on lines with a junction. The theory developed has furthermore direct applications in traffic planing and real-time control.
On the offline traffic planing side, the fundamental diagrams of this thesis have demonstrated that two variables have to be optimized to obtain a desired frequency with regard to the system parameters run times, run time margin, passenger travel demand and safe separations times: The number of trains and the difference between the number of trains on the branches.
The theory can be developed further by, first, considering systems with a junction which is operated with a rule different from one-over-two. Second, modeling systems with more than two branches, and third, modeling the effect of a skip-stop policy.

On the online real-time control side, the model proposes a dwell time accounting for passenger accumulation on the platform and controlled run times to guarantee traffic stability within a run time margin. It is necessary to further develop the control for the transient regime. In case of a serious perturbation, an efficient control for the transient regime is required in order to reach the stationary regime in a reasonable time.

\backmatter
\bibliographystyle{abbrv}
\bibliography{bibliography} 
\end{document}